\documentclass[11pt]{article}

\usepackage{amssymb}
\usepackage{amsmath}
\usepackage{multirow}
\usepackage{tikz}
\usetikzlibrary{bayesnet}
\usepackage[]{algorithm}

\makeatletter

\usepackage{longtable}
\usepackage{amsmath,amsbsy,amsgen,amscd,amsthm,amsfonts,amssymb} 

\usepackage{enumerate}
\usepackage{subcaption}
\usepackage{algcompatible}
\usepackage{graphicx} 

\counterwithin{table}{section}
\counterwithin{algorithm}{section}
\usepackage{amsmath,cases}
\long\def\acks#1{\vskip 0.3in\noindent{\large\bf Acknowledgments and Disclosure of Funding}\vskip 0.2in
\noindent #1}

\usepackage{enumitem}
\usepackage{enumerate}
\usepackage{bookmark}
\usepackage{float}
\usepackage{tikz}
\usetikzlibrary{calc}
\newcommand{\E}{\mathbb{E}} 
\newcommand{\R}{\mathbb{R}} 
\renewcommand{\P}{\mathbb{P}} 

\newcommand{\D}{\mathcal{D}}
\newcommand{\G}{\mathcal{G}}
\newcommand{\Y}{\mathcal{Y}}

\newcommand{\Z}{\mathcal{Z}}
\newcommand{\T}{\mathcal{T}}
\newcommand{\Q}{\mathcal{Q}}
\newcommand{\X}{\mathcal{X}}
\newcommand{\vnu}{\boldsymbol{\nu}}
\newcommand{\veta}{\boldsymbol{\eta}}
\newcommand{\vtheta}{\boldsymbol{\theta}}
\newcommand{\vbeta}{\boldsymbol{\beta}}

\newcommand{\mSigma}{\boldsymbol{\Sigma}}
\newcommand{\mOmega}{\boldsymbol{\Omega}}

\newcommand{\mY}{\mathbf{Y}}
\newcommand{\mA}{\mathbf{A}}
\newcommand{\mQ}{\mathbf{Q}}
\newcommand{\mX}{\mathbf{X}}
\newcommand{\mD}{\mathbf{D}}
\newcommand{\mW}{\mathbf{W}}
\newcommand{\mv}{\mathbf{v}}
\newcommand{\mB}{\mathbf{B}}
\newcommand{\mV}{\mathbf{V}}
\newcommand{\mU}{\mathbf{U}}
\newcommand{\mO}{\mathbf{O}}
\newcommand{\mS}{\mathbf{S}}
\newcommand{\mL}{\mathbf{L}}
\newcommand{\mH}{\mathbf{H}}
\newcommand{\mZ}{\mathbf{Z}}
\newcommand{\mT}{\mathbf{T}}

\usepackage{amsmath}

\newtheorem{theorem}{Theorem}[section]

\newtheorem{lemma}{Lemma}[section]

\newtheorem{proposition}[theorem]{Proposition} 
\newtheorem{remark}{Remark}[section]
\newtheorem{corollary}{Corollary}[theorem]
\newtheorem{definition}[theorem]{Definition}

\newtheorem{assumption}{Assumption}[section]
\hypersetup{colorlinks, citecolor=blue
}
\usepackage[affil-it]{authblk}

\title{Tensor Topic Modeling Via HOSVD}

\author[1]{Yating Liu\thanks{yatingliu@uchicago.edu}}
\author[1]{Claire Donnat\thanks{cdonnat@uchicago.edu}}
\affil[1]{Department of Statistics,  University of Chicago}

\date{}
\begin{document}
\maketitle
\begin{abstract}

By representing documents as mixtures of topics, topic modeling has allowed the successful analysis of datasets across a wide spectrum of applications ranging from ecology to genetics. An important body of recent work has demonstrated the computational and statistical efficiency of probabilistic Latent Semantic Indexing (pLSI)-- a type of topic modeling--- in estimating both the topic matrix (corresponding to  distributions over word frequencies), and the topic assignment matrix.
However, these methods are not easily extendable to the incorporation of additional temporal, spatial, or document-specific information, thereby potentially neglecting useful information in the analysis of spatial or longitudinal datasets that can be represented as tensors. Consequently, in this paper, we propose using a modified higher-order singular value decomposition (HOSVD) to estimate topic models based on a Tucker decomposition, thus accommodating the complexity of tensor data. 
Our method exploits the strength of tensor decomposition in reducing data to lower-dimensional spaces and successfully recovers lower-rank topic and cluster structures, as well as a core tensor that highlights interactions among latent factors. 
We further characterize explicitly the convergence rate of our method in entry-wise $\ell_1$ norm. Experiments on synthetic data demonstrate the statistical efficiency of our method and its ability to better capture patterns across multiple dimensions. Additionally, our approach also performs well when applied to large datasets of research abstracts and in the analysis of vaginal microbiome data.
\vspace{5mm}
\noindent


\vspace{7mm}\noindent 

\end{abstract}

\newpage
\section{Introduction}\label{sec: introduction}

Consider a corpus of $N$ documents, each consisting of $M$ words taking values in a dictionary of size $R$. From a mathematical perspective, this data can be summarized in a document-term matrix $\mY \in \R^{N \times R}$ whose $(i,r)^{th}$ entry  records the number of times word $r$ has occurred in document $i$. The objective of topic modeling is to extract a set of $K$ topics from the matrix $\mY$ (usually encoded as a matrix $\mA \in \R^{K \times R}$, where rows correspond to different topics), as well as a topic assignment matrix $\mW \in \R^{N \times K}$ that captures the topic proportions present in each document.  In this setting, topics are understood as distinct distributions over word frequencies. It is customary to assume that each row of the matrix $\mY $ follows a multinomial distribution:
\begin{equation}\label{eq:plsi}
    \mY_{i \cdot} \sim \text{Multinomial}(M,\{\sum_{k=1}^K \mW_{ik} \mA_{kr}\}_{r=1}^R)
\end{equation}
where $\{\sum_{k=1}^K \mW_{ik} \mA_{kr}\}_{r=1}^R$ is a $R$-dimensional vector representing the expected frequency of each word $r$ in document $i$ as a mixture of $K$ different distributions (or topics), whose weights are denoted as $\mW_{ik}$. The rows of $\mY$ are here sampled independently of one another, with no explicit relationship between documents.

Since its introduction, topic modeling has proven to be a useful unsupervised learning technique for extracting latent structure from document corpora and allowing the subsequent clustering of documents according to their topic proportions \cite{lee2012ivisclustering,xie2013integrating, yau2014clustering,zhou2024much}. Beyond the text analysis sphere, topic modeling has been successfully applied to the analysis of count data across multiple applications, including image processing \cite{li2010building,sharma2015image,tu2016topic}, social network analyses \cite{cha2012social,curiskis2020evaluation}, as well as genetics \cite{kho2017novel,pritchard2000inference,yang2019characterizing,liu2016overview}, and microbiome studies \cite{okui2020bayesian,sankaran2019latent,symul2023sub} where topics are interpreted as networks of interacting genes or communities of co-abundant bacteria, respectively. 

From a methods perspective, topic modeling approaches can broadly be classified as belonging to one of two classes: Bayesian approaches, mostly expanding on the work by Blei et al. \cite{blei2003latent}, and frequentist approaches. Amongst the latter, the probabilistic Latent Semantic Indexing (pLSI) model \cite{hofmann1999probabilistic} has recently benefited from an increased amount of interest \cite{ke2022using, klopp2021assigning} in the statistics community, as it allows the development of fast, statistically sound methods for inferring both $\mA$ and $\mW$. 

However, in many applications, the data often includes additional dimensions, such as temporal or spatial attributes. In these settings, the data is best represented as a tensor rather than a matrix to reflect the natural structure of the data.
Instances of this setting occur in a wide number of applications, including:
\begin{itemize}
    \item \textit{Example 1  (text data): analysis of scientific reviews:} Consider the analysis of a set of reviews written by $N^{(1)}$ reviewers on $N^{(2)}$ academic papers. In this setting, it might be desirable to simultaneously cluster reviewers according to their \textit{reviewer types} (e.g. favorable or unfavorable to a given research topic, focused on theoretical considerations or the applicability of the method, etc.), and papers according to their topic proportions \cite{zheng2016topic}. 
    \item \textit{Example 2 (marketing): ``Market-Basket'' analysis:}  Consider a set of $N^{(1)}$ customers (the ``documents''), and their purchases of various items (the ``words'') through time (e.g. months or seasons). Tensor topic modeling can here be used to analyze purchasing patterns over time, identifying latent purchasing behaviors that correlate product categories with customer segments and temporal/seasonal variations (see Section~\ref{sec: real data experiments}).
    \item \textit{Example 3 (biology): longitudinal microbiome studies:} topic modeling has been proven to be a valuable tool for the analysis of microbiome samples\cite{chen2012estimating,movassagh2021vaginal,okui2020bayesian}. In this setting, microbime samples from a given patient can be interpreted as documents, with words corresponding to bacteria types, and topics corresponding to communities of co-abundant bacteria \cite{sankaran2019latent}. In many settings, it is not unusual to have samples from the same subject collected across time \cite{dethlefsen2011incomplete,symul2023sub,romero2014composition}, or across different sampling sites (e.g. vaginal, saliva samples and gut samples, or samples from different parts of the gastro-intestinal tract \cite{shalon2023profiling}). In this case, inferring clusters of patients, time points, as well as communities of bacteria, and understanding the interactions between these latent structures is important for gaining deeper biological insights. (see Section~\ref{sec: real data experiments}).
\end{itemize}

In this paper, we thus propose the first adaptation of pLSI to \textbf{t}ensor \textbf{t}opic \textbf{m}odeling (\textbf{TTM}), enabling us to appropriately leverage tensors' inherent structure to efficiently estimate latent factors across all dimensions of the input tensor matrix.  We formalize the description of our method for the analysis of text documents as described in Example 1 due to its broader ease of interpretation, but, as we will show in Section~\ref{sec: real data experiments}, our method is broadly applicable to the analysis of tensor count data at large. Let us therefore suppose that we observe $N^{(1)}\times N^{(2)}$ documents written by $N^{(1)}$ reviewers for $N^{(2)}$ papers using a vocabulary of $R$ words. For each $i\in[N^{(1)}]$ and $j\in[N^{(2)}]$, let $M_{ij}$ denote the length of document (or review) $(i,j)$. For all $i\in[N^{(1)}]$, $j\in[N^{(2)}]$ and $r\in[R]$, we observe the empirical frequency of each word in each document, defined as:
$$
\mathcal{Y}_{ijr}=\frac{\text{number of times word } r \text{ has appeared in reviewer $i$'s review of paper $j$}}{M_{ij}}.
$$
By analogy with the original topic model framework (Equation \eqref{eq:plsi}), we assume that each document $(i,j)$  is sampled independently, so that the entries $\{M_{ij}\mathcal{Y}_{ij\star}\}_{i\in[N^{(1)}], j\in[N^{(2)}]}$ are sampled independently as: 
\begin{align}\label{eq: model multinominal}
M_{ij}\mathcal{Y}_{ij\star}\sim\text{Multinomial}\left(M_{ij},\mathcal{D}_{ij\star}\right)
\end{align}
for some tensor $\mathcal{D}\in\mathbb{R}^{N^{(1)}\times N^{(2)}\times R}$ representing the expected word frequencies in each document: $\mathbb{E}(\mathcal{Y})=\mathcal{D}$. Letting $\mathcal{Z}:=\mathcal{Y}-\mathcal{D}$, the observed word frequencies $\mathcal{Y}$ can thus be written in a ``signal plus noise" form:
\begin{equation}\label{eq:signal_plus_noise}
\mathcal{Y}=\mathcal{D}+\mathcal{Z}.    
\end{equation}

In the context of tensor topic modeling, each entry of $\D$ is assumed to be non-negative and to sum up to 1 along the ``vocabulary axis''-- that is, for each document $(i,j)$, $\sum_{r=1}^R \D_{ijr} =1$.
As an analogue of the matrix pLSI model,  tensor topic modeling (TTM) imposes a low-rank structure on $\mathcal{D}$. In this paper, we propose studying TTM under a nonnegative Tucker decomposition (NTD) framework \cite{kim2007nonnegative}, which decomposes the tensor $\mathcal{D}$ as the product of a smaller core tensor $\mathcal{G}\in\mathbb{R}^{K^{(1)}\times K^{(2)}\times K^{(3)}}$ and three factor matrices $\mA^{(1)}\in\mathbb{R}^{N^{(1)}\times K^{(1)}}$, $\mA^{(2)}\in\mathbb{R}^{N^{(2)}\times K^{(2)}}$ and $\mA^{(3)}\in\mathbb{R}^{R\times K^{(3)}}$. In this setting, the order-3 tensor $\mathcal{D}$ is said to admit a rank-$(K^{(1)}, K^{(2)}, K^{(3)})$ Tucker decomposition if it can be written as:
\begin{align}\label{eq: TTM}
    \mathcal{D}=\mathcal{G} \cdot (\mA^{(1)},\mA^{(2)},\mA^{(3)}),
\end{align}
where $\cdot$ denote the tensor multiplication.
In element-wise form, this model is equivalent to assuming that each entry $\D_{ijr}$ can be re-written as: 
\begin{align}\label{eq: model defined entry-wise}
\D_{ijr}=\sum_{k^{(1)}=1}^{K^{(1)}} \sum_{k^{(2)}=1}^{K^{(2)}} \sum_{k^{(3)}=1}^{K^{(3)}} \mathcal{G}_{k^{(1)}k^{(2)}k^{(3)}} \mA^{(1)}_{ik^{(1)}}\mA^{(2)}_{jk^{(2)}}\mA^{(3)}_{rk^{(3)}}
\end{align}

Entries of the latent factors and the core tensor of $\D$ are also assumed to be nonnegative, and can be interpreted either as probabilities or mixed membership proportions.
In particular, under model~\eqref{eq: model defined entry-wise}, each reviewer $i$ is assumed to belong to a mixture of reviewer types, as indicated by the row  $\mA^{(1)}_{i\star}$ -- thereby allowing a mixed membership assignment of reviewers to clusters. Similarly, $\mA^{(2)}_{jk^{(2)}}$ represents paper $j$'s soft membership assignment to \textit{paper category} $k^{(2)}$.
Finally, each column $\mA^{(3)}_{\star k^{(3)}}$ corresponds to a probability distribution on the word frequencies, thereby defining \textit{topic} $k^{(3)}$. We thus refer to the latent space $\mA^{(1)}$ as the \textit{reviewer types}, the latent space $\mA^{(2)}$ as the \textit{paper category} space, and $\mA^{(3)}$ as the latent \textit{topic} space for words, as is usually done in matrix topic models \cite{blei2003latent}. 
By contrast, the core tensor $\G$ captures the interactions between latent topics and clusters, so that $\G_{k^{(1)}k^{(2)}k^{(3)}}$ can be understood as the probability of observing topic $k^{(3)}$ in a review authored by an author of type $k^{(1)}$ in a paper of type $k^{(2)}$.  Therefore, \eqref{eq: model defined entry-wise} can be understood as enforcing the following law of total probabilities for any word $r$ in document $(i,j)$:
\begin{align*}
    \mathbb{P}\left(\text{word }r | \text{ reviewer } i, \text{ paper } j\right)=\sum_{k^{(1)},k^{(2)},k^{(3)}} &\mathbb{P}\left(\text{topic }k^{(3)} | \text{ reviewer type } k^{(1)}, \text{ paper category } k^{(2)}\right)\\
    &\times\mathbb{P}\left(\text{reviewer type }k^{(1)} | \text{ reviewer }i\right)\\&\times\mathbb{P}\left(\text{paper category }k^{(2)} | \text{ paper }j\right)\\&\times\mathbb{P}\left(\text{word }r| \text{ topic }k^{(3)} \right)
\end{align*}
These components satisfy:
   \begin{equation}\label{eq:constraints A}
       \begin{split}
        \sum_{k^{(1)}=1}^{K^{(1)}} \mA^{(1)}_{ik^{(1)}}=1, \sum_{k^{(2)}=1}^{K^{(2)}} \mA^{(2)}_{jk^{(2)}}&=1, \sum_{r=1}^R \mA^{(3)}_{rk^{(3)}}=1\\
             \sum_{k^{(3)}=1}^{K^{(3)}} \mathcal{G}_{k^{(1)}k^{(2)}k^{(3)}}&=1 , \text{ and }\sum_{r=1}^R\mathcal{D}_{ijr}=1.
       \end{split}
   \end{equation}

In real applications, the size of the core $\mathcal{G}$ is less than any of the original dimensions of the tensor: $\max\left(K^{(1)} ,K^{(2)}, K^{(3)}\right)\ll \min\left(N^{(1)}, N^{(2)},R \right)$. For technical reasons, we assume throughout this paper that $K^{(a)}$ for $a\in[3]$ is fixed while $N^{(1)}, N^{(2)},R $ and the document length $M_{ij}$ vary. 
Since the nonnegative Tucker decomposition introduced in \eqref{eq: TTM} may not be unique, we introduce the following additional separability conditions to ensure identifiability.

\begin{definition}[Anchor document assumption]\label{assumption: anchor doc} We call document $i^*$ an {\it anchor document} for a given cluster $k\in[K]$ if document $i^*$ belongs solely to cluster $k$(i.e., 
$\P[\text{cluster } k  | \text{document } i^* ] =1$).

Adapting this definition to the context of tensor topic modeling, we will assume in the rest of this paper that there exists for each reviewer cluster $k^{(1)}$ an ``anchor reviewer" $i^*$ such that:
\begin{equation*}
\mA^{(1)}_{i^*, k^{(1)}} = 1, \qquad \text{and} \qquad \forall k' \neq k^{(1)},\quad  \mA^{(1)}_{i^*, k'} = 0.
\end{equation*}
Similarly, we assume that for each paper category $k^{(2)}$,  there exists an ``anchor paper" $j^*$ such that:
\begin{equation*}
\mA^{(2)}_{j^*, k^{(2)}} = 1, \qquad \text{and} \qquad \forall k' \neq k^{(2)},\quad  \mA^{(2)}_{j^*, k'} = 0.
\end{equation*}
    
\end{definition}
\begin{definition}[Anchor word assumption]\label{def: anchor word} We call word $r^*\in[R]$ an \textit{anchor word} for topic $k^{(3)}\in[K^{(3)}]$ if the expected frequency of word $j$ in topic $k^{(3)}$ is nonzero, but its expected frequency is zero in all other topics.

In this paper, we assume that every topic $k^{(3)}$ has at least one anchor word:
    \begin{equation}\label{eq:anchor_word}
         \forall k^{(3)} \in [K^{(3)}], \exists r^*:\quad  \mA^{(3)}_{r^*,k^{(3)}} >0, \qquad \text{and} \qquad \forall k' \neq k^{(3)},\quad  \mA^{(3)}_{r^*, k'} = 0.
    \end{equation} 
\end{definition}

\begin{remark}
  These conditions are sufficient, but not necessary for recovering the latent structure in the tensor-pLSI setting considered in this paper. In fact, we provide in Appendix \ref{sec: vh hunting} a discussion on extending the method to settings where the anchor word assumption is not satisfied. However, we note that these assumptions also enhance the interpretability of the model. In essence, these conditions imply the existence of pure words as well as archetypal reviewers and papers. Unlike centroids, archetypes highlight extremes rather than averages and are often used in machine learning and statistics to provide greater interpretability \cite{cutler1994archetypal,morup2012archetypal}. This is particularly relevant in our context, since, usually, to understand or represent topics or clusters in mixed membership settings, it is rather natural to turn to their ``purest"/ "most extremes" representants  rather than to average members of that group.
\end{remark}


\subsection{Our contributions}\label{sec:our contributions}
\par In this paper, we propose a novel estimation procedure to recover the core tensor $\mathcal{G}$, along with the three factor matrices $\mA^{(1)}, \mA^{(2)}, \mA^{(3)}$ in~\eqref{eq: TTM}. We further characterize the statistical efficiency of our method by providing theoretical guarantees under the anchor word, anchor reviewer and anchor paper assumptions. 

 More specifically, inspired by the use of Singular Value Decomposition as the cornerstone step in the estimation of matrix topic models \cite{ke2022using}, we propose to use higher-order SVD (HOSVD) \cite{xia2019sup} for tensor topic modeling. Since $\mathcal{D}$ is assumed to be low-rank, HOSVD is indeed a powerful and robust tool for dimensionality reduction that can (a) be rigorously analyzed under the anchor assumptions (Assumptions~\ref{def: anchor word} and \ref{assumption: anchor doc}); (b) is more interpretable than alternative tensor decompositions (e.g. the CP decomposition) by allowing to model the interactions between all latent components; and (c) has the potential to be computationally fast and efficient.   
We summarize our main contributions below:
\begin{itemize}
    \item We propose a new HOSVD-based method for estimating the four tensor topic components of \eqref{eq: TTM}: the core tensor $\mathcal{G}$ and three factor matrices $\mA^{(1)}, \mA^{(2)}, \mA^{(3)}$ (Section \ref{sec: preliminary and method}).
    \item We characterize the statistical efficiency of our procedure by deriving an error bound in terms of the entry-wise $\ell_1$ loss in Theorem \ref{theorem: main theorem}. Our method relies on the anchor word and anchor document assumptions (Assumptions~\ref{def: anchor word} and \ref{assumption: anchor doc}). Under mild assumptions on the smallest frequency of the word, our error bound is shown to be valid for all parameter regimes for which $K^{(a)}$ for $a\in[3]$ are considered fixed.  
    \item Our proposed procedure has further been shown to be adaptive to the column-wise $\ell_q$-sparsity assumption (Assumption \ref{ass: weak lq sparsity}) in Section \ref{sec: preliminary and method}. This sparsity assumption, motivated by Zipf's law \cite{zipf2013psycho}, aligns with the natural distribution of word frequencies in topics, where only a few words appear frequently, and most others are rare \cite{tran2023sparse}.
    \item We highlight the performance of our method on an extensive set of synthetic (Section~\ref{sec: experiments with synthetic data}) and real-world experiments (Section~\ref{sec: real data experiments}), highlighting the increased accuracy and interpretability of our proposed approach.
\end{itemize}

We note that theoretical guarantees for the parameters in matrix topic modeling  have already been established in the literature \cite{ke2022using,klopp2021assigning}. Consistent estimates of the three-factor matrices in TTM (Tensor Topic Modeling) can be obtained using similar proof techniques by flattening the tensor into matrices across three different dimensions. However, flattening the tensor is not sufficient to provide a consistent and efficient estimate for the core tensor. To the best of our knowledge, our estimating procedure is the first to leverage the advantages of HOSVD and offer a consistent estimate of the core tensor in Tucker-based TTM.

\subsection{Notations}
For any $k \in \mathbb{N}$, let $[k]$ denote the index set $\{1, \dots, k\}$. We use $\mathbf{1}_d$ to denote the vector in $\mathbb{R}^d$ with all entries equal to 1. For a general vector $\mv \in \mathbb{R}^d$, let $\|\mv\|_r$ denote the vector $\ell_r$ norm, for $r =0, 1, \dots, \infty$, and let $\text{diag}(\mv)$ denote the $d \times d$ diagonal matrix with diagonal entries equal to entries of $\mv$. For any $a, b \in \mathbb{R}$, let $a \vee b := \max(a, b)$ and $a\wedge b := \min(a,b)$. 
\par Let $\mathbf{I}_m$ denote the $m \times m$ identity matrix. For a general matrix $\mQ \in \mathbb{R}^{m \times n}$ and $r = 0, 1, \dots, \infty$, let $\|\mQ\|_r$ denote the vector $\ell_r$-norm of $\mQ$ if one treats $\mQ$ as a vector. Let $\|\mQ\|_F$ and $\|\mQ\|_\text{op}$ denote the Frobenius (i.e. $\|\mQ\|_F = \|\mQ\|_2$) and operator norms of $\mQ$ respectively. For any index $j \in [m]$ and $i \in [n]$, let $\mQ_{ji}$, $\mQ_{j,i}$ or $\mQ(j,i)$ denote the $(j,i)$-entry of $\mQ$. Also, let $\mQ_{j\star}$ and $\mQ_{\star i}$ denote the $j^\text{th}$ row and $i^{th}$ column of $\mQ$ respectively. For an integer $k \leq m \wedge n$, let $\sigma_k(\mQ)$ denote the $k^\text{th}$ largest singular value of $\mQ$, and if $\mQ$ is a square matrix then, if applicable, let $\lambda_k(\mQ)$ denote the $k^\text{th}$ largest eigenvalue of $\mQ$.
 We write the \textit{Kronecker product} between two matrices $\mA\in\mathbb{R}^{n_1\times n_2}$ and $\mB\in\mathbb{R}^{m_1\times m_2}$ as: $\mA\otimes \mB\in \mathbb{R}^{n_1m_1\times n_2m_2}$ with $(\mA\otimes \mB)_{ik,jl}=\mA_{ij}\mB_{kl}$. For any matrix $\mA \in \R^{n_1\times n_2}$, the vectorized $\mathbb{R}^{n_1n_2}$-vector $\text{vec}(\mA)$ is obtained by stacking the columns of $\mA$. Conversely, a vector $\mathbf{u}\in \mathbb{R}^{n_1n_2}$ can be written in matrix form $\mathbf{U}=\operatorname{mat}(\mathbf{u})=\operatorname{vec}^{-1}(\mathbf{u})\in\mathbb{R}^{n_1\times n_2}$.
\par We define the unfolding operations of tensors along each dimensions as follows. The mode-1 matricization of a tensor $\mathcal X \in \R^{n_1 \times n_2 \times n_3}$ matrix $\mX^{(1)}$  with entries: $\mX^{(1)}(i,j_k)$ for $j_k=n_3(j-1)+k$. We also write $(jk)=j_k$. The mode-2 and mode-3 matricization  (defined as $\mathcal{M}_2(\mathcal{X})$ and $\mathcal{M}_3(\mathcal{X})$, respectively)
can be defined in a similar manner as the matrices $\mathcal X \in \R^{n_2 \times n_1 \times n_3}$ matrix $\mX^{(2)}$  and $\mathcal X \in \R^{n_3 \times n_1 \times n_2}$ matrix $\mX^{(3)}$ such that:
$$ \mX^{(2)}_{j, (ik)} = \mathcal X_{ijk} \quad \text{and}\quad  \mX^{(3)}_{k, (ij)} = \mathcal X_{ijk}. $$
\par Throughout this paper, we denote matrices with simple bold capital letters (e.g. $\textbf{A}\in \R^{n_1 \times n_2}$), vectors with bold lower-case letters (e.g. $\textbf{a}\in \R^{n}$), and 3-way tensors using italics (e.g. $\mathcal{A} \in \R^{n_1 \times n_2 \times n_3}$). Let $C, c >0$ denote absolute constants that may depend on $K^{(a)}$; we assume that $K^{(a)
}$ for $a\in[3]$ are fixed, unobserved constants. Let $C^*, c^* > 0$ denote numerical constants that do not depend on the unobserved quantities like $K^{(a)}$ for $a\in[3]$. The constants $C, c , C^*, c^*$ may change from line to line. For the simplicity of notation, we write $K^{(1,2,3)}=K^{(1)}K^{(2)}K^{(3)}$, and $N^{(1,2)}=N^{(1)}N^{(2)}$.
\vspace{4mm}

\section{Tensor Topic Modeling}\label{sec: preliminary and method}
Recall that $\mathcal{Y}\in\mathbb{R}^{N^{(1)}\times N^{(2)}\times R}$ denotes the observed corpus tensor and by assumption, under model \eqref{eq: model multinominal}, $\mathbb{E}[\mathcal{Y}]=\mathcal{D}$. In our paper, for ease of presentation, we assume that the document lengths are all equal: $M_
{11} = \dots = M_
{N^{(1)}N^{(2)}} = M$. Our results also hold if we assume the document lengths satisfy $\max_{i,j \in [N^{(1)}]\times [N^{(2)}]} M_{ij} \leq C^* \min_{i,j \in [N^{(1)}]\times [N^{(2)}]} M_{ij}$, in which case $M = \frac{1}{N^{(1)}N^{(2)}}\sum_{i,j=1}^{N^{(1)},N^{(2)}} M_{ij}$ denotes the average document length.  Before introducing our method, we begin with some preliminary definitions and properties of tensor matricization, an operation that is central to our proposed method.
\vspace{5mm}\\
\textbf{Mode-$a$ matricization of $\mathcal{D}$}: As discussed in Section \ref{sec:our contributions}, estimating the three factor matrices can be achieved by flattening (or {\it matricizing}) the input tensor $\mathcal{Y}$ across each of the different modes. Specifically, for $a\in [3]$, we define the mode-$a$ matricization of a tensor  $\mathcal{D}$ with rank $(K^{(1)}, K^{(2)}, K^{(3)})$ 
Tucker-decomposition $\mathcal{D} = \G \cdot (\mA^{(1)},\mA^{(2)},\mA^{(3)})$ as:
\begin{align}\label{eq:matricization}
\mD^{(1)}:&=\mathcal{M}_1(\mathcal{D})=\mA^{(1)} \mathcal{M}_1(\mathcal{G})(\mA^{(2)}\otimes \mA^{(3)})^\top=\mA^{(1)}\mW^{(1)}\in\mathbb{R}^{N^{(1)}\times (N^{(2)}R)}\tag{D.1}\\
\mD^{(2)}:&=\mathcal{M}_2(\mathcal{D})=\mA^{(2)} \mathcal{M}_2(\mathcal{G})(\mA^{(1)}\otimes \mA^{(3)})^\top=\mA^{(2)}\mW^{(2)}\in\mathbb{R}^{N^{(2)}\times (N^{(1)}R)}\tag{D.2}\\
    \mD^{(3)}:&=\mathcal{M}_3(\mathcal{D})=\mA^{(3)} \mathcal{M}_3(\mathcal{G})(\mA^{(1)}\otimes \mA^{(2)})^\top=\mA^{(3)}\mW^{(3)}\in\mathbb{R}^{R\times (N^{(1)}N^{(2)})}\tag{D.3}
\end{align}
where $\mW^{(1)}\in\mathbb{R}^{K^{(1)}\times (N^{(2)}R)}$, $\mW^{(2)}\in\mathbb{R}^{K^{(2)}\times (N^{(1)}R)}$ and $\mW^{(3)}\in\mathbb{R}^{K^{(3)}\times (N^{(1)}N^{(2)})}$. By direct application of Bayes' theorem, we have:
\begin{align*}
\mW^{(1)}_{k^{(1)}(jr)}&= \sum_{k^{(2)}=1}^{K^{(2)}} \sum_{k^{(3)}=1}^{K^{(3)}} \mathcal{G}_{k^{(1)}k^{(2)}k^{(3)}} \mA^{(2)}_{jk^{(2)}}\mA^{(3)}_{rk^{(3)}} = \mathbb{P}\left(\text{word }r | \text{ reviewer type } k^{(1)},\text{ paper }j \right)\tag{W.1}\\
\mW^{(2)}_{k^{(2)}(ir)}&=\sum_{k^{(1)}=1}^{K^{(1)}}\sum_{k^{(3)}=1}^{K^{(3)}} \mathcal{G}_{k^{(1)}k^{(2)}k^{(3)}} \mA^{(1)}_{ik^{(1)}}\mA^{(3)}_{rk^{(3)}}= \mathbb{P}\left(\text{word }r | \text{ paper category } k^{(2)},\text{ reviewer }i \right)\tag{W.2}\\
\mW^{(3)}_{k^{(3)}(ij)}&=\sum_{k^{(1)}=1}^{K^{(1)}} \sum_{k^{(2)}=1}^{K^{(2)}}\mathcal{G}_{k^{(1)}k^{(2)}k^{(3)}} \mA^{(1)}_{ik^{(1)}}\mA^{(2)}_{jk^{(2)}}=\mathbb{P}\left(\text{topic }k^{(3)} | \text{ reviewer } i,\text{ paper }j \right)\tag{W.3}
\end{align*}

Since each of the entries of the matrix $\mW^{(a)}, a\in [3]$ can be interpreted as probabilities, the following constraints must hold:
\begin{equation}\label{eq:constraints_on_Ws}
    \sum_{r=1}^{R} \mW^{(1)}_{k^{(1)}(jr)}=1, \quad \sum_{r=1}^{R}\mW^{(2)}_{k^{(2)}(ir)}=1, \text{ and }\sum_{k^{(3)}=1}^{K^{(3)}} \mW^{(3)}_{k^{(3)}(ij)}=1
\end{equation}
Using matricization on the tensor $\mathcal{D}$, the estimation of the factor matrices $\mA^{(a)}$ is similar to the traditional pLSI setting, which we detail in Section \ref{sec: procedure for factor matrices}. However, estimating the core tensor $\mathcal{G}$ is more difficult, as is establishing its consistency. To see this, note that a matricization along any axis $a$ of the tensor $\mathcal{D}$ would yield, in vector form:
$$
\text{vec}\left(\mD^{(a)}\right)=\left(\mA^{(\backslash a)\top}\otimes \mA^{(a)\top} \right)\text{vec}\left(\mathcal{M}_{a}(\mathcal{G})\right)
$$
where $\mA^{(\backslash a)}:=\mA^{(b)}\otimes \mA^{(c)}$ for $b,c\neq a$ and $b<c$. However, unlike for the factor matrices, we cannot rely on any anchor assumption to reliably estimate $\mathcal{M}_{a}(\mathcal{G})$. A solution would be to formulate the recovery of $\mathcal{M}_{a}(\mathcal{G})$ as a regression problem,  using the estimated factor matrices in lieu of $\mA^{(\backslash a)}$. However, this would yield an error bound in the estimation of $\G$ that increases with the error between $\mA:=\left(\mA^{(\backslash a)\top}\otimes \mA^{(a)\top} \right)$ and its estimate $\hat A:=\left(\hat \mA^{(\backslash a)\top}\otimes \hat \mA^{(a)\top} \right)$. Indeed, for any (compatible) matrices $\mA$ and $\mB$,
$$\|\hat A\otimes \hat B-A\otimes B\|_1\leq \|\hat A-A\|_1\|\hat B\|_1+\|\hat B-B\|_1\| A\|_1,$$ and in our case, $\| \mA^{(1)}\|_1 = N^{(1)}, \| \mA^{(2)}\|_1 = N^{(2)}, \| \mA^{(3)}\|_1 = K^{(3)}$. Thus, this discrepancy scales with $N_R=\max\left(N^{(1)},N^{(2)},R\right)$. Additionally, this method is computationally slow due to the vectorization of the tensor data. To address these challenges, we propose an alternative algorithm in Section \ref{sec: procedure for core tensor} that constructs an estimate of $\mathcal{G}$ and establishes its consistency.

\subsection{\texorpdfstring{Oracle procedure for the estimation of the factor matrices $\mA^{(1)}, \mA^{(2)}, \mA^{(3)}$}%
{Oracle procedure for the estimation of the factor matrices A(1), A(2), A(3)}}
\label{sec: procedure for factor matrices}

 We first consider the oracle case in which the underlying frequency matrix $\mathcal{D}$ is observed, and derive an oracle procedure to estimate the factor matrices. In the next subsection, we will show how to modify the oracle procedure to deal with stochastic noise (the $\Z$ in Equation~\eqref{eq:signal_plus_noise}), which gives our final method for estimating factor matrices.
\par As discussed in the Introduction, in this paper, we propose a procedure to estimate the factor matrices $\mA^{(1)}, \mA^{(2)}$, and $\mA^{(3)}$ based on the traditional pLSI framework. We begin by introducing a few preliminary concepts that are central to our estimation procedure.
\begin{definition}[Ideal Simplex]\label{def: ideal simplex} Let $\{\vnu_k\}_{k=1}^K \in\mathbb{R}^{p}$ denote a set of $K$ vertices in $\mathbb{R}^p.$ Let $\{\mathbf{y}_i\}_{i\in[n]}\in\mathbb{R}^{p}$  be a set of $n$ data points in the convex hull of the vertices $\{\vnu_k\in\mathbb{R}^{p}\}_{k=1}^K$:   $$\forall i\in [n], \qquad \mathbf{y}_i=\sum_{k=1}^K \omega_{ik} \vnu_{k}, \qquad \text{ with weights $\omega_{ik}\geqslant 0$ for all $k\in [K]$ and $\sum_{k=1}^K \omega_{ik}=1$}.$$ The set $\{\mathbf{y}_i\}_{i\in[n]}\in\mathbb{R}^{p}$ is known as an ``ideal simplex" if for each $k\in [K]$, there exists an index $i\in[p]$ such that $\omega_{ik}=1$ and $\omega_{ik'}=0$ for all $k'\neq k$.
\end{definition}
For points lying on an ideal simplex $\{\mathbf{y}_i\}_{i\in[n]}$,   the $K$ vertices $\{\vnu_k\}_{k=1}^K \in\mathbb{R}^{p}$ can be identified using a ``vertex hunting'' algorithm. This algorithm returns data points corresponding to the vertices of the simplex $\{\vnu_k\}_{k\in[K]}$.
\begin{definition}[Vertex hunting] Given $K$ and a point cloud $\{\mathbf{y}_i\}_{i\in[n]}\in\mathbb{R}^{p}$ belonging to the ideal vertex defined above,  a vertex hunting algorithm $\mathcal{V}(\cdot)$ is a function that returns a subset $K$  of points $\{\vnu_1,\dots,\vnu_K\} \subset \{\mathbf{y}_i\}_{i\in[n]} $ corresponding to the vertices.    
\end{definition}

\subsubsection{\texorpdfstring{Estimating $\mA^{(1)}$ and $\mA^{(2)}$}{Estimating A(1) and A(2)}}

Under the anchor-reviewer and anchor-document assumptions, the matricizations $ \mD^{(1)} $ and $ \mD^{(2)} $  form an ideal simplex (Definition~\ref{def: ideal simplex}), and the matrices $\mA^{(1)}$ and $\mA^{(2)}$ satisfy the condition of $\{\omega_{ik}\}_{i
\in[n],k\in[K]}$ defined in its definition \ref{def: ideal simplex}, since:
$$ \mD^{(1)} =  \mA^{(1)} \mW^{(1)} \quad \text{with } \mA^{(1)}_{ik} = \mathbb{P}[\text{reviewer type } k | \text{ reviewer }i] ,\quad  \sum_{k=1}^{K^{(1)}} \mA^{(1)}_{ik}=1,$$
$$ \mD^{(2)} =  \mA^{(2)} \mW^{(2)} \quad \text{with } \mA^{(2)}_{jk} = \mathbb{P}[\text{paper category  } k | \text{  document }j] ,\quad  \sum_{k=1}^{K^{(2)}} \mA^{(2)}_{jk}=1.$$

To obtain estimates for $\mA^{(1)}$ and $\mA^{(2)}$, a simple procedure consists in first, recovering vertices from the point clouds $\mD^{(1)}$ and $\mD^{(2)}$ respectively, yielding two sets of vertices $\mH^{(a)}, a\in \{1,2\}$. We note that there exists several methods for inferring vertices from point clouds. Options include for instance the successive projection algorithm (SPA) (see \cite{klopp2021assigning}).
The second step of the algorithm is to express each of the observations as a convex combination of the vertices: $\mD^{(a)} = \mA^{(a)} \mH^{(a)}$ for $a\in\{1,2\}$, and obtain estimates of $\mA^{(a)}$ by simply right multiplying $\mD^{(a)} $ by the inverse of $\mH^{(a)}$. However, as shown in Appendix~\ref{Sec: SVD analysis}, the output matrix $\mH^{(a)}$ of vertices from the point cloud $\{\mD^{(a)}_{i\star}\}_{i\in[N^{(a)}]}$ is not invertible. 
We would thus need to estimate  $\mA^{(a)}$ by regressing $\mD^{(a)}$ unto $\mH^{(a)} $, which would make the error bounds for $\mA^{(a)}$ scale linearly with $N^{(a)}$ for $a\in\{1,2\}$.

\par We propose instead building upon the matrix topic modeling given by \cite{klopp2021assigning,ke2022using,tran2023sparse}, which provides a more efficient way of solving for $\mA^{(a)}$. We start by considering the HOSVD of tensor $\mathcal{D}$. Under the Tucker decomposition framework, each matricization of  $\mathcal{D}$ along axis $a$ admits a singular value decomposition (SVD) of the form:
\begin{align*}
\mD^{(a)}=\mathbf{\mathbf{\Xi}}^{(a)}\mathbf{\Lambda}^{(a)} \mB^{(a)\top} 
\end{align*}
A key observation is that, if $\lambda_{K^{(a)}}\left(\mD^{(a)}\right)>0$, then the matrix $\mathbf{\mathbf{\Xi}}^{(a)}$ can be represented as 
\begin{align}\label{eq: Xi=AtildeV}
    \mathbf{\Xi}^{(a)}=\mA^{(a)}\tilde \mV^{(a)}
\end{align}
where $\tilde \mV^{(a)}$ is unique and non-singular (proof in Appendix \ref{Sec: SVD analysis}). If the matrix $\mA^{(1)}$ satisfies the anchor-reviewer (and $\mA^{(2)}$  the anchor-paper) assumption, the rows of the matrix $\mathbf{\mathbf{\Xi}}^{(a)}, a\in \{1,2\}$ form an ideal simplex in $\mathbb{R}^{K^{(a)}}$ with vertices given by the row of the matrix $\tilde \mV^{(a)}$. The matrices $\mA^{(a)}$ for $a\in\{1,2\}$ are the corresponding weights, and, since $\tilde \mV^{(a)}$  is invertible, $\mA^{(a)}$ can be estimated as: $  \mA^{(a)} = \mathbf{\Xi}^{(a)}   (\tilde\mV^{(a)})^{-1}$.

\subsubsection{\texorpdfstring{Estimating $\mA^{(3)}$}{Estimating A(3)}}\label{sec: estimating A3}
For $\mA^{(3)}$, unlike for the first two dimensions, the anchor-word assumption does not naturally imply that the matricization $\mD^{(3)}$ is an ideal simplex. Inspired by \cite{ke2022using}, our procedure uses the SCORE normalization, which was originally introduced to deal with degree heterogeneity in network analysis in \cite{jin2015fast}, to reduce the word point cloud $\mD^{(3)}$ to an ideal simplex. 
\begin{definition}\label{def: SCORE normalization}[SCORE normalization $s(\mathbf{\Xi})$] Given a $n\times K$ matrix $\mathbf{\Xi}$ where the first column is strictly positive, the SCORE normalization  $s(\mathbf{\Xi}):\mathbb{R}^{n\times K}\rightarrow\mathbb{R}^{n\times K-1}$ is the procedure that normalizes each row of the matrix $\mathbf{\Xi}$ by its first coordinate. The resulting matrix is defined as $ \mS:=s(\mathbf{\Xi})=[\mS_{1\star},\dots,\mS_{n\star}]\in\mathbb{R}^{n\times (K-1)}$ such that for all $i\in[n]$:
    $$\mS_{ik}=\mathbf{\Xi}_{i(k+1)}/\mathbf{\Xi}_{i1} \text{ for } k= 1,\dots, K-1 .$$
    
\end{definition}
Under the anchor-word assumption, we can indeed show that the rows of $\mS$ form an ideal simplex (Lemma \ref{lem: hat S row wise bound mode 3} in the Appendix, similar to \cite{ke2022using}). As for the first two dimensions, we can thus employ a vertex-hunting procedure to recover the vertices of the simplex. Expressing each row of $s(\mathbf{\Xi}^{(3)})$ as a convex combination of these  vertices (whose weights are denoted by  the matrix $\mathbf{\Omega}^{(3)}$), the matrix $\mA^{(3)}$ can then be recovered as:
\begin{align}\label{eq:definition of Pi}
\mathbf{\Omega}^{(3)}:=\left[\text{diag}(\mathbf{\Xi}_{\star 1}^{(3)})\right]^{-1} \mA^{(3)} [\text{diag}(\tilde \mV^{(3)}_{\star 1})] 
\end{align}
The detailed proof is provided by Lemma \ref{lem: hat S row wise bound mode 3} in the Appendix. 

\subsubsection{Estimating the Latent Factors}

We summarize the estimation for all three latent components as follows.

\paragraph{Oracle Procedure:} Let $\mD:=\mD^{(a)}\in\mathbb{R}^{n\times p}$ denote a matrix associated with the factor index $a\in[3]$, where the dimensions $n, p$ may vary depending on the choice of $a$.
\begin{enumerate}
    \item (SVD) Compute the singular value decomposition (SVD) of $\mD$ to obtain the first $K$ left singular vectors. Denote the resulting matrix as $\mathbf{\Xi}=[\Xi_{\star 1},\dots, \Xi_{\star K}]\in\mathbb{R}^{n\times K}$.
    \item (Post-SVD preprocessing) Define the matrix $S$ as:
    \begin{equation*}
       \mS=\begin{cases}
            \mathbf{\Xi}\in\mathbb{R}^{n\times \tilde K} \text{ for }\tilde K:=K &\text{ if } a\in\{1,2\}\\
            s(\mathbf{\Xi})\in\mathbb{R}^{n\times \tilde K} \text{ for }\tilde K:=K-1  &\text{ if } a=3 
            \end{cases}
    \end{equation*}
    where $s$ is defined in Definition~\ref{def: SCORE normalization}.
    \item (VH algorithm) Apply a vertex hunting algorithm on the rows of $\mS$ to obtain a vertex set $\{\mV_{1\star},\dots, \mV_{K\star}\}\in\mathbb{R}^{\tilde K}$. Define $\mV=[\mV_{1\star},\dots, \mV_{K\star}]^\top\in\mathbb{R}^{K\times\tilde K }$
    \item (Recovery) Recover the weight matrix $\mathbf{\Omega}^*\in\mathbb{R}^{n\times K}$ by solving the linear equation $\mathbf{\Omega}^* \mV^* =\mS^*$ where
    \begin{equation*}
        \mV^*=\begin{cases}
            \mV\\
              [\mathbf{1}_K, \mV_{\star 1}, \dots \mV_{\star 
              \tilde K}]
        \end{cases}
        \mS^*=\begin{cases}
            \mS&\text{ if } a\in\{1,2\}\\
            [\mathbf{1}_n, \mS_{\star1},\dots \mS_{\star \tilde K}]&\text{ if } a=3
        \end{cases}
    \end{equation*}
  
     \item (Output) Compute $ \mA$ as
     \begin{equation*}
        \mA^{(a)}=\begin{cases}
            \mathbf{\Omega}^* & \text{ if } a\in\{1,2\}\\
            \text{ColNorm}(\text{diag}
(\mathbf{\Xi}_{\star 1}) \mathbf{\Omega}^*)& \text{ if } a=3
        \end{cases}
     \end{equation*}
     where ColNorm($\mA$) for matrix $\mA$ is to normalize each column of matrix $\mA$ to have unit $l_1$ norm
\end{enumerate}
\begin{remark}\label{remark: oracle of A}
Our oracle procedure differs from that of \cite{ke2022using} and \cite{klopp2021assigning} in the first step of the singular value decomposition (SVD). Indeed, we perform an SVD on $\mD$ instead of $\mL^{1/2}\mD$ where $\mL$ is defined differently in \cite{ke2022using} and \cite{klopp2021assigning} to adapt to different noise levels. We prefer not preconditioning the SVD as in \cite{ke2022using} as this preconditioning yields a point cloud that can be more easily stretched and destabilized by low-frequency words. As shown in \cite{tran2023sparse}, the robustness of the SVD preconditioning requires indeed that the frequencies of the words in $\mA^{(3)}$ have the same order. To increase the signal-to-noise ratio in the procedure, an alternative is to truncate or threshold words altogether as in \cite{tran2023sparse}, which is easy to adapt in our oracle procedure (see Section \ref{sec: theoretical results}).
\par Additionally, considering $\mD$ instead of $\mL^{1/2}\mD$ simplifies some parts of our theoretical analysis. In particular, the estimation of the core tensor via HOSVD based on  $\mD$  (as opposed to $\mL^{1/2}\mD$ with different $\mL$s) is easier to analyze (See Lemma \ref{lemma: hooi svd bound}).
\end{remark}

\subsection{Estimation procedure in the presence of noise }
\par We now extend the oracle procedure described in the previous subsection to real-world settings, in which we observe a noisy version of $\mD^{(a)}$, $\mY^{(a)}:=\mathcal{M}_{a}(\mathcal{Y})$. 

In this case, the procedure for recovering the latent factors is essentially the same as in the oracle case, replacing the first step by an eigendecomposition of the matrix $\hat \mQ$ to obtain a matrix $ \mathbf{\hat\Xi}\in\mathbb{R}^{n\times K}$ corresponding to its leading eigenvectors $K$, where $\hat \mQ$ is defined as 
    \begin{equation}\label{eq: spectral decom}
        \hat \mQ:=\begin{cases}
          \mY\mY^\top&\text{ if } a\in\{1,2\}\\
          \mY\mY^\top-\frac{1}{M}\text{diag}(\mY\mathbf{1}_p)&\text{ if } a=3
        \end{cases}
    \end{equation}

Furthermore, we choose to use the Successive Projection (SP)\cite{araujo2001successive} as our vertex hunting algorithm in all the experiments, but other choices like Sketched Vertex Search (SVS) \cite{jin2017estimating} and Archetypal Analysis (AA) \cite{javadi2020nonnegative} can also be applied. The difference between those vertex hunting algorithms has been studied in \cite{ke2022using, tran2023sparse}.
\noindent\\
\begin{remark}\label{remark: estimating A3}
When $
a=3$, adjusting the matrix $\mY^{(3)}\mY^{(3)\top}$ by subtracting $\frac{1}{M}\text{diag}(\mY^{(3)}\mathbf{1}_{N^{(1)}N^{(2)}})$ helps center the variance of the multinomial distribution. This adjustment leads to a more accurate SVD, thereby improving the match between the true and estimated singular vectors. 
\par To ensure that the SCORE normalization is well-conditioned during the estimation procedure, the leading eigenvectors obtained should be strictly positive. According to Perron's theorem, the signs of the leading singular vectors are necessarily positive in the oracle procedure. However, in the noisy case,  $\mathbf{\hat\Xi}_{\star1}$ may contain negative entries. Therefore, any word $r\in[R]$ for which $\mathbf{\hat \Xi}_{r1}$  is negative should have the corresponding rows in $\mA^{(3)}$ clipped to zero after Step 2. Subsequent steps are performed only with $\mathbf{\hat \Xi}_{r1}>0$ for $r\in[R]$.  In our theoretical analysis, this situation is unlikely to occur with high probability (see Lemma \ref{lem: row-wise HOSVD} in the Appendix).
\end{remark}

\subsection{\texorpdfstring{Estimating the core tensor $\G$}{Estimating the core tensor G}}\label{sec: procedure for core tensor}
For a tensor $\mathcal{D}$ with rank $r(\mathcal{D})=\left(K^{(1)},K^{(2)},K^{(3)}\right)$,
its higher-order singular value decomposition (HO-SVD) is defined as the decomposition 
$$
\mathcal{D}=\mathcal{S}\cdot \left(\mathbf{\Xi}^{(1)},\mathbf{\Xi}^{(2)},\mathbf{\Xi}^{(3)}\right).
$$
The left singular vectors $\mathbf{\Xi}^{(1)}, \mathbf{\Xi}^{(2)}$ and $\mathbf{\Xi}^{(3)}$ can be obtained from the SVDs of the matricizations $\mD^{(1)}, \mD^{(2)}$ and $\mD^{(3)}$ respectively. 
The $K^{(1)}\times K^{(2)}\times K^{(3)}$ HO-SVD core tensor $\mathcal{S}$ can thus be obtained through the following equation:
$$
\mathcal{S}:=\mathcal{D}\cdot\left(\mathbf{\Xi}^{(1)\top},\mathbf{\Xi}^{(2)\top},\mathbf{\Xi}^{(3)\top}\right)
$$
due to the orthogonality of left singular vectors.

However, in topic modeling, such a simple relationship between the core tensor $\mathcal{G}$ and the three mode matrices $\mA^{(a)}$ for $a\in[3]$ does not exist. Mode matrices in topic modeling are indeed neither non-singular nor orthogonal. While $\mathcal{G}$ can be estimated using a constrained least square regression, as highlighted in the introduction, this naive approach does not yield good estimates.

However, as detailed in Equation \eqref{eq: Xi=AtildeV}, the singular vectors $    \mathbf{\Xi}^{(a)}$ of the HO-SVD of $\mathcal{D}$ can be decomposed as $ \mathbf{\Xi}^{(a)}=\mA^{(a)}\tilde \mV^{(a)}$. Thus, if $\lambda_{K^{(a)}}(\mD^{(a)})>0$, then the matrix $\tilde \mV^{(a)}$ is invertible and $\mA^{(a)}=\mathbf{\Xi}^{(a)} \left(\tilde \mV^{(a)}\right)^{-1}$.

This allows us to rewrite the tensor topic model of \eqref{eq: TTM} as
\begin{align*}
\mathcal{D}=&\mathcal{G}\cdot\left(\mA^{(1)},\mA^{(2)}, \mA^{(3)} \right)\\
=&\mathcal{G}\cdot\left(\mathbf{\Xi}^{(1)} (\tilde \mV^{(1)})^{-1},\mathbf{\Xi}^{(2)}  (\tilde \mV^{(2)})^{-1},\mathbf{\Xi}^{(3)} (\tilde \mV^{(3)})^{-1}\right)\\
=&\left\{\mathcal{G}\cdot \left( (\tilde \mV^{(1)})^{-1}, (\tilde \mV^{(2)})^{-1}, (\tilde \mV^{(3)})^{-1}\right)\right\}\cdot \left(\mathbf{\Xi}^{(1)},\mathbf{\Xi}^{(2)},\mathbf{\Xi}^{(3)}\right)
\end{align*}
Right multiplying the two sides of the previous equation by the singular vectors, we see that: 
$$\mathcal{G}=\mathcal{S}\cdot \left( \tilde \mV^{(1)}, \tilde \mV^{(2)},\tilde \mV^{(3)}\right).
    $$

For $a \in \{1,2\}$, the matrix  $\tilde{\mV}^{(a)}$ corresponds simply to the vertices extracted by the vertex hunting procedure: $\tilde{\mV}^{(a)} = \mV^{(a)} = \mathcal{V}(\mathbf{\Xi}^{(a)})$. For $a = 3$, Lemma \ref{lem: hat S row wise bound mode 3} shows that:
\begin{equation}\label{eq: tilde V and Vstar}
    \left[\text{diag}(\tilde \mV^{(3)}_{\star 1})\right]^{-1}\tilde \mV^{(3)}=\mV^{*(3)}=[\mathbf{1}_{K^{(3)}}, \mV^{(3)}]
\end{equation}
We now introduce an explicit procedure to recover $\mathcal{G}$ from $\mathcal{S}$ and the three mode matrices $\tilde \mV^{(a)}$ for $a\in[3]$. This procedure relies on the following lemma.
\begin{lemma}\label{lemma: relationship of G and S} Suppose $\sigma_{K^{(3)}}(\mA^{(3)})\geqslant c^*$ for some constant $c^*>0$. Then, there exists a positive vector $\mathbf{q}_0\in\mathbb{R}^{K^{(3)}}$ such that $ \tilde \mV^{(3)}=\text{diag}(\mathbf{q}_0) \mV^{*(3)}$ and
   \begin{align}
    \mA^{(3)}\text{diag}(\mathbf{q}_0)=\text{diag}(\mathbf{\Xi}^{(3)}_{\star 1})\mathbf{\Omega}^{*(3)}\label{eq: lemma 2.1 A3}
    \end{align}
    where $\mV^{*(3)}$ is the set of vertices extracted by the vertex hunting procedure.
\end{lemma}
Lemma \ref{lemma: relationship of G and S} establishes the connection between equations \eqref{eq:definition of Pi} and \eqref{eq: tilde V and Vstar}, and allows the estimation of $\tilde \mV^{(3)}$ from the estimation procedure for $\mA^{(3)}$. Since each column of $\mA^{(3)}$ has unit $\ell_1-$norm, the right hand size of \eqref{eq: lemma 2.1 A3} means that 
the vector $\mathbf{q}_0$ can be estimated as the $\ell_1$-norm of each column of $\text{diag}(\mathbf{\Xi}^{(3)}_{\star 1})\mathbf{\Omega}^{*(3)}$. In the oracle procedure, the vector $\mathbf{q}_0$ corresponds to the vector $\tilde  \mV_{\star 1}^{(3)}$. The entries of $\tilde \mV_{\star 1}^{(3)}$  are all positive if $\lambda_{K^{(a)}}(\mD^{(a)})>0$ (see the proof of \eqref{eq: bound on first value of tilde V} in Section \ref{sec: vh hunting}). In the noisy setting, our theoretical analysis guarantees that the entries of $\mathbf{q}_0$ are positive with high probability.

\par We summarize the full estimation procedure for $\mathcal{G}$ in the noisy setting:
\noindent\\
\textbf{[Estimation procedure for $\mathcal{G}$]} Given an input tensor $\mathcal{Y}$
\begin{enumerate}
    \item (HOSVD) Obtain estimates of the singular matrices $\mathbf{\hat\Xi}^{(a)}$ for all $a\in\{1,2,3\}$, from the spectral decomposition in \eqref{eq: spectral decom}.
    \item (HOSVD Core recovery) Estimate the core tensor $\hat{\mathcal{S}}$ as $\mathcal{\hat S}:=\mathcal{Y}\cdot\left(\mathbf{\hat \Xi}^{(1)\top},\mathbf{\hat \Xi}^{(2)\top},\mathbf{\hat\Xi}^{(3)\top}\right)$
    \item (Mode Recovery) Estimate $\check \mV^{(a)}$ from the matrices of vertices of $\hat \mV^{*(a)}$ such that 
    \begin{align*}
        \check \mV^{(1)}=\hat \mV^{*(1)},\quad\check \mV^{(2)}=\hat \mV^{*(2)} \quad\text{and}\quad \check \mV^{(3)}=\text{diag}(\mathbf{\hat q}_0) \hat \mV^{*(3)}
    \end{align*}
    where $\mathbf{\hat q}_0=\left(\|\hat \mA^*_{\star 1}\|_1,\dots  \|\hat \mA^*_{\star K^{(3)}}\|_1\right)^\top$ with $\hat \mA^*=\text{diag}(\mathbf{\hat \Xi}^{(3)}_{\star 1}) \mathbf{\hat\Omega}^{*(3)}$
    \item (Core Tensor Recovery) Set the core tensor to be: $ \tilde{\mathcal{G}}=\mathcal{\hat \mS}\cdot \left( \check \mV^{(1)}, \check \mV^{(2)},\check \mV^{(3)}\right)$.
      Set $\hat{\mathcal{G}}$ to be a clipped version of $\mathcal{\tilde G}$, so that  any negative entries in $\mathcal{\tilde G}$ are set to zero and then all rows are appropriately re-normalized so that $\hat{\mathcal{G}}_{k^{(1)}k^{(2)}\star}$ sums up to 1.
\end{enumerate}

\begin{remark}[Comparison with HOOI]
Suppose the underlying noise tensor $\mathcal{Z}$ consists of $i.i.d.$ subGaussian entries with mean zero and variance $\sigma$, 
\cite{richard2014statistical}  and \cite{zhang2018tensor} showed that the method of high-order orthogonal iteration (HOOI)\cite{de2000best} yields optimal estimates of the singular vectors for the Tucker decomposition when the signal-to-noise ratio  is $\gtrsim p^{3/4}$, where $p$ is the maximum tensor dimension. In contrast, HOSVD tends to yield sub-optimal results under such conditions. Our method can be easily adapted to work with HOOI by incorporating power iterations into steps 1 and 2 using the data $\mathcal{Y}$ with no centering step (described in \eqref{eq: spectral decom}). This adaptation would in principle allow us to leverage the advantages of HOOI and deploy them to our subGaussian setting.
\par However, under  model \eqref{eq: model multinominal}, the stochastic noise $\mathcal{Z}$ follows a multinomial distribution, and entries are not independent across dimensions $r\in[R]$. We can show that in this scenario, HOOI achieves performance bounds comparable to those of HOSVD (see Lemma \ref{lemma: hooi svd bound}). 
\end{remark}

\begin{remark}[Comparison with truncated HOSVD] Truncated HOSVD and STAT-SVD \cite{zhang2019optimal} are methods designed to identify low-rank structures in tensor data while assuming a certain sparsity structure. Specifically, these methods often rely on mode-wise hard sparsity, requiring at least one row of the mode matrices $\mA^{(a)}$ to have all entries equal to zero. However, this assumption is not suitable for topic modeling: if a word does not appear in any of the documents, it is absent from the vocabulary and does not affect the model.
\par As discussed in \cite{tran2023sparse}, assuming element-wise sparsity or row-wise sparsity in the mode matrix $\mA^{(3)}$ is not appropriate for dealing with the large vocabulary size $R$. Instead, a weak column-wise $\ell_q$ sparsity assumption (Assumption \ref{ass: weak lq sparsity}) would be more appropriate. This assumption is based on the empirical observation that word frequency in text data often decreases inversely proportionally with word rank. Similar assumptions are used in other statistical contexts, such as sparse PCA \cite{ma2013sparse} and sparse covariance estimation \cite{cai2012optimal}.
Our method is easily adapted to accommodate this weak $\ell_q$ sparsity. This approach is detailed in Lemma \ref{lemm: sparse ttm} and Section \ref{sec: theoretical results}, showing how our method can effectively handle the sparsity structure relevant to tensor topic modeling.
\end{remark}


\subsection{Theoretical Results}\label{sec: theoretical results}
Given a matrix parameter $\mA\in\mathbb{R}^{n\times K}$, we measure the performance of its estimator $\hat \mA$ through the element-wise $\ell_1$ error (subject to a permutation matrix $\mathbf{\Pi}$):
\begin{align}\label{eq: l1 error in matrices}
    \mathcal{L}\left(\hat \mA, \mA\right)\equiv\sum_{k=1}^K \|\hat \mA_{\star \pi(k)}-\mA_{\star k}\|_1
\end{align}
The core tensor is also permuted consistently with the factor matrices. We therefore measure
\begin{align}\label{eq: l1 error in tensor}
    \mathcal{L}\left(\mathcal{\hat G}, \mathcal{G}\right) \equiv \sum_{k^{(1)}=1}^{K^{(1)}}\sum_{k^{(2)}=1}^{K^{(2)}}\sum_{k^{(3)}=1}^{K^{(3)}}|\mathcal{\hat G}_{\pi^{(1)}(k^{(1)})\pi^{(2)}(k^{(2)})\pi^{(3)}(k^{(3)})}-\mathcal{G}_{k^{(1)}k^{(2)}k^{(3)}}|
\end{align}
where $\mathbf{\Pi}^{(1)},\mathbf{\Pi}^{(2)},\mathbf{\Pi}^{(3)}$ are permutation matrices for the factor matrices $\mA^{(1)},\mA^{(2)},\mA^{(3)}$ respectively. 

\par We begin by stating the assumptions underlying our analysis.
\begin{assumption}\label{ass: min singular value}
For some constant $c^*\in(0,1)$
$$\sigma_{K^{(1)}}(\frac{1}{N^{(1)}}\mA^{(1)\top}\mA^{(1)})\geqslant c^*, \sigma_{K^{(2)}}(\frac{1}{N^{(2)}}\mA^{(2)\top}\mA^{(2)})\geqslant c^*, \text{ and } \sigma_{K^{(3)}}(\mA^{(3)})\geqslant c^*\sqrt{K^{(3)}}$$
$$\sigma_{K^{(a)}}(\frac{1}{K^{(1)}K^{(2)}}\mathcal{M}_a(\mathcal{G})\mathcal{M}_a(\mathcal{G})^\top)\geqslant c^* \text{ for all } a\in[3]$$
\end{assumption}
\begin{assumption}[Anchor conditions]\label{ass:identifiability} For the factor matrices defined in the introduction section, we assume that:
\begin{itemize}
    \item $\mA^{(1)}$ and $\mA^{(2)}$ has at least one anchor entry for each class $k$. That is, for each class $k$, there exist an index $j$ such that $\mA_{jk}=1$ and $\mA_{jk'}=0$ for all other class $k'.$
    \item $\mA^{(3)}$ has at least one anchor word for each topic $k$. That is, for each topic $k$, there exists a word $r$ such that $\mA_{rk}>0$ but $\mA_{rk'}=0$  for all other topic $k'.$
\end{itemize}
\end{assumption}
\begin{assumption}\label{ass: min value A3}
    For some constant $c^*>0$, we assume that:
    $$
    \min_{k,k'\in[K^{(3)}]} \left(\mA^{(3)\top}\mA^{(3)}\right)_{kk'}\geqslant c^*
    $$
\end{assumption}
Assumption \ref{ass: min singular value} assumes the factor matrices and the core tensor are well-conditioned. This assumption is commonly assumed in the literature on matrix topic modeling \cite{ke2022using,klopp2021assigning,tran2023sparse}. Since the entries of the four components of Tucker decomposition in topic modeling are all probabilities, we always have $\sigma_1(\mA^{(3)})\leqslant\sqrt{K^{(3)}}$, $\sigma_1(\frac{1}{N^{(a)}}\mA^{(a)\top}\mA^{(a)})\leqslant 1$ for $a\in\{1,2\}$, and $\sigma_1(\mathcal{M}_a(\mathcal{G})^\top \mathcal{M}_a(\mathcal{G}))\leqslant K^{(1)}K^{(2)}$ for any $a\in[3]$. Combined with Assumption~\ref{ass: min singular value}, this implies that the condition numbers of the Tucker components are all bounded by some constant.

\par Assumption \ref{ass:identifiability} requires the lower-dimensional classes or topics from the factor matrices to each have at least one anchor entity. For comparison, identifiability in matrix topic modeling is guaranteed when assuming an anchor condition on only one dimension (e.g. anchor document or anchor word, but not necessarily both). In the tensor case, matricization across a mode does not guarantee identifiability on another, and therefore we have to assume an anchor condition for each target factor matrix. This is thus a much stronger assumption than for matrix topic modeling. However, this assumption is only used to guarantee that there exists a vertex hunting algorithm that recovers the vertices such that $\max_{k\in[K]}\|\hat \vnu_k-\vnu_k\|_2\leqslant c^* \max_{i\in[n]}\|\mathbf{\hat y}_i-\mathbf{y}_i\|_2$(up to a label permutation). We refer the reader to the discussion of \cite{tran2023sparse} for weaker assumptions in Section 2.4, under which our main theorem \ref{theorem: main theorem} may also hold.
\par Assumption \ref{ass: min value A3} pertains the correlation between topics,  effectively assuming that topics are reasonably correlated. We do not perceive this assumption to be restrictive, since any topic in $\mA^{(3)}$  will assign significant weights to common, less informative words.  This condition implies a large eigengap, meaning the $K^{(3)}$ first largest eigenvalues along the third mode are well-separated from the smaller eigenvalues (due to noise). This makes the model well-conditioned and facilitates the subsequent analysis.

\vspace{0.4cm}

\par We now state our main theorem.
\begin{theorem}\label{theorem: main theorem} Suppose that the assumptions \ref{ass: min singular value},\ref{ass:identifiability} and \ref{ass: min value A3} hold. Define $N_R=\max\{N^{(1)},N^{(2)},R\}$ and $f_r=\sum_{k\in[K^{(3)}]}\mA^{(3)}_{rk}$. For some constants $c,c^*,C>0$, we further assume that the number of words per document $M$ is such that $M\geqslant C\log(N_R)$, and that the word frequencies $f_r = \sum_{k=1}^K \mA^{(3)}_{jk}$ are such that:
$\min_{r\in[R]} f_r\geqslant c \sqrt{\frac{\log(N_R)}{N^{(1)}N^{(2)}M}}$. 

Then, with probability $1-o(N_R^{-1})$, 
 \begin{align}
    &\mathcal{L}\left(\hat \mA^{(1)},\mA^{(1)}\right)\leq c^* N^{(1)}\sqrt{\frac{\log(N_R)}{M}}\label{eq L A1}\\
      &\mathcal{L}\left(\hat \mA^{(2)},\mA^{(2)}\right)\leq c^*  N^{(2)}\sqrt{\frac{\log(N_R)}{M}}\label{eq L A2}\\
    &  \mathcal{L}\left(\hat \mA^{(3)},\mA^{(3)}\right)\leq c^* \left(\frac{\log(N_R)}{N^{(1)}N^{(2)}M}\right)^{1/4}\label{eq L A3}\\
    &\mathcal{L}\left(\mathcal{\hat G}, \mathcal{G}\right)\leq c^*  \max\left\{\sqrt{\frac{\log N_R}{M}},\left(\frac{\log N_R}{N^{(1)}N^{(2)}M}\right)^{1/4}\right\}\label{eq L G}
\end{align}   
\end{theorem}
The proof of this theorem is provided in Appendix~\ref{app:proof}. When document lengths $M$ is sufficiently large such that 
$M\gg N^{(1)}N^{(2)}$, the error in estimating $\mathcal{G}$ is thus primarily driven by the SVD error associated with the mode matrix $\mA^{(3)}$. Conversely, if 
$M\ll N^{(1)}N^{(2)}$,
the estimation error is dominated by the errors in the mode matrices $\mA^{(1)}$ and $\mA^{(2)}$. 

\vspace{0.4cm}

We also compare our results in Theorem \ref{theorem: main theorem} with the existing literature. For the matrices $\mA^{(1)}, \mA^{(2)}$, the rates in Equations \eqref{eq L A1} and \eqref{eq L A2} are consistent with the minimax rates provided in \cite{klopp2021assigning}, up to a logarithmic factor. For the topic-word matrix $\mA^{(3)}$, our rate only depends on the logarithm of the vocabulary size
 $R$,
 and the rate in \eqref{eq L A3} aligns with the rate derived  in \cite{tran2023sparse} for high-frequency words. \par

However, our analysis is based on the condition that the lowest frequency of the entire vocabulary is bounded below by $\sqrt{\frac{\log N_R}{N^{(1)}N^{(2)}M}}$. 
If this condition is not realistic, an alternative is to assume weak $\ell_q$ sparsity on $\mA^{(3)}$ as in Assumption \ref{ass: weak lq sparsity}, as discussed in \cite{tran2023sparse}. 
\begin{assumption}[Column-wise $\ell_q$ sparsity]\label{ass: weak lq sparsity}Given $k\in[K^{(3)}]$, let the entries of each column $\mA^{(3)}_{\star k}$ be ordered as $\mA^{(3)}_{[1]k}\geqslant \cdots \geqslant \mA^{(3)}_{[R]k}$. For some $q\in(0,1)$ and $s_0>0$, the columns of $\mA^{(3)}$ satisfy
$$
\max_{k\in[K^{(3)}]}\left\{\max_{r\in[R]} r\left(\mA^{(3)}_{[r]k}\right)^q\right\}\leqslant s_0
$$
    
\end{assumption}
This assumption imposes a condition on the decay rate of the ordered entries in the columns of $\mA^{(3)}$. It allows most entries of $\mA^{(3)}$ to be small but nonzero, making it feasible to estimate these low-frequency entries as zero if they fall below a certain threshold. We set this threshold at $\tau=c'\sqrt{\frac{\log N_R}{N^{(1)}N^{(2)}M}}$ for some constant $c'>0$ ($c'=0.005$ as suggested in \cite{tran2023sparse}). Consequently, the estimation procedure previously outlined in Section \ref{sec: procedure for core tensor} operates only on a subset of the data 
$\mathcal{D}_{\mathcal{J}}\in\mathbb{R}^{N^{(1)}\times N^{(2)}\times |\mathcal{J}|}$
where the set of words $\mathcal{J}$ is defined as 
\begin{align}
    \mathcal{J}:=\left\{r\in[R]:\frac{1}{N^{(1)}N^{(2)}}\sum_{i\in[N^{(1)}],j\in[N^{(2)}]}\mathcal{Y}_{ijr}\geqslant c'\sqrt{\frac{\log N_R}{N^{(1)}N^{(2)}M}} \right\}
\end{align}
For the subset of word frequencies in the matrix $\mA^{(3)}_{\mathcal{J}^c\star}$, where $\mathcal{J}^c$ denotes the set of indices with frequencies below the threshold, we set these entries to zero.
\begin{lemma}[Sparse tensor topic modeling]\label{lemm: sparse ttm}Suppose the Assumption \ref{ass: min singular value},\ref{ass:identifiability},\ref{ass: min value A3} and \ref{ass: weak lq sparsity} are satisfied. If $M\geq c \log N_R$, and $s_0\left(\frac{\log N_R}{N^{(1)}N^{(2)}M}\right)^{\frac{1-q}{2}}=o(1)$, we then can show that, with  probability at least $1-o(N_R^{-1})$, the error in the third factor is such that:
\begin{align}\label{eq: rate sparse A3}
    \mathcal{L}(\hat \mA^{(3)},\mA^{(3)})\leq c^*\left(\frac{\log N_R}{N^{(1)}N^{(2)}M}\right)^{1/4}+s_0\left(\frac{\log N_R}{N^{(1)}N^{(2)}M}\right)^{\frac{1-q}{2}}
\end{align}
    
\end{lemma}

The first part of the error term in \eqref{eq: rate sparse A3} accounts for the errors associated with the words within the set $\mathcal{J}$ where $\min_{r\in\mathcal{J}}f_r\gtrsim \sqrt\frac{\log N_R}{N^{(1)}N^{(2)}M}$. In the strong sparsity regime where $s_0=O(1)$ and $0<q<1/2$, this component of the error dominates. The second part of the error term, on the right-hand side of the rate, comes from setting the entries corresponding to the words outside the set $\mathcal{J}$ to zero. This simplification is based on the assumption that these entries are below a certain threshold and can be safely approximated as zero. However, in some cases, especially when the number of such entries is large and $1/2<q<1$, this component can dominate the overall error. 
\par Since the estimation methods (e.g., HOSVD, core recovery in Section \ref{sec: procedure for core tensor}) utilize the subset 
$\mathcal{J}$ with a specified minimum frequency condition, the error rates for $\mA^{(1)}$, $\mA^{(2)}$ and the core tensor $\mathcal{G}$ remain unchanged.

\section{Experiments on synthetic data}\label{sec: experiments with synthetic data}
To validate the performance of our method, we present here the results of extensive experiments on synthetic datasets\footnote{The code for our method and all the experiments presented in this section can be found on Github at the following link: \url{https://github.com/yatingliu2548/tensor-topic-modeling}} where the ground truth is known. This allows us in particular to understand the method's performance as a function of the data dimensions and properties.

\paragraph{Data Generation Procedure.} We adapt the synthetic data generation procedure of \cite{ke2022using} and \cite{tran2023sparse} to the tensor setting. More specifically, we generate soft cluster assignments along the first and second modes by sampling $\mA^{(1)}$ and $\mA^{(2)}$ rowwise from a Dirichlet distribution with parameter $\mathbf{1}_{K^{(a)}}, a \in \{1,2\}$. We generate a core matrix $\G \in \R^{K^{(1)} \times K^{(2)} \times K^{(3)}}$ by sampling each slice $(k^{(1)},k^{(2)}$ from a Dirichlet distribution with parameter $\mathbf{1}_{K^{(3)}}$. Finally, we create the matrix $\mA^{(3)}$ (whose rows indicate the frequencies of the words in each topic) by sampling its entries from a uniform distribution and normalizing the columns appropriately. The data $\Y$ is then sampled from a multinomial distribution with parameters $\D = \G \cdot (\mA^{(1)}, \mA^{(2)}, \mA^{(3)})$ as \eqref{eq: model multinominal}.

\paragraph{Benchmarks} Throughout this section, we  compare the performance of our method against five different benchmarks:
\begin{itemize}[itemsep=0.01em]
    \item {\it Tensor-LDA:} we implement a fully Bayesian counterpart to our method (the extension of LDA for tensor data), where all latent factors and core $\mathcal{G}$ are endowed with Dirichlet priors. The posteriors for all latent $\mA^{(a)}, a \in [3]$ and $\mathcal{G}$ are fitted using Variational Bayes in Rstan \cite{guo2020package}. We provide further details on the model and implementation of this method in Appendix~\ref{app:benchmarks}.
        \item {\it Hybrid-LDA:} We also compare our method with a simple refinement of the basic LDA results to our setting: we begin by fitting an LDA model to the matricization of the data along the third mode, thereby allowing us to extract the latent topic matrix $\mA^{(3)}$ and a topic assignment matrix $\mW^{(3)} \in \R^{K^{(3)}\times N^{(1)}N^{(2)}}$. We subsequently apply a vertex hunting algorithm on the resulting topic assignment matrix $\mW^{(3)}$ to extract the remaining two factors, $\mA^{(1)}$ and $\mA^{(2)}$, and estimate the core using constrained linear regression. The full estimation procedure is  detailed in Appendix~\ref{app:benchmarks}.
       \item {\it Structural Topic Model (STM)\cite{roberts2014structural}:} Since the previous method (Hybrid LDA) is oblivious to existing correlations in topics amongst documents, we replace the Latent Dirichlet Allocation step in the previous bullet point by a Structural Topic Model (STM)\cite{roberts2014structural}. The main difference between the two models lies in the explicit encoding of correlations between documents in STM. To this end, it replaces the Dirichlet prior on the mixture matrix $\mW^{(3)}$ by a lognormal distribution:
       $$ \vtheta \sim \text{lognormal} (\mX\vbeta, \mathbf{\Sigma})$$
       where $\mX$ is a matrix of covariates and $\vbeta$ is a regression coefficient. In our case, we take $\mX$ to be a $N^{(1)} N^{(2)} \times (N^{(1)}+N^{(2)})$ matrix such that:
       $$ \forall i \in[N^{(1)}], j \in [N^{(2)}], \quad \mX_{(ij), l} = \begin{cases}
           1 \quad \text{ if } l = N^{(1)}  \\ 
           0 \quad \text{ if } l <N^{(1)} \text{ and }  l \neq N^{(2)} \\
           1 \quad \text{ if } l = N^{(1)} + N^{(2)}  \\ 
           0 \quad \text{ if } l >N^{(1)} \text{ and }  l \neq N^{(2)} + N^{(1)} \\
       \end{cases}$$
        \item {\it Non-Negative Tucker Decomposition (NTD):} To compare our method against another simple heuristic, we also apply a simple non-negative Tucker decomposition on our data, and normalize the estimated latents and core appropriately so that their rows (or columns when appropriate) lie on the simplex. 
    \item {\it Topic-Score TTM:} Finally, we further compare our method against a variant using the topic score method of \cite{ke2022using} to compute the topic matrix $\mA^{(3)}$, but using the same method as we propose to estimate the rest of the latent factors and core.  This allows us in particular to assess the efficiency of our proposed normalization (Step 1 of the HOSVD, Equation~\eqref{eq: spectral decom}) 
\end{itemize}

\subsection{Visualizing Topic Interactions}
We begin by illustrating the performance of our method on a toy example consisting of a $30\times 10\times 50$ tensor $\mathcal{D}$, shown in Figure \ref{fig:setting1_D_mixed} in Appndix~\ref{app:synthetic}. We take the number of words per document to be identical across documents and set it to $M=100$. Modes 1 and 2 index the documents, while mode 3 indexes the words. We let $\mathcal{Y}$ denote the observed tensor, and $\mathcal{D}$ to denote the  ground-truth one, which is built to exhibit the following structure:
\begin{itemize}
    \item \textbf{First Mode (30 indices, 2 clusters)}: The first 10 indices belong to a first cluster, the last 10 to another, and the middle 10 users have split membership between clusters 1 and 2.

\item \textbf{Second Mode (10 indices, 2 clusters)}: indices along the second mode are split into two separate clusters, with indices 1 through 5 belonging to cluster 1 and 6 through 10 belonging to cluster 2.

\item \textbf{Third Mode (50 words, 3 topics)}: the third mode represents the topics. 
\end{itemize}

The results are shown in Figures \ref{fig:sim1_ours method}-\ref{fig:sim1_hybridLDA}, with the associated reconstruction errors $ \|\mathcal{D}-\mathcal{\hat D}\|_1$ detailed in the figure captions, where $\mathcal{\hat D} = \hat{\mathcal{G}} \cdot ( \hat{\mA}^{(1)}, \hat{\mA}^{(2)}, \hat{\mA}^{(3)})$ represents the reconstructed data tensor using the estimated components. We note that our method yields the smallest reconstruction $\ell_1$ error. 

\par All algorithms, (ours, NTD, tensor LDA, and hybrid-LDA) successfully identify the two groups along the first mode, as well as the mixed membership of the middle 10 indices.
The results in the second mode are more mixed. Hybrid LDA, shown in Figure \ref{fig:sim1_hybridLDA},  fails to correctly recognize clusters along the second mode. By contrast, Tucker-decomposition-based methods (ours in Figure \ref{fig:sim1_ours method}, Non-negative Tucker Decomposition (NTD) in Figure \ref{fig:sim1_NTD} and Tensor-LDA in Figure \ref{fig:sim1_TLDA}) successfully recover the mode-2 clusters. However,  TTM-HOSVD and NTD feature stronger membership assignments compared to Tensor-LDA. Moreover, the core tensors derived from our method (TTM-HOSVD) in Figure \ref{fig:sim1_ours method} and  Tensor-LDA in Figure \ref{fig:sim1_TLDA} reveal clear interactions between clusters along all modes. In particular, these methods show the first cluster of indices along the first mode switches from topic 1 to topic 2 for as the documents to which they correspond switch from cluster 1 to cluster 2 along mode 2.
In contrast, neither NTD nor hybrid-LDA shows such clear interaction patterns in the core. 
\begin{figure}[ht]
    \centering
    \begin{subfigure}[b]{0.10\textwidth}
    \includegraphics[width=\textwidth]{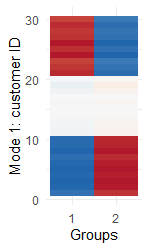}
    \end{subfigure}
    \begin{subfigure}[b]{0.10\textwidth}
    \includegraphics[width=\textwidth]{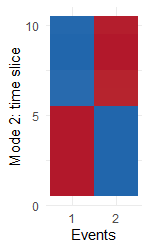}
    \end{subfigure}
    \begin{subfigure}[b]{0.10\textwidth}
    \includegraphics[width=\textwidth]{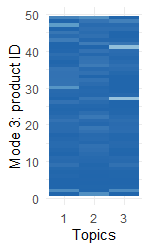}
    \end{subfigure}
    \begin{subfigure}[b]{0.22\textwidth}
    \includegraphics[width=\textwidth]{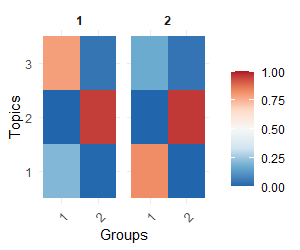}
    \end{subfigure}
    \caption{Tucker components $\mA^{(1)},\mA^{(2)},\mA^{(3)},\mathcal{G}$ (Left to Right respectively) estimated by our method with ranks $(2,2,3)$ for the tensor from Figure \ref{fig:setting1_D_mixed}. Note that the factor $\mA^{(1)}$ shows a clear clustering pattern. The estimated factor $\mA^{(2)}$ also exhibits two clusters. Factor $\mA^{(3)}$ shows the topics. The core tensor $\mathcal{G}$ highlights interactions between the clusters of multiple modes. Each slice in the last subfigure represents the clusters derived from $\mA^{(2)}$. The reconstruction error is $\|\mathcal{\hat D}-\mathcal{D}\|_1=22.687$}
    \label{fig:sim1_ours method}
\end{figure}

\begin{figure}[ht]
    \centering
    \begin{subfigure}[b]{0.10\textwidth}
    \includegraphics[width=\textwidth]{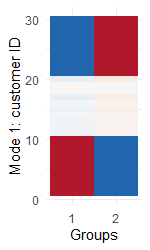}
    \end{subfigure}
    \begin{subfigure}[b]{0.10\textwidth}
    \includegraphics[width=\textwidth]{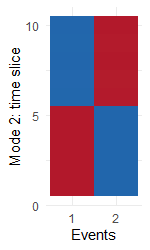}
    \end{subfigure}
    \begin{subfigure}[b]{0.10\textwidth}
    \includegraphics[width=\textwidth]{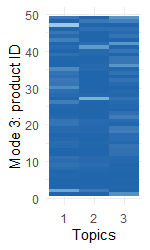}
    \end{subfigure}
    \begin{subfigure}[b]{0.22\textwidth}
    \includegraphics[width=\textwidth]{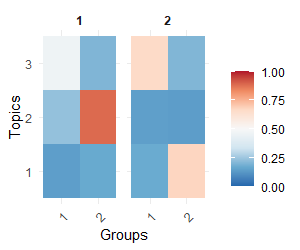}
    \end{subfigure}
    \caption{Tucker components $\mA^{(1)},\mA^{(2)},\mA^{(3)},\mathcal{G}$ (Left to Right respectively) estimated by Nonnegative Tucker decomposition (NTD) with ranks $(2,2,3)$. The reconstruction error is $\|\mathcal{\hat D}-\mathcal{D}\|_1=91.543$. }
    \label{fig:sim1_NTD}
\end{figure}

\begin{figure}[ht]
    \centering
    \begin{subfigure}[b]{0.10\textwidth}
    \includegraphics[width=\textwidth]{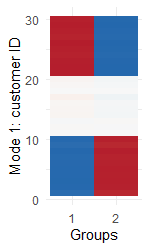}
    \end{subfigure}
    \begin{subfigure}[b]{0.10\textwidth}
    \includegraphics[width=\textwidth]{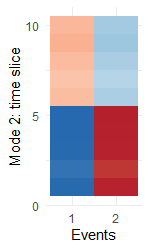}
    \end{subfigure}
    \begin{subfigure}[b]{0.10\textwidth}
    \includegraphics[width=\textwidth]{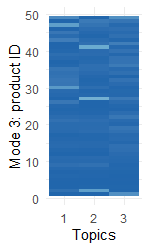}
    \end{subfigure}
    \begin{subfigure}[b]{0.22\textwidth}
    \includegraphics[width=\textwidth]{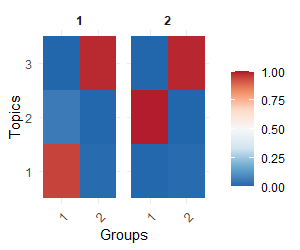}
    \end{subfigure}
    \caption{Tucker components $\mA^{(1)},\mA^{(2)},\mA^{(3)},\mathcal{G}$ (Left to Right respectively) estimated by Tensor LDA with ranks $(2,2,3)$. The reconstruction error is $\|\mathcal{\hat D}-\mathcal{D}\|_1=27.262$. }
    \label{fig:sim1_TLDA}
\end{figure}

\begin{figure}[ht]
    \centering
    \begin{subfigure}[b]{0.10\textwidth}
    \includegraphics[width=\textwidth]{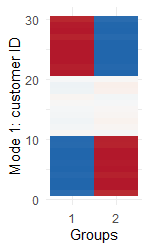}
    \end{subfigure}
    \begin{subfigure}[b]{0.10\textwidth}
    \includegraphics[width=\textwidth]{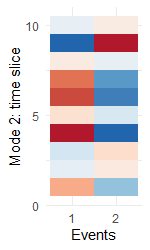}
    \end{subfigure}
    \begin{subfigure}[b]{0.10\textwidth}
    \includegraphics[width=\textwidth]{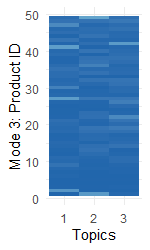}
    \end{subfigure}
    \begin{subfigure}[b]{0.22\textwidth}
    \includegraphics[width=\textwidth]{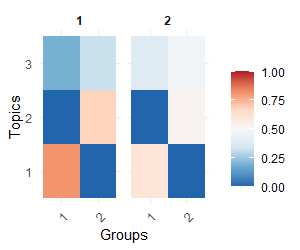}
    \end{subfigure}
    \caption{Tucker components $\mA^{(1)},\mA^{(2)},\mA^{(3)},\mathcal{G}$ (Left to Right respectively) estimated by Hybrid-LDA with ranks $(2,2,3)$. The reconstruction error is $\|\mathcal{\hat D}-\mathcal{D}\|_1=82.703$. }
    \label{fig:sim1_hybridLDA}
\end{figure}

\par Our first experiment highlights the importance of Tucker decomposition in topic modeling for multidimensional data. In particular, this approach allows us to understand interactions among latent factors (e.g. interactions between clusters of indices and topics).
This approach offers a more comprehensive understanding of the latent structure within the data.

Figure \ref{fig:boxplot} further demonstrates that our HOSVD-based method is deterministic, ensuring consistent results across multiple runs. In contrast, tensor LDA, hybrid LDA, and NTD are sensitive to initialization and prone to getting stuck in local optima, often requiring multiple runs with different initializations to find a global solution. More specifically, hybrid LDA exhibits the largest variance in reconstruction error, as  the estimation of the matrix 
 $\mW^{(3)}$ is critical for determining  $\mA^{(1)}$, $\mA^{(2)}$. As shown in Appendix \ref{app:syn:tucker}, if the estimation of $\mW^{(3)}$  is sufficiently accurate, the rest of the procedure to estimate $\mA^{(1)}$ and $\mA^{(2)}$ can effectively capture the data patterns. However, since $\mW^{(3)}$ is nondeterministic in LDA, Hybrid-LDA often exhibits higher variance. The instability of tensor LDA is also evident, with significant variation in the reconstruction error and outlier points in Figure \ref{fig:boxplot}. Finally, while NTD is more stable compared to Bayesian methods, it exhibits the largest average reconstruction error.

\begin{figure}[H]
    \centering
    \includegraphics[width=0.5\linewidth]{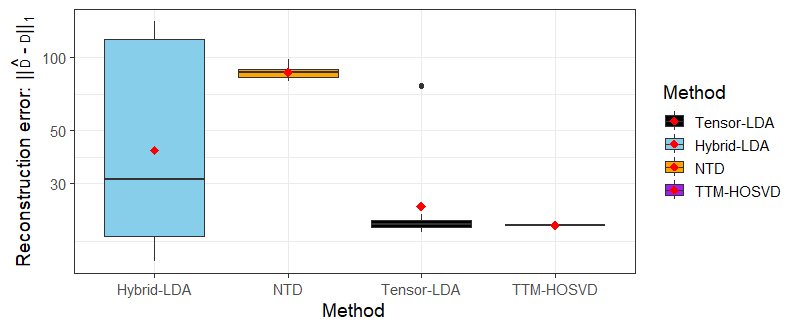}
    \caption{Boxplot of the reconstruction errors over 30 runs on the same dataset from Figure \ref{fig:setting1_D_mixed}. The middle line is the median value over 30 runs. The red point refers to the mean value.}
    \label{fig:boxplot}
\end{figure}
We provide the results for STM in Appendix \ref{app:syn:tucker}, since it is a variation on the Hybrid LDA approach. Overall, STM offers a much less variable alternative to LDA, but still yields a higher reconstruction rate than our method.

\subsection{Varying the data dimensions}

We now vary the data dimensions and assess the performance of our method and benchmarks in estimating the different components of the model. Figure~\ref{fig:synthetic_effect_q1} shows the estimators' respective errors as a function of the first dimension $N^{(1)}$ (with $N^{(2)}$ being equal to $N^{(1)}$) for various dictionary sizes, and Figure~\ref{fig:synthetic_effect_r} shows the error as a function of the dictionary size $R$ for various dimensions $(N^{(1)}, N^{(2)}).$ Overall, we observe that our method consistently performs better than all other benchmarks, especially in estimating the latent factors $\mA^{(a)}, a \in \{1,2,3\}.$ The naive NTD approach does not succeed in correctly estimating the latent factors. The performance of the tensor LDA method that we suggest in this paper seem to drop as the size of the dictionary increases. Overall, our method seems to be the most robust to varying dictionary sizes.

\begin{figure}[H]
    \centering
    \includegraphics[width=\linewidth]{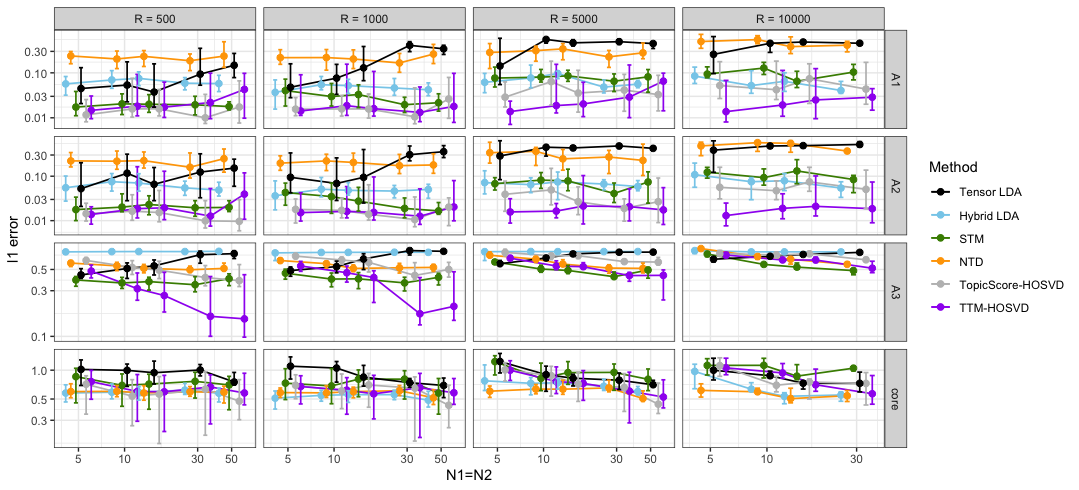}
    \caption{TTM's performance as a function of $N^{(1)}$, compared to various benchmarks. $M$ is here set to 10,000 and $K^{(1)}=K^{(2)}=2$ while $K^{(3)}=4$.}
    \label{fig:synthetic_effect_q1}
\end{figure}

\begin{figure}[H]
    \centering
    \includegraphics[width=\linewidth]{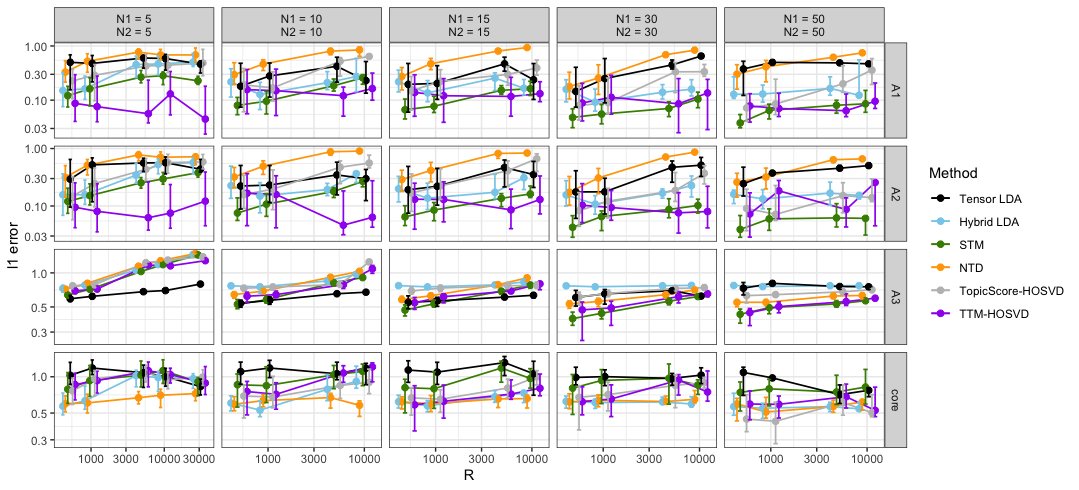}
    \caption{TTM's performance as a function of $R$, compared to various benchmarks. $M$ is here set to 10,000 and $K^{(1)}, K^{(2)}=2$ while $K^{(3)}=4$.}
    \label{fig:synthetic_effect_r}
\end{figure}


We also evaluate the effect of the number of words per document $M$ in Figure~\ref{fig:synthetic_effect_M}, Our method and STM seem to have steeper slopes, indicating a more efficient use of the information than in other methods.

\begin{figure}[H]
    \centering
    \includegraphics[width=0.8\linewidth]{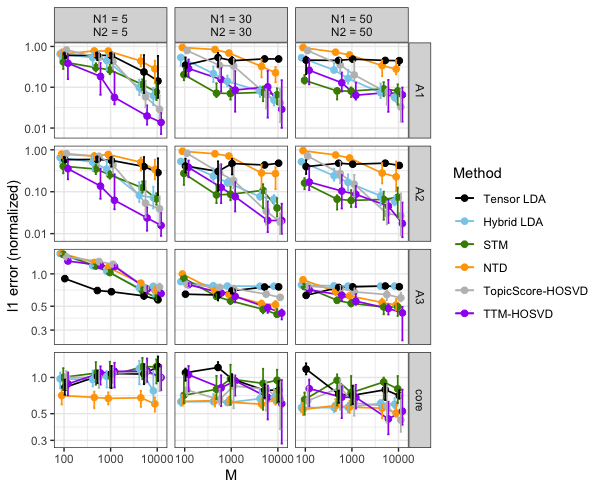}
    \caption{TTM's performance as a function of $M$, compared to various benchmarks. $R$ is here set to 500 and $K^{(1)}=K^{(2)}=2$ while $K^{(3)}=4$.}
    \label{fig:synthetic_effect_M}
\end{figure}

\section{Real data experiments}\label{sec: real data experiments}
In this section, we apply our method to real-world datasets, specifically focusing on text, microbiome as well as marketing data  as outlined in the Introduction (Section \ref{sec: introduction}).



\subsection{ArXiv's paper Analysis}
For our first experiment, we consider a corpus of 5,094 research abstracts belonging to (at least) one of four categories: Computer Science, Mathematics, Physics, and Statistics from 2007 to 2024. \footnote{The data is available on Kaggle at \url{https://www.kaggle.com/datasets/Cornell-University/arxiv}} after some pre-processing of the data (See detail in Appendix \ref{appsec:arvix}). The structure of the obtained tensor is of dimensions (283, 18, 8837) where the first dimension represents sampled articles, the second dimension corresponds to the publication years, and the third dimension represents the vocabulary size. In this study, we aimed to analyzed scholarly articles from the arXiv repository to investigate trends in theoretical and applied research over 18 years. We run our Tensor Topic Modeling method of the dataset with ranks (2,3,4). (As described in \cite{ke2022using,klopp2021assigning,tran2023sparse}, the spectrum of $\mD^{(a)}$ is useful for estimating $K^{(a)}$. We note that it is often possible to determine the eigenvalue cutoff by inspecting the scree plot of $\mD^{(a)}$'s eigenvalues. Figure \ref{fig:articles: singular value} in the Appendix displays the scree plot for each mode with different values of $K^{(a)}$.)

\paragraph{Analysis and Discussion} 
We show the top 10 words for each of the four topics.  The results are presented in Figure \ref{subplot:arxiv ours:A3}, and \ref{fig: arxiv topic others}. For the topics of Computer Science, Math and Physics, the top 10 most representative words produced by our method agree with 50\%, 60\%, and 70\% of NTD's most representative words in the corresponding topics, respectively. Upon closer inspection, we find that some of the words produced by our method (such as ``quantum'') are more indicative of Physics. In contrast, the Bayesian methods show less agreement with the top 10 words derived from our method. Specifically, most of the top 10 words produced by Hybrid LDA are generic and could easily appear in other categories, reducing the specificity of the topics. With the inclusion of external information, STM improves upon LDA by extracting more exclusive words relevant to the topics, producing results more similar to those from our method. Meanwhile, Tensor-LDA includes many non-typical words such as ``mesh",``talk'', ``pinpoint", ``institutions", and ``beginequation" which are less representative of the core topics. This highlights the strength of our method in capturing more relevant and representative words for each topic, especially compared to other decomposition techniques.

\par To investigate the trends in theoretical and applied research across subject topics from 2007 to 2024, we further examine the core tensor in Figure \ref{subplot:arxiv ours:core} (with exact values provided in Table \ref{arvix table: core: ours} in the Appendix). Within the theoretical group, we observe that the proportion of statistics research increases over the time slice, although mathematics still holds a significant proportion. However, the proportion of mathematics gradually decreases over time (from 76\% to 69\% as shown in Table \ref{arvix table: core: ours}), suggesting a shift of focus towards statistics in recent years. Additionally, computer science is categorized purely to application, and its overall proportion continues to increase over the years. This shows the surge of computer science and ongoing technological advancements in recent years. Meanwhile, physics is categorized as purely theoretical and the proportion is stable over the time. Due to space limitations, we include the core values of the benchmarks in Appendix \ref{appsec:arvix}.

\begin{figure}[H]
    \centering
    \begin{subfigure}[b]{0.25\textwidth}
    \includegraphics[width=\textwidth]{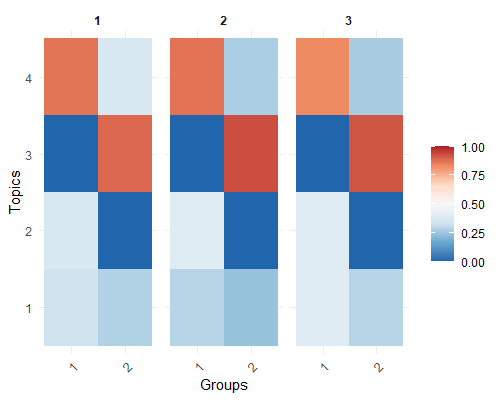}
    \subcaption{Core} \label{subplot:arxiv ours:core}
    \end{subfigure}\hspace{1.0cm}
    \begin{subfigure}[b]{0.22\textwidth}
\includegraphics[width=\textwidth]{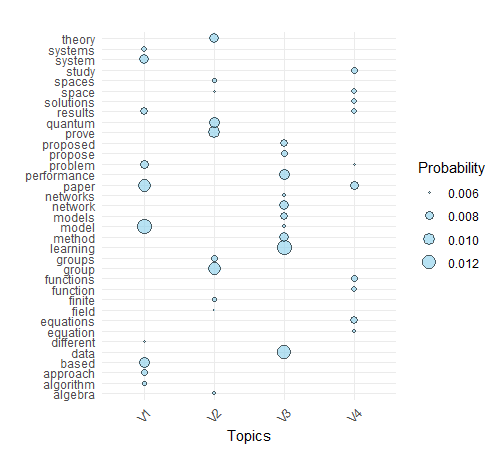} \subcaption{$\mA^{(3)}$}\label{subplot:arxiv ours:A3}
    \end{subfigure}
    \caption{The Tucker components $\mA^{(3)}$ (right) and the core tensor (left) are derived from our method (TTM-HOSVD). The values in the core tensor have been sqrt-transformed in the colorbar to enhance visual interpretation. Each dot represents the frequency of the corresponding words in a topic, with the size of the dots reflecting the frequency of the items. Based on the analysis, we can refer to the topics as: v1: Statistics, v2: Physics, v3: Computer Science, v4: Math. Group 1 refers to the theoretical cluster, while Group 2 corresponds to the application cluster (Group alignment using \textit{cosine} distance metric).}
    \label{fig: arxiv ours}
\end{figure}

\begin{figure}[H]
    \centering
    \begin{subfigure}[b]{0.2 \textwidth}
    \includegraphics[width=\textwidth]{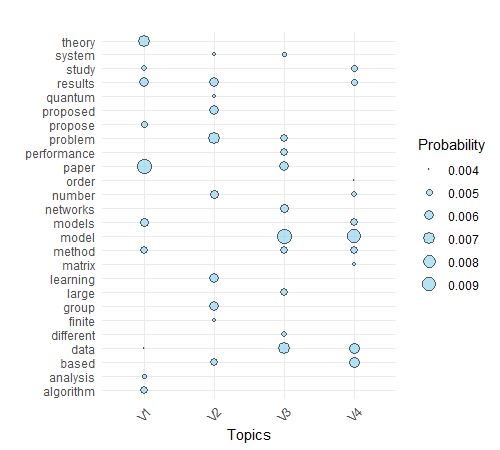}
    \subcaption{Hybrid-LDA}\label{fig: arxiv topic lda}
    \end{subfigure}
    \begin{subfigure}[b]{0.2 \textwidth}
\includegraphics[width=\textwidth]{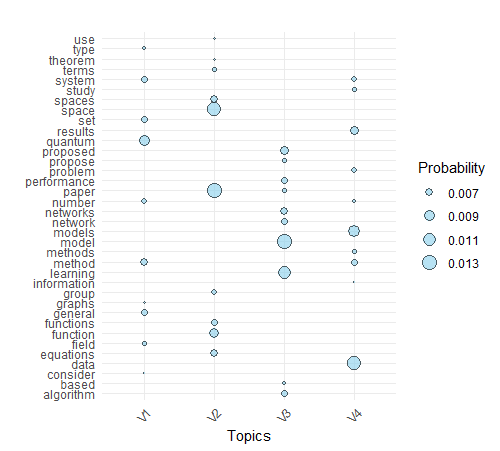} \subcaption{STM}\label{fig: arxiv topic stm}
    \end{subfigure}
        \begin{subfigure}[b]{0.2\textwidth}
\includegraphics[width=\textwidth]{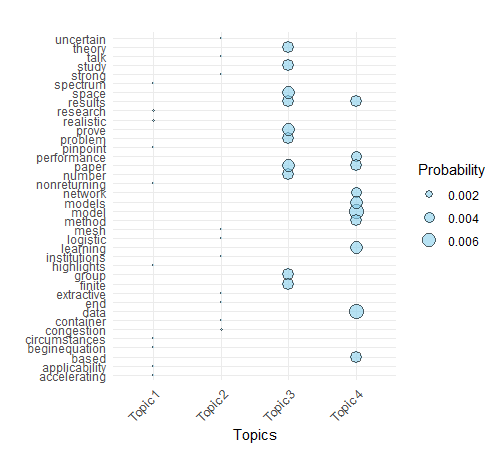} \subcaption{Tensor-LDA}\label{fig: arxiv topic tlda}
    \end{subfigure}
      \begin{subfigure}[b]{0.2 \textwidth}
\includegraphics[width=\textwidth]{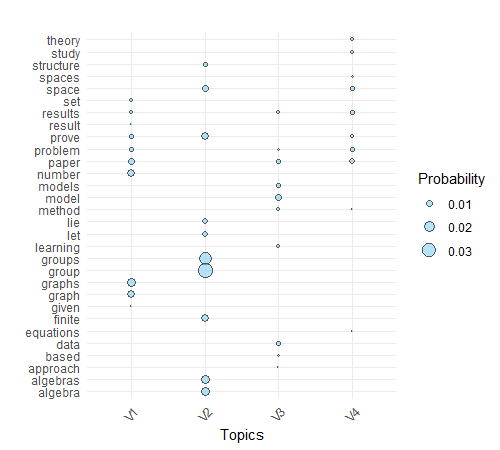} \subcaption{NTD}\label{fig: arxiv topic NTD}
    \end{subfigure}
    \caption{Tucker components $\mA^{(3)}$ derived from benchmarks. It is required to permute the topics for comparison. Hybrid-LDA: v1: Statistics, v2: Physics, v3: Computer Science, v4: Math; STM: v1: Physics, v2: Math, v3: Computer Science, v4: Statistics; Tensor-LDA: v1: Physics, v2: Statistics, v3: Math, v4: Computer Science; NTD: v1: Statistics, v2: Physics, v3: Computer Science, v4: Math.}
    \label{fig: arxiv topic others}
\end{figure}

\subsection{Sub-communities of the vaginal microbiota in non-pregnant women}
In this subsection, 
we revisit the analysis of women vaginal microbiota presented by Symul et al. \cite{symul2023sub}.
This dataset consists of vaginal samples collected daily from 30 nonpregnant women enrolled at the University of Alabama, Birmingham. Each sample is represented here as a vector of counts of rRNA amplicon sequence variants (ASVs), agglomerated by taxonomic assignment (see \cite{symul2023sub}). 

As in \cite{symul2023sub}, we propose to analyze the data using topic models; In this context, each sample from individual $i$ on day $t$ can be viewed as a document, with ASV corresponding to words ($R = 1,338$). We construct the data tensor using a subset of vaginal samples (since many participants do not have a sample for every single day of the cycle) and align the samples based on the patient's menstrual cycle. Specifically, for each participant, we selected 11 samples corresponding to various stages of their menstrual cycle: 2 days from the menstrual phase, 3 days from the follicular phase, 2 days from the ovulation phase, and 4 days from the luteal phase. Keeping only participants for whom this construction does not produce missing entries, this results in a three-mode tensor of size $24\times 11\times 1,338$, as shown in the left plot of Figure \ref{fig:microbiota}.

\begin{figure}[H]
    \centering
    \begin{subfigure}[b]{0.35\textwidth}
    \includegraphics[width=\textwidth]{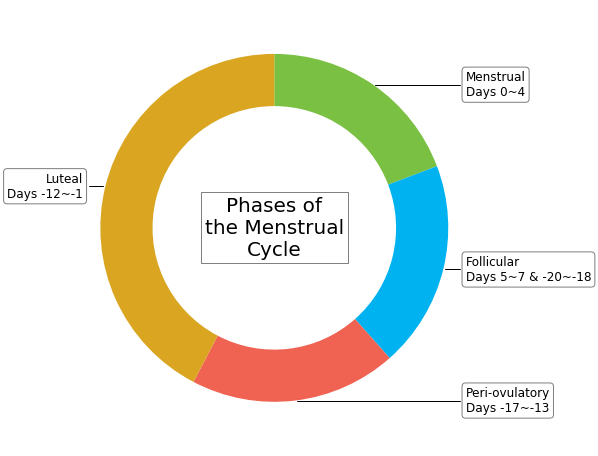}
    \end{subfigure}
    \hspace{1.5cm}
    \begin{subfigure}[b]{0.45\textwidth}
    \includegraphics[width=\textwidth]{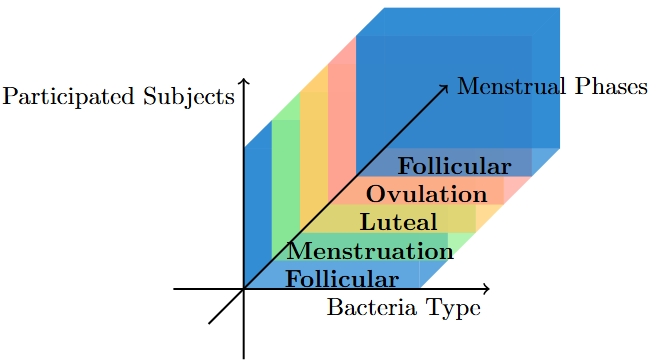}
    \end{subfigure}
 
\caption{Visualization of the construction of the vaginal microbiome tensor ($24\times 11\times 1,338$). Left: Menstrual cycle information. Cycles (28 days in total) were then aligned from day $-20$ (i.e. 20 days before menses) to day $+7$ (i.e. 7 days after the first day of menses). Right: The tensor represents data collected from 24 participants, each with vaginal samples taken across 11 days of the menstrual cycle. The participants include 3 individuals of Hispanic$/$Latino origin, 4 who identify as White, and 17 who identify as African American. The participants' ages range from 19 to 45.}
       \label{fig:microbiota}
\end{figure}

The first mode of the resulting tensor indexes the participants, the second indexes the cycle stage, and the third mode represents the different ASVs. Additionally, we have demographic information for each participant, including race and age.


\paragraph{TTM-HOSVD provides more accurate representation of vaginal microbiotas.} To evaluate the topics estimated by our method, we propose using two techniques: {\it (a) the refinement and coherence} of the topic hierarchy computed by the package \texttt{alto} \cite{fukuyama2021multiscale} as a proxy for topic stability. Developed as a diagnostic and visualization tool to select an appropriate number of topics for LDA, \texttt{alto} aligns topics as $K$ grows based on the document assignments $\mW$ in order to visualize the induced topic hierarchy, as well as the progressive splitting or recombination of topics as $K$ increases. In this setting, topics with better refinement (few ancestor per topic) and coherence (persistence of the topic as $K$ grows) are considered more stable; and {\it (b) the median topic resolution \cite{tran2023sparse}}, computed by splitting the dataset in half, fitting the model on each half, and evaluating the agreement between the recovered topics using a cosine similarity . 

Figure~\ref{fig: vaginal:topic refinement}  shows the refinement and coherence of the topics recovered by our method as $K^{(3)}$ increases, and compares them with those obtained by the hybrid LDA, STM, NTD and Tensor LDA approaches described in the previous section. Our method seems to provide topics with higher refinement (fewer ancestors per topic) and higher coherence. For instance, figure \ref{subfig: vaginal: refinement: hybrid-LDA} shows that hybrid LDA produces topics that have a tendency of recombining (e.g.  topic 7 merges with 4 at $K^{(3)}=6$).  The other figures (\ref{subfig: vaginal: refinement: ntd} and \ref{subfig: vaginal: refinement: tensor LDA}) show even worse rates of topic recombination as $K^{(3)}$ increases. 




Figure \ref{fig: vaginal: topic resolution} shows that our method also exhibits significantly better results in terms of median topic resolution than both LDA and NTD for up to $K^{(3)}=10$ topics -- at which point NTD begins to outperform all other methods. NTD relies on the minimization of the KL divergence between the observed data and the low-rank approximation --- a procedure which could perhaps allow it to better capture fine-grained patterns in the data, particularly when dealing with a larger number of topics. However, this comes at a significant computational cost compared to our method. 

The results for STM (Structural Topic Model)\cite{roberts2014structural} are presented in the Appendix  (Section \ref{sec:app:real_exp:vaginal}). However, this method performs the worst in both the topic refinement and topic resolution plots, which could indicate that the method places too much emphasis on enhancing relationships derived from the meta information. 


\begin{figure}[ht]
    \centering
    \begin{subfigure}[b]{0.45\textwidth}
    \includegraphics[width=\textwidth]{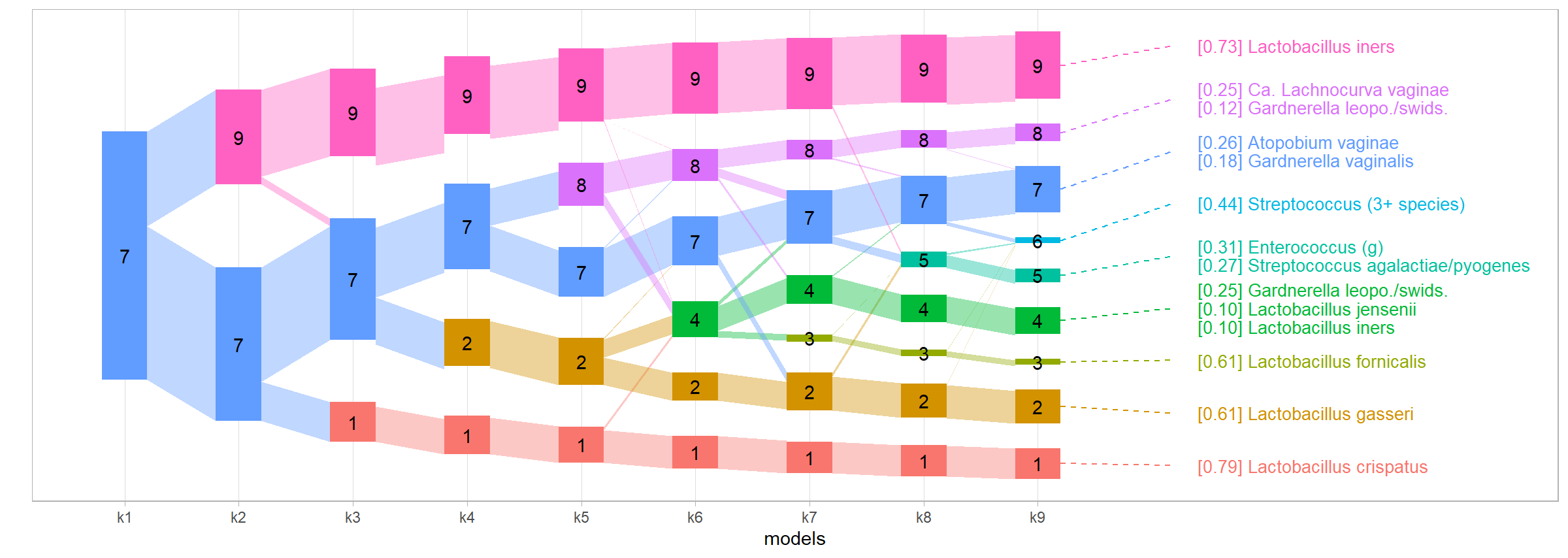}
    \subcaption{Our methods}\label{subplot:vaginal:resolution:ours}
    \end{subfigure}
    \begin{subfigure}[b]{0.45\textwidth}
    \includegraphics[width=\textwidth]{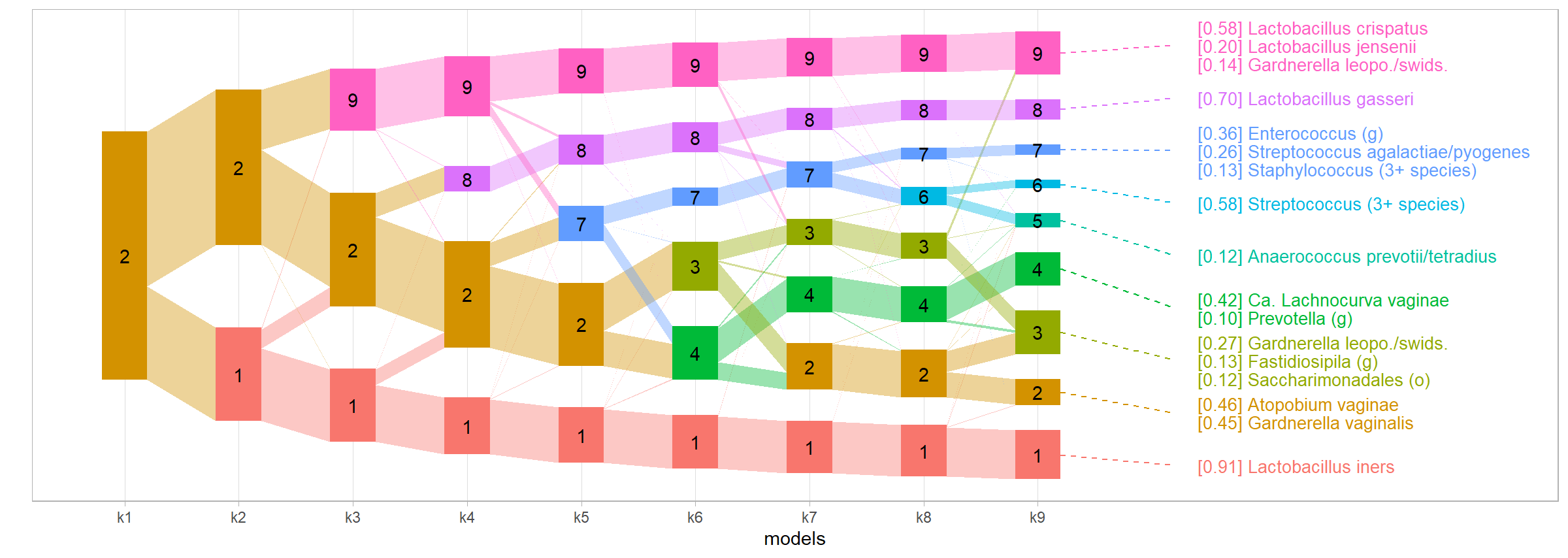}
\subcaption{Hybrid-LDA}\label{subfig: vaginal: refinement: hybrid-LDA}
    \end{subfigure}
    \\
    \begin{subfigure}[b]{0.4\textwidth}
    \includegraphics[width=\textwidth]{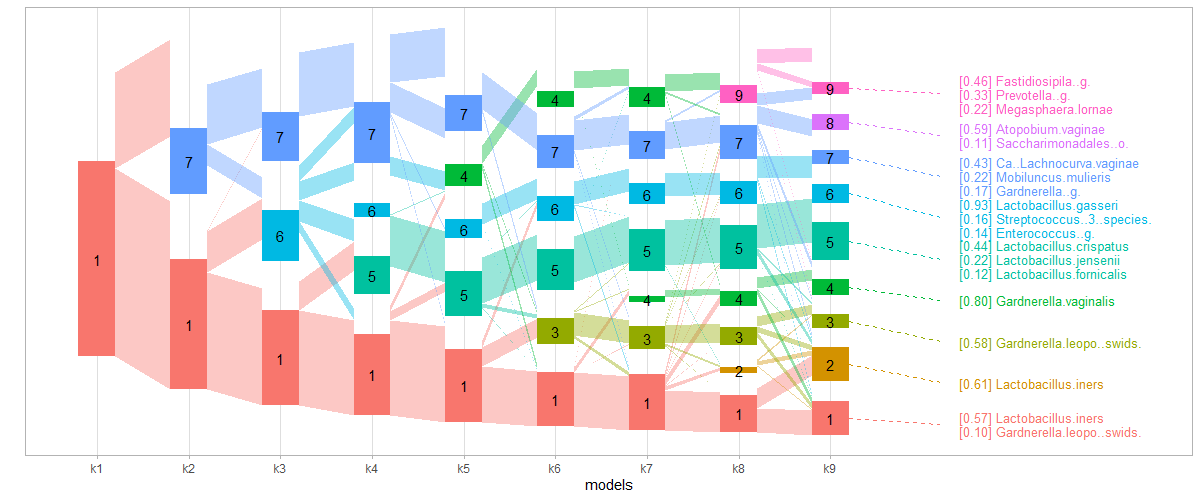} 
    \subcaption{NTD}\label{subfig: vaginal: refinement: ntd}
    \end{subfigure}
    \begin{subfigure}[b]{0.45\textwidth}
    \includegraphics[width=\textwidth]{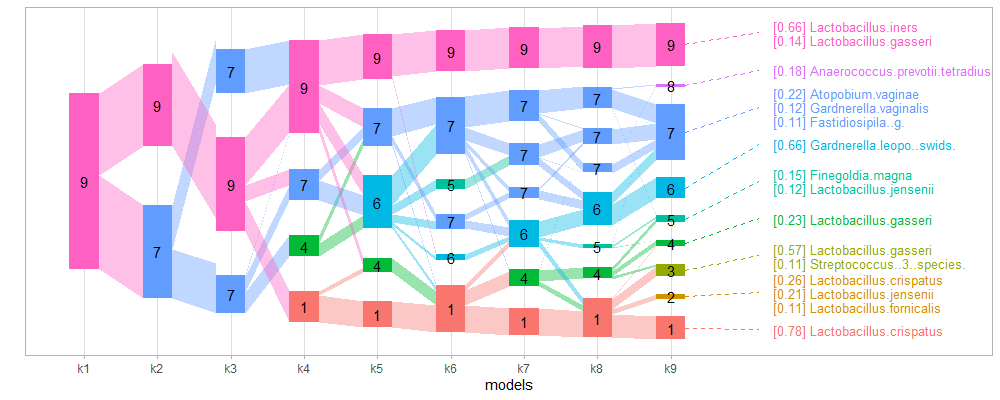}
\subcaption{Tensor-LDA}\label{subfig: vaginal: refinement: tensor LDA}
    \end{subfigure}
 
    \caption{Comparison of the refinement and coherence of topics recovered using our method (a), NTD (c) and LDA based methods (b),(d), in which we fix $K^{(1)}=3, K^{(2)}=4$.}
    \label{fig: vaginal:topic refinement}
\end{figure}

\begin{figure}[ht]
    \centering
    \begin{subfigure}
{0.5\textwidth}
    \includegraphics[width=\textwidth]{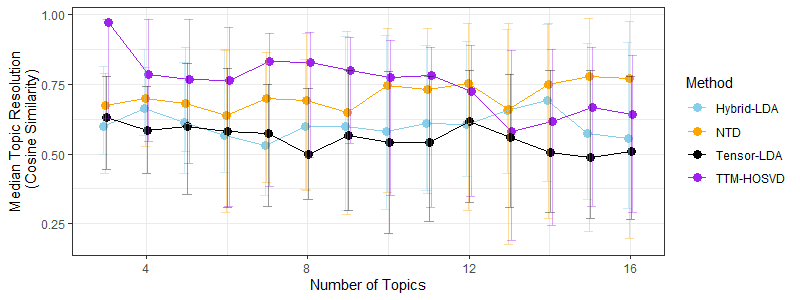}
    \subcaption{Median Topic Resolution}
    \label{fig: vaginal: topic resolution}
    \end{subfigure}
     \begin{subfigure}
{0.35\textwidth}
    \includegraphics[width=\textwidth]{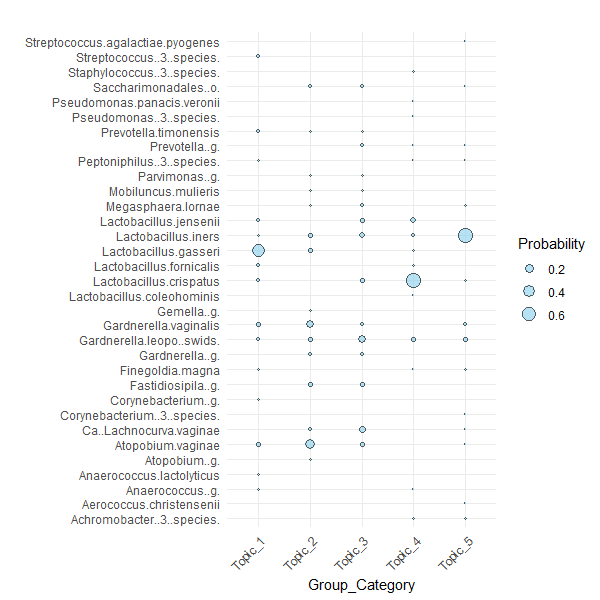}
    \subcaption{$\mA^{(3)}$ by TTM-HOSVD}\label{fig:vaginal:ous A3}
    \end{subfigure}
 \caption{Left: Median Topic Resolution as a function of $K^{(3)}$ on the vaginal ecosystem data of Symul et al \cite{symul2023sub}. Vertical error bars represent the interquartile range for the average topic resolution scores over 25 trials. Right: Topic composition for $K^{(3)}$ = 5 topics. The proportion of each species ($y$-axis) within each topic ($x$-axis) is encoded by the dot size.}
    
\end{figure}
 
\paragraph{TTM-HOSVD reveals associations between topic composition and menstrual phases.} 
We here propose to test whether the recovered latent factors from $\mA^{(1)}$, $\mA^{(2)}$ and $\mW^{(3)}$ are associated with relevant covariates, such as demographic characteristics and menstrual phases. 
Following the procedure in \cite{symul2023sub}, we employed Dirichlet regression to analyze the estimated document-topic matrix $\hat \mW^{(3)}=\mathcal{M}_3(\mathcal{\hat G})\left(\hat \mA^{(1)}\times\hat \mA^{(2)} \right)^\top$. Our HOSVD-based approach (Figure \ref{subplot:dirichlet ours:W}) reveals that participants' race is significantly associated with topic proportions, a result consistent with that of the LDA-based approach of the original paper (shown in Figure 3(c) in \cite{symul2023sub}). However, compared to matrix LDA, our method identifies more significant points, demonstrating a more pronounced relationship between vaginal microbiota composition, demographic characteristics, and menstrual phases. 
(For comparison, we also reproduced the LDA-based results, which are presented in Appendix Figure \ref{fig: dirichlet others method}.) 

To further explore the ability of the data to project participants and cycle days into lower-dimensional representations that reflect demographic characteristics and menstrual phases, we analyzed the heatmaps of $\mA^{(1)}$ and $\mA^{(2)}$. The heatmap of $\mA^{(2)}$ estimated by our method in Figure \ref{subplot:vaginal_main:ours:heatmap:A2} shows that samples from the same menstrual phase exhibit similar loadings, successfully recovering the underlying menstrual phase structure. Additionally, the heatmap of  $\mA^{(1)}$ in Figure \ref{subplot:vaginal: heatmap: ours:A1} indicates that the groups are not solely associated with a single demographic feature. This finding aligns with the Dirichlet regression results, which reveal that the lower-dimensional representation of subjects is significantly associated with both age and racial information.

Figures \ref{subplot:dirichlet ours:A1} and \ref{subplot:dirichlet ours:A2} further validate the effectiveness of Tucker decomposition in this dataset. Significant associations in $\mA^{(1)}$ and $\mA^{(2)}$ are concentrated around their respective covariates, demonstrating that topic composition is conditionally independent given the lower-dimensional factors representing time (menstrual phases) and demographic features. These findings underscore the utility of Tucker decomposition in uncovering topics in multi-dimensional data, providing stronger and more interpretable signals compared to traditional matrix LDA-based approaches.


\begin{figure}[H]
    \centering
    \begin{subfigure}[b]{0.285\textwidth}
    \includegraphics[width=\textwidth]{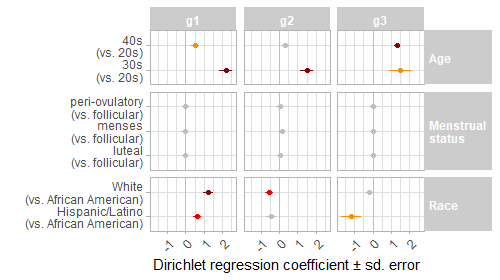}
    \subcaption{$\mA^{(1)}$} \label{subplot:dirichlet ours:A1}
    \end{subfigure}
    \begin{subfigure}[b]{0.285\textwidth}
    \includegraphics[width=\textwidth]{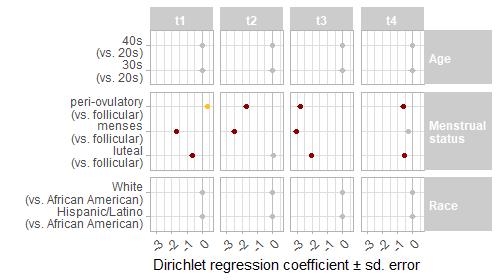}
\subcaption{$\mA^{(2)}$}\label{subplot:dirichlet ours:A2}
    \end{subfigure}
    \begin{subfigure}[b]{0.38\textwidth}
\includegraphics[width=\textwidth]{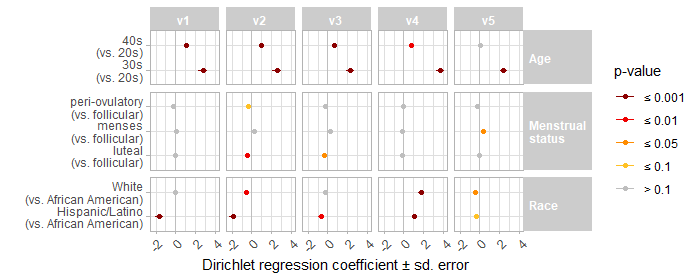}
    \subcaption{$\mW^{(3)}$}\label{subplot:dirichlet ours:W}
    \end{subfigure}
    \caption{Dirichlet regression estimated coefficients (x-axis) quantifying the associations between race, menstrual status, age (y-axis) and categorical proportions (horizontal panels). Colours indicate the statistical significance. The left two plots are corresponds to the categorical proportion for mode matrices $\mA^{(1)}$ and $\mA^{(2)}$, respectively, from HOSVD using our method. The subplot (c) is for estimated document topic matrix $\hat \mW^{(3)}$. Each column of the plot indicates the corresponding lower-dimensional clusters for each mode.}
    \label{fig: dirichlet}
\end{figure}
\begin{figure}[H]
    \centering
       \begin{subfigure}[b]{0.40\textwidth}
    \includegraphics[width=\textwidth]{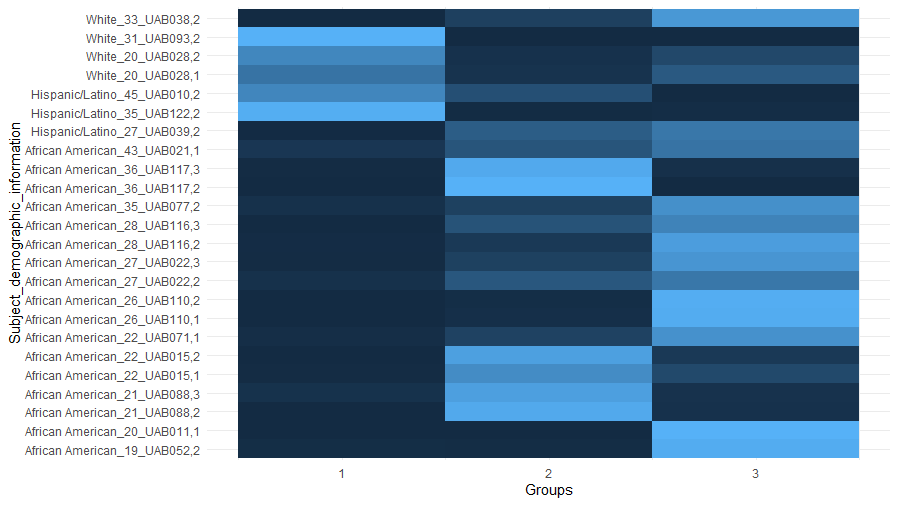}
    \subcaption{$\mA^{(1)}$} \label{subplot:vaginal: heatmap: ours:A1}
    \end{subfigure} \hspace{1.2cm} \begin{subfigure}[b]{0.31\textwidth}\includegraphics[width=\textwidth]{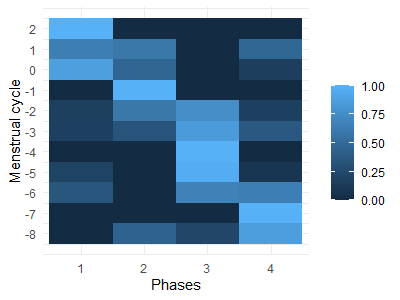}\subcaption{$\mA^{(2)}$}\label{subplot:vaginal_main:ours:heatmap:A2}\end{subfigure}
    \caption{Matrix components $\mA^{(1)}$ and  $\mA^{(2)}$ derived from HOSVD. The y-axis of $\mA^{(1)}$ labels each subject's racial and age information, with ``UAB0XX" denoting the subject ID and the following number indicating the menstrual cycle number. The y-axis of $\mA^{(2)}$ corresponds to the four menstrual phases: luteal, menstrual, follicular, and ovulatory, respectively. }
    \label{fig:vaginal A1 and A2 ours heatmap}
\end{figure}

\paragraph{TTM-HOSVD highlights dynamic changes in microbiome composition:}

Figure \ref{subplot:vaginal_main:ours:heatmap:core} displays the heatmap of the core tensor derived from TTM-HOSVD, revealing key patterns in microbiome topic activity throughout the menstrual cycle. Notably, Topic 4, which is dominated by Lactobacillus crispatus (as shown in Figure \ref{fig:vaginal:ous A3}), emerges as the most active sub-community across the menstrual cycle in Group 3. This dominance is further reflected in Figure \ref{subplot:vaginal_main:uab022}, where Lactobacillus crispatus exhibits heightened activity in Subject UAB022, who is primarily associated with Group 3. The figure also highlights the reduced presence of microbiome Topic 2 in Group 3, as observed in UAB022. Additionally, Microbiome Topic 1 shows a noticeable decline during menstrual phases 3 and 4, which correspond to the luteal and ovulation phases. Figure \ref{subplot:vaginal_main:uab022} further reveals that the proportion of Lactobacillus gasseri (the dominant bacterium in microbiome Topic 1) peaks during the menstrual phase for Subject UAB022.

These findings demonstrate that our Tucker decomposition-based method effectively captures meaningful associations between microbiome topic activity.

\begin{figure}[H]
    \centering
  
 \begin{subfigure}[b]{0.4\textwidth}
\includegraphics[width=\textwidth]{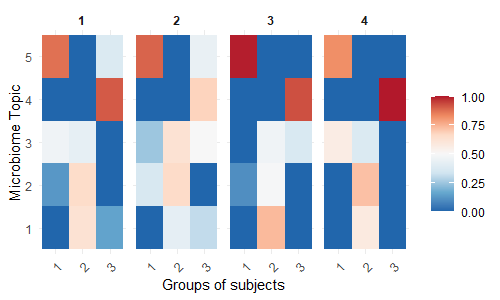}
\subcaption{Core tensor}\label{subplot:vaginal_main:ours:heatmap:core}
\end{subfigure} 
\hspace{1.0cm}
\begin{subfigure}[b]{0.37\textwidth}
\includegraphics[width=\textwidth]{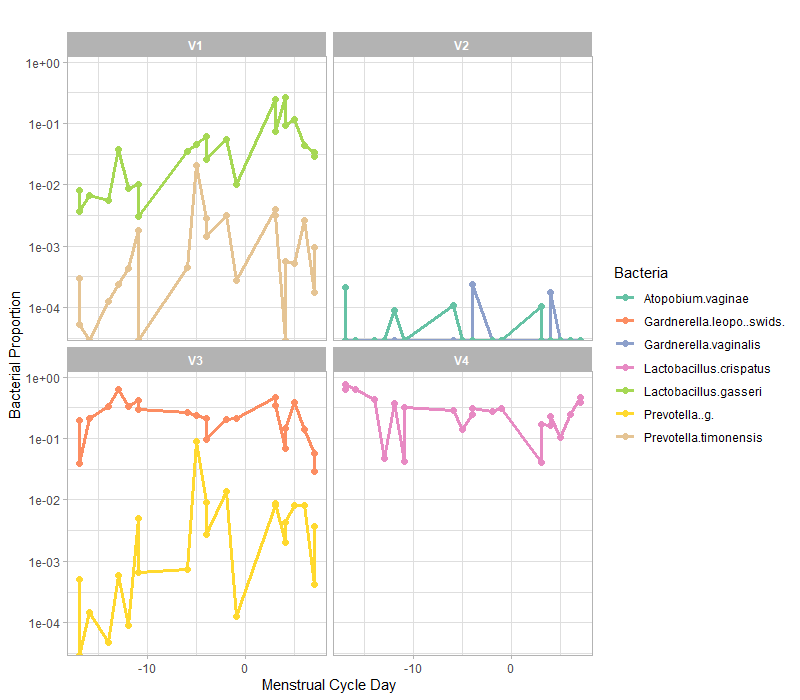}
\subcaption{UAB022}\label{subplot:vaginal_main:uab022}
\end{subfigure}
    \caption{Left: Core tensor derived from our method (TTM-HOSVD). Right: Some selected the bacterial proportions over the menstrual days for Subject ``UAB022" who is likely categorized to Group 3 in Subject. The bacteria have been categorized into topics based on the maximum value of the proportion in $\mA^{(3)}$ using the TTM-HOSVD method. The y-axis has been log-scaled. Each facet plot corresponds to a microbiome topic, with the topic numbers matching those in Figure \ref{fig:vaginal:ous A3}.}
\label{fig:vaginal:ours:Tucker heatmap}
\end{figure}

To understand the gap in performance between the benchmarks and our method, we visualize the plots presented in this subsection for the benchmarks in Appendix \ref{sec:app:real_exp:vaginal}. These results show that our method outperforms the benchmarks in recovering the expected menstrual phases using matrix  $\mA^{(2)}$, and capturing the patterns in the core tensor.

\subsection{Market-Basket Analysis}

Finally, we propose deploying our method to the analysis of market data. More specifically, we consider the analysis of the Dunnhumby dataset\footnote{The dataset is publicly available at the following url: \url{https://www.kaggle.com/datasets/frtgnn/dunnhumby-the-complete-journey}.}, a public dataset tracking the purchases of 2,500 households who frequently shop over a two-year period in a retail store. After a light preprocessing of the data (see details in Appendix~\ref{app:market}), we retain a total of 176 households, whose purchases are observed at 26 time points (corresponding to sets of two weeks) in the year. Multiple purchases in a given period of two weeks are aggregated, yielding baskets (our ``documents'') consisting of an average 10 items.   The dictionary here consists of a list of 2,102 possible retail items. We run our Tensor Topic Modeling method of the dataset to extract $K^{(3)}=10$ retail ``topics'' (i.e. sets of items that are frequently bought together), as well as attempt classifying shoppers in $K^{(1)}=4$ categories and time points in $K^{(2)}=4$ clusters. The values for $K^{(1)}$ and $K^{(3)}$ were chosen using the scree plot of the corresponding matricized data, and the value of $K^{(2)}$ was chosen to match the 4 seasons.

\paragraph{Analysis of the results.}
The market-basket topics recovered by our TTM-HOSVD method are presented in Figure~\ref{fig:topics-hosvd}, where we show the top 30 words for each of the 10 topics. Due to space constraints, we present the results for STM,  Hybrid-LDA, as well as TopicScore-HOSVD in Appendix~\ref{app:market} (Figures~\ref{fig:market_topic_stm}, \ref{fig:market_lda2} and \ref{fig:market-topics-topic-score hosvd}).  We were unfortunately unable to run the Bayesian tensor LDA approach on a problem of this size, as this approach is unfortunately unable to scale to larger dictionaries and core sizes. We summarize the frequency of each category among the top 30 most common items in Table~\ref{tab:market} of the Appendix.

Overall, our method is able to detect interesting patterns. Topic 4, for instance, is highly enriched in the ``baby product" category (with items from this category making up as much as 40\% of the topic (Table~\ref{tab:market})), as it is enriched in items from the baby product category, as well as items from other categories targeted to children (e.g. baby juices or chocolate milk). Topic 9, on the other hand, is enriched in pet products (dog and cat food), with pet products representing  24\% of the topic. Topic 8 is the most enriched in snacks ( 28\% of the topic), while Topic 10 is enriched in Eggs \& Dairy (20\%), Topic 3 and 7 in ``fruits and vegetables'' (16.5\% and 15.8\% respectively), and Topic 6 in beverages (16\%). By contrast, the hybrid LDA model does not recover any "baby-product" or "beverage"-related topic, and most of its topics are either enriched in Candy\& Snacks, Meals,or Fruits and Vegetables (Figure~\ref{fig:market_lda2} of the Appendix) --- thereby shedding less light into customer-buying patterns.  STM fares better than the Hybrid-LDA approach, recovering a clear pet-product topic, but seems to mix Beverages and Baby product items (Figure~\ref{fig:market_topic_stm}).

Turning to the analysis of the latent factors along mode 2 ($\mA^{(2)}$), we observe that our proposed method returns cluster assignments that seem to be smooth in time (Figure~\ref{fig:market_time}). Topics 3 and 4, in particular, are complementary (first and second half of the year, respectively), while topic 1 seems to be at the midway point, and topic 2 serves as the transition between topic 3 and 4. LDA and STM, however, return patterns in the time dimensions, with Topic 4 spanning most of the year, and topic 2 and 3 spanning the extremities (Figure~\ref{fig:market_time}).

\begin{figure}[htbp]
    \centering
\includegraphics[width=\linewidth]{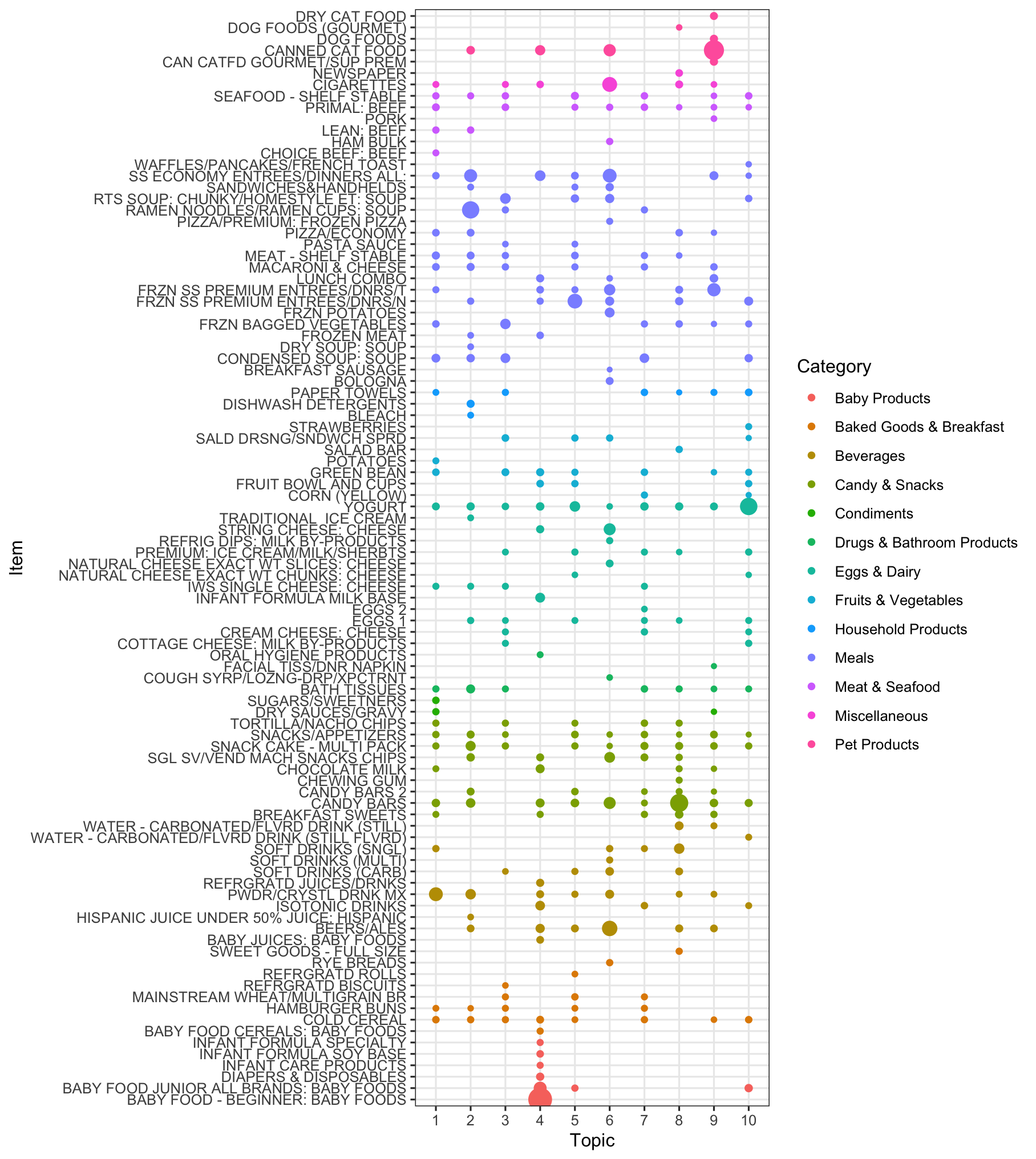}
    \caption{Topics recovered using our HO-SVD tensor topic modeling approach. Each dot represents the frequency of the corresponding item in the topic, with size reflecting the frequency of the item. Items are here divided in 13 categories.}
    \label{fig:topics-hosvd}
\end{figure}

\begin{figure}[ht]
    \centering
    \begin{subfigure}[b]{0.33\textwidth}
    \includegraphics[width=\linewidth]{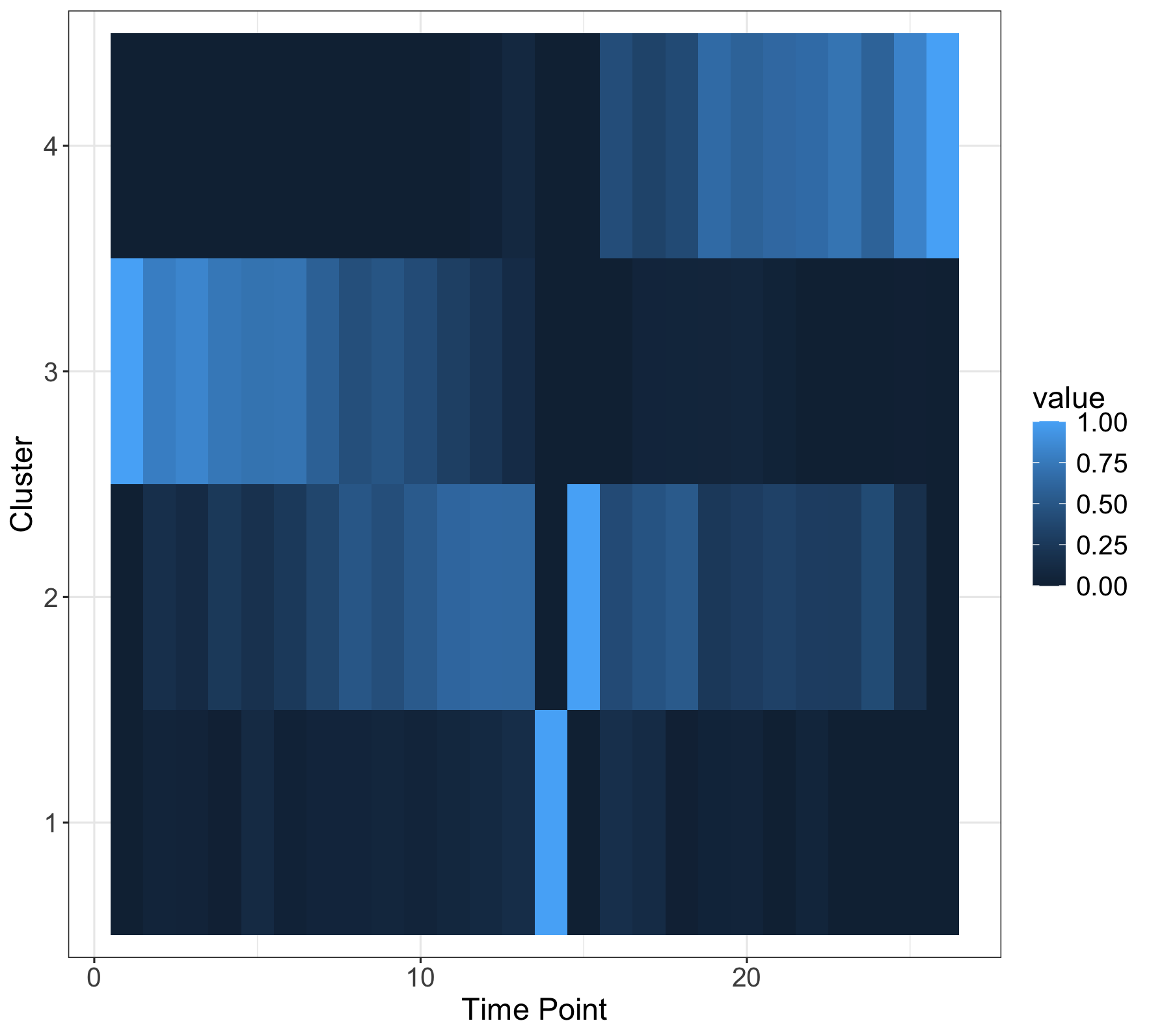}
    \caption{Clusters in the time dimension: HOSVD}
    \label{fig:enter-label}
    \end{subfigure}
        \begin{subfigure}[b]{0.33\textwidth}
    \includegraphics[width=\linewidth]{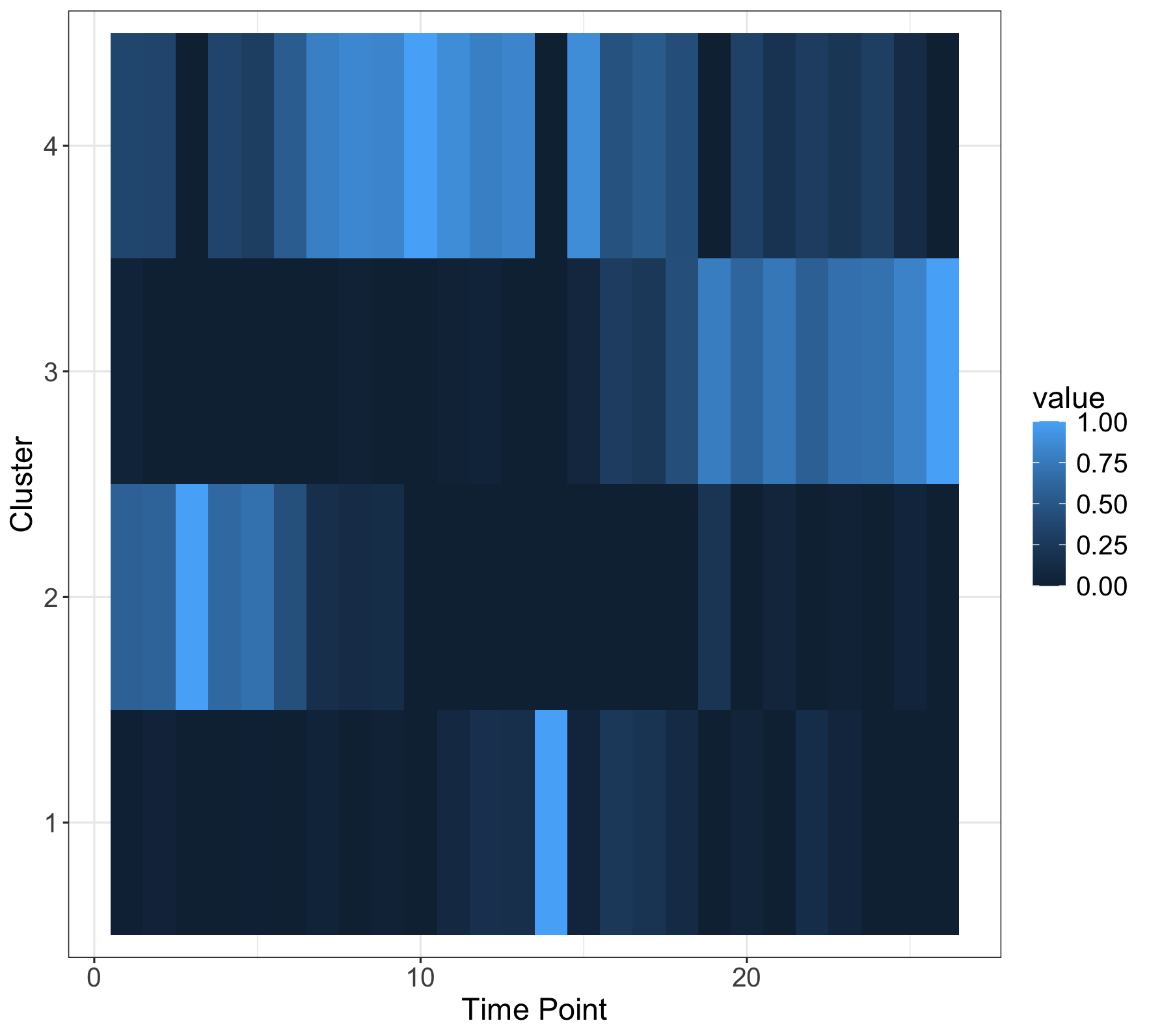}
    \caption{Clusters in the time dimension: STM}
    \label{fig:market_time}
    \end{subfigure}
    \caption{Estimated Latent Clusters along the time dimension.}
\end{figure}

\section{Related Literature}\label{sec:literature}
\subsection{Relationship to existing literature}\label{sec:relationship to literature}

Before further discussion, we briefly describe related work and emphasize the potential significance of tensor topic modeling approaches for handling modern datasets.

\paragraph{Topic Modeling.} Topic modelling approaches usually fall into on of two categories: Bayesian approaches, which typically involve variations around the Latent Dirichlet Allocation (LDA) model of \cite{blei2003latent}; and its frequentist counterpart,  probabilistic latent semantic indexing (pLSI). While the data generation mechanism of~\eqref{eq:plsi} is the same across methods, the two types of approaches differ mainly in the assumptions that they make on the data and in the tools that they use for inference. Beyond these two main categories, additional methods such as Latent Semantic Analysis and standard Non-negative Matrix Factorization have also been employed to solve the same problem. However, since their outputs are non longer probabilities and do not benefit from the same interpretability, we will not discuss them here.

Latent Dirichlet Allocation assumes that the topic assignment and topic matrices, $\mW$  and $\mA$ respectively, are random variables, which need to be endowed with appropriate priors (such as for instance, Dirichlet distributions \cite{blei2003latent}). In this setting, the goal of the inference is to adjust the posterior distribution of $\mA$ and $\mW$ given the observed data. This is typically achieved through a version of the variational Expectation-Maximization  (EM) algorithm, which is known to be fast and scalable.  Further refinements of LDA have subsequently been proposed to accommodate additional side information, such as accounting for correlation between documents by the same author, or longitudinal information in the analysis of microbiome data. For instance, 
in the matrix setting,
\cite{roberts2014structural} extend the usual LDA setting to structural LDA by incorporating document-level covariates $X_i$ and propose modeling correlations between documents by replacing the Dirichlet prior on each row of the mixture matrix $\mW$ with a logistic normal sampling step: $\mW_{i\star} = \text{softmax}(\veta_{i})$, where $\veta_{i} \sim N(\mX_{i}\vbeta, \mSigma)$.
However, fewer works have considered the tensor setting. Amongst the latter, most seem to have primarily focused on the supervised setting, where documents are paired with numerical ratings. For example, in the context of reviewer-item rating dyads (a similar context to our reviewer example),  FLDA \cite{agarwal2010flda} and FLAME \cite{wu2015flame} learn to explain document ratings using a latent variable model, where scores are explained by a combination of latent factors on the items, reviewers and topics. 
To our knowledge, the only adaptation of LDA to the tensor setting is found in the works by \cite{guo2013lda} and \cite{anandkumar2014tensor}. In \cite{guo2013lda}, the authors consider document corpora with a time dimension. The authors propose a method that automatically categorizes similar topics into groups over time using a two-step process, which fits topics to documents using LDA --- thereby yielding a topic tensor ---, and subsequently refines the topics by grouping them using a canonical polyadic (CP) decomposition. 
While this approach allows to understand correlations between topics through time, the two-step procedure does not allow the estimation of the topics to borrow strength across tie points. 
\cite{anandkumar2014tensor}, on the other hand, consider the application of LDA on tensor data, using the moments to recover the latent topics using a CP decomposition. 
This model consequently does not allow to understand the interactions between topics and latent time variables,   Furthermore, to the best of our knowledge, no Bayesian model for tensor topic modeling has yet been suggested in the literature.

Probabilistic Latent Semantic Indexing, on the other hand, can be viewed as a special instance of non-negative matrix factorization (NMF). Its consistency is generally based on a separability condition, also referred as ``anchor word'' or ``anchor document'' assumption \cite{ke2022using, bing2020optimal,klopp2021assigning} (definitions \ref{assumption: anchor doc} and \ref{def: anchor word}). In other words, anchor words act as signatures for a topic, while anchor documents act as ``archetypes'', in that they correspond to documents that have only a single topic. Either of the anchor assumptions is a sufficient condition for identifiability of parameters \cite{chen2023learning} in matrix pLSI. Existing methods \cite{ke2022using,arora2012learning,bing2020fast,tran2023sparse,klopp2021assigning} estimate the parameters of the model by first, identifying the anchors and subsequently using them to recover the remaining parameters, usually by expressing the remaining rows as convex combinations of the anchors. 

A major drawback of both types of approaches is that they are inherently designed to analyze data based on two-dimensional interactions between documents and topics. This can oversimplify complex datasets where documents also have temporal or spatial dimensions, or demographic characteristics leading to latent clusters. In other words, simply flattening all dimensions and neglecting the correlations between documents fails to accommodate their unique properties and relationships. This approach can result in a loss of important information that could substantially aid in topic inference. Moreover, while traditional topic models focus solely in reducing the vocabulary space by describing documents as mixture of topics, the automatic clustering of documents and time points might also be of scientific interest. In particular, we might be interested in understanding interactions between all latent components. For example, in our guiding example, reviewer and paper information can be treated as additional information that can be useful to improve the topic and mixed-membership estimation. In microbiome studies, temporal and demographic characteristics of samples may impact microbioa abundances. Understanding the interaction between types of patients (e.g. pregnant vs non-pregnant women) and the different abundances of microbiota might be of scientific significance. In this setting, representing the data as a tensor holds richer promise.

 The Nonnegative Matrix Factorization literature has a richer history in accommodating tensor data. Methods such as fixed NMF \cite{cichocki2007nonnegative} and direct NMF \cite{paatero1994positive} operate by slicing the tensor along one of the dimensions (for instance, time) and estimate the topic proportions over that particular dimension. However, it still fails to fully capture the nuanced information of that dimension since these methods assume a prior distribution of topic proportion over the third mode. For direct NMF, the data are treated as ``independent" across the side information, while the data in fixed NMF are assumed to share the same latent topics over the third dimension. It fails to capture the topic changes over third mode. 
 Therefore, it is crucial to employ a tensor decomposition method that captures the complete information and detects changes in topics across different dimensions.

\paragraph{Tensor Decompositions.} There are two main tensor decompositions: the  {\it CANDECOMP/ PARAPAC} (CP) decomposition and the {\it Tucker} decomposition. Due to the discrete and non-negative nature of count data, some studies have devised constrained version of these decompositions for the analysis of text data, namely non-negative CP \cite{ahn2021large} (NNCPD) or non-negative Tucker decomposition \cite{romeo2014multi} (NTD) in tensor topic modeling. 
CP can be viewed as a special case of Tucker decomposition, which require a super-diagonal core tensor. In tensor topic modeling, CP decomposition finds the analogy sharing by all dimensions, while tucker decomposition allows more complex interactions across modes. However, these tensor decomposition methods still face identifiability issue in tensor topic modeling. The uniqueness of NNCPD is based on Kruskal’s rank condition \cite{bhaskara2014uniqueness}. Unlike the separable condition alike anchor assumption, it is hard to interpret its meaning in pLSI model. MFLBM\cite{zheng2016topic} was proposed to estimate the tucker-based topic modeling by variational EM algorithm but as the matrix LDA, it do not explore the anchor conditions. Therefore, these algorithms face the identifiability issue and do not guarantee to produce consistent estimates of core tensor and three factor matrices.

\section{Conclusion}
In this paper, we introduce \textit{Tensor Topic modeling} (TTM), a HOSVD-based estimation procedure  to incorporate spatial, temporal, and other document-specific information, providing a more nuanced analysis for datasets that inherently span multiple dimensions. Our procedure is also shown to adapt effectively under the $\ell_q$-sparsity assumption \ref{ass: weak lq sparsity} on word-topic matrix $\mA^{(3)}$, which is exhibited by many real-world text datasets. By thresholding low-frequency words, we reduce the inclusion of specific, potentially identifiable terms, which helps mitigate such privacy risks \cite{manzonelli2024membership}. Our threshold HOSVD-based method can be potentially used as a pre-processing step in differentially private (DP) vocabulary selection to improve privacy without significantly impacting model performance. 

Based on this paper, some potential research directions emerge. First, investigating core sparsity within tensor decompositions could offer insights, particularly in enhancing computational efficiency and interpretability. For instance, exploring the AL$\ell_0$CORE approach \cite{hood2024all0core}, where the core tensor’s $\ell_0$-norm is constrained, allows only a limited number of non-zero entries (up to a user-defined budget $Q^*$). This sparse core structure is an adaptation of Tucker decomposition that maintains computational tractability similar to CP decomposition while benefiting from Tucker’s more flexible latent structure.

Also, the minimax-optimal $\ell_1$-error rate for the Tucker components with and without $\ell_q$-assumption \ref{ass: weak lq sparsity} remains an open problem.

\acks{The authors would like to acknowledge support for this project
from the National Science Foundation (NSF grant IIS-2238616). This work was completed in part with resources provided by the University of Chicago’s Research Computing Center. }

\vspace{4mm}
\newpage
\bibliographystyle{abbrv}
\bibliography{sample}

\vspace{4mm}
\newpage

\appendix
\section{Appendix: Proofs}\label{app:proof}
\section{Properties of the data}

\subsection{Tensor Topic Modeling}

In this paper, we generalize the usual matrix topic modeling framework assumed by Probabilistic Latent Semantic Indexing ($\mD=\mA\mW$) to accommodate the analysis of tensor count data assumed to have a latent Tucker factor decomposition ($\mathcal{D} = \mathcal{G} \cdot (\mA^{(1)}, \mA^{(2)}, \mA^{(3)})$). The crux of this approach consists in considering the matricized versions of $\mathcal{D}$ (as defined in Equation \eqref{eq:matricization} of the main text) along each axis: 
\begin{enumerate}
    \item $\mathcal{M}_1(\mathcal{D})=\mD^{(1)}=\mA^{(1)}\mW^{(1)}\in\mathbb{R}^{N^{(1)}\times (N^{(2)}R)}$, which we assume to have  rank $K^{(1)}$,
     \item $\mathcal{M}_2(\mathcal{D})=\mD^{(2)}=\mA^{(2)}\mW^{(2)}\in\mathbb{R}^{N^{(2)}\times (N^{(1)}R)}$, which we assume to have rank $K^{(2)}$,
      \item $\mathcal{M}_3(\mathcal{D})=\mD^{(3)}=\mA^{(3)}\mW^{(3)}\in\mathbb{R}^{R\times (N^{(1)}N^{(2)})}$, which we assume to have   rank $K^{(3)}$.
\end{enumerate}
Given the inherent structure of the problem, each one of the Tucker factor matrices $\mA^{(1)}, \mA^{(2)}, \mA^{(3)}$ belong to one of two matrix types. Letting $\mA\in\mathbb{R}^{n\times K}$ denote a general matrix, the factor matrices can be either considered as instances of either:
\begin{align}
 \text{Case 1: Row-Wise Stochastic Matrices: }\quad    &\sum_{k\in[K]} \mA_{ik}=1 \text{ for any }i\in[n]
     \tag{C1}\label{case:1}\\
   \text{Case 2: Column-Wise Stochastic Matrices:}  \quad&\sum_{i\in[n]} \mA_{ik}=1 \text{ for any }k\in[K]
    \tag{C2}\label{case:2}
\end{align}

\begin{remark}\label{remark:symmetry_mode12}
As highlighted in the main text (Equations~\eqref{eq:constraints A} and \eqref{eq:constraints_on_Ws}), the matrices $\mA^{(1)},\mA^{(2)}$, $\mathcal{M}_1(\mathcal{G})/K^{(2)}$ and $\mathcal{M}_2(\mathcal{G})/ K^{(1)}$ in our proposed tensor topic model are assumed to be row-wise stochastic (Case \eqref{case:1}), while the matrices $\mathcal{M}_3(\mathcal{G})$ and  $\mA^{(3)}$ are column-stochastic (Case \eqref{case:2}). This implies that, in the rest of the appendix, the estimation of $\mA^{(1)}$ and $\mA^{(2)}$ can be achieved symmetrically (using the same tools and methods), but the estimation of $\mA^{(3)}$ must be tailored to account for its peculiar structure.
\end{remark}
\subsection{Properties of the  one-hot word encoding vectors.}\label{app:propertiesone hot}

Throughout our theoretical analysis, it will often be easier to analyze the behavior of the observed empirical frequencies $\mY_{ir}$ of each term $r$ in the vocabulary in document $i$ by writing them as averages of one-hot word encoding  vectors.

\par Let $\{\mT_{im}:i\in[n],m\in[M]\}$ represent the set of independent $R$-dimensional  vectors corresponding to the one-hot word encoding vectors for each of the  $m$ words in in document $i$. Under the pLSI model: 
\begin{equation}\label{eq:def_T_im}
\mT_{im}\in\mathbb{R}^R \sim \text{Multinomial}(1,d_i)
\end{equation} 
for $\mathbf{d}_i\in\mathbb{R}^R$. We have $\sum_{r\in[R]}\mathbf{d}_i(r)=1$. For any fixed $r\in[R]$, we define  $\mX_{im}$ to be the centered version of $\mT_{im}$:
\begin{equation}\label{eq:def_X_im}
    \mX_{im}(r)=\mT_{im}(r)-d_i(r).
\end{equation} 
where $\mathbf{d}_i(r) = \mathbb{E}[\mT_{im}(r)].$
It is easy to show that the following properties on the moments of $\mX_{im}$ hold true:
\begin{itemize}
\item \textbf{First moment of $\mX_{im}$}: 
\begin{equation}\label{eq:first_moment_X_im}
|\mX_{im}(r)|\leqslant 1, \qquad \text{ and } \qquad \mathbb{E}[\mX_{im}(r)]=0 \qquad  
\end{equation} 
\item \textbf{Second moment of $\mX_{im}$ and $\mT_{im}$}: 
\begin{equation}\label{eq:second moment X im}
\mathbb{E}[\mT_{im}(r)^2]=\mathbf{d}_{i}(r),\qquad \mathbb{E}[\mT_{im}(r)^2]=\mathsf{Var}\left[\mX_{im}(r)\right]=\mathbf{d}_{i}(r)-\mathbf{d}_{i}(r)^2
\end{equation} 
and if $r \neq  l$:
\begin{equation}\label{eq:second moment X im_cross}
\mathbb{E}[\mT_{im}(r)\mT_{im}(l)]=0,\quad \mathbb{E}[\mX_{im}(r)\mX_{im}(l)]=-\mathbf{d}_{i}(r)\mathbf{d}_{i}(l).
\end{equation} 

\item \textbf{Fourth moment of $\mX_{im}$ and $\mT_{im}$}:

\begin{equation}\label{eq:fourth_moment_X_im_cross}
\begin{split}
\mathbb{E}[\mX_{im}(r)^4] &\leqslant \mathbf{d}_{i}(r) \text { and } \mathsf{Var}\left[\mX_{im}(r)^2\right]\leqslant \mathbf{d}_i(r),\\
    \text{And if } r\neq l,\hspace{3cm}& \\
    \quad \mathbb{E}[\mX_{im}(r)^2\mX_{im}(l)^2]&\leqslant \mathbf{d}_{i}(r)\mathbf{d}_{i}(l)
    ,\\
    \quad\mathsf{Cov}(\mX_{im}(r)^2,\mX_{im}(l)^2)&=\E[\mX_{im}(r)^2\mX_{im}(l)^2] -\E[\mX_{im}(r)^2]\E[\mX_{im}(l)^2]\leqslant  \mathbf{d}_{i}(r)\mathbf{d}_{i}(l)
\end{split}
\end{equation} 

\end{itemize}
\begin{remark}\label{remark: z multinomial}
    Recall the model in \eqref{eq: model multinominal}. Define $\mathcal{X}_{ijm}(r)=\mathcal{T}_{ijm}(r)-\mathbb{E}[\mathcal{T}_{ijm}(r)]$ where $M\mathcal{T}_{ijm}(\cdot)\in\mathbb{R}^{R}\sim \text{Multinomial}(M,\mathcal{D}_{ij\star})$ for any $i\in[N^{(1)}], j\in[N^{(2)}]$ and $r\in [R]$. Define:
    \begin{align}\label{eq: Z= 1/M sum Xijmr}
    \mathcal{Z}_{ijr}=\frac{1}{M}\sum_{m=1}^M \mathcal{X}_{ijm}(r)\quad\quad\text{ for } i\in[N^{(1)}],j\in[N^{(2)}],m\in[M]
    \end{align}
where $\mathcal{Z}:=\mathcal{Y}-\mathcal{D}$ with $\mathcal{Y}$ the observed tensor corpus, and  $\mathcal{D}$ its expectation. By the previous remarks (Equations~\eqref{eq:second moment X im} and \ref{eq:fourth_moment_X_im_cross}), noting that each document in the tensor $\Y$ is indexed by a tuple $(i,j)$, we directly deduce that for each entry of $\mathcal{Z}$:
\begin{equation}\label{eq: expection of Z}
\begin{aligned}
    \mathbb{E}[\mathcal{Z}_{i_1j_1r_1}\mathcal{Z}_{i_2j_2r_2}]=\frac{1}{M}
     \begin{cases}
     \mathcal{D}_{ijr}-\mathcal{D}_{ijr}^2 &\text{if }i_1=i_2=i,j_1=j_2=j,r_1=r_2=r\\
          -\mathcal{D}_{ijr_1}\mathcal{D}_{ijr_2}& \text{if } (i_1=i_2=i,j_1=j_2=j) \text{ but } r_1\neq r_2\\
     0 &\text{as soon as } (i_1\neq i_2) \text{ or } (j_1\neq j_2)\\
    \end{cases}    
\end{aligned}\end{equation}
Properties on the concentration and norms of $\mathcal{Z}$ are provided in Section \ref{Sec: noise analysis}
\end{remark}
\begin{lemma}\label{lem: Dijr less than fr}
Define $f_r:=\sum_{k\in[K^{(3)}]}\mA^{(3)}_{rk}$. For any $i\in[N^{(1)}],j\in[N^{(2)}]$ and $r\in[R]$, we have 
$$\mathcal{D}_{ijr}\leqslant f_r.$$
\end{lemma}
\begin{proof}
Recall that, in our definition of $\mathcal{D}$, for each entry $(i,j, r)$, we have
\begin{align*}
\mathcal{D}_{ijr}&=\sum_{k^{(1)}=1}^{K^{(1)}} \sum_{k^{(2)}=1}^{K^{(2)}} \sum_{k^{(3)}=1}^{K^{(3)}} \mathcal{G}_{k^{(1)}k^{(2)}k^{(3)}} \mA^{(1)}_{ik^{(1)}}\mA^{(2)}_{jk^{(2)}}\mA^{(3)}_{rk^{(3)}}\qquad \text{ where } \mathcal{G}_{k^{(1)}k^{(2)}k^{(3)}}  \leq 1\\
&\leqslant\sum_{k^{(1)}=1}^{K^{(1)}} \sum_{k^{(2)}=1}^{K^{(2)}} \sum_{k^{(3)}=1}^{K^{(3)}}\mA^{(1)}_{ik^{(1)}}\mA^{(2)}_{jk^{(2)}}\mA^{(3)}_{rk^{(3)}}=\sum_{k^{(3)}=1}^{K^{(3)}}\mA^{(3)}_{rk^{(3)}}=f_r
\end{align*}
where the last line follows from the fact that $\mA^{(1)}$ and $\mA^{(2)}$ are column-wise stochastic (entrywise non-negative) matrices (Case \ref{case:1}).
\end{proof}
\begin{lemma}[Maximum singular values of the factor matrices]\label{lemma: max singular values}
The eigenvalues of the factor matrices defined in Model~\eqref{eq: TTM} verify the following inequalities:
$$
\sigma_1\left(\frac{1}{N^{(1)}}\mA^{(1)\top }\mA^{(1)}\right)\leqslant 1,\qquad \sigma_1\left(\frac{1}{N^{(2)}}\mA^{(2)\top}\mA^{(2)}\right)\leqslant 1,\qquad \sigma_1\left(\mA^{(3)}\right)\leqslant \sqrt{K^{(3)}} 
$$ and 
$$\sigma_{1}\left(\frac{1}{K^{(1)}K^{(2)}}\mathcal{M}_a(\mathcal{G})^\top\mathcal{M}_a(\mathcal{G})\right)\leqslant 1 \text{ for any } a\in[3]$$
\end{lemma}
\begin{proof}
For any matrix of the factor matrices $\mA$ such that each of its entries is less than 1:
$$
\sigma_1(\mA^\top \mA)\leqslant\text{tr}(\mA^\top \mA)=\sum_{k=1}^K \sum_{i=1}^n \mA_{ik}^2\leqslant\sum_{k=1}^K \sum_{i=1}^n \mA_{ik}
$$
To conclude, we note that this last double sum simplifies given the stochastic nature of the Tucker factor matrices in TTM (with $\mA^{(1)}, \mA^{(2)}$ and $\mathcal{M}_a(\mathcal{G})$ belonging to Case  \eqref{case:1} and $\mA^{(3)}$ belonging to Case \eqref{case:2}).

\end{proof}
\begin{lemma}[Maximum singular values of the matricizations of  $\D$]\label{lemma: singular value D}
Given any matricization $\mD^{(a)}$ of $\mathcal{D}$ along axis $a \in [3]$, we have for all $k\in[\max_{a\in[3]}K^{(a)}]$,
$$ c^* K^{(1)} K^{(2)}K^{(3)}N^{(1)}N^{(2)}\leqslant\lambda_k(\mD^{(a)}\mD^{(a)\top}) \leqslant K^{(1)} K^{(2)}K^{(3)}N^{(1)}N^{(2)}.$$
\end{lemma}
\begin{proof}
For any matrices $\mA$ and $\mB$ for which we can form the Kronecker product $\mA\otimes \mB$, the following holds true:

$$\lambda_1[(\mA\otimes \mB)(\mA\otimes \mB)^\top]=\lambda_1[(\mA\mA^\top)\otimes (\mB\mB^\top)]=\lambda_1(\mA\mA^\top)\cdot\lambda_1(\mB\mB^\top),$$ 
$$\text{and } \lambda_{\min}[(\mA\otimes \mB)(\mA\otimes \mB)^\top]=\lambda_{\min}[(\mA\mA^\top)\otimes (\mB\mB^\top)]=\lambda_{\min}(\mA\mA^\top)\cdot\lambda_{\min}(\mB\mB^\top),$$ 
In the previous equations, the first equality follows by the mixed-product property of the Kronecker product, and the second, by property of the extremal eigenvalues of the Kronecker product \cite{schacke2004kronecker}.
Moreover, $\lambda_1(\mA\mA^\top) = \lambda_1(\mA^\top \mA).$
Therefore, for any set of indices $a_1 \neq a_2 \neq a_3$ with $a_1, a_2, a_3 \in [3]$, the matrix $\mW^{(a_1)}=\mathcal{M}_{a_1}(\mathcal{G})\cdot(\mA^{(a_2)}\otimes \mA^{(a_3)})$ is such that:
\begin{align*}
\lambda_1(\mW^{(a_1)}\mW^{(a_1)\top})&=\lambda_1[\mathcal{M}_{a_1}(\mathcal{G})\cdot(\mA^{(a_2)}\otimes \mA^{(a_3)})(\mA^{(a_2)}\otimes \mA^{(a_3)})^\top \cdot\mathcal{M}_{a_1}(\mathcal{G})^\top]\\&\leqslant \lambda_1(\mathcal{M}_{a_1}(\mathcal{G})^\top\mathcal{M}_{a_1}(\mathcal{G}))\cdot \lambda_1(\mA^{(a_2)\top}\mA^{(a_2)})\cdot\lambda_1(\mA^{(a_3)\top}\mA^{(a_3)})
\end{align*}
Therefore,
$$\lambda_1(\mD^{(a_1)}\mD^{(a_1)\top})=\lambda_1(\mA^{(a_1)}\mW^{(a_1)}\mW^{(a_1)\top}\mA^{(a_1)\top})\leqslant \lambda_1(\mathcal{M}_{a_1}(\mathcal{G})^\top \mathcal{M}_{a_1}(\mathcal{G}))\Pi_{a=1}^3 \lambda_{1}(\mA^{(a)\top}\mA^{(a)})$$ Combining this with Lemma \ref{lemma: max singular values}, for any $a\in[3]$, we conclude that :
$$
\lambda_1(\mD^{(a)}\mD^{(a)\top})\leqslant K^{(1)} K^{(2)} K^{(3)} N^{(1)} N^{(2)}
$$
Using Assumption \ref{ass: min singular value}, we also have
\begin{equation}
\begin{split}
    \lambda_{K^{(a)}}(\mD^{(a)}\mD^{(a)\top})& =\lambda_{K^{(a)}}(\mA^{(a_1)}\mW^{(a_1)}\mW^{(a_1)\top}\mA^{(a_1)\top})\\
    & \geq \lambda_{K^{(a)}}(\mA^{(a_1)}\mA^{(a_1)\top}) \lambda_{K^{(a)}}(\mW^{(a_1)}\mW^{(a_1)\top})\\
    &\geqslant c^*K^{(1)} K^{(2)} K^{(3)} N^{(1)} N^{(2)}.\\
    \end{split}
\end{equation}
\end{proof}

\begin{lemma}[From Lemma F.2 of \cite{ke2022using}]\label{lem: eigen gap}
Let $\mathbf{\Theta}\in\mathbb{R}^{K\times K}$ be a symmetric matrix such that all entries are lower-bounded by a constant $c_0$. If $\|\mathbf{\Theta}\|_{op}\leqslant c^*$, then there exists some constant $c>0$ such that
$$\lambda_1(\mathbf{\Theta})-\lambda_2(\mathbf{\Theta})\geqslant c.$$
Furthermore, the entries of the unit-norm leading positive eigenvector of $\mathbf{\Theta}$ are all lower bounded by some constant $c_1>0$.
\end{lemma}
\begin{proof}
To make this manuscript self-contained, we provide here the proof of Lemma~\ref{lem: eigen gap}, which is an extract of the proof of Lemma F.2 in \cite{ke2022using}. The argumentation is essentially a proof by contradiction.

Consider a sequence $\{\mathbf{\Theta}^{(n)}\}$ of symmetric matrices such that all the elements are bounded below by $c_0$. Suppose that the gap between the first two eigenvalues of each of its elements goes to zero as $n \to \infty$. Since $\|\mathbf{\Theta}\|_{op}\leqslant c^*$, by the Bolzano-Weierstrass  theorem, there exists a subsequence $\{\mathbf{\Theta}^{(n_m)}\}_{m=1}^\infty$ such that $\lim_{m \to \infty}\mathbf{\Theta}^{(n_m)}= \mathbf{\Theta}_0$ for a fixed matrix $\mathbf{\Theta}_0$. By the assumptions on the sequence $\{\mathbf{\Theta}^{(n)}\}$, the matrix $\mathbf{\Theta}_0$ has entries that are all bounded below by $c_0$ and the gap between the first two eigenvalues of $\mathbf{\Theta}_0$ are zero. 
By Perron's theorem \cite{horn2012matrix}, such a matrix $\mathbf{\Theta}_0$ does not exist.

\par To prove the second point, let $\mathbf{\eta}_1^{(n)}$ be the unit-norm leading positive eigenvector of $\mathbf{\Theta}^{(n)}$. By Perron's Theorem \cite{horn2012matrix}, the leading eigenvector of strictly positive matrix is also strictly positive. Note $\mathbf{\eta}^{(n)}_1$ indicates the dependency of $n$. Assume the contradiction for some $k\in[K]$
$$
\lim_{n\to \infty} \inf \mathbf{\eta}_1^{(n)}(k)=0.
$$
This also implies that there exists a subsequence $\{n_m\}_{m=1}^{\infty}$ such that $\lim_{m\to \infty} \mathbf{\eta}_1^{(n_m)}(k)=0$ and $\lim_{m\to \infty}\mathbf{\Theta}^{n_m}\to \mathbf{\Theta}_0$. By the continuity, we have $ \mathbf{\eta}_1^{(n_m)}\to \mathbf{\eta}_0$ where $\mathbf{\eta}_0$ is unit norm leading eigenvector of $\mathbf{\Theta}_0$ up to a multiple of $\pm 1$ on $\mathbf{\eta}_1^{(n_m)}$. This would implies $\mathbf{\eta}_0(k)=0$ but contradicts with the claim that the leading eigenvector of strictly positive matrix $\mathbf{\Theta}_0$ is also all positive. This ends the proof.
\end{proof}

\section{Basic Bernstein Concentration Inequalities}

\begin{lemma}[Bernstein inequality (Theorem 2.8.4, \cite{vershynin2018high})]\label{lemma: bernstein ineq} Let $X_1, \dots, X_n$ be independent random variables with $|X_i|\leqslant b$, $\mathbb{E}[X_i] = 0$ and $\text{Var}[X_i] \leq \sigma_i^2$ for all $i$. Let $\sigma^2 := n^{-1}\sum_{i=1}^n \sigma_i^2$. Then for any $t > 0$, 
$$ \mathbb{P}\left(n^{-1}\left|\sum_{i=1}^n X_i\right| \geq t \right) \leq 2 \exp\left(-\frac{nt^2/2}{\sigma^2 + bt/3}\right)$$
\end{lemma}

\begin{lemma}[Matrix Bernstein Inequality \cite{tropp2015introduction}] \label{lemma: matrix bernstein} Let $\mX_1,\dots \mX_n\in\mathbb{R}^{p\times r}$ be independent random matrices such that $\|\mX_i\|_{op}\leqslant b$ and  $\mathbb{E}[\mX_i]=0$ for all $i$. Let $$\sigma^2=\max\{ \|\sum_{i=1}^n \mathbb{E}[\mX_i \mX_i^\top] \|_{op}, \|\sum_{i=1}^n \mathbb{E}[\mX_i ^\top \mX_i] \|_{op}\}.$$ Then, for all $t>0$, we have
$$\mathbb{P}\left(\left\|\frac{1}{n}\sum_{i=1}^n \mX_i\right\|_{op}\geqslant t\right)\leqslant (p+r)\exp\left(-\frac{t^2n^2/2}{\sigma^2+btn/3}\right)$$
\end{lemma}

\begin{lemma}[Bernstein Inequality with permutations]\label{lemma: Berstein involving permutation}
  Let $X_{1,1},\dots, X_{N_1,N_2}$ for all $i\in[N_1]$, $j\in[N_2]$  be independent random variables such that $|X_{i,j}|\leqslant 1 $, $\mathbb{E}[X_{i,j}]=0$ and $\text{Var}[X_{i,j}]=\sigma^2$. Define $Q_{i,M}= Q(X_{i,1},\dots, X_{i,N_2})=\sum_{m=1}^M X_{i,m}$ for some integer $M\leqslant N_2$, and for any permutation $\pi$ of $[N_2]$, define $Q_{i,\pi}=Q(X_{i,\pi(1)},\dots, X_{i,\pi(N_2)})$. With these notations, $Q_{i,M}$ is the identity permutation of $Q_{i,\pi}$. 
  
  Let $V=V(X_{1,1},\dots X_{N_1,N_2})$ be the random variable defined as $V=\frac{1}{N_2!}\sum_{i\in[N_1]}\sum_{\pi\in\Pi(N_2)} Q_{i,\pi}$ for some permutation $\pi$ on set $[N_2]$ (where $\Pi$ denotes all possible permutations of the set $[N_2]$).
  
  Then, for all $t>0$, we have
    $$\mathbb{P}(V\geqslant t)\leqslant \exp\left(-\frac{t^2/2}{MN_1\sigma^2+t/3}\right)$$
    
\end{lemma}
\begin{proof}
Consider first the case where $N_1=1$ (fix $i$), so $V$ can simply be rewritten as: $V=\frac{1}{N_2!}\sum_{\pi\in\Pi(N_2)} Q_{\pi}$.

    By Markov's inequality and the convexity of the exponential function, we have for any $s,t>0$,
    \begin{align}\label{eq: prob v in bernstein permute}
        \mathbb{P}\left(V\geq t\right)\leqslant e^{-st}\mathbb{E}[e^{sV}]\leqslant e^{-st}\left(\frac{1}{N_2!}\sum_{\pi\in\Pi(N_2)}\mathbb{E}[e^{sQ_{\pi}}]\right). 
    \end{align}
    Let $G(x)=\frac{e^x-1-x}{x^2}$, which is an increasing function of $x$.
  Expanding the exponential, we have:
    \begin{align*}
        \mathbb{E}\left[e^{sX_m}\right]&=\mathbb{E}\left[1+sX_m+\frac{s^2X_m^2}{2}+\dots\right]\\
        &=\mathbb{E}\left[1+s^2X_m^2\left(\frac{e^{sX_m}-1-sX_m}{s^2X_m^2}\right)\right]\quad\quad \quad(\text{ since } \mathbb{E}[X_m]=0)\\
        &=\mathbb{E}\left[1+s^2X_m^2G(sX_m)\right]\\
        &\leqslant \mathbb{E}\left[1+s^2X_m^2G(s) \right]\quad\quad\quad( \text{ since } |X_m|\leqslant 1)\\
        &=1+s^2\sigma^2G(s)\leqslant e^{s^2\sigma^2G(s)}
    \end{align*}
    This inequality holds for any $m\in[M]$ and permutation ${\pi}$.  Since the variables $X_{m}$ for $m\in[M]$  are independent, $\E[e^{sQ_{\pi}}] = \prod_{m=1}^M \E[e^{sX_{\pi(m)}}]$ and we are able to rewrite \eqref{eq: prob v in bernstein permute} as:
   \begin{align}\label{eq: bernstein permutation indep sum}\mathbb{P}\left(V\geqslant t\right)&\leqslant \exp(-st+Ms^2\sigma^2G(s))=\exp\left\{-st+M\sigma^2\left(e^s-1-s\right)\right\}
   \end{align}
   Let $s=\log\left(1+\frac{t}{M\sigma^2}\right)>0$, then: 
   \begin{align*}
       \mathbb{P}\left(V\geqslant t\right)&\leqslant \exp\left\{-t\log\left(1+\frac{t}{M\sigma^2}\right)+M\sigma^2\left(\frac{t}{M\sigma^2}-\log\left(1+\frac{t}{M\sigma^2}\right)\right)\right\}\\
       &= \exp\left\{-M\sigma^2\left[\left(1+\frac{t}{M\sigma^2}\right)\log\left(1+\frac{t}{M\sigma^2}\right)-\frac{t}{M\sigma^2}\right]\right\}
   \end{align*}
   Let $H(x)=(1+x)\log(1+x)-x$ and note that $H(x)\geqslant \frac{3x^2}{6+2x}$. Evaluating the function $H(x)$ at  $x=\frac{t}{M\sigma^2}$, the previous bound guarantees:
   $$
   \mathbb{P}\left(V\geqslant t\right)\leqslant
  \exp\left\{-M\sigma^2\left(\frac{3t^2}{M^2\sigma^4(6+2t/(M\sigma^2))}\right)\right\} =\exp\left\{-\frac{t^2/2}{M\sigma^2+t/3}\right\}.$$
  
  We now consider the case where $N_1>1$. Since all $\{X_{i,m}\}_{i\in[N^{(1)}],m\in [M]}$ are independent, for any value of the tuple $(i,m)$, \eqref{eq: prob v in bernstein permute} can be rewritten as:  
  $$\P[V\geq t] \leq e^{-st}\E[e^{sV}] =  e^{-st}\E[e^{s \sum_{i \in [N_1]} \frac{1}{N_2!}\sum_{\pi} Q_{i\pi}}] =  e^{-st}\prod_{i \in [N_1]}\E[e^{s \frac{1}{N_2!}\sum_{\pi} Q_{i\pi}}]
 \leq  e^{-st} e^{N_1Ms^2\sigma^2G(s)},$$
 where the last inequality follows by applying Equations \eqref{eq: prob v in bernstein permute} and \eqref{eq: bernstein permutation indep sum} to each of the $i$ terms in the product. Therefore, the result follows by a similar reasoning to the case where $N_1=1$ by replacing $M$ with $MN_1$.
\end{proof}

\section{\texorpdfstring{Properties of the matricized noise $\mZ^{(a)} = \mathcal{M}_{(a)}(\mathcal{Y} - \mathcal{D})$}%
{Properties of the matricized noise Z(a) = Ma(Y - D)}}%
\label{Sec: noise analysis}

\subsection{Properties of the moments of the matricized noise}

Let $N_R=\max\{N^{(1)},N^{(2)},R\}$.
\begin{lemma}\label{lemma: expectation of zz} Assuming the multinomial model of TTM in Equation~\eqref{eq: TTM} and letting $\mZ^{(a)} = \mathcal{M}_{(a)}(\mathcal{Y} - \mathcal{D}) $ denote the matricization of the noise $\mathcal{Z}$ along each axis $a\in [3]$, the following equalities hold:
    \begin{itemize}
    \item $\mathbb{E}[\mZ^{(1)}\mZ^{(1)\top}]=\frac{1}{M}\text{diag}\left(N^{(2)}-\|\mD^{(1)}_{1\star}\|_2^2\dots,N^{(2)}-\|\mD^{(1)}_{N^{(1)}\star}\|_2^2\right)\in\mathbb{R}^{N^{(1)}\times N^{(1)}}$
    \item  $\mathbb{E}[\mZ^{(2)}\mZ^{(2)\top}]=\frac{1}{M}\text{diag}\left(N^{(1)}-\|\mD^{(2)}_{1\star}\|_F^2\dots,N^{(1)}-\|\mD^{(2)}_{ N^{(2)}\star}\|_2^2\right)\in\mathbb{R}^{N^{(2)}\times N^{(2)}}$
    \item $\mathbb{E}[\mZ^{(3)}\mZ^{(3)\top}]=\frac{1}{M}\left[\text{diag}\left(\mD^{(3)}\mathbf{1}_{N^{(1)}N^{(2)}}\right)-\mD^{(3)}\mD^{(3)\top}\right]\in\mathbb{R}^{R\times R}$
\end{itemize}
\end{lemma}
\begin{proof}

\begin{itemize}
\item Consider first the case of mode-1 matricization, with $\mZ^{(1)} = \mathcal{M}_1(\mathcal{Z})$. By definition, $        \forall i_1, i_ 2 \in [N^{(1)}]$:
    \begin{equation*}
    \begin{split}
\E[(\mathcal{M}_1(\mathcal{Z})\mathcal{M}_1(\mathcal{Z})^\top)_{i_1i_2}]&=\E[\sum_{j\in [N^{(2)}]}\sum_{r \in [R]}  \mathcal{Z}_{i_1jr}\mathcal{Z}_{i_2jr}]\\
       & =\frac{1}{M} \begin{cases}0 \text{ if } i_1 \neq i_2\\
       \sum_{j\in [N^{(2)}]}\sum_{r \in [R]} \left(\mathcal{D}_{i_1jr}-\mathcal{D}^2_{i_1jr}\right) \qquad \text{ }
       \end{cases}\\
        & =\frac{1}{M} \begin{cases}0 \text{ if } i_1 \neq i_2\\
       N^{(2)} - \|\mathcal{D}^2_{i_1, \cdot}\|_2^2 \qquad \text{ otherwise}
       \end{cases}\\
           \end{split}
        \end{equation*} 
        where the penultimate equality follows by applying Remark A.1 (Equation \eqref{eq: expection of Z}), and the last line follows from the properties of $\mathcal{D}$ established in Equation \eqref{eq:constraints A}.

\item Since for mode-2 matricization, the structure of the problem is identical to that of mode-1 matricization, it suffices to replace $N^{(1)}$ by $N^{(2)}.$

\item For mode-3 matricization, for any $ r_1, r_2 \in [R]$, we note that by Equation~\eqref{eq: expection of Z}:
    \begin{equation}
    \begin{split}
\E[(\mathcal{M}_3(\mathcal{Z})\mathcal{M}_3(\mathcal{Z})^\top)_{r_1r_2}]&=\E[\sum_{i\in [N^{(1)}]} \sum_{j\in [N^{(2)}]} \mathcal{Z}_{ijr_1}\mathcal{Z}_{ijr_2}]\\
       & = \frac{1}{M}\begin{cases}
       -\sum_{i\in [N^{(1)}]} \sum_{j\in [N^{(2)}]} \mathcal{D}_{ijr_1}\mathcal{D}_{ijr_2} \text{ if } r_1 \neq r_2\\
       \sum_{i\in [N^{(1)}]} \sum_{j\in [N^{(2)}]} \left(\mathcal{D}_{ijr_1}-\mathcal{D}^2_{ijr_1}\right) \qquad \text{ if }r_1 = r_2
       \end{cases}\\
           \end{split}
        \end{equation} 
    Which concludes the  proof.
    \end{itemize}
\end{proof}

\begin{lemma}\label{lemma: op norm of expectation of zz} Under the same assumptions as Lemma \ref{lemma: expectation of zz},  we have:
\begin{align*}
    c^*\frac{N^{(2)}}{M}\leqslant&\|\mathbb{E}[\mZ^{(1)}\mZ^{(1)\top}]\|_{op}\leqslant \frac{N^{(2)}}{M},  \quad  c^*\frac{N^{(1)}}{M}\leqslant\|\mathbb{E}[\mZ^{(2)}\mZ^{(2)\top}]\|_{op}\leqslant \frac{N^{(1)}}{M},\\
    &\quad \text{and }c^* \frac{N^{(1)} N^{(2)}}{M}\leqslant \|\mathbb{E}[\mZ^{(3)}\mZ^{(3)\top}]\|_{op}\leqslant \frac{N^{(1)} N^{(2)}}{M}
\end{align*}
    
\end{lemma}

\begin{proof}
    The proof is a direct consequence of Lemma \ref{lemma: expectation of zz}, and follows trivially for the mode 1 and 2 matricizations of $\mathcal{Z}$ given the diagonal nature of the covariance matrix, with $c^*$ any constant such that $c^* \leq N^{(a)} - \max_{i} \|\mathcal{D}^{(a)}_{i,\cdot}\|_{2}$ for $ a \in \{1,2\}.$

    For our mode-3 matricization of $\mathcal{Z}$, we note that:
$$ \E[(\mathcal{M}_3(\mathcal{Z})\mathcal{M}_3(\mathcal{Z})^\top)] = \frac{1}{M}\text{diag}\left(\sum_{i\in [N^{(1)}]} \sum_{j\in [N^{(2)}]} \mathcal{D}_{ijr_1}\right) - \frac{1}{M}\mathcal{M}^{(3)}(\mathcal{D})\mathcal{M}^{(3)}(\mathcal{D})^T $$
For any $\mathbf{u} \in \mathbb{R}^{R}$ such that $\| \mathbf{u}\|_2=1:$
    \begin{align*}
       \mathbf{u}^T  \mathbb{E}[\mZ^{(3)}\mZ^{(3)}]\mathbf{u} &= \frac{1}{M}\sum_{r_1 \in \R } \sum_{i\in [N^{(1)}]} \sum_{j\in [N^{(2)}]} \left(\mathcal{D}_{ijr_1} \mathbf{u}_{r_1}^2\right) -\sum_{r_1, r_2 \in \R }\frac{1}{M}\sum_{i\in [N^{(1)}]} \sum_{j\in [N^{(2)}]} \mathcal{D}_{ijr_1}\mathcal{D}_{ijr_2} \mathbf{u}_{r_1}\mathbf{u}_{r_2}  \\
       &=  \frac{1}{M}\sum_{r_1 \in \R } \sum_{i\in [N^{(1)}]} \sum_{j\in [N^{(2)}]} \mathcal{D}_{ijr_1}\mathbf{u}_{r_1}^2 - \frac{1}{M}\sum_{i\in [N^{(1)}]} \sum_{j\in [N^{(2)}]} \left( \sum_{r_1 \in \R } \mathcal{D}_{ijr_1} \mathbf{u}_{r_1} \right)^2   \\
              &\leq   \frac{1}{M}\sum_{r_1 \in \R } \sum_{i\in [N^{(1)}]} \sum_{j\in [N^{(2)}]} \mathcal{D}_{ijr_1}\mathbf{u}_{r_1}^2\\
            &\leq   \frac{1}{M}\sum_{r_1 \in \R } \sum_{i\in [N^{(1)}]} \sum_{j\in [N^{(2)}]} \mathcal{D}_{ijr_1} \quad \text{ since } \sum_{r_1} \mathbf{u}_{r_1}^2 =1\\
            &\leq  \frac{N^{(1)}N^{(2)}}{M}
 \end{align*}
    
Therefore: $$\| \mathbb{E}[\mZ^{(3)}\mZ^{(3)\top}] \|_{op} \leq 
 \frac{N^{(1)}N^{(2)}}{M}$$
We also have, for any $r_0\in [R]$, letting $\tilde{\mathbf{u}}$ be the indicator vector of $r_0$ (ie $\tilde{\mathbf{u}}_{r_0} = 1$, and $\tilde{\mathbf{u}}_r =0 $ for any $r\neq r_0$
):

\begin{align*}
\|\mathbb{E}[\mZ^{(3)}\mZ^{(3)\top}]\|_{op} &\geq  \frac{1}{M}\sum_{r_1 \in R } \sum_{i\in [N^{(1)}]} \sum_{j\in [N^{(2)}]} \mathcal{D}_{ijr_1}\tilde{\mathbf{u}}_{r_1}^2 - \frac{1}{M}\sum_{i\in [N^{(1)}]} \sum_{j\in [N^{(2)}]} \left( \sum_{r_1 \in \R } \mathcal{D}_{ijr_1} \tilde{\mathbf{u}}_{r_1} \right)^2 \\
&= \frac{1}{M} \sum_{i\in [N^{(1)}]} \sum_{j\in [N^{(2)}]} \mathcal{D}_{ijr_0} - \frac{1}{M}\sum_{i\in [N^{(1)}]} \sum_{j\in [N^{(2)}]} \left(  \mathcal{D}_{ijr_0}  \right)^2 \\
&\geq c^*\frac{N^{(1)}N^{(2)}}{M} 
 \end{align*}
with $c^* = \min_{r} \mathcal{D}_{ijr}(1-\mathcal{D}_{ijr}). $

\end{proof}

\subsection{Concentration bounds of the matricized noise}

\begin{lemma}[Concentration of the maximum of the entries in $\Z.$]\label{lemma: max Z}
Suppose $\frac{\log(N_R)}{M} \leq c_0$ for some constant $c_0$.
Then with probability at least $1-o(N_R^{-1})$, for every $i\in[N^{(1)}],j\in[N^{(2)}]$ and $r\in[R]$, we have
$$
|\Z_{ijr}|\leqslant \sqrt{\frac{\log(N_R)}{M}}
$$
\end{lemma}

\begin{proof}
Recall that, in Remark~\ref{remark: z multinomial}, we defined $\Z_{ijr}$ as the average of $M$ i.i.d centered variables: $\mathcal{Z}_{ijr}=\frac{1}{M}\sum_{m\in[M]}\mathcal{X}_{ijm}(r)$ where $|\mathcal{X}_{ijm}(r)|\leqslant 1$ and $\mathbb{E}[\mathcal{X}_{ijm}(r)]=0$. 
We further note that by direct application  of Equation~\eqref{eq:second moment X im}, $\text{Var}[\mathcal{X}_{ijm}(r)]\leqslant \mathcal{D}_{ijr}\leqslant 1$
. By Bernstein's inequality \cite{vershynin2018high} (Lemma \ref{lemma: bernstein ineq}), we then have, for any $t>0$:
$$
\mathbb{P}\left(|\mathcal{Z}_{ijr}|\geqslant t\right)\leqslant 2\exp\left(-\frac{Mt^2/2}{1+t/3}\right)
$$
Taking $t=C^*\sqrt{\frac{\log(N_R)}{M}}$ for some constant $C^*$, we have, by a simple union bound:
\begin{equation}
    \begin{split}
        \P[\exists (i,j,r) : |\mathcal{Z}_{ijr}|\geqslant C^*\sqrt{\frac{\log(N_R)}{M}}] &\leq 2 N^{(1)}N^{(2)} R\exp\left(-\frac{ \frac{(C^*)^2}{2} \log(N_R)}{1+ \frac{C^*}{3} \sqrt{\frac{\log(N_R)}{M}}}\right)\\
         &\leq \exp\left(3 \log(N_R) -\frac{ \frac{(C^*)^2}{2} \log(N_R)}{1+ \frac{C^*}{3} \sqrt{\frac{\log(N_R)}{M}}}\right)\\
         &\leq \exp( -(C_0 -3) \log(N_R)).
    \end{split}
\end{equation}
for some constant $C_0$, assuming that the ratio $\frac{\log(N_R)}{M}$ is bounded above by a constant $c_0.$
This concludes the proof.
\end{proof}

\begin{lemma}[Concentration of the operator norm of the matricized noise]\label{lem: Z l2 no}
Suppose $\frac{\log(N_R)}{M} \leq c_0$ for some constant $c_0$. Then, with probability at least $1-o(N_R^{-1})$:
    \begin{equation}\label{eq:norm op Z1}
        \begin{split}
              &   \|\mZ^{(1)}\|_{op} \leqslant C^*\sqrt{\frac{N^{(1)}N^{(2)        }\log(N_R)}{M}} \\
   &    \|\mZ^{(2)}\|_{op}\leqslant C^*\sqrt{\frac{N^{(1)}N^{(2)        }\log(N_R)}{M}} 
         \\
         &\|\mZ^{(3)}\|_{op}\leqslant C^*\sqrt{\frac{N^{(1)}N^{(2)        }\log(N_R)}{M}}.    
        \end{split}
    \end{equation}

    Moreover, with the same probability:
        \begin{equation}      \label{eq:bound norm Z1}  
\|\mZ^{(1)}_{i\star}\|_{2}\leqslant C^*\sqrt{\frac{  N^{(2)}\log(N_R)}{M}}\qquad \text{ and } \qquad \|\mZ^{(2)}_{j\star}\|_{2}\leqslant C^*\sqrt{\frac{  N^{(1)}\log(N_R)}{M}} 
        \end{equation}
\end{lemma}

\begin{proof}
\textbf{Analysis of Mode-1 and Mode-2 matricization:} 
Consider first the random matrix $\mZ^{(1)}\in\mathbb{R}^{N^{(1)}\times (N^{(2)}R)}$. By Remark \ref{remark: z multinomial}, the rows of the matrix $\mZ^{(1)}$ are independent, and, as before, we write $\mZ^{(1)}$ as an average of $M$ variables:
$$
\mZ^{(1)\top}=\frac{1}{M}\sum_{m=1}^M \mX^{(1)}_m
$$
where the set of matrices $\{\mX^{(1)}_m\}_{m\in[M]}\in\mathbb{R}^{N^{(2)}R\times N^{(1)}}$  are independent matrices  such that $\mX^{(1)}_m(jr,i)=\mathcal{X}_{ijm}(r)$. The product $\mX^{(1)}_m \mX^{(1)\top}_m$ has entries $\left(\mX^{(1)}_m \mX^{(1)\top}_m\right)(j_1r_1,j_2r_2)=\sum_{i\in[N^{(1)}]}\mathcal{X}_{ij_1m}(r_1)\mathcal{X}_{ij_2m}(r_2)$. By \eqref{eq: expection of Z}:
\begin{equation*}
    \mathbb{E}[\left(\mX^{(1)}_m \mX^{(1)\top}_m\right)(j_1r_1,j_2r_2)]=\begin{cases}
        -\sum_{i\in[N^{(1)}]}\mathcal{D}_{ijr_1}\mathcal{D}_{ijr_2} &\text{ if } j_1=j_2=j, r_1\neq r_2\\
       \sum_{i\in[N^{(1)}]} (\mathcal{D}_{ijr}-\mathcal{D}_{ijr}^2)&\text{ if } j_1=j_2=j, r_1=r_2=r\\
        0 & \text{ otherwise }
     \end{cases}
\end{equation*}
Therefore:

\begin{align*}
\left\|\sum_{m}\mathbb{E}[\mX^{(1)}_m\mX^{(1)\top}_m]\right\|_{op}&=\left\|\sum_{m=1}^M\mathbb{E}[\mX^{(1)}_m\mX^{(1)\top}_m]\right\|_{op}\leqslant\left\|\sum_{m=1}^M\mathbb{E}[\mX^{(1)}_m\mX^{(1)\top}_m]\right\|_F\leqslant\sum_{m=1}^M\left\|\mathbb{E}[\mX^{(1)}_m\mX^{(1)\top}_m]\right\|_F\\
&\leqslant M \sqrt{\sum_{j\in [N^{(2)}], r \in[R]} \left( \sum_{i \in [N^{(1)}]} \mathcal{D}_{ijr}(1-\mathcal{D}_{ijr}) \right)^2+\sum_{j \in [N^{(2)}]}\sum_{r_1\neq r_2}\left(\sum_{i \in [N^{(1)}]} \mathcal{D}_{ijr_1}\mathcal{D}_{ijr_2}\right)^2}\\
&\leqslant M \sqrt{N^{(1)}} \sqrt{\sum_{i,j,r}  \mathcal{D}_{ijr}^2(1-\mathcal{D}_{ijr})^2+\sum_{i,j}\sum_{r_1\neq r_2}(\mathcal{D}_{ijr_1}\mathcal{D}_{ijr_2})^2} \text{ by Cauchy Schwarz}\\
&\leqslant M \sqrt{N^{(1)}} \sqrt{\sum_{i,j,r} ( \mathcal{D}_{ijr}^2-2\mathcal{D}_{ijr}^3 + \mathcal{D}_{ijr}^4)+\sum_{i,j}\sum_{r_1\neq r_2}(\mathcal{D}_{ijr_1}\mathcal{D}_{ijr_2})^2}\\
&= M \sqrt{N^{(1)}}\sqrt{\sum_{i,j,r}  \mathcal{D}_{ijr}^2(1-2\mathcal{D}_{ijr})+\sum_{i,j}\sum_{r_1, r_2}(\mathcal{D}_{ijr_1}\mathcal{D}_{ijr_2})^2}\\
&= M \sqrt{N^{(1)}}\sqrt{\sum_{i,j,r}  \mathcal{D}_{ijr}^2(1-2\mathcal{D}_{ijr})+\sum_{i,j}(\sum_{r_1}\mathcal{D}_{ijr_1}^2)^2}\\
&\leq M \sqrt{N^{(1)}}\sqrt{\sum_{i,j,r}  \mathcal{D}_{ijr}+\sum_{i,j}(\sum_{r_1}\mathcal{D}_{ijr_1})^2}\\
&\leqslant M \sqrt{N^{(1)}}\sqrt{\sum_{i,j} 1+\sum_{i,j}1} \qquad \text{since } \sum_r \D_{ijr} = 1\\
&\leqslant  M N^{(1)} \sqrt{2 N^{(2)}}\leqslant  M N^{(1)} N^{(2)}.
\end{align*}
On the other hand, we have $[\mX^{(1)\top}_m\mX^{(1)}_m]_{i_1i_2}=\sum_{j,r}\mathcal{X}_{i_1jm}(r)\mathcal{X}_{i_2jm}(r)$. 
We have, by \eqref{eq: expection of Z}:

$$\mathbb{E}[\mX^{(1)\top}_m\mX^{(1)}_m]=\text{diag}\left(\sum_{jr}(\mathcal{D}_{1jr}-\mathcal{D}_{1jr}^2),\dots,\sum_{jr}(\mathcal{D}_{N^{(1)}jr}-\mathcal{D}_{N^{(1)}jr}^2)\right).$$ Thus 
$\|\mathbb{E}[\mX^{(1)\top}_m\mX^{(1)}_m]\|_{op}\leqslant N^{(2)}$ and $\|\sum_m\mathbb{E}[\mX^{(1)\top}_m\mX^{(1)}_m]\|_{op}\leqslant N^{(2)}M$. 

We further note that for any $i \in [N^{(1)}]$, $$\|(\mX^{(1)}_{m})_{i\star }\|^2_2=\sum_{j,r}\mathcal{X}_{ijm}(r)^2=\sum_{j,r}(\mT_{ijm}(r)-\mathbb{E}[\mT_{ijm}(r)])^2\leqslant 2N^{(2)}.$$

 Thus, for any unit norm vector $\mathbf{u}\in\mathbb{R}^{N^{(1)}},  \mathbf{v} \in\mathbb{R}^{N^{(2)}R}$,
 $$\mathbf{u}^\top (\mX^{(1)}_{m}) \mathbf{v}  = \sum_i \sum_{j,r} \mathbf{u}_i\mathcal{X}_{ijm}(r)  \mathbf{v} _{(j_r)} \leqslant  \sqrt{\sum_{i} \mathbf{u}_i^2} \sqrt{\sum_i ( (X_m)_{i \star}  \mathbf{v}  )^2}
 \leqslant \sqrt{2N^{(1)} N^{(2)}} .$$

It follows that
$$\|\mX^{(1)}_m\|_{op}^2=\sup_{\mathbf{u}\in\mathcal{S}^{N^{(1)}-1}, \mv\in\mathcal{S}^{N^{(2)}R-1}}\|\mathbf{u}^\top \mX^{(1)}_{m}\mv\|_2^2\leqslant 2N^{(2)} N^{(1)} $$
By the Matrix Bernstein inequality (Lemma \ref{lemma: matrix bernstein}), we have 
$$
\mathbb{P}\left(\left\|\frac{1}{M}\sum_{m=1}^M \mX^{(1)}_m\right\|_{op}\geqslant t\right)\leqslant (N^{(1)} + N^{(2)}R) \exp\left(-\frac{t^2 M^2}{2( \sigma^2 +LtM/3)}\right)
$$
where $\sigma^2 =  N^{(1)}N^{(2)}M$ and $\mL =  \sqrt{2N^{(1)}N^{(2)}}$.
\par Choosing $t=C^*\sqrt{\frac{N^{(1)} N^{(2)}\log (N_R)}{M}}$ for some constant $C^*$, and noting that $N^{(1)} + N^{(2)}R \leq 2N_R^2$, the previous inequality becomes:
\begin{equation*}
    \begin{split}
        \mathbb{P}\left(\left\|\frac{1}{M}\sum_{m=1}^M \mX^{(1)}_m\right\|_{op}\geqslant C^*\sqrt{\frac{N^{(1)} N^{(2)}\log (N_R)}{M}}\right)&\leqslant  2N_R^2\exp\left(-\frac{(C^*)^2MN^{(1)} N^{(2)}\log (N_R)}{2( N^{(1)}N^{(2)}M + \frac{\sqrt{2}}{3}\sqrt{N^{(1)}N^{(2)}M\log(N_R)} })\right)\\
        &=  \exp\left(2\log(N_R) + \log(2)-\frac{\frac{(C^*)^2}{2}\log (N_R)}{ 1+ \frac{\sqrt{2}}{3} \sqrt{\frac{\log(N_R)}{N^{(1)}N^{(2)}M}}}\right).\\
    \end{split}
\end{equation*}
Therefore, for $C^*$ large enough, the result in the first line of Equation~\eqref{eq:norm op Z1} holds. 
\par We now turn to the bound on the row norm $\|\mZ^{(1)}_{i\star}\|_2$ (Equation~\eqref{eq:bound norm Z1}). In this case, we note that for any tuple $i,j, r$, since $|\mathcal X_{imj}(r) | \leq 1$ and $\text{Var}(\mathcal X_{imj}(r)) \leq 1$ (see Remark \ref{remark: z multinomial}),  by Bernstein's inequality (Lemma ~\ref{lemma: bernstein ineq}), for a choice of $t  = C^* \sqrt{\frac{ \log(N_R)}{M}}$ with $C^*$ a constant:

$$ \P[ | \frac{1}{M} \sum_{m=1}^M \mathcal X_{ijm}(r) | > C^* \sqrt{\frac{\log(N_R)}{M}} ] \leq 2 \exp ( - \frac{(C^*)^2\log(N_R)/2 }{1 + \frac{C^*}{3} \sqrt{\frac{ \log(N_R) }{M}}}  )$$
Equation~\eqref{eq:bound norm Z1} follows by a simple union bound, with a choice of $C^*$ large enough and provided that the ratio $\frac{\log(N_R)}{M} \leq c_0$.

\par The results for $\mZ^{(2)}$ follow directly from those on $\mZ^{(1)}$, due to the symmetry of the roles played by the first two dimensions.
\vspace{1cm}
\noindent\\
\textbf{Analysis of Mode-3 matricization:} 
We now adapt the proof to show the results for $\mZ^{(3)}$, due to its row-stochastic properties (matrix of type \ref{case:2}). In this case, the matrix $\mZ^{(3)}\in\mathbb{R}^{R\times (N^{(1)}N^{(2)})}$ has independent column vectors.  We set define matrix $\mX^{(3)}_m\in\mathbb{R}^{R\times N^{(1,2)}} $ such that $[\mX^{(3)}_m \mX^{(3)\top}_m](r_1,r_2)=\sum_{ij}\mathcal{X}_{ijm}(r_1)\mathcal{X}_{ijm}(r_2)$. Then, by Remark \ref{remark: z multinomial}, Equation~\eqref{eq: expection of Z}:
$$\mathbb{E}[\mX^{(3)}_m \mX^{(3)\top}_m]=\text{diag}(\mD^{(3)}\mathbf{1}_{N^{(1,2)}})-\mD^{(3)}\mD^{(3)\top},$$
which, 
Lemma \ref{lemma: singular value D}, implies $\|\sum_{m=1}^M\mathbb{E}[\mX^{(3)}_m \mX^{(3)\top}_m]\|_{op}\leqslant M \|(\mD^{(3)}\mD^{(3)\top}\|_{op} \leq K^{(1,2,3)} N^{(1,2)}M . $
Furthermore,
$$
\mathbb{E}[\mX^{(3)\top}_m \mX^{(3)}_m]=\text{diag}\left(\{1-\sum_{r\in[R]}\mathcal{D}^2_{ijr}\}_{(i,j)\in[N^{(1)}]\times [N^{(2)}]}\right)
$$
By the same arguments as above, we have $\|\sum_{m\in[M]}\mathbb{E}[\mX^{(3)}_m \mX^{(3)\top}_m]\|\leqslant M$. We then define: 
$\sigma^2=K^{(1)}K^{(2)}K^{(3)} N^{(1)}N^{(2)}\lesssim  N^{(1)}N^{(2)} M$. 
To apply the matrix Bernstein, we now need to bound the operator norm $\| \mX^{(3)}_m\|$.  Since, by definition of $\mathcal{X}$, for each column $(i,j)$,
 $\|(\mX^{(3)}_m)_{\star (i,j)}\|_2^2=\sum_r (\mathcal{X}_{ijm}(r))^2\leqslant 2$, it is easy to see that $\|\mX^{(3)}_m\|_{op}^2\leq \|\mX^{(3)}_m\|_{F}^2 \leqslant 2N^{(1)} N^{(2)}$ as well.  The result on the operator norm of $\mZ^{(3)}$ in Equation~\eqref{eq:norm op Z1} thus follows by similar arguments to mode-1 matricization, using the Matrix Bernstein's inequality.
 

\end{proof}


\begin{lemma}\label{lemma: sumij zij}
Suppose $f_r\geqslant {\frac{\log(N_R)}{N^{(1)}N^{(2)}M}}$. Then with probability at least $1-o(N_R^{-1})$,  the following inequality holds:
\begin{align}
   \left|\frac{1}{N^{(1)}N^{(2)}}\sum_{i,j=1}^{N^{(1)},N^{(2)}}\mathcal{Z}_{ijr}\right|\leqslant C^*\sqrt{\frac{f_r\log(N_R)}{N^{(1)}N^{(2)}M}} \text{ for all }r\in[R]\label{eq:average_X}
\end{align}

\end{lemma}
\begin{proof}
    We note again that $\mathcal{Z}_{ijr}=\frac{1}{M}\sum_m \mathcal{X}_{ijm}(r)$ where $\mathcal{X}_{ijm}(r)$ is independent across $i\in[N^{(1)}]$, $j\in[N^{(2)}]$ and $m\in[M]$. Moreover, each $\mathcal{X}_{ijm}(r)$ is such that $|\mathcal{X}_{ijm}(r)|\leqslant 1$ and $\mathbb{E}[\mathcal{X}_{ijm}(r)]=0$, $\text{Var}[\mathcal{X}_{ijm}(r)]=\mathcal{D}_{ijr}-\mathcal{D}_{ijr}^2\leqslant\mathcal{D}_{ijr}\leqslant f_r$.
    
    Therefore, by Berstein's inequality (Lemma \ref{lemma: bernstein ineq}), for any $t>0$:
     $$ \mathbb{P}\left(\frac{1}{N^{(1)}N^{(2)}M}\left|\sum_{i,j,m} \mathcal{X}_{ijm}(r)\right| \geq t \right) \leq 2 \exp\left(-\frac{N^{(1)}N^{(2)}Mt^2/2}{f_r + t/3}\right)$$
    Choosing $t=C^*\sqrt{\frac{f_r\log(N_R)}{N^{(1)}N^{(2)}M}}$, the previous equation becomes:
    \begin{equation}
        \begin{split}
\mathbb{P}\left(\frac{1}{N^{(1)}N^{(2)}M}\left|\sum_{i,j,m} \mathcal{X}_{ijm}(r)\right| \geq C^*\sqrt{\frac{f_r\log(N_R)}{N^{(1)}N^{(2)}M}} \right) &\leq 2 \exp\left(-(C^*)^2\frac{\log(N^{(r)})/2}{1 + \frac{C^*}{3}\sqrt{\frac{\log(N_R)}{f_rN^{(1)}N^{(2)}M}}}\right)\\
&\leq \exp(-(C^* \log(N_R))
        \end{split}
    \end{equation}
    as long as $f_r\geqslant\frac{\log(N_R)}{N^{(1)}N^{(2)}M}$. By application of a simple union bound along the $R\leq N_R$ words, we prove equation~\eqref{eq:average_X}.
    
\end{proof}
\begin{lemma} \label{lemma: zijr square} Suppose $f_r\geqslant \sqrt{\frac{\log(N_R)}{N^{(1)}N^{(2)}M}}$ for all $r\in[R]$. Then, with probability at least $1-o(N_R^{-1})$,
     \begin{align}
      &\left|\frac{1}{N^{(2)}}\sum_{j,r=1}^{N^{(2)},R}\mathcal{Z}_{ijr}^2-\frac{1}{M} (1 - \frac{1}{N^{(2)}}\sum_{j, r} \D_{ijr}^2)\right|\leqslant C^*\sqrt{\frac{\log(N_R)}{N^{(2)}M}} \label{eq:z2 over jr}\\
      &\left|\frac{1}{N^{(1)}}\sum_{i,r=1}^{N^{(1)},R}\mathcal{Z}_{ijr}^2-\frac{1}{M} (1 - \frac{1}{N^{(1)}}\sum_{i, r} \D_{ijr}^2)\right|\leqslant C^*\sqrt{\frac{\log(N_R)}{N^{(1)}M}} \label{eq:z2_over_ir}\\
      &\left|\frac{1}{N^{(1)}N^{(2)}}\sum_{i,j=1}^{N^{(1)},N^{(2)}}\mathcal{Z}_{ijr}^2-\frac{1}{M}\cdot\frac{1}{N^{(1)} N^{(2)}}\sum_{i,j}(\mathcal{D}_{ijr}-\mathcal{D}_{ijr}^2)\right|\leqslant C^*\sqrt{\frac{f_r\log(N_R)}{N^{(1)}N^{(2)}M}} +C^*\sqrt{\frac{f_r^2\log(N_R)}{N^{(1)}N^{(2)}M}} \label{eq:z2_over_ij}
      \end{align}
\end{lemma}

\begin{proof}
We note that for the squared entries $\mathcal{Z}_{ijr}^2$, the following equalities hold:
    \begin{equation}\label{eq: z square expansion}
    \begin{aligned}
    \mathcal{Z}_{ijr}^2&=\frac{1}{M^2}\sum_{m}\sum_{m'} \mathcal{X}_{ijm}(r)\mathcal{X}_{ijm'}(r)\\
    &=\frac{1}{M^2}\sum_{m} \mathcal{X}_{ijm}(r)^2+\frac{1}{M^2}\sum_{m\neq m'} \mathcal{X}_{ijm}(r)\mathcal{X}_{ijm'}(r)\\
    &=\frac{1}{M^2}\sum_{m} \mathcal{X}_{ijm}(r)^2+\frac{M-1}{M}\frac{1}{M(M-1)}\sum_{m\neq m'} \mathcal{X}_{ijm}(r)\mathcal{X}_{ijm'}(r)\\
    &=\frac{1}{M^2}\sum_{m} \mathcal{X}_{ijm}(r)^2+\frac{M-1}{M}\frac{1}{M!}\sum_{\pi\in{\Pi}[M]}\frac{1}{\lfloor M/2\rfloor}\sum_{m\in \lfloor M/2\rfloor} \mathcal{X}_{ij\pi(2m-1)}(r)\mathcal{X}_{ij\pi(2m)}(r)
    \end{aligned}
     \end{equation}
     We first note that by Remark~\ref{remark: z multinomial} Equation \eqref{eq: expection of Z}, $$\E[\Z_{ijr}^2] = \frac{1}{M} \D_{ijr}(1-\D_{ijr}). \qquad \eqref{eq: prob v in bernstein permute}  $$
Moreover, for the first term, leftmost in the sum above, we propose using Bernstein's inequality  (Lemma \ref{lemma: bernstein ineq}), while we will consider Lemma \ref{lemma: Berstein involving permutation} to bound the second summand. The use of Bernstein inequalities requires characterizing the means and variances of the corresponding random variables $\mathcal{X}_{ijm}(r)^2$. 

Consequently, we split the analysis depending on the matricization mode:
\begin{itemize}
    \item \textbf{Mode 1 \& 2 matricization.} 
As highlighted in Remark~\ref{remark:symmetry_mode12}, Modes 1 and 2 exhibit similar behaviours. We analyse here mode 2, and the results for mode 1 follow by replacing $N^{(2)}$ with $N^{(1)}$. 
\begin{enumerate}
\item {\bf Expectation:} In this case, given observation \eqref{eq: prob v in bernstein permute} above, the expectation $\frac{1}{N^{(1)}}\sum_{i,r}\Z_{ijr}^2 $ simplifies to:
\begin{equation*}
    \begin{split}
\frac{1}{N^{(1)}}\sum_{i,r}\mathbb{E}[\Z_{ijr}^2]&=   \frac{1}{N^{(1)}}\sum_{i,r}   \frac{1}{M} \D_{ijr}(1-\D_{ijr})  = \frac{1}{M} - \frac{1}{MN^{(1)}}\sum_{i,r}\D_{ijr}^2. 
    \end{split}
\end{equation*}
    \item {\textbf{First summand:}}
    Define $\tilde{\Q}_{ijm}=\sum_{r\in[R]}\mathcal{X}_{ijm}^2(r)$. It is easy to see  that $|\tilde{\Q}_{ijm}|\leqslant 2$. Indeed, letting $r_0$ denote the index for which $ \T_{ijm}(r_0) = 1$, we have: 
    $$\tilde{\Q}_{ijm} = \sum_{r\in[R]} \big( \T_{ijm}(r)  -2 \T_{ijm}(r)\D_{ijr} + \D_{ijr}^2) \leq 1 -2 \D_{ijr_0}  + 1 \leq 2.$$
    Furthermore, we have:
    $\mathbb{E}[\tilde{\Q}_{ijm}]=\sum_{r}\mathcal{D}_{ijr}(1-\mathcal{D}_{ijr})$, and 
    \begin{equation*}
\begin{split}
    \text{Var}[\tilde{\Q}_{ijm}] &= \sum_{r,r'} \E[ \mathcal{X}^2_{ijm}(r) \mathcal{X}^2_{ijm}(r')] - \E[\tilde{Q}_{ijm}]^2\\
&\leq \sum_{r}\mathcal{D}_{ijr}(1-\mathcal{D}_{ijr})(1-3\mathcal{D}_{ijr}(1-\mathcal{D}_{ijr}))\\
&+\sum_{r\neq r'}\mathcal{D}_{ijr}\mathcal{D}_{ijr'} \left( -7\mathcal{D}_{ijr}\mathcal{D}_{ijr'} + 3(\mathcal{D}_{ijr}+\mathcal{D}_{ijr'} ) -1\right)\\
&\leq \sum_{r}\mathcal{D}_{ijr} +2\sum_{r, r'}\mathcal{D}_{ijr}\mathcal{D}_{ijr'} \\
&\leqslant 3
\end{split}
    \end{equation*} 
    Therefore, by Bernstein's inequality (Lemma \ref{lemma: bernstein ineq}), for any $t>0$: 
$$\mathbb{P}\left(\left|\frac{1}{N^{(1)}M}\sum_{i,m}\tilde{\Q}_{ijm}-\mathbb{E}[\frac{1}{N^{(1)}M}\sum_{i,m}\tilde{\Q}_{ijm}]\right|\geq t\right)\leqslant 2 \exp\left(-\frac{N^{(1)}M t^2/2}{3+2t/3}\right)$$
Choosing $t=C^*\sqrt{\frac{\log (N_R)}{N^{(1)}M}}$ for some constant $C^*$, we obtain:
\begin{equation}\label{eq:1st summand}
\mathbb{P}\left(\left|\frac{1}{N^{(1)}M}\sum_{i,m}\tilde{\Q}_{ijm}-\mathbb{E}[\frac{1}{N^{(1)}M}\sum_{i,m}\tilde{\Q}_{ijm}]\right|\geq C^*\sqrt{\frac{\log (N_R)}{N^{(1)}M}} \right)\leqslant 2 \exp\left(-C \log(N_R)\right).
\end{equation}
Therefore, by a simple union bound, provided that $c$ is large enough, there exists a constant ${C}$ such that:
$$\mathbb{P}\left(\forall j \in [N^{(2)}]: \left|\frac{1}{N^{(1)}M}\sum_{i,m}\tilde{\Q}_{ijm}-\mathbb{E}[\frac{1}{N^{(1)}M}\sum_{i,m}\tilde{\Q}_{ijm}]\right|\geq c\sqrt{\frac{\log (N_R)}{N^{(1)}M}} \right)\leqslant  e^{-{C} \log(N_R)}.$$
    \item {\textbf{Second summand:}} {For the terms involving $m\neq m'$}, we define $\Q_{ijmm'}=\sum_{r}\mathcal{X}_{ijm}(r)X_{ijm'}(r)$. In this case, it is also possible to verify that, for fix $i,j,m,m'$,
    $$|\Q_{ijmm'}| \leqslant 2,\qquad \mathbb{E}[\Q_{ijmm'}]=0,$$ 
        \begin{equation*}
\begin{split}
    \text{Var}[{\Q}_{ijmm'}] &= \sum_{r,r'} \E[ \mathcal{X}_{ijm}(r) \mathcal{X}_{ijm'}(r)\mathcal{X}_{ijm}(r')\mathcal{X}_{ijm'}(r')] - \E[{\Q}_{ijmm'}]^2\\
    &= \sum_{r,r'} \E[ \mathcal{X}_{ijm}(r) \mathcal{X}_{ijm}(r')]\E[  \mathcal{X}_{ijm'}(r)\mathcal{X}_{ijm'}(r')]\\
&= \sum_{r}\mathcal{D}_{ijr}^2(1-\mathcal{D}_{ijr})^2 -\sum_{r\neq r'}\mathcal{D}_{ijr}^2\mathcal{D}_{ijr'}^2\\
&\leqslant 1.
\end{split}
    \end{equation*} 

Applying Lemma \ref{lemma: Berstein involving permutation}, we obtain that, with probability $1-\mO(\frac{1}{N_R})$ 
\begin{equation}\label{eq:2nd summand}
    \left|\frac{1}{N^{(1)}M!}\sum_{i, \pi}\sum_{m}\Q_{ij,\pi}\right|\leqslant C^* \sqrt{\frac{ M\log (N_R)}{N^{(1)}}}.
\end{equation}
By a simple union bound argument, we thus have:
$$ \P[\forall m, m':\quad    \left|\frac{1}{N^{(1)}M!}\sum_{i, \pi}\sum_{m}\Q_{ij,\pi}\right|\leqslant C^* \sqrt{\frac{ M\log (N_R)}{N^{(1)}}}]\leq \exp(-C\log(N_R)). $$
Adding Equations~\eqref{eq:1st summand} and \eqref{eq:2nd summand} yields the results in \eqref{eq:z2 over jr} and subsequently, by substitution of $i$ to $j$ and $N^{(1)}$ by $N^{(2)}$, in \eqref{eq:z2_over_ir}.

\end{enumerate}
\item \textbf{Mode 3 matricization} 
\begin{enumerate}
    \item {\textbf{First summand:}} we have, for any tuple $i,j,r$: $|\mathcal{X}_{ijm}(r)^2|\leqslant 2$, $\mathbb{E}[\mathcal{X}_{ijm}(r)^2]=\mathcal{D}_{ijr}-\mathcal{D}_{ijr}^2$ and $\text{Var}[\mathcal{X}_{ijm}(r)^2]\leqslant \mathcal{D}_{ijr}\leqslant f_r$ (see Appendix~\ref{app:propertiesone hot}, noting that here documents are indexed by the tuple $(i,j)$).
    
    By direct application of the Bernstein Inequality (Lemma \ref{lemma: bernstein ineq}), as long as $f_r\geqslant\frac{\log (N_R)}{MN^{(1)}N^{(2)}}$, then with probability at least $1-\mO(\frac{1}{N_R}):$
    \begin{equation}\label{eq:1st summand mode3}
    \left|\frac{1}{MN^{(1)}N^{(2)}}\sum_{ijm}\mathcal{X}_{ijm}(r)^2-\mathbb{E}[\frac{1}{MN^{(1)}N^{(2)}}\sum_{ijm}\mathcal{X}_{ijm}(r)^2]\right|\leqslant C^*\sqrt{\frac{f_r\log (N_R)}{MN^{(1)}N^{(2)}}}.
        \end{equation}
    \item {\textbf{Second summand:}} For $m\neq m'$, we consider $\Q^r_{ijm} = \X_{ij(2m-1)}(r)\X_{ij(2m)}(r)$.
    It is easy to see that for any fixed $r$, the $\Q^r$s are such that, for any tuple of indices $(i,j,m)$:
    $|\Q^r_{ijm}| = |\X_{ij(2m-1)}(r)\X_{ij(2m)}(r)|\leqslant 1$, $\E[\Q^r_{ijm}]=\mathbb{E}[\X_{ij(2m-1)}(r)\X_{ij(2m)}(r)]=0$ and $\text{Var}[\X_{ij(2m-1)}(r)\X_{ij(2m)}(r)]=(\mathcal{D}_{ijr}-\mathcal{D}_{ijr}^2)^2\leqslant \mathcal{D}_{ijr}\mathcal{D}_{ijr}\leqslant  f_r^2$. Let $\tilde M=\lfloor M/2\rfloor$. Applying Lemma \ref{lemma: Berstein involving permutation}, for any  $t>0$:
     $$   \mathbb{P}\left(\left|\frac{1}{N^{(1,2)}}\frac{1}{M!}\sum_{\pi\in \Pi[M]}\sum_{i,j}\sum_{m=1}^{\tilde M}\X_{ij(2m-1)}(r)\X_{ij(2m)}(r)\right|\geqslant t\right)\leqslant\exp\left(-\frac{N^{(1,2)}t^2/2}{ \tilde M f_r^2+t/3}\right)$$
Choosing $t=C^*\sqrt{\frac{f_r^2\tilde M\log (N
_R)}{N^{(1,2)}}}$ as long as  $f_r\geqslant c^*\sqrt{\frac{\log (N_R)}{MN^{(1)}N^{(2)}}}$, we obtain a bound on the second term that is similar to \eqref{eq:1st summand mode3}, and adding the two, we obtain result~\eqref{eq:z2_over_ij}. 
\end{enumerate}
\end{itemize}

\end{proof}

\begin{lemma}\label{lemma: zz} Suppose $f_r\geqslant \sqrt{\frac{\log(N_R)}{N^{(1)}N^{(2)}M}}$ for all $r\in[R]$. Then, with probability at least $1-o(N_R^{-1})$:
     \begin{align}
     &  \left|\frac{1}{N^{(2)}}\sum_{j,r=1}^{N^{(2)},R}\mathcal{Z}_{ijr}\mathcal{Z}_{i'jr}\right|\leqslant C^*\sqrt{\frac{\log(N_R)}{N^{(2)}M}} \text{ for all } i\neq i' \label{eq:z cross over jr}\\
     &  \left|\frac{1}{N^{(1)}}\sum_{i,r=1}^{N^{(1)},R}\mathcal{Z}_{ijr}\mathcal{Z}_{ij'r}\right|\leqslant C^*\sqrt{\frac{\log(N_R)}{N^{(1)}M}} \text{ for all } j\neq j' \label{eq:z cross over ir}\\
     &
    \left|\frac{1}{N^{(1)}N^{(2)}}\sum_{i,j=1}^{N^{(1)},N^{(2)}}\mathcal{Z}_{ijr}\mathcal{Z}_{ijr'}+\frac{1}{M}\mathcal{D}_{ijr}\mathcal{D}_{ijr'}  \right|\leqslant C^*\sqrt{\frac{f_r f_{r'}\log(N_R)}{N^{(1)}N^{(2)}M}} \text{ for all } r\neq r'
     \end{align}
     
\end{lemma}

\begin{proof}
We rewrite the product $\mathcal{Z}_{ijr}\mathcal{Z}_{i'jr}$, $\mathcal{Z}_{ijr}\mathcal{Z}_{ij'r}$ and $\mathcal{Z}_{ijr}\mathcal{Z}_{ijr'}$ for $i\neq i'$, $j\neq j'$, and $r\neq r'$ in a similar way to Equation \eqref{eq: z square expansion}.
Like in the proof of Lemma \ref{lemma: zijr square}, we use Bernstein inequalities  (Lemma \ref{lemma: bernstein ineq} and \ref{eq: bernstein permutation indep sum}) to link two terms involved in \eqref{eq: z square expansion}.
\begin{itemize}
    \item \textbf{Mode 1 \& 2 matricization} 
Since mode-1 and  mode-2 matricizations (Equations \eqref{eq:z cross over jr} and \eqref{eq:z cross over ir} respectively) have similar behaviors, we provide here the proof for mode-1 matricization.
\begin{enumerate}
\item Adapting Equation \eqref{eq: z square expansion}, we have for all $i \neq i'$:
\begin{equation*}
    \begin{split}
         \left|\frac{1}{N^{(2)}}\sum_{j,r=1}^{N^{(2)},R}\mathcal{Z}_{ijr}\mathcal{Z}_{i'jr}\right|&=\left|\frac{1}{M^2 N^{(2)}}\sum_{j,r=1}^{N^{(2)},R} \sum_{m,m'}^{M}\mathcal{X}_{ijm}(r)\mathcal{X}_{i'jm'}(r) \right|
    \end{split}
\end{equation*}
For a fixed tuple $i,i'$ such that $i\neq i'$, let $\Q_{jmm'}=\sum_{r\in[R]}\mathcal{X}_{ijm}(r)\mathcal{X}_{i'jm'}(r)$.
Then:
\begin{equation*}
    \begin{split}
         \left|\frac{1}{N^{(2)}}\sum_{j,r=1}^{N^{(2)},R}\mathcal{Z}_{ijr}\mathcal{Z}_{i'jr}\right|&=\left|\frac{1}{M^2 N^{(2)}}\sum_{j}^{N^{(2)}} \sum_{m,m'=1}^{M}\Q_{jmm'}\right|\\
         &=\left|\frac{1}{M^2 N^{(2)}}\sum_{j}^{N^{(2)}} \sum_{m}^{M}\Q_{jmm} + \frac{1}{M^2 N^{(2)}}\sum_{j}^{N^{(2)}} \sum_{m\neq m'}^{M}\Q_{jmm'}\right|
    \end{split}
\end{equation*}

    \item  By direct application of Cauchy-Schwarz:
    $$|\Q_{jmm}|\leq \sqrt{\sum_r\mathcal{X}_{ijm}(r)^2} \sqrt{\sum_r \mathcal{X}_{ijm'}(r)^2} \leq \sqrt{\sum_r\mathcal{X}_{ijm}(r)} \sqrt{\sum_r\mathcal{X}_{ijm'}(r)}= 1.$$
    By independence of the variables $\mathcal{X}_{ijm}(r)$ for $i \neq i'$, we have: 
    $\mathbb{E}[\Q_{jmm}]=0$, and 
    \begin{equation*}
    \begin{aligned}
    \text{Var}[\Q_{jmm}]=&\sum_{r_1}\sum_{r_2}\text{Cov}\left(\mathcal{X}_{ijm}(r_1)\mathcal{X}_{i'jm}(r_1),\mathcal{X}_{ijm}(r_2)\mathcal{X}_{i'jm}(r_2)\right)\\
    \leqslant&\sum_{r_1}\sum_{r_2}\mathcal{D}_{ijr_1}\mathcal{D}_{ijr_2}\mathcal{D}_{i'jr_1}\mathcal{D}_{i'jr_2}\leqslant c
        \end{aligned}
\end{equation*}
    By Bernstein's inequality Lemma \ref{lemma: bernstein ineq}, we have 
$$\mathbb{P}\left(\left|\frac{1}{N^{(2)}M}\sum_{j,m}\Q_{ii'jm}\right|\geq t\right)\leqslant 2 \exp\left(-\frac{N^{(2)}M t^2/2}{c+t/3}\right)$$
We then choose $t=c\sqrt{\frac{\log (N_R)}{N^{(2)}M}}$ to bound the first term.
    \item For $m\neq m'$, $\Q_{jmm'}=\sum_{r}\mathcal{X}_{ijm}(r)X_{i'jm'}(r)$ also verifies $|\Q_{jmm'}|\leqslant 1$. $\mathbb{E}[\Q_{jmm'}]=0$, $\text{Var}[\Q_{jmm'}]\leqslant c$.
By application of Lemma \ref{lemma: Berstein involving permutation}, we obtain that with probability at least $1-o(N_R^{-1})$: 
$$\left|\frac{1}{N^{(2)}M!}\sum_{j, \pi}\sum_{m}Q_{i'ij,\pi}\right|\leqslant C^* \sqrt{\frac{ M\log (N_R)}{N^{(2)}}}$$

\end{enumerate}
\item \textbf{Mode 3 matricization}. We treat the case of Mode 3 matricization separately, due to its row-stochastic structure.
\begin{enumerate}
    \item $|\mathcal{X}_{ijm}(r)\mathcal{X}_{ijm}(r')|\leqslant 1$, $\mathbb{E}[\mathcal{X}_{ijm}(r)\mathcal{X}_{ijm}(r')]=-\mathcal{D}_{ijr}\mathcal{D}_{ijr'}$ and $\text{Var}[\mathcal{X}_{ijm}(r)\mathcal{X}_{ijm}(r')]\leqslant  f_rf_{r'}$. By Lemma \ref{lemma: bernstein ineq}, if $f_r\geqslant\sqrt{\frac{\log (N_R)}{MN^{(1)}N^{(2)}}}$, with high probability,
    $$\left|\frac{1}{MN^{(1)}N^{(2)}}\sum_{ijm}\mathcal{X}_{ijm}(r)\mathcal{X}_{ijm}(r')-\mathbb{E}[\star]\right|\leqslant C^*\sqrt{\frac{f_{r'}f_r\log (N_R)}{MN^{(1)}N^{(2)}}}$$
    \item For $m\neq m'$, consider $\X_{ij(2m-1)}(r)\X_{ij(2m)}(r')$ such that
    $|\X_{ij(2m-1)}(r)\X_{ij(2m)}(r')|\leqslant 1$, and $\mathbb{E}[\X_{ij(2m-1)}(r)X_{ij(2m)}(r')]=0$ and $\text{Var}[\X_{ij(2m-1)}(r)\X_{ij(2m)}(r')]\leqslant \mathcal{D}_{ijr}\mathcal{D}_{ijr'}\leqslant  f_rf_{r'}$. Let $\tilde M=\lfloor M/2\rfloor$. Applying Lemma \ref{lemma: Berstein involving permutation}, 
     $$   \mathbb{P}\left(\left|\frac{1}{N^{(1,2)}}\frac{1}{M!}\sum_{\pi\in \Pi[M]}\sum_{i,j}\sum_{m=1}^{\tilde M}\X_{ij(2m-1)}(r)\X_{ij(2m)}(r')\right|\geqslant t\right)\leqslant\exp\left(-\frac{N^{(1,2)}t^2/2}{ \tilde M f_rf_{r'}+t/3}\right)$$
Let $t=C^*\sqrt{\frac{f_rf_{r'}\tilde M\log (N
_R)}{N^{(1,2)}}}$ if  $f_r\geqslant c^*\sqrt{\frac{\log (N_R)}{MN^{(1)}N^{(2)}}}$ to bound the second term. 
\end{enumerate}
 \end{itemize}
\end{proof}
\begin{lemma}\label{lemma: zz-Ezz l2 bound}
    
Under the condition in Lemma \ref{lemma: zijr square}, with probability at least $1-o(N_R^{-1})$, we have
    \begin{align}
       \|\mZ^{(1)}\mZ^{(1)\top}-\mathbb{E}[\mZ^{(1)}\mZ^{(1)\top}]\|_F\leqslant C^*\left(N^{(1)}\sqrt{\frac{N^{(2)}\log N_R}{M}}\right) \label{eq:norm z1z1t} \\
       \|\mZ^{(2)}\mZ^{(2)\top}-\mathbb{E}[\mZ^{(2)}\mZ^{(2)\top} ]\|_F\leqslant C^*\left(N^{(2)}\sqrt{\frac{N^{(1)}\log N_R}{M}}\right)\label{eq:norm z2z2t}\\
       \|\mZ^{(3)}\mZ^{(3)\top}-\mathbb{E}[\mZ^{(3)}\mZ^{(3)\top} ]\|_F\leqslant C^*\left(\sqrt{\frac{N^{(1,2)}\log N_R}{M}}\right) \label{eq:norm z3z3t}
    \end{align}
\end{lemma} 
\begin{proof}
We first start with Equation~\eqref{eq:norm z1z1t}. We have:
$$\|\mZ^{(1)}\mZ^{(1)\top} -\mathbb{E}[\mZ^{(1)}\mZ^{(1)\top}]\|_F^2=\sum_{i,i'\in[N^{(1)}]}\left(\sum_{j,r}\mathcal{Z}_{ijr}\mathcal{Z}_{i'jr}-\mathbb{E}[\mZ^{(1)}\mZ^{(1)\top}]\right)^2.$$
By Lemmas \ref{lemma: zijr square} and \ref{lemma: zz}, with probability at least $1-o(N_R^{-1})$, we have
    \begin{align*}
        \|\mZ^{(1)}\mZ^{(1)\top}-\mathbb{E}[\mZ^{(1)}\mZ^{(1)\top} ]\|_F^2&\lesssim \sum_{i\in[N^{(1)}]}\frac{N^{(2)}\log N_R}{M}+\sum_{i\neq i'} \frac{N^{(2)}\log N_R}{M}\\
        &\lesssim \left(N^{(1)}\right)^2\frac{N^{(2)}\log N_R}{M}.
    \end{align*}
 Equation~\eqref{eq:norm z2z2t} is obtained in a similar fashion, due to the symmetry in the roles of $\mZ^{(2)}$ and $\mZ^{(1)}.$
 Finally, for Equation ~\eqref{eq:norm z3z3t}, by Lemmas \ref{lemma: zijr square} and \ref{lemma: zz}:
    \begin{align*}
        \|\mZ^{(3)}\mZ^{(3)\top}-\mathbb{E}[\mZ^{(3)}\mZ^{(3)\top}]\|_F^2\lesssim&\sum_{r\in[R]} \frac{(f_r+f_r^2)N^{(1)} N^{(2)}\log(N_R)}{M}+\sum_{r\neq r'} \frac{N^{(1)} N^{(2)}f_rf_{r'}\log N_{R}}{M}\\
        \lesssim& N^{(1)} N^{(2)}\frac{\log(N_R)}{M}
    \end{align*}
\end{proof}

\begin{proof}
    The proof of this lemma follows directly from Lemma \ref{lemma: zz-Ezz l2 bound}
\end{proof}
\begin{lemma}\label{lemma: dz} Given $f_r\geqslant \frac{\log(N_R)}{N^{(1)}N^{(2)}M}$ for all $r$, with probability at least $1-o(N_R^{-1})$
\begin{align*}
   \left| \frac{1}{N^{(1)}N^{(2)}}\sum_{i,j=1}^{N^{(1)},N^{(2)}}\mathcal{Z}_{ijr} \mW^{(3)}_{k,(ij)} \right|&\leqslant C^*\sqrt{\frac{f_r\log(N_R)}{N^{(1)}N^{(2)}M}}\text{ for all } r\in[R], k\in[K^{(3)}]
\end{align*}
\end{lemma}
\begin{proof}
Note
$$ \sum_{ij}\mathcal{Z}_{ijr} \mW^{(3)}_{k,(ij)}=\frac{1}{N^{(1,2)}M} \sum_{ijm}N^{(1,2)}\mathcal{X}_{ijm} (r)\mW^{(3)}_{k,(ij)}$$
We have $|N^{(1,2)}\mathcal{X}_{ijm}(r) \mW^{(3)}_{k,(ij)}|\leqslant N^{(1,2)}$, $\mathbb{E}[N^{(1,2)}\mathcal{X}_{ijm}(r) \mW^{(3)}_{k,(ij)}]=0$ and $\text{Var}[N^{(1,2)}\mathcal{X}_{ijm}(r) \mW^{(3)}_{k,(ij)}]\leqslant (N^{(1,2)})^2 D_{ijr}\leqslant (N^{(1,2)})^2 f_r$. Applying Berstein's inequality again, we have
$$
\mathbb{P}\left(\left|\sum_{ij} Z_{ijr}\mW^{(3)}_{k,ij}\right|\geqslant t\right)\leqslant 2 \exp\left(-\frac{M t^2}{N^{(1,2)}f_r+t/3}\right)
$$
Let $t=C^*\sqrt{\frac{N^{(1)}N^{(2)}f_r\log(N_R)}{M}}$ if $f_r\geqslant \frac{\log(N_R)}{N^{(1)}N^{(2)}M}$.
\end{proof}

\section{Analysis of the SVD Procedure}\label{Sec: SVD analysis}

Let $\mD = \mA\mW \in \R^{ p \times n}$ be matrix as defined in the standard pLSI model provided in  Equation~\eqref{eq:plsi}, and denote its SVD as $\mD = \mathbf{{\Xi}} \mathbf{{\Lambda}}\mathbf{\mB}^\top $, with $ \mathbf{{\Xi}}^\top \mathbf{{\Xi}}  = \mB^{\top}\mB =\mathbf{I}_{K} $ and $\mathbf{{\Lambda}}=\text{diag}(\lambda_1,\dots \lambda_K)$. Let $\tilde{\mV} \in\mathbb{R}^{K\times K}$  be the matrix defined as $\tilde \mV=\mW\mB{\mathbf{\Lambda}}^{-1}$.

\begin{lemma}[Adaptation of Lemma D.2 in \cite{ke2022using}]\label{lemma:v_tilde_properties} Suppose the Assumption \ref{ass: min singular value} is satisfied. The following properties hold true for $\tilde{\mV}$:
\begin{enumerate}
    \item The matrix $\tilde{\mV}=\mW\mB \mathbf{\Lambda}^{-1}\in\mathbb{R}^{K\times K}$ is invertible, unique and such that $\tilde{\mV}=(\mathbf{\Xi}^\top \mA)^{-1}$.
    \item $\tilde{\mV}\tilde{\mV}^\top=(\mA^\top \mA)^{-1}$, and thus the singular values of $\tilde{\mV}$ are the inverses of the singular values of $\mA$.
    \item $\tilde{\mV}_{\star 1},\dots \tilde{\mV}_{\star K}$ are right eigenvectors of $\mW\mW^\top \mA^\top \mA$ associated with eigenvalues $\lambda^{2}_k$(these eigenvectors are not necessarily orthogonal with each other)
    \item Finally, we have $\mathbf{\mathbf{\Xi}}=\mA\tilde{\mV}$, and $\|\mathbf{\mathbf{\Xi}}_{j\star}\|_2\leqslant \|\mA_{j\star}\|_2 \sigma_{K}(\mA)^{-1}$.
\end{enumerate}
\end{lemma}
\begin{proof}
The proof for the previous lemma is essentially a readaptation of Lemma D.2 in \cite{ke2022using}, which we provide here to make this manuscript self-contained, but we emphasize that the original proof in a more general setting is in \cite{ke2022using}.
    \begin{enumerate}
    \item By definition:
    $$ \mD  =  \mathbf{\mathbf{\Xi}} \mathbf{\Lambda} \mB^\top ,$$ so we write:
    $$
        \mathbf{\Xi}  = \mD \mB \mathbf{\Lambda}^{-1} =  \mA (\mW \mB \mathbf{\Lambda}^{-1}) = \mA \tilde{\mV}.\qquad \eqref{eq: prob v in bernstein permute}
$$
    Since: $$\mathbf{\Xi}^\top \mathbf{\Xi} = \mathbf{I}_K = \mathbf{\Xi}^{\top} \mA \tilde{\mV}=(\mathbf{\Xi}^{\top} \mA) \tilde{\mV}, $$ we deduce that $\tilde{\mV}$ is unique and non singular, and that it is the inverse of 
    $\mathbf{\Xi}^{\top} \mA$. 
    \item 
In this case, plugging in $\mathbf{\Xi} = \mA\tilde{\mV}$ in $\mathbf{\Xi}^{\top} \mathbf{\Xi} = \mathbf{I}_K$, we obtain:
$$ \tilde{\mV}^{\top} \mA^{\top} \mA \tilde{\mV}= I_K$$
Multiplying left and right by $\tilde{\mV}$ and $\tilde{\mV}^{\top}$ respectively:
$$ (\tilde{\mV} \tilde{\mV}^{\top}) \mA^{\top} \mA (\tilde{\mV} \tilde{\mV}^{\top})= \tilde{\mV} \tilde{\mV}^{\top} \implies \tilde{\mV} \tilde{\mV}^{\top} = (\mA^{\top} \mA)^{-1}.$$
    \item Let $\lambda_k$ denote the singular values of $\mD=\mA\mW$. Then, for any column $\mathbf{\xi}_k$ of the matrix $\mathbf{\Xi}$, by definition:
    $$ \mD \mD^T \xi_k = \lambda_k^2 \mathbf{\xi}_k $$
    where by \eqref{eq: prob v in bernstein permute}, $\mathbf{\xi}_k =  \mA \tilde{\mV}_{\cdot k}.$
        Therefore:
    $$ \mD \mD^T \xi_k =  \mA\mW\mW^T\mW\mA^T \mA \tilde{\mV}_{\cdot k} = \lambda_k^2 \mA \tilde{\mV}_{\cdot k} $$
        $$ \implies    \mA^T\mA\mW\mW^T\mW\mA^T \mA \tilde{\mV}_{\cdot k} = \lambda_k^2 \mA^T\mA \tilde{\mV}_{\cdot k} $$
                $$ \implies    \mW^T\mW\mA^T \mA \tilde{\mV}_{\cdot k} = \lambda_k^2 \tilde{\mV}_{\cdot k} $$
                where the last line follows from the fact that $\mA^T\mA$ is assumed to be non-singular and invertible.
    \item By \eqref{eq: prob v in bernstein permute}, $\mathbf{\Xi}=\mA\tilde{\mV}$. The inequality follows by definition of the operator norm of $\tilde{\mV}$, and by application of point 2 in which we noted that the singular values of $\tilde{\mV}$ are the inverses of the singular values of $\mA$.
\end{enumerate}
\end{proof}

\subsection{Properties of the SVD}
\begin{proposition}[Lemma 5.1 in \cite{lei2015consistency}]\label{prop: sintheta} Let $\mQ\in\mathbb{R}^{n\times n}$ be a symmetric matrix with rank $K$, and $\hat \mQ\in\mathbb{R}^{n\times n}$ be any symmetric matrix. Let $\hat \mU,\mU\in\mathbb{R}^{n\times K}$ be $K$ leading left singular vector matrices of $ \hat \mQ$ and $ \mQ$ respectively. Then there exists a $K\times K$ orthogonal matrix $\mO$ such that $$ \|\hat \mU\mO -\mU\|_F\leqslant \frac{2\sqrt{2K}\| \hat \mQ-\mQ\|_{op}}{\lambda_K( \mQ)}$$
\end{proposition}


\begin{lemma}[Lemma F.1 of \cite{ke2022using}]\label{lem:svd ke}   Let $\mQ$ and $\hat \mQ$ be $n \times n$ symmetric matrices with $\text{rank}(\mQ) = K$, and assume $\mQ$ is positive semi-definite. For $1 \leq k \leq K$, let $\lambda_k$ and $\hat \lambda_k$ be the $k^\text{th}$ largest eigenvalues of $\mQ$ and $\hat \mQ$ respectively, and let $\mathbf{u}_k$ and $\mathbf{\hat u}_k$ be the $k^\text{th}$ eigenvectors of $\mQ$ and $\hat \mQ$. Fix $1 \leq s \leq k \leq K$. For some $c \in (0,1)$, suppose that: $$\min(\lambda_{s-1} - \lambda_s, \lambda_k - \lambda_{k+1}, \min_{l \in [K]} |\lambda_l|) \geq c\|\mQ\|_2, \quad \|\hat \mQ-\mQ\|_2 \leq (c/3)\|\mQ\|_2,$$
where by convention, if $s = 1$, we take $\lambda_{s-1} - \lambda_s = \infty$.
Write $\mU= [\mathbf{u}_s, \dots, \mathbf{u}_k], \hat \mU = [\mathbf{\hat u}_s, \dots, \mathbf{\hat u}_k]$ and $\mU^* = [\mathbf{u}_1, \dots, \mathbf{u}_K]$. There exists an orthonormal matrix $\mO$ such that for all $1 \leq j \leq n$,   \begin{align}
 \label{eq: rowwise ke lemma }   
\|(\hat \mU\mO - \mU)_{j*}\|_2 \leq \frac{5}{c\|\mQ\|_2}(\|\hat \mQ-\mQ\|_2\|\mU^*_{j*}\|_2 + \sqrt{K}\|(\hat \mQ-\mQ)_{j*}\|_2) \end{align}
\end{lemma}
\vspace{0.5cm}

\subsection{Accuracy of the HOSVD decomposition}\label{sec: HOSVD}
Given the mode-$a$ matricization $\mD^{(a)}=\mathcal{M}_a(\mathcal{D})\in\mathbb{R}^{n_1\times n_2}$ for arbitrary dimensions $n_1\in\{N^{(1)},N^{(2)},R\}$ and $n_2\in\{N^{(1)}R,N^{(2)}R,N^{(1)}N^{(2)}\}$, define the matrix $\mL^{(a)}\in\mathbb{R}^{n_1\times n_1}$ and the scalar $\alpha$ such that
\begin{align*}
    \mL^{(a)}:=\begin{cases}
        0_{N^{(1)}\times N^{(1)}}
        &\text{ if } a=1\\
        0_{N^{(2)}\times N^{(2)}}
        &\text{ if } a=2\\
         \frac{1}{M}\text{diag}(\mY^{(3)}\mathbf{1}_{N^{(1,2)}})
         &\text{ if } a=3
    \end{cases} \quad \quad
     \alpha^{(a)}:=\begin{cases}
       0&\text{ if } a\in\{1,2\}\\
        \frac{1}{M}&\text{ if } a=3
    \end{cases}
\end{align*}
Next, for each mode of matricization,  we define a pair of $(\mQ^{(a)},\hat \mQ^{(a)})$ to adapt Proposition \ref{prop: sintheta} to our tensor setting. Define 
\begin{align}\label{def: hat mQ mQ}
    \mQ^{(a)}=(1-\alpha^{(a)})\mD^{(a)}(\mD^{(a)})^\top\quad\text{and}\quad \hat{\mQ}^{(a)}=\mY^{(a)}(\mY^{(a)})^\top-\mL^{(a)}
\end{align}
Recall that the $\mathbf{\Xi}_{\star k}$ is the $k$-th left singular vector of $\mD$, and $ \mathbf{\hat\Xi}_{\star k}$ is the $k$-th left singular vector of the observed frequency matrix $Y$. Equivalently, $\mathbf{\Xi}_{\star k}$ and $\mathbf{\hat \Xi}_{\star k}$ are  the $k$-th eigenvectors of $\mQ$ and $\hat \mQ$, respectively.
\par To apply Proposition \ref{prop: sintheta}, we need to study in particular the difference $\|\mQ^{(a)}-\hat{\mQ}^{(a)}\|_{op}$. Since $\mY^{(a)} = 
 \mD^{(a)}  - \mZ^{(a)}$, we re-write the difference as:
\begin{align}
    \hat \mQ^{(a)}-\mQ^{(a)}=&\mY^{(a)}(\mY^{(a)})^\top-\mL^{(a)}  - (1-\alpha^{(a)})\mD^{(a)}(\mD^{(a)})^\top\notag\\
    =&(\mD^{(a)} + \mZ^{(a)})(\mD^{(a)} + \mZ^{(a)})^\top-\mL^{(a)}  - (1-\alpha^{(a)})\mD^{(a)}(\mD^{(a)})^\top\notag\\
=&\mD^{(a)}\mZ^{(a)\top}+\mZ^{(a)}\mD^{(a)\top}+\mZ^{(a)}\mZ^{(a)\top}-\mathbb{E}[\mZ^{(a)}\mZ^{(a)\top}]\notag\\
&+\mathbb{E}[\mZ^{(a)}\mZ^{(a)\top}]-\mL^{(a)}+\alpha \mD^{(a)}\mD^{(a)^\top} \label{eq: hat mQ-mQ}
\end{align}
\begin{lemma}\label{lem: HOSVD error} Under the Assumption \ref{ass: min singular value}, if $M\geqslant\log N_R/c_0$ for some constant $c_0>0$, then with probability at least $1-o(N_R^{-1})$
\begin{align}
    \|\mathbf{{\hat \Xi}}^{(1)} \mO^{(1)}-\mathbf{\Xi}^{(1)}\|\leqslant C^*\sqrt{\frac{\log N_R}{M}}\label{eq: hat xi 1}\\
    \|{\mathbf{\hat \Xi}}^{(2)} \mO^{(2)}-\mathbf{\Xi}^{(2)}\|\leqslant C^*\sqrt{\frac{\log N_R}{M}}\label{eq: hat xi 2}
\end{align}
If we further assume $f_r\geqslant\sqrt{\frac{\log N_R}{N^{(1)}N^{(2)}M}}$, we have 
\begin{align}\label{eq: hat xi 3}
    \|\mathbf{\mathbf{\hat \Xi}}^{(3)} \mO^{(3)}-\mathbf{\Xi}^{(3)}\|\leqslant C^*\sqrt{\frac{\log N_R}{N^{(1,2)}M}}
\end{align}
\end{lemma}
\begin{proof}
The proof idea is to control data error $\|\hat \mQ^{(a)}-\mQ^{(a)}\|$ and apply the results in Proposition \ref{prop: sintheta}.
\begin{itemize}
    \item \textbf{Mode 1 \& 2 matricization:} Since the analyses of mode 1 and 2 are identical, we only provide a detailed analysis of mode 1 here. Note that $\alpha=0$ and $\mL=0$ in mode 1 and mode 2 matricization. Therefore, we re-write 
    \eqref{eq: hat mQ-mQ} as
    $$
    \hat \mQ^{(1)}-\mQ^{(1)}=(\mY^{(1)}-\mD^{(1)})\mY^{(1)\top}+\mD^{(1)}(\mY^{(1)}-\mD^{(1)})^\top
    $$
 By Lemma \ref{lem: Z l2 no}, with probability at least $1-o(N_R^{-1})$:
    \begin{align*}
        \|\hat \mQ^{(1)}-\mQ^{{(1)}}\|_{op}&\leqslant \left(\|\mY^{(1)}\|_{op}+\|\mD^{(1)}\|_{op}\right)\|\mZ^{(1)}\|_{op}\leqslant \left(\|\mD^{(1)}\|_{op}+\|\mD^{(1)}\|_{op}+\|\mZ^{(1)}\|_{op}\right)\|\mZ^{(1)}\|_{op}\\
       & \leqslant C^*\sqrt{N^{(1,2)}}\cdot \sqrt{\frac{N^{(1,2)}\log N_R}{M}}=C^*N^{(1,2)}\sqrt{\frac{\log N_R}{M}}
    \end{align*}
Therefore, using proposition \ref{prop: sintheta}, under assumption \ref{ass: min singular value}, we then obtain the results in \eqref{eq: hat xi 1} and \eqref{eq: hat xi 2}.

\item\textbf{Mode-3 matricization:} We use the decomposition explicited in  \eqref{eq: hat mQ-mQ} to get an upper bound of $\|\hat \mQ^{(3)}-\mQ^{(3)}\|$. With probability at least $1-o(N_R^{-1})$,  the following holds true:
    \begin{enumerate}
        \item \textbf{Analysis of $\mD^{(3)}(\mZ^{(3)})^\top+\mZ^{(3)}(\mD^{(3)})^\top$}: We have: 
$$\mD^{(3)}(\mZ^{(3)})^\top=\sum_{k}\mA^{(3)}_{\star k} (\mZ^{(3)}(\mW^{(3)}_{k\star})^\top)^{\top}.$$
Using Lemma \ref{lemma: dz} and the fact that $\sum_{k} \|\mA_{\star k}\|_2\leqslant \sum_{k} \|\mA_{\star k}\|_1=K^{(3)}$ and $\sum_r f_r\leqslant K^{(3)}$
\begin{align*}
    \|\mD^{(3)}(\mZ^{(3)})^\top+\mZ^{(3)}(\mD^{(3)})^\top\|_F&\leqslant 2\|\mD^{(3)}(\mZ^{(3)})^\top\|_F\\&\leqslant2\sum_k\|\mA^{(3)}_{\star k}\|_2\|\mZ^{(3)}(\mW^{(3)}_{k\star})^\top\|_2\\
    &=2\sum_k\|\mA^{(3)}_{\star k}\|_2\sqrt{\sum_r(
\sum_{ij}\mathcal{Z}_{ijr}\mW^{(3)}_{k,ij})^2}\\
&\leqslant C^*\sqrt{\frac{N^{(1)}N^{(2)}\log(N_R)}{M}}
\end{align*}
        \item \textbf{Analysis of $\|\mZ^{(3)}(\mZ^{(3)})^\top-\mathbb{E}[\mZ^{(3)}(\mZ^{(3)})^\top]\|_{op}$}: We use the upper bound given by Lemma \ref{lemma: zz-Ezz l2 bound}:
        $$
        \|\mZ^{(3)}(\mZ^{(3)})^\top-\mathbb{E}[\mZ^{(3)}(\mZ^{(3)})^\top]\|\leqslant C^*\left(\sqrt{\frac{N^{(1,2)}\log N_R}{M}}\right)
        $$
        \item \textbf{Analysis of $\|\mathbb{E}[\mZ^{(3)}(\mZ^{(3)})^\top]-\mL^{(3)}+\frac{1}{M}\mD^{(3)}(\mD^{(3)})^\top\|$}: We use Lemmas \ref{lemma: expectation of zz} and \ref{lemma: sumij zij}:
        $$
        \|\mathbb{E}[\mZ^{(3)}(\mZ^{(3)})^\top]-\mL^{(3)}+\frac{1}{M}\mD^{(3)}(\mD^{(3)})^\top\|=\max_r\frac{1}{M}|\sum_{ij}\mathcal{Z}_{ijr}|\leqslant \max_r C^*\left(\sqrt{\frac{f_r N^{(1,2)}\log N_R}{M}}\right)
        $$
Using the fact that $f_r\leqslant K^{(3)}$, we have
$$\|\hat \mQ^{(3)}-\mQ^{(3)}\|\leqslant\sqrt{\frac{N^{(1,2)}\log N_R}{M}}.$$
Applying proposition \ref{prop: sintheta}, we immediately obtain \eqref{eq: hat xi 3}

    \end{enumerate}
    
\end{itemize}
\end{proof}

\subsection{Accuracy of the HOOI procedure}
For the purpose of keeping this appendix clear and (almost) self-contained, we summarize the HOOI procedure as follows.
Let $\mathbf{\hat \Xi}_0=\mathbf{\hat \Xi}_{\text{hosvd}}$. 
Calculate and repeat until convergence: 
\begin{align*}
    \mathbf{\hat \Xi}_t^{(1)}&=\textbf{SVD}_{K_1}\left(\mY^{(1)}\cdot \left(\mathbf{\hat \Xi}_{t-1}^{(2)\top}\otimes \mathbf{\hat \Xi}_{t-1}^{(3)\top} \right)\right)\\
      \mathbf{\hat \Xi}_t^{(2)}&=\textbf{SVD}_{K_2}\left(\mY^{(2)}\cdot \left(\mathbf{\hat \Xi}_{t-1}^{(1)\top}\otimes \mathbf{\hat \Xi}_{t-1}^{(3)\top} \right)\right)\\
        \mathbf{\hat \Xi}_t^{(3)}&=\textbf{SVD}_{K_3}\left( \mY^{(3)}\cdot \left(\mathbf{\hat \Xi}_{t-1}^{(1)\top}\otimes \mathbf{\hat \Xi}_{t-1}^{(2)\top} \right)\right)\\
        t &\leftarrow t+1
\end{align*}

\begin{lemma}\label{lemma: hooi svd bound}
Define $L_t:=\max_{a=1,2,3} \left\|\sin \Theta (\mathbf{\hat \Xi}^{(a)}_t, \mathbf{\Xi}^{(a)})\right\|,t=0,1,2,\dots$. Whenever
$$
\lambda_K(\mD\mD^\top)\geqslant c^*N^{(1,2)},
$$
after at most $t_{\max}=C^*\log\left(\frac{N^{(1,2)}}{\lambda_K(\mD\mD^\top)}\vee 1\right)$ iterations in algorithm, the following upper bounds hold with high probability,
\begin{align*}
    \max_{a\in[3]}\|\mathbf{\hat \Xi}^{(a)}_{\text{hooi}}\mO^{(a)}-\mathbf{\Xi}^{(a)}\|\leqslant C^*\sqrt{\frac{\log N_R}{M}}
\end{align*}
\end{lemma}
\begin{proof}
Define 
\begin{equation}\label{eq:def_z1t}
    \mZ_t^{(1)}:=\mZ^{(1)}\left(\mathbf{\hat \Xi}^{(2)\top}_t\otimes \mathbf{\hat \Xi}^{(3)\top}_t\right)\in \R^{N^{(1)} \times (K^{(2)} K^{(3)}) },
\end{equation} and $\mZ_t^{(2)}$ and $\mZ_t^{(3)}$ in a similar fashion:

$$\mZ_t^{(2)}:=\mZ^{(2)}\left(\mathbf{\hat \Xi}^{(1)\top}_t\otimes \mathbf{\hat \Xi}^{(3)\top}_t\right)\in \R^{N^{(2)} \times (K^{(1)} K^{(3)}) },$$
$$\mZ_t^{(3)}:=\mZ^{(3)}\left(\mathbf{\hat \Xi}^{(1)\top}_t\otimes \mathbf{\hat \Xi}^{(2)\top}_t\right)\in \R^{R \times (K^{(1)} K^{(2)} ) }.$$
We follow the proof of Theorem 1 and inequality (31) in \cite{zhang2018tensor}. Similar to \cite{zhang2018tensor}, we define:
$$ \mY^{(1)}_t = \mathcal{M}_1\left(\Y \times_2 \mathbf{\hat \Xi}^{(2)}_t\times_3 \mathbf{\hat \Xi}^{(3)}_t\right) =  \mY^{(1)} \cdot \left(\mathbf{\hat \Xi}^{(2)\top}_t\otimes \mathbf{\hat \Xi}^{(3)\top}_t\right) \in \R^{N^{(1)} \times (K^{(2)} K^{(3)}) }$$
$$ \mD^{(1)}_t = \mathcal{M}_1\left(\D \times_2 \mathbf{\hat \Xi}^{(2)}_t\times_3 \mathbf{\hat \Xi}^{(3)}_t\right) =  \mD^{(1)} \cdot \left(\mathbf{\hat \Xi}^{(2)\top}_t\otimes \mathbf{\hat \Xi}^{(3)\top}_t\right)\in \R^{N^{(1)} \times (K^{(2)} K^{(3)}) }.$$
By Wedin's $\sin \Theta$ theorem \cite{wedin1972perturbation}, we know that:

\begin{equation}\label{eq:bound_sinTheta}
    \| \sin \Theta (\mathbf{\hat{\Xi}}^{(t+1)}, \mathbf{\Xi}^{(1)})\| \leq \frac{\| \mZ_t^{(1)}\|_{op}}{\sigma_r(\mD_t^{(1)})}.
\end{equation}
We thus need to bound the numerator and denominator in the previous equation.

\paragraph{Bound on $\| \mZ_{t+1}^{(1)}\|_{op}$.}
We have, by definition of $\mZ^{(1)}_{t+1}$ (Equation ~\eqref{eq:def_z1t}):

\begin{equation}\label{eq:bound_z1}
    \begin{aligned}
    \|\mZ^{(1)}_{t+1}\|_{op}& = \| \mZ^{(1)} \cdot (\mathbf{\hat \Xi}^{(2)}_{t}\otimes \mathbf{\hat \Xi}^{(3)}_t) \|_{op} \\
    & \leq \| \mZ^{(1)}\|_{op} \|\mathbf{\hat \Xi}^{(2)}_t\otimes \mathbf{\hat \Xi}^{(3)}_t \|_{op} \\
    & = \| \mZ^{(1)}\|_{op} \| ( \mathbf{\hat\Xi}^{(2)}_t-\mathbf{\Xi}^{(2)})\otimes \mathbf{\hat \Xi}^{(3)}_t  + \mathbf{\Xi}^{(2)}\otimes (\mathbf{\hat \Xi}^{(3)}_t  -\mathbf{\Xi}^{(3)}) + \mathbf{\Xi}^{(2)} \otimes \mathbf{\Xi}^{(3)}\|_{op} \\
    & \leq \| \mZ^{(1)}\|_{op} \left(2 L_{t}  +1 \right) \\
    &\leqslant 2\| \mZ^{(1)}\|_{op} (1+L_t).  
    \end{aligned}
\end{equation}

By the proof of Equation (31) in \cite{zhang2018tensor}, we have:
\begin{equation}
    \begin{split}
        \| \mZ^{(1)}_{t+1}\|_{op} &\leq \|  \mZ^{(1)} (\mathbf{\Xi}^{(2)}\otimes \mathbf{\Xi}^{(3)}) \|_{op} +  \|  \mZ^{(1)} \cdot(P_{\mathbf{\Xi}^{(2)}_{\perp}} \mathbf{\hat \Xi}_{t}^{(2)}\otimes P_{\mathbf{\Xi}^{(3)}} \mathbf{\hat \Xi}_{t}^{(3)}) \|_{op} \\
        &+  \| \mZ^{(1)} \cdot(P_{\mathbf{\Xi}^{(2)}} \mathbf{\hat \Xi}_{t}^{(2)}\otimes P_{\mathbf{\Xi}^{(3)}_{\perp}} \mathbf{\hat \Xi}_{t}^{(3)}) \|_{op}+ \| \mZ^{(1)} \cdot(P_{\mathbf{\Xi}^{(2)}_{\perp}} \mathbf{\hat \Xi}_{t}^{(2)}\otimes P_{\mathbf{\Xi}^{(3)}_{\perp}} \mathbf{\hat \Xi}_{t}^{(3)}) \|_{op}\\
    &\leq \| \mZ^{(1)} (\mathbf{\Xi}^{(2)}\otimes \mathbf{\Xi}^{(3)}) \|_{op} +  \| \mZ^{(1)} \cdot(P_{\mathbf{\Xi}^{(2)}_{\perp}} \mathbf{\hat \Xi}_{t}^{(2)}\otimes P_{\mathbf{\Xi}^{(3)}} \mathbf{\hat \Xi}_{t}^{(3)}) \|_{op} \\
        &+  \| \mZ^{(1)} \cdot(P_{\mathbf{\Xi}^{(2)}} \mathbf{\hat \Xi}_{t}^{(2)}\otimes P_{\mathbf{\Xi}^{(3)}_{\perp}} \mathbf{\hat \Xi}_{t}^{(3)}) \|_{op}+ \| \mZ^{(1)} \cdot(P_{\mathbf{\Xi}^{(2)}_{\perp}} \mathbf{\hat \Xi}_{t}^{(2)}\otimes P_{\mathbf{\Xi}^{(3)}_{\perp}} \mathbf{\hat \Xi}_{t}^{(3)}) \|_{op}\\
    \end{split}
\end{equation}

\paragraph{Bound on $\sigma_r(\mD_1^{(t)})$.} To this end, by the same arguments as the proof of Equation 30 in \cite{zhang2018tensor}, we obtain:
\begin{equation}\label{eq:bound_sigma_k1_D1}
    \sigma_{K^{(1)}}(\mD_t^{(1)}) \geq \sigma_{K^{(1)}}(\mD^{(1)}) \cdot (1-L_t^2).
\end{equation}


Therefore, combining Equations~\ref{eq:bound_z1} and \ref{eq:bound_sigma_k1_D1}, equation~\ref{eq:bound_sinTheta} becomes:

$$ L_{t+1}\leq  \frac{2\|\mZ^{(1)}\|_{op} (1+L_t)}{\sigma_K( \mD^{(1)})(1-L_t^2)}=  \frac{2\|\mZ^{(1)}\|}{\sigma_K(\mD^{(1)})(1-L_t)}.$$

Using inequality (32) in \cite{zhang2018tensor} and $L_0\leqslant 1/2$ from HOSVD, we apply Lemmas \ref{lemma: singular value D} and \ref{lem: Z l2 no} with sintheta theorem  
By induction, we can sequentially prove that for all $t\geqslant 0$, 
$$
L_{t+1}\stackrel{\text{w.h.p}}{\lesssim}(1+L_t)\sqrt{\frac{\log N_R}{M}}\leqslant 1/2 
$$
This then yields, with high probability, let $\epsilon=\sqrt{\frac{\log N_R}{M}}$
\begin{align*}
    &L_{t+1}-\epsilon\leqslant4\epsilon(L_t-\epsilon) \text{ as } \epsilon\leqslant 1/2\\
    \Rightarrow \quad\quad &L_{t_{\max}}-\epsilon\leqslant (4\epsilon)^{t_{\max}}(L_0-\epsilon)\\
    \Rightarrow \quad\quad &L_{t_{\max}}\leqslant \epsilon+\frac{L_0}{2^{t_{\max}}}
\end{align*}
From the results in HOSVD, we have with high probability, \begin{align*}
 |L_0|\leqslant C^*\epsilon
 \end{align*}
Thus,
$$L_{t_{\max}}\leqslant \epsilon+\cdot \frac{C^*}{2^{t_{\max}}}\cdot \frac{N^{(1,2)}\sqrt{\log N_R/M}}{\lambda_K(\mD\mD^\top)}.$$
Let $t_{\max}\geqslant C^*\left(\log(\frac{N^{(1,2)}}{\lambda_K(\mD\mD^\top)}\vee 1)\right)$. Therefore we have the following upper bound for $\ell_2$ distance for $\mathbf{\hat \Xi}_{t_{\max}}$,
$$
\max_{a\in[3]}\|\mathbf{\hat \Xi}^{(a)}_{\text{hooi}}\mO^{(a)}-\mathbf{\Xi}^{(a)}\|\leqslant L_{t_{\max}}\leqslant C^*\sqrt{\frac{\log N_R}{M}}
$$
\end{proof}

\subsection{Row-wise HOSVD bound}\label{sec: row-wise HOSVD}
\begin{lemma}\label{lem: row-wise HOSVD}
    Suppose the assumption \ref{ass: min singular value} is satisfied and $M\geqslant \log N_R/c_0$ for some constant $c_0>0$, then with probability at least $1-o(N_R^{-1})$, for every $i\in[N^{(1)}]$, $j\in[N^{2}]$, there exist orthogonal matrices $\mO^{(1)}\in\mathbb{O}_{K^{(1)},K^{(1)}}$, $\mO^{(2)}\in\mathbb{O}_{K^{(2)}, K^{(2)}}$
    \begin{align}
    E_{1}:=&\|(\mathbf{\mathbf{\hat \Xi}}^{(1)} \mO^{(1)}-\mathbf{\Xi}^{(1)})_{i\star}\|_2\lesssim
\sqrt{\frac{\log (N_R)}{N^{(1)}M}}\tag{$E 1$}\label{eq: E 1}\\
E_{2}:=&\|(\mathbf{\mathbf{\hat \Xi}}^{(2)} \mO^{(2)}-\mathbf{\Xi}^{(2)})_{j\star}\|_2\lesssim
\sqrt{\frac{\log (N_R)}{N^{(2)}M}}\tag{$E 2$}\label{eq: E 2}
\end{align}
Furthermore, if the assumption \ref{ass: min value A3} is further satisfied, and $f_r\geqslant c^*\sqrt{\frac{\log N_R}{N^{(1)}N^{(2)}M}}$ for all $r\in[R]$, we have probability at least $1-o(N_R^{-1})$, there exists an orthogonal matrices $\mO^{(3)}=\text{diag}(o,\mO^*)$ where $o\in\{\pm1\}$ and orthonormal matrix $\mO^*\in\mathbb{O}_{K-1,K-1}$,
\begin{align}
  E_{3}(r):=\|(\mathbf{\mathbf{\hat \Xi}}^{(3)} \mO^{(3)}-\mathbf{\Xi}^{(3)})_{r\star}\|_2\lesssim
\sqrt{\frac{f_r\log (N_R)}{N^{(1)}N^{(2)}M}} \text{ for all } r\in[R]\tag{$E 3$}\label{eq: E 3}
\end{align}
\end{lemma}
\begin{proof}
 To obtain the bound above with high probability, we use Lemma \ref{lem:svd ke} and need to study $\|\mU^{*(a)}_{j\star}\|$ and $\|\hat \mQ^{(a)}_{j\star}-\mQ^{(a)}_{j\star}\|$ for each $a$.
Given the results from Section \ref{sec: HOSVD}, we have with probability at least $1-o(N_R^{-1})$,
\begin{align}\label{eq: hat mQ-mQ l2 norm}
\max_{a\in[3]}\|\hat \mQ^{(a)}-\mQ^{(a)}\|\lesssim N^{(1,2)}\sqrt{\frac{\log N_R}{M}}\leqslant c^*(1-\alpha) N^{(1,2)}\stackrel{\text{Lemma \ref{lemma: singular value D}}}{\leqslant} \lambda_K(\mQ^{(a)})
\end{align}
\begin{itemize}
    \item \textbf{Mode 1 \& 2 matricization:} We here provide the analysis of mode 1, and that of mode 2 alike. Let $s=1$ and $k=K$ in both mode 1 and mode 2. From the Lemma \ref{lemma: singular value D} and \eqref{eq: hat mQ-mQ l2 norm}, we verify the conditions in Lemma \ref{lem:svd ke} with $\hat \mU^{(1)}=\mathbf{\mathbf{\hat \Xi}}^{(1)}$, and $\mU^{(1)}=\mathbf{\Xi}^{(1)}$.
    \begin{enumerate}
        \item Let $\hat \mU^{(1)}=\mathbf{\mathbf{\hat \Xi}}^{(1)}$, $\mU^{*(1)}=\mU^{(1)}=\mathbf{\Xi}^{(1)}$. We have
$$\|\mathbf{\Xi}^{(1)}_{i\star}\|_2^2\leqslant\|\mA^{(1)}_{i\star}\|_2^2\sigma_K(\mA^{(1)})^{-2}\leqslant\|\mA^{(1)}_{i\star}\|_1\sigma_K(\mA^{(1)})^{-2}\leqslant 1/N^{(1)}
$$
        \item  For $\|(\hat \mQ^{(1)}-\mQ^{(1)})_{i\star }\|$, we have
        \begin{align*}
           \|(\hat \mQ^{(1)}-\mQ^{(1)})_{i\star }\|=& \|\mY^{(1)}(\mY^{(1)}-\mD^{(1)})_{i\star}+(\mY^{(1)}-\mD^{(1)}) \mD^{(1)}_{i\star}\| \\
\leqslant &\|\mY^{(1)}\|\|\mZ^{(1)}_{i\star}\|+\|\mZ^{(1)}\|\|\mD^{(1)}_{i\star}\|  \\
\leqslant & \sqrt{N^{(1,2)}}\|\mZ^{(1)}_{i\star}\|+\sqrt{\frac{N^{(1,2)}\log N_R}{M}}\|\mD^{(1)}_{i\star}\| 
\\\leqslant & N^{(2)}\sqrt{\frac{N^{(1)}\log N_R}{M}}
\end{align*}
where the last second inequality is given by Wely's inequality and Lemmas \ref{lemma: singular value D}, \ref{lem: Z l2 no} and the last inequality is from the results of Lemma \ref{lem: Z l2 no} and the fact that 
$$\|\mD^{(1)}_{i\star}\|_2^2=\sum_{jr}P(r|i,j)^2\leqslant N^{(2)}$$
    \end{enumerate}
We then use Lemma \ref{lem:svd ke} with two bounds above, with probability at least $1-o(N_R^{-1})$,
\begin{align*}
    \|(\mathbf{\mathbf{\hat \Xi}}^{(1)} \mO^{(1)}-\mathbf{\Xi}^{(1)})_{i\star }\|&\leqslant \sqrt\frac{\log N_R}{M}\cdot \sqrt{\frac{1}{N^{(1)}}}+N^{(2)}\sqrt{\frac{N^{(1)}\log N_R}{M}}\cdot \frac{1}{N^{(1,2)}}\\
   & \leqslant C^*\sqrt{\frac{\log N_R}{N^{(1)}M}}
\end{align*}
    \item\textbf{Mode 3 matricization:} We divide the eigenvectors of $\mQ^{(3)}$ into two groups. The first group only includes the leading eigenvector $\mathbf{\Xi}^{(3)}_{\star 1}$ as $s=1, k=1$, while the second group contains the left $K-1$ eigenvectors as $s=2$ and $k=K^{(3)}$ in lemma \ref{lem:svd ke}. We need to figure out the lower bound of the gap of the first two eigenvalues of $\mQ^{(3)}$.  Let $n=N^{(1)}N^{(2)}$. Define 
$\mathbf{\Theta}=\frac{1}{n}\mW\mW^\top \mA^\top \mA $, where $\mW=\mW^{(3)}$ and $\mA=\mA^{(3)}$.
Note that $\mW=\mathcal{M}_3(\mathcal{G})(\mA^{(1)}\otimes \mA^{(2)})^\top$, we have $$\sigma_K(\frac{1}{n}\mW\mW^\top)\geqslant\frac{1}{N^{(1)}}\sigma_{\min}(\mA^{(1)\top} \mA^{(1)})\cdot \frac{1}{N^{(2)}}\sigma_{\min}(\mA^{(2)\top} \mA^{(2)}) \cdot \sigma_K(\mathcal{M}(\mathcal{G})^\top\mathcal{M}(\mathcal{G}) )\geqslant c^* K^{(1)}K^{(2)}$$
On the other side, by lemma \ref{lemma: max singular values}, we have
$$
\sigma_1(\frac{1}{n}\mW\mW^\top)\leqslant  K^{(1)}K^{(2)}
$$
By assumptions \ref{ass: min singular value} and \ref{ass: min value A3}, we have
\begin{align}
    \mathbf{\Theta}_{k_1k_2 }&=\sum_{k=1}^K\frac{1}{n}(\mW\mW^\top)_{k_1 k} \cdot (\mA^\top \mA)_{k k_2}\notag\\
   & \geqslant \min_{s,s'}(\mA^\top \mA)_{s,s'} \cdot \sum_{k=1}^K \frac{1}{n}(\mW\mW^\top)_{k_1 k}\notag\\
   &\geqslant \min_{s,s'}(\mA^\top \mA)_{s,s'} \cdot \frac{1}{n}(\mW\mW^\top)_{k_1 k_1}\notag\\
   &\geqslant \min_{s,s'\in[K^{(3)}]}(\mA^\top \mA)_{s,s'} \cdot\sigma_K(\frac{1}{n}\mW\mW^\top )\geqslant c^*\label{eq: entry of theta 3}
\end{align}
Also, 
$$
\|\mathbf{\Theta}\|\leqslant \|\mA\|_{op}^2\|\frac{1}{n}\mW\mW^\top\|_{op}\leqslant \prod_{i=1}^3 K^{(i)}
$$
Applying lemma \ref{lem: eigen gap}, we have
$$
\lambda_1(\mathbf{\Theta})-\lambda_2(\mathbf{\Theta})\geqslant c^*
$$
Note that 
$$
\lambda_i(\mQ^{(3)})=\left(1-\frac{1}{M}\right)\lambda_i(\mA\mW\mW^\top \mA^\top)=\left(1-\frac{1}{M}\right)\lambda_i(\mW\mW^\top \mA^\top \mA)
$$
so we have
$$
\lambda_1(\mQ^{(3)})-\lambda_2(\mQ^{(3)})\geqslant c^*N^{(1,2)}
$$
and thus 
$$
\lambda_1(\mQ^{(3)})\geqslant c^*N^{(1,2)} +\max_{k\in[K^{(3)}]}\lambda_k(\mQ^{(3)})
$$
By lemma \ref{lemma: singular value D} and $\mQ^{(3)}=\left(1-\frac{1}{M}\right)\mD^{(3)}\mD^{(3)\top}$, we have
$$
c^*N^{(1,2)}\leqslant\lambda_k(\mQ^{(3)})\leqslant N^{(1,2)}\text{ for all } k\in[K^{(3)}]
$$
Plugging these values into lemma \ref{lem:svd ke} with matrix $\mO^{(3)}=\text{diag}(o,\mO^*)\in\mathbb{R}^{K\times K}$ where $o\in\{\pm1\}$ and orthonormal matrix $\mO^*\in\mathbb{R}^{(K-1)\times (K-1)}$, we achieve the result with high probability. Note that we set $o\in\{\pm 1\}$ because the singular vector is determined up to a sign flip.
\par We then satisfy the conditions in Lemma \ref{lem:svd ke}, and need to study the terms left in \eqref{eq: rowwise ke lemma }.
\begin{enumerate}
    \item  Let $\mU^{*(3)}=\mathbf{\Xi}^{(3)}$, we have
   $$\|\mathbf{\Xi}^{(3)}_{r\star}\|_2^2\leqslant\|\mA_{r\star}\|_2^2\sigma_K(\mA)^{-2}\leqslant\|\mA_{r\star}\|_1\sigma_K(\mA)^{-2}\leqslant f_r/K$$
    \item The upper bound$ \|(\mQ^{(3)}-\hat \mQ^{(3)})_{r\star}\|_2\leqslant C^*\sqrt{\frac{f_rN^{(1)}N^{(2)}\log (N_R)}{M}}$. w.h.p. is given by 
    \begin{align*}
        \|(\mD^{(3)}\mZ^{(3)\top}+\mZ^{(3)}\mD^{(3)\top})_{r\star}\|_2\leqslant& \sum_k \mA_{rk}\|\mZ^{(3)}\mW_{k\star}\|+\sum_k\left|\sum_{ij}\mathcal{Z}_{ijr}\mW_{k,ij}\right|\|\mA_{\star k}\|_2\\
        \stackrel{\text{Lemma \ref{lemma: dz}}}{\lesssim}& \sqrt{\frac{f_rN^{(1,2)}\log N_R}{M}}\\
        \left\|\left(\mZ^{(3)}\mZ^{(3)\top}-\mathbb{E}[\mZ^{(3)}\mZ^{(3)\top}]\right)_{r\star}\right\|\leqslant&\sqrt{\sum_{r'}\left(\sum_{ij}\mathcal{Z}_{ijr}\mathcal{Z}_{ijr'}-\mathbb{E}[\mathcal{Z}_{ijr}\mathcal{Z}_{ijr'}]\right)^2}\\
        \stackrel{\text{Lemmas \ref{lemma: zz},\ref{lemma: zijr square}}}{\lesssim} &\sqrt{\frac{f_rN^{(1,2)}\log N_R}{M}} \quad \text{ with } f_r\leqslant K\\
        \frac{1}{M}\|\text{diag}(\mZ^{(3)}\mathbf{1}_{n})_{r\star}\|_2\stackrel{\text{Lemma \ref{lemma: sumij zij}}}{\lesssim}&  \frac{1}{M}\sqrt{\frac{f_rN^{(1,2)}\log N_R}{M}}
    \end{align*}
    Combined the two points above, we apply Lemma \ref{lem:svd ke} for two groups of eigenvectors to get the bound in \eqref{eq: E 3}
\end{enumerate}
\end{itemize}
\end{proof}

\section{Analysis of the Vextex Hunting Algorithm}\label{sec: vh hunting}

In our analysis, we will make the following assumption on the efficiency the vertex hunting algorithm (similar to Assumption 4 in \cite{tran2023sparse}).

\begin{assumption}\label{ass: vh efficient}[Efficiency of the VH algorithm]
Given $K$ and a set of data points $\{\mathbf{y}_i\}_{i\in[n]}\in\mathbb{R}^{p}$ such that $\mathbf{y}_i=\sum_{k=1}^K \omega_{ik} \vnu_{k}$ satisfying the ideal simplex condition \ref{def: ideal simplex}, the vertex hunting function $\mathcal{V}$ that outputs estimates $\mathcal{V}( \{\mathbf{y}_i\}_{i=1}^n) = \{\hat\vnu_k\}$ for the vertices of the point cloud $\{\mathbf{y}_i\}_{i\in[n]}$  satisfies, for some absolute constant $C>0$:
\begin{align}\label{eq: efficiency of VH algo}
\max_{k\in[K]}\|\hat \vnu_{\pi(k)}-\vnu_k\|\leqslant C^* \max_{i\in[n]} \|\mathbf{y}_i- \mathbf{\hat y}_i\|
\end{align}
subject to a label permutation $\pi$.
\end{assumption}
This assumption is commonly used to control vertex estimation error by requiring to have points near each vertex \cite{ke2022using,tran2023sparse}. The Successive Projection (SP) method \cite{araujo2001successive} has been demonstrated to satisfy this assumption (see Section 2.2 in \cite{ke2022using}). Another suitable vertex hunting (VH) algorithm, Sketched Vertex Search (SVS) \cite{jin2017estimating}, is more robust to noise in the simplex under mild conditions, although it is computationally slower. Additionally, the Archetype Analysis (AA) vertex hunting procedure can achieve the bound in \eqref{eq: efficiency of VH algo} under conditions that are more relaxed than the ideal simplex condition described in \ref{def: ideal simplex} (see Theorem F.1 in \cite{javadi2020nonnegative}). 
\par Remark that as long as the bound in \eqref{eq: efficiency of VH algo} is satisfied with high probability, our main results remain valid even under these relaxed conditions, which do not require the anchor conditions outlined in Assumption \ref{ass:identifiability}
. The anchor conditions are primarily used to satisfy the ideal simplex condition in Assumption \ref{ass: vh efficient}. We retain these anchor conditions to help the readers better understand the concepts and to keep the proofs straightforward and concise.

Next, we demonstrate how the anchor conditions ensure that the simplex formed by the rows of $ S$ in step 2 of the Post-SVD process satisfies the ideal simplex condition in Assumption \ref{ass: vh efficient}. Recall that in oracle procedure, we use $\mV$ to denote the matrix of vertices returned by the VH algorithm, while $\mV^*$ denotes the vertices matrix used in the recovery procedure (Algorithm~\ref{sec: estimating A3}). Define $\Omega$ as the weight matrix with $\mV$ returned by the VH algorithm in the oracle procedure, and $\mathbf{\Omega}^*$ its counterpart in the recovery procedure.
\begin{lemma}[Mode 1 \& 2 matricization]\label{lem: hat S row wise bound mode 1 and 2} Given $a\in\{1,2\}$,
suppose we define the simplex 
$$\mS^{(a)}=\mathbf{\Xi}^{(a)}=\mA^{(a)}\tilde{\mV}^{(a)}$$
with vertex matrix $\mV^{(a)}=\tilde{\mV}^{(a)}$ and weight matrix $\mathbf{\Omega}^{(a)}=\mA^{(a)}$. Under the conditions of Lemma \ref{lem: row-wise HOSVD}, we have probability at least $1-o(N_R^{-1})$,
      \begin{align*}
              &\|(\mathbf{\hat S}^{(1)} \mO^{(1)}-\mathbf{S}^{(1)})_{i\star}\|_2\lesssim
\sqrt{\frac{\log (N_R)}{N^{(1)}M}}\\
&\|(\mathbf{\hat S}^{(2)} \mO^{(2)}-\mathbf{S}^{(2)})_{j\star}\|_2\lesssim
\sqrt{\frac{\log (N_R)}{N^{(2)}M}}
      \end{align*}

If $\mA^{(a)}$ follows \eqref{case:1} such that $\sum_{k\in K^{(a)}}\mA_{ik}=1$ for all $i\in[N^{(a)}]$ and satisfies the anchor document condition \ref{assumption: anchor doc}, then  the simplex $\mathbf{S}^{(a)}$ is an ideal simplex.
      
\end{lemma}
\begin{proof}
The SVD of $\mD^{(a)}$ in lemma \ref{lemma:v_tilde_properties} yields
$$
\mathbf{\Xi}^{(a)}_{i\star}=\sum_{k=1}^{K^{(a)}}\mA^{(a)}_{ik} \tilde{\mV}^{(a)}_{k\star}
$$
Fix $a\in\{1,2\}$, by observation~\eqref{eq:constraints A}, for any $i\in[N^{(a)}]$, $\sum_{k\in[K^{(a)}]}\mA^{(a)}_{ik}=1$. Furthermore, by the anchor document assumption on the matrix $\mA$ (Assumption~\ref{assumption: anchor doc}), for each $k$, there exists an index $i$ such that $\mA^{(a)}_{ik}=1$ and $\mA^{(a)}_{ik'}=0$ for $k\neq k'$. Thus, the linear equation $\mathbf{\Xi}^{(a)}=\mA^{(a)}\tilde{\mV}^{(a)}$ defines an ideal simplex (Definition~\ref{def: ideal simplex}), with the matrix $\mA^{(a)}$ playing the role of the weights $\mathbf{ \Omega}^{(a)}$. Consequently, we run the VH algorithm on the rows of the matrix $\mathbf{\Xi}^{(a)}$ to estimate the corresponding vertices. 
\end{proof}

\begin{lemma}[Mode 3 matricization]\label{lem: hat S row wise bound mode 3}
Suppose the simple $\mathbf{S}^{(3)}$ is given by
    $$[\mathbf{1}_R, \mathbf{S}^{(3)}]=\left(\text{diag}(\mathbf{\Xi}^{(3)}_{\star 1})\right)^{-1} \mA^{(3)} \tilde{\mV}^{(3)}$$
  with vertex matrix $\mV^{(3)}$ and weight matrix $\mathbf{\Omega}^{(3)}$ such that
    $$
    [\mathbf{1}_{K^{(3)}}, \mV^{(3)}]=[\text{diag}(\tilde{\mV}^{(3)}_{\star 1})]^{-1}\tilde{\mV}^{(3)}\quad \text{and}\quad \mathbf{\Omega}^{(3)}=[\text{diag}(\mathbf{\Xi}^{(3)}_{\star 1})]^{-1} \mA^{(3)} [\text{diag}(\tilde{\mV}^{(3)}_{\star 1})]$$ 
     If we assume the same conditions in Lemma \ref{lem: row-wise HOSVD}, we have probability at least $1-o(N_R^{-1})$
    \begin{align}
         \|\mO^{*\top}\mathbf{\hat S}^{(3)}_{r\star}-\mathbf{S}^{(3)}_{r\star}\|_2\lesssim 
         \sqrt{\frac{\log N_R}{f_r N^{(1)}N^{(2)}M}}\text{ for all }r \in[R]
    \end{align}
    where $\mO^*\in\mathbb{O}_{K^{(3)}-1,K^{(3)}-1}$ comes from the orthogonal matrix $\mO^{(3)}=\text{diag}(o,\mO^*)$.
   Observed that $\mA^{(3)}$ follows \eqref{case:2} such that $\sum_{r\in[R]}\mA^{(3)}_{rk}=1$, if $\mA^{(3)}$ further satisfies the anchor word condition \ref{def: anchor word}, then the point cloud of the rows of $\mathbf{S}^{(3)}$ forms an ideal simplex.
  
\end{lemma}
\begin{proof} For notation simplicity, let $\mO=\mO^{(3)}=\text{diag}(o,\mO^*)$ where $o\in\{\pm 1\}$. Let $\mA=\mA^{(3)}$, $\mW=\mW^{(3)}$, $\mathbf{\Xi}=\mathbf{\Xi}^{(3)}$, etc. Do SCORE normalization \ref{def: SCORE normalization} on $\mathbf{\Xi}^{(3)}$, we get 
$$
[\mathbf{1}_R, \mS ]=\left(\text{diag}(\mathbf{\Xi}_{\star 1})\right)^{-1}\mathbf{\Xi}=\left(\text{diag}(\mathbf{\Xi}_{\star 1})\right)^{-1} \mA \tilde{\mV}
$$
where $\mS\in\mathbb{R}^{R \times (K-1)}$ works data point cloud in VH algorithm.
\par To find an ideal simplex for VH algorithm, the equation above is not enough. Consider \begin{align}\label{def: Vstar in 3}
    \mV^{*}=[\text{diag}(\tilde{\mV}_{\star 1})]^{-1}\tilde{\mV}.
\end{align} The entries of the first column of $\mV^*$ are all equal to 1. To make $\mV^*$ and $\mS$ well-defined, we need to prove that the entries of $\tilde{\mV}_{\star 1}$ and $\mathbf{\Xi}_{r1}$ are all nonzero. We will show later that for some constant $c^*,C^*>0$ such that
\begin{align}
    \frac{c^*}{\sqrt{K^{(3)}}}\leqslant\min_{k\in[K^{(3)}]} \tilde{\mV}_{k1}^{(3)}\leqslant\max_{k\in[K^{(3)}]} \tilde{\mV}_{k1}^{(3)}\leqslant \frac{C^*}{\sqrt{K^{(3)}}}\label{eq: bound on first value of tilde V}\\
     \frac{c^* f_r}{\sqrt{K^{(3)}}}\leqslant\mathbf{\Xi}_{r1}^{(3)}\leqslant \frac{C^* f_r}{\sqrt{K^{(3)}}} \text{ for all } r\in[R]\label{eq: bound of xi 1 entry}
\end{align}
Therefore, all entries in $\tilde{\mV}_{\star 1}$ and $\mathbf{\Xi}_{\star 1}$ are positive. By the definition of $\mV^*$, it is of form 
\begin{equation*}
    \mV^{*\top}=\begin{pmatrix}
        1 & \dots & 1\\
        \mV_{1\star} &\dots &\mV_{K\star}
    \end{pmatrix}
\end{equation*}
where $\mV_{1\star},\dots, \mV_{K\star}$ are supposed to be the simplex vertices of the point cloud generated by the rows of $S$. We could rewrite 
$$
\mathbf{S}^*=[\mathbf{1}_R, \mS]=\left(\text{diag}(\mathbf{\Xi}_{\star 1})\right)^{-1} \mA [\text{diag}(\tilde{\mV}_{\star 1})] \mV^*=\mathbf{\Omega}^* \mV^*
$$
To prove $\mathbf{S}^*=\mathbf{\Omega}^*  \mV^*$ ($\mathbf{S}^*_{r\star}=\sum_{k=1}^K \Omega_{rk}^* \mV_{k\star}^*$) and $\mS=\mathbf{\Omega}^* \mV$ ($    \mS_{r\star}=\sum_{k=1}^K \mathbf{\Omega}_{rk}^* \mV_{k\star}$) are indeed of form of ideal Simplex, we need to prove $\mathbf{\Omega}^*$ is non negative and the rows of it sum up to 1. Remark that 
\begin{align*}
\mathbf{\Omega}^*:=\left[\text{diag}(\mathbf{\Xi}_{\star 1})\right]^{-1} \mA [\text{diag}(\tilde{\mV}_{\star 1})] 
\end{align*}
we have 
$$
\mathbf{\Omega}^*_{rk}=\frac{\tilde{\mV}_{k1}\mA_{rk}}{\mathbf{\Xi}_{r1}} \quad \text{and}\quad \sum_{k=1}^K \mathbf{\Omega}^*_{rk}=1
$$
Since the entries of the matrices $\text{diag}(\mathbf{\Xi}_{\star 1}),\text{diag}(\tilde{\mV}_{\star 1})$ and $\mA$ are all nonnegative, the entries of $\mathbf{\Omega}^*$ is then also nonnegative. If $r$ is an anchor word index, then $\mA_{rk}\neq 0$ but $\mA_{rk'}=0$ for all $k'\neq k$. Since $\sum_{k=1} \mathbf{\Omega}^*_{rk}=1$, $\mathbf{\Omega}^*_{rk}$ can only be equal to 1 at the anchor word index $r$. Thus the property of $\mathbf{\Omega}^*$ guarantees that $\mathbf{S}^*=\mathbf{\Omega}^* \mV^*$ and $\mS=\mathbf{\Omega}^* \mV$ are of form of ideal simplex. To get the bound of the error given the VH, we need to first bound the error between point cloud of $\mS$ and the data $\mathbf{\hat S}$.

By the definition, we have
\begin{equation}\label{eq: score normalization form}
    \begin{pmatrix}
        1\\
        \mS_{r\star }
    \end{pmatrix}=\mathbf{\Xi}_{r\star }/(\mathbf{\Xi}_{r1}), \quad
    \begin{pmatrix}
        1\\
        \mO^{*\top} \hat \mS_{r\star }
    \end{pmatrix}=\mO \mathbf{\hat \Xi}_{r\star }/(\mathbf{\hat \Xi}_{r1})
\end{equation}
To make \eqref{eq: score normalization form} to be well-defined, we also need $\mathbf{\hat \Xi}_{r 1}$ for all $r\in[R]$ are nonzero. By \eqref{eq: E 3}, we know that with high probability when $f_r\geqslant c^*\sqrt{\frac{\log(N_R)}{N^{(1,2)}M}}$, for all $r\in[R]$,
$$
|o \mathbf{\hat \Xi}_{r1}-\mathbf{\Xi}_{r1}|\leqslant C^*\sqrt{\frac{f_r\log (N_R)}{N^{(1,2)} M}}\leqslant C^* f_r\left(\frac{\log (N_R)}{N^{(1,2)} M}\right)^{1/4}
$$
and by \eqref{eq: bound of xi 1 entry}, we have 
$$
\mathbf{\Xi}_{r1}\geqslant c^*f_r>0.
$$
We can see that $|o \mathbf{\hat \Xi}_{r1}-\mathbf{\Xi}_{r1}|\ll \mathbf{\Xi}_{r1}$ with high probability as $N^{(1,2)}M$ is sufficiently large. Thus we have $|o \mathbf{\hat \Xi}_{r1}|\geqslant \mathbf{\Xi}_{r1}/2$. Therefore, we have $\mathbf{\Xi}_{\star 1}$ is strictly positive and thus $\mathbf{\hat \Xi}_{\star 1}$ is strictly positive with high probability. WLOG, we set $o=1$. By \eqref{eq: score normalization form}, we have 
\begin{align}
    \|\mO^{*\top}\hat \mS_{r\star}-\mS_{r\star}\|_2&=\left\|\frac{1}{\mathbf{\hat \Xi}_{r1}}\mO^{\top} \mathbf{\hat\Xi}_{r\star}-\frac{1}{ \mathbf{\Xi}_{r1}}\mathbf{\Xi}_{r\star}\right\|_2\notag\\
    &=\left\|\frac{1}{\mathbf{\hat \Xi}_{r1}}\left(\mO^{\top} \mathbf{\hat\Xi}_{r\star}-\mathbf{\Xi}_{r\star }\right)-\frac{\mS_{r\star}}{ \mathbf{\hat \Xi}_{r1}}\left(\mathbf{\hat\Xi}_{r\star}-\mathbf{\Xi}_{r\star }\right)\right\|_2\notag\\
    &\leqslant |\mathbf{\hat\Xi}_{r1}|^{-1}\left(\|\mO \mathbf{\hat \Xi}_{r\star}-\mathbf{\Xi}_{r\star }\|+\| \mS_{r\star}\|\cdot |\mathbf{\hat \Xi}_{r1}-\mathbf{\Xi}_{r1}|\right)\notag\\
    &\leqslant c^*\frac{1}{f_r}\cdot E_{3}(r)\left(1+\|\mS_{r\star}\|_2\right)\text{ using the bound from \eqref{eq: E 3}}\notag\\
    &\leqslant C^*\sqrt{\frac{\log N_R}{f_r N^{(1,2)}M}} \left(1+\|\mS_{r\star}\|_2\right)\label{eq: bound of H error in C2}
\end{align}
Since each row $\mS_{r\star}$ is in the simplex, it follows that for some constant $C^*>0$
\begin{align}\label{eq: Sr in mode 3}
\|\mS_{r\star}\|_2\leqslant \max_{k\in[K]} \|\mV_{k\star }\|_2\leqslant\max_{k\in[K]} \|\mV_{k\star }^*\|_2\leqslant \|\mV^*\|_{2}\leqslant C^*
\end{align}
where the last inequality comes from the results of Lemma \ref{lemma: Vstar eigenvalue in C2}. 
\noindent
\\
\vspace{0.5cm}
\\
\textbf{Proving inequalities \eqref{eq: bound on first value of tilde V} and \eqref{eq: bound of xi 1 entry} under the Assumption \ref{ass: min singular value}}\\
By Assumption \ref{ass: min singular value}, we have $\sigma_{K^{(3)}}(\mA^{(3)})\geqslant c^* \sqrt{K^{(3)}}$. Then
$$
\max_{k\in[K^{(3)}]} \tilde{\mV}^{(3)}_{k1}\leqslant \|\tilde{\mV}^{(3)}_{\star 1}\|_2\leqslant \sigma_1(\tilde{\mV}^{(3)})=\sigma_{K^{(3)}}(\mA^{(3)})^{-1}\leqslant \frac{C^*}{\sqrt{K^{(3)}}}
$$
Similarly, we have 
$$
\|\tilde{\mV}^{(3)}_{\star 1}\|_2\geqslant \sigma_{K^{(3)}}(\tilde{\mV}^{(3)})=\sigma_1(\mA^{(3)})^{-1}\geqslant\frac{1}{\sqrt{K^{(3)}}}
$$
By inequality \eqref{eq: entry of theta 3}, we have $\mathbf{\Theta}=\frac{1}{N^{(1,2)}}\mW^{(3)}\mW^{{(3)}\top} \mA^{(3)\top} \mA^{(3)}$ such that each entry of $\mathbf{\Theta}$ is lower bounded by a constant. $\mathbf{\Theta}$ is a strictly positive matrix. By lemma \ref{lem: eigen gap}, the leading eigenvector of $\mathbf{\Theta}$, which is $\tilde{\mV}^{(3)}_{\star 1}$ by item 3 in Lemma \ref{lemma:v_tilde_properties}, is a strictly positive vector.
$$
\min_{k\in[K^{(3)}]} \tilde{\mV}^{(3)}_{k1}=\|\tilde{\mV}^{(3)}_{\star 1}\|_2\min_{k}\left\{\frac{\tilde{\mV}^{(3)}_{k1}}{\|\tilde{\mV}^{(3)}_{\star 1}\|_2}\right\}\geqslant \frac{c^*}{\sqrt{K^{(3)}}}
$$
We end the proof for \eqref{eq: bound on first value of tilde V}. Next, by the definition $\mathbf{\Xi}^{(3)}=\mA^{(3)}\tilde{\mV}^{(3)}$, we have $\mathbf{\Xi}^{(3)}_{\star 1}=\sum_{k=1}^{K^{(3)}}\mA^{(3)}_{\star k}\tilde{\mV}^{(3)}_{k1}$. For any $r\in[R]$, we have
$$
c^*\frac{f_r}{\sqrt{K^{(3)}}}\leqslant\min_{k} \tilde{\mV}^{(3)}_{k1}f_r\leqslant\mathbf{\Xi}^{(3)}_{r 1}\leqslant f_r\max_{k}\tilde{\mV}^{(3)}_{k1}\leqslant C^*\frac{f_r}{\sqrt{K^{(3)}}}
$$
\end{proof}
\begin{corollary}\label{coro: V star row-wise bound}
Under the conditions in Lemma \ref{lem: row-wise HOSVD}, if there exists an efficient VH algorithm such that the bound in \eqref{eq: efficiency of VH algo} could be achieved with high probability, then we have probability at least $1-o(N_R^{-1})$, subject to label permutations $\pi_1,\pi_2,\pi_3$,
\begin{align*} &\max_{k\in[K^{(1)}]}\|\mO^{(1)\top}\hat \mV^{*(1)}_{\pi_1(k)\star}- \mV^{*(1)}_{k\star}\|\lesssim  \sqrt{\frac{\log N_R}{N^{(1)}M}}\\&\max_{k\in[K^{(2)}]}\|\mO^{(2)\top}\hat \mV^{*(2)}_{\pi_2(k)\star}- \mV^{*(2)}_{k\star}\|\lesssim\sqrt{\frac{\log N_R}{N^{(2)}M}} \\
&\max_{k\in[K^{(3)}]}\|\mO^{(3)\top}\hat \mV^{*(3)}_{\pi_3(k)\star}- \mV^{*(3)}_{k\star}\|\lesssim\max_r\sqrt{\frac{\log N_R}{f_rN^{(1)}N^{(2)}M}}
\end{align*}
where $\mV^{*(a)}=\mV^{(a)}$ for $a\in\{1,2\}$, $\mV^{*(3)}=[\mathbf{1}_{K^{(3)}}, \mV^{(3)}]$, and $\mathbf{\Omega}^{*(a)}=\mathbf{\Omega}^{(a)}$ for all $a\in[3]$.
\end{corollary}
\begin{lemma}\label{lemma: Vstar eigenvalue in C1} If $\mA^{(1)}$ and $\mA^{(2)}$ satisfy Assumption \ref{ass: min singular value} and $\mV^{*(a)}=\tilde{\mV}^{(a)}$ for $a\in\{1,2\}$, then we have for some constants $c^*,C^*>0$
    \begin{align*}   
    \frac{c^*}{N^{(1)}}\leqslant\lambda_{K^{(1)}}(\mV^{*(1)} \mV^{*(1)\top})\leqslant\lambda_{1}(\mV^{*(1)}  \mV^{*(1)\top})\leqslant \frac{C^*}{N^{(1)}}\\
     \frac{c^*}{N^{(2)}} \leqslant\lambda_{K^{(2)}}(\mV^{*(2)} \mV^{*(2)\top})\leqslant\lambda_{1}(\mV^{*(2)} \mV^{*(2)\top})\leqslant \frac{C^*}{N^{(2)}}
     \end{align*}
\end{lemma}
\begin{proof}
    Applying the results from Lemma \ref{lem: hat S row wise bound mode 1 and 2} and corollary \ref{coro: V star row-wise bound}, we have  
    $\mV^{*(a)}= \mV^{(a)}=\tilde{\mV}^{(a)}$ for $a\in\{1,2\}$. Item 2 in Lemma \ref{lemma:v_tilde_properties} yields that $\tilde{\mV}^{(a)}\tilde{\mV}^{(a)\top}=(\mA^{(a)\top}\mA^{(a)})^{-1}$ is non-degenerate, we have
    $$
    \lambda_{K^{(a)}}(\mV^{*(a)}\mV^{{*(a)}\top})=\lambda_{1}(\mA^{(a)\top} \mA^{(a)})^{-1}\stackrel{\text{Lemma \ref{lemma: max singular values}}}{\geqslant} \frac{1}{N^{(a)}}>0
    $$
   Using Assumption \ref{ass: min singular value}, we have the upper bound
    $$ \lambda_{1}(\mV^{*(a)}\mV^{*{(a)}\top})\leqslant\frac{C^*}{N^{(a)}}$$
\end{proof}
\begin{lemma}\label{lemma: Vstar eigenvalue in C2} Under the Assumptions \ref{ass: min singular value}, let $\tilde{\mV}^{*(3)}=[\text{diag}(\tilde{\mV}^{(3)}_{\star 1})]^{-1}\tilde{\mV}^{(3)}$, for some constants $c^*,C^*>0$, we have
    $$
    c^* \leqslant\lambda_{K^{(3)}}(\mV^{*(3)} \mV^{*(3)\top})\leqslant\lambda_{1}(\mV^{*(3)} \mV^{*(3)\top})\leqslant C^*.
    $$
\end{lemma}
\begin{proof}
As $ \mV^{*(3)}=[\mathbf{1}_{K^{(3)}}, \mV^{(3)}]=[\text{diag}(\tilde{\mV}^{(3)}_{\star 1})]^{-1}\tilde{\mV}^{(3)}$, we could write
\begin{align*}
    \|\mV^{*(3)}\|_2^2=&\lambda_{\max}(\mV^{*(3)} \mV^{*(3)\top})\\
    =&\lambda_{\max}([\text{diag}(\tilde{\mV}^{(3)}_{\star 1})]^{-1} \tilde{\mV}^{(3)} \tilde{\mV}^{(3)\top}[\text{diag}(\tilde{\mV}^{(3)}_{\star 1})]^{-1} )\\
    \leqslant&\lambda_{\max}([\text{diag}(\tilde{\mV}^{(3)}_{\star 1})]^{-1}[\text{diag}(\tilde{\mV}^{(3)}_{\star 1})]^{-1})\lambda_{\max}(\tilde{\mV}^{(3)}\tilde{\mV}^{(3)\top})\\
    \stackrel{\text{Lemma \ref{lemma:v_tilde_properties}}}{\leqslant} &\left[\min_{k}\tilde{\mV}^{(3)}_{k1}\right]^{-2} [\lambda_{\min}(\mA^{(3)\top} \mA^{(3)})]^{-1} \quad \text{ as }  \tilde{\mV}^{(3)}\tilde{\mV}^{(3)\top}=(\mA^{(3)\top} \mA^{(3)})^{-1}\\
    \leqslant &\frac{\sqrt{K^{(3)}}}{c^*}\cdot c^* K^{(3)}\text{ by Assumption \ref{ass: min singular value} and inequality \eqref{eq: bound on first value of tilde V}}\\
    \leqslant& c^* \sqrt{K^{(3)}}\leqslant C^*
\end{align*}
On the other side, Lemma \ref{lemma: max singular values} and \eqref{eq: bound on first value of tilde V} enable
$$
\lambda_{K^{(3)}}(\mV^{*(3)}\mV^{*(3)\top})\geqslant \left(\frac{1}{\max_{k}\tilde{\mV}^{(3)}_{k1}}\right)^2 \frac{1}{\lambda_{1}(\mA^{(3)\top} \mA^{(3)})}\geqslant c^*
$$
\end{proof}

\section{\texorpdfstring{Recovery Error for weight matrix $\mOmega^*$}{Recovery Error for weight matrix Omega*}}
Recall that $\mS^*=\mOmega^*\mV^*$.
To get the recovery error, the first key step is to prove $\mV^*$ and $\hat \mV^*$ is non-degenerate. 
\begin{lemma}\label{lem: hat V is nondeg}
Suppose we have matrix $\mV^*\in\mathbb{R}^{K\times K}$ is non-degenerate such that $\sigma_{\min}(\mV^*)\geqslant \epsilon_n>0$. Given a permutation matrix $\Pi$ and orthogonal matrix $\mO$, if $\|\Pi \hat \mV^* \mO-\mV^*\|\leqslant \epsilon_1$ for some $\epsilon_n>\epsilon_1>0$,
then $\hat \mV^*$ is non-degenerate with $\sigma_{\min}(\hat \mV^*)\geqslant \epsilon_n-\epsilon_1$.
\end{lemma}
\begin{proof}
Since $\mV^*$ is non-degenerate, $\sigma_{\min}(\mV^*)>0$.
    By Wely's inequality, we have
     $$
    \sigma_{\min}(\hat \mV^*)=\sigma_{\min}(\Pi\hat \mV^*\mO)\geqslant \sigma_{\min}(\mV^*)-\|\Pi\hat \mV^*\mO-\mV^*\|\geqslant \epsilon_n-\epsilon_1>0
    $$
\end{proof}
\begin{lemma}\label{lem: hat omega bound} Suppose the VH algorithm is efficient as in Assumption \ref{ass: vh efficient}. Under the conditions in Lemma \ref{lem: row-wise HOSVD}, we have with probability at least $1-o(N_R^{-1})$, for every $i\in[N^{(1)}]$, $j\in[N^{(2)}]$ and $r\in[R]$
\begin{align*}
     &\|\mathbf{\hat \Omega}^{*(1)}- \mathbf{\Omega}^{*(1)} \Pi^{(1)}\|_F\lesssim \sqrt{\frac{N^{(1)}\log N_R}{M}} \quad\text{and}\quad  \|\mathbf{\hat \Omega}^{*(1)}_{i\star}- \Pi^{(1)\top}\mathbf{\Omega}^{*(1)}_{i\star} \|_1\lesssim \sqrt{\frac{N^{(1)}\log N_R}{M}}\\
        &\|\mathbf{\hat \Omega}^{*(2)}- \mathbf{\Omega}^{*(2)} \Pi^{(2)}\|_F\lesssim \sqrt{\frac{N^{(2)}\log N_R}{M}} \quad \text{and}\quad  \|\mathbf{\hat \Omega}^{*(2)}_{j\star}- \Pi^{(2)\top}\mathbf{\Omega}^{*(2)}_{j\star} \|_1\lesssim\sqrt{\frac{N^{(2)}\log N_R}{M}}\\
        &   \|(\mathbf{\hat \Omega}^{*(3)}- \mathbf{\Omega}^{*(3)} \Pi^{(3)})_{r\star }\|_2\lesssim \left(\frac{\log N_R}{N^{(1)}N^{(2)}M}\right)
\end{align*}
for some permutation matrices $\Pi^{(1)}$, $\Pi^{(2)}$, and $\Pi^{(3)}$ which are consistent as before.
\end{lemma}
\begin{proof}
Recall that the estimator $ \mathbf{\tilde\Omega}^{*(a)}$ is obtained by $\mathbf{\tilde \Omega}^{*(a)}=\mathbf{\hat S}^{*(a)}\left(\hat \mV^{*(a)}\right)^{-1} $. Lemmas \ref{lem: hat S row wise bound mode 1 and 2} and \ref{lem: hat S row wise bound mode 3} control the error between $\mathbf{\hat S}^{*(a)} \mO^{(a)}$ and $\mathbf{S}^{*(a)}$, while an efficient VH algorithm enables then to control the discrepancy between $ \Pi^{(a)}\hat \mV^{*(a)}\mO^{(a)}$ and $\mV^{*(a)}$ with a label permutation $\Pi^{(a)}$ and orthogonal matrix $\mO^{(a)}$. Thereby, $\mathbf{\tilde \Omega}^{*(a)} \left(\Pi^{(a)}\right)^{-1}=\mathbf{\hat S}^{*(a)} \mO^{(a)} \mO^{(a)\top}\left(\hat \mV^{*(a)}\right)^{-1} \left(\Pi^{(a)}\right)^{-1}$. 

We will prove later that for any $a\in[3]$, $\hat \mV^{*(a)}$ is non-degenerate with probability at least $1-o(N_R^{-1})$.
Here we assume it is non-degenerate. For simplicity of notation, let $\mO$, $\Pi$ be general matrice for $\mO^{(a)}$, and $\Pi^{(a)}$ respectively. Then
    \begin{align}
        &\left\|\mathbf{\tilde \Omega}^{*(a)} -\mathbf{\Omega}^{*(a)} \Pi^{(a)}\right\|_F=\left\|\mathbf{\hat S}^{*(a)}  \mO\mO^{\top}\left(\hat \mV^{*(a)}\right)^{-1}-\mathbf{S}^{*(a)} \left( \mV^{*(a)}\right)^{-1}\Pi\right\|_F\notag\\
        =&\left\|(\mathbf{\hat S}^{*(a)} \mO  -\mathbf{S}^{*(a)} )\mO^{\top}\left(\hat \mV^{*(a)}\right)^{-1}+\mathbf{S}^{*(a)}(\mO^{\top}\left(\hat \mV^{*(a)}\right)^{-1}\Pi  - \left( \mV^{*(a)}\right)^{-1}\Pi\right\|_F\notag\\
        \leqslant &\left\|(\mathbf{\hat S}^{*(a)} \mO -\mathbf{S}^{*(a)} )\right\|_F\cdot \left\|\left(\hat \mV^{*(a)}\right)^{-1}\right\|_2+\left\|\mO^{\top}\left(\hat \mV^{*(a)}\right)^{-1}\Pi^{-1} - \left( \mV^{*(a)}\right)^{-1}\right\|_F\notag\\
        \leqslant& \left\|\mathbf{\hat S}^{*(a)} \mO  -\mathbf{S}^{*(a)} \right\|_F\cdot \left\|\left(\hat \mV^{*(a)}\right)^{-1}\right\|+\left\|\left(\hat \mV^{*(a)}\right)^{-1} \right\|\left\|\left( \mV^{*(a)}\right)^{-1}\right\|\left\|\Pi\hat \mV^{*(a)}\mO-\mV^{*(a)}\right\|_F\label{eq: tilde Omega error}
        \end{align}
        Similarly, we also have for any abstract row index $r$,
 \begin{align}
    &\|\mathbf{\tilde \Omega}^{*(a)} _{r\star }-\Pi^{(a)\top}\mathbf{\Omega}^{*(a)} _{r\star}\|_2=\|\left(\hat \mV^{*(a)}\right)^{-\top}\mO(\mO^\top \mathbf{\hat S}^{*(a)}_{r\star})-\Pi^\top \left( \mV^{*(a)}\right)^{-\top}\mathbf{S}^{*(a)}_{r\star}\|_2\notag\\
    =&\left\|\left(\hat \mV^{*(a)}\right)^{-1}\mO \left(\mO^\top \mathbf{\hat S}^{*(a)}_{r\star}-\mathbf{S}^{*(a)}_{r\star }\right)+\left\{\left(\hat \mV^{*(a)}\right)^{-\top}\mO-\Pi^{-1}\left( \mV^{*(a)}\right)^{-\top}\right\}\mathbf{S}^{*(a)}_{r\star}\right\|_2\notag\\
    \leqslant&\left\| \left(\hat \mV^{*(a)}\right)^{-1}\right\|\cdot \|\mO^* \mathbf{\hat S}^{*(a)}_{r\star}- \mathbf{S}^{*(a)}_{r\star}\|_2+\|\mO^\top \left(\hat \mV^{*(a)}\right)^{-1}\Pi^{-1}-\left( \mV^{*(a)}\right)^{-1}\|\cdot \|\mS_{r\star}^{*(a)}\|_2\notag\\
    \leqslant& \left\| \left(\hat \mV^{*(a)}\right)^{-1}\right\| \|\mO^* \mathbf{\hat S}^{*(a)}_{r\star}- \mathbf{S}^{*(a)}_{r\star}\|+\left\|\left(\hat \mV^{*(a)}\right)^{-1} \right\|\left\|\left( \mV^{*(a)}\right)^{-1}\right\|\left\|\Pi\hat \mV^{*(a)}\mO-\mV^{*(a)}\right\|_F\|\mS_{r\star}^{*(a)}\|\label{eq: rowwise tilde Omega error}
\end{align}
\begin{itemize}
    \item \textbf{Mode 1 \& 2 matricization: } with probability at least $1-o(N_R^{-1})$, we have
    \begin{enumerate}
        \item $\hat \mV^{*(a)}$ for $a\in\{1,2\}$ is non-degenerate with $\sigma_{\min}(\hat \mV^{*(a)})\geqslant\frac{1}{2}\sqrt{\frac{c^*}{N^{(a)}}}$: Since Corollary \ref{coro: V star row-wise bound} yields that
        $$
        \left\|\Pi^{(a)}\hat \mV^{*(a)}\mO^{(a)}-\mV^{*(a)}\right\|_F\leqslant C^*\sqrt{\frac{\log N_R}{N^{(a)}M}} \quad \text{for any } a\in\{1,2\},
        $$
        and Lemma \ref{lemma: Vstar eigenvalue in C1} gives $\sigma_{\min}(\mV^{*(a)})\geqslant  \sqrt{\frac{c^*}{N^{(a)}}}$, we can use the assumption $M\geqslant \log N_R/c_0$ with $c_0< \frac{c^*}{4C^{*2}}$ to derive
        $\sqrt{\frac{c^*}{N^{(a)}}}-C^*\sqrt{\frac{\log N_R}{N^{(a)}M}}\geqslant \frac{1}{2}\sqrt{\frac{c^*}{N^{(a)}}}>0$. Therefore, 
we have the non-degeneracy of $\hat \mV^{*(a)}$ by lemma \ref{lem: hat V is nondeg}.
        \item Lemma \ref{lem: hat S row wise bound mode 1 and 2} can further admit that
        $
        \left\|\mathbf{\hat S}^{*(a)} \mO  -\mathbf{S}^{*(a)} \right\|_F^2=\sum_{i}\left\|(\mathbf{\hat S}^{*(a)} \mO  -\mathbf{S}^{*(a)} )_{i\star}\right\|_2^2\lesssim \frac{\log N_R}{M}.
        $
        Then, \eqref{eq: tilde Omega error} gives
        \begin{align*}
            \left\|\mathbf{\tilde \Omega}^{*(a)} -\mathbf{\Omega}^{*(a)} \Pi^{(a)}\right\|_F&\lesssim\sqrt{\frac{\log N_R}{M}}\sqrt{N^{(a)}}+N^{(a)}\cdot \sqrt{\frac{\log N_R}{N^{(a)}M}}\\
            &\leqslant\sqrt{\frac{N^{(a)}\log N_R}{M}}\quad \text{ for } a\in\{1,2\}
        \end{align*}
        \end{enumerate}
        \item \textbf{Mode 3 matricization: } with probability at least $1-o(N_R^{-1})$,
        \begin{enumerate}
            \item $\hat \mV^{*(3)}$ is non-degenerate with $\sigma_{\min}(\hat \mV^{*(3)})\geqslant c^*/2$:  
        Lemma \ref{lemma: Vstar eigenvalue in C2} gives $\sigma_{\min}(\mV^{*(3)})\geqslant  c^*$. By the assumption $f_r\geqslant \sqrt{\frac{\log N_R}{N^{(1,2)}M}}$, Corollary \ref{coro: V star row-wise bound} gives that $ \left\|\Pi^{(3)}\hat \mV^{*(3)}\mO^{(3)}-\mV^{*(3)}\right\|_F\lesssim\left(\frac{\log N_R}{N^{(1,2)}M}\right)^{1/4}$. As $c^*-C^*\left(\frac{\log N_R}{N^{(1,2)}M}\right)^{1/4}\geqslant c^*-o(1)\geqslant c^*/2>0$, Lemma \ref{lem: hat V is nondeg} gives the results.
        \item As $\mathbf{S}^{*(3)}=[1_R,\mathbf{S}^{(3)}]$, and $\mathbf{\hat S}^{*(3)}=[1_R,\mathbf{\hat S}^{(3)}]$, Lemma \ref{lem: hat S row wise bound mode 3} also yields that
        $
        \left\|(\mathbf{\hat S}^{*(3)} \mO^{(3)}  -\mathbf{S}^{*(3)} )_{r\star }\right\|\lesssim \sqrt{\frac{\log N_R}{f_rN^{(1,2)}M}}.
        $ with $\mO^{(3)}=\text{diag}(\pm 1, \mathbf{\Omega}^*)$.
        By \eqref{eq: Sr in mode 3}, we have $\|\mathbf{S}^{*(3)}_{r\star}\|_2\leqslant\sqrt{1+\|\mathbf{S}^{(3)}_{r\star}\|_2^2}\leqslant C^*$. Then, \eqref{eq: rowwise tilde Omega error} gives 
        \begin{align*}
           \|\mathbf{\tilde \Omega}^{*(3)} _{r\star }-\Pi^{(3)\top}\mathbf{\Omega}^{*(3)} _{r\star}\|_2 \lesssim\sqrt{\frac{\log N_R}{f_rN^{(1,2)}M}}+\left(\frac{\log N_R}{N^{(1,2)}M}\right)^{1/4}\lesssim\left(\frac{\log N_R}{N^{(1,2)}M}\right)^{1/4}
        \end{align*}
            \end{enumerate}
        \end{itemize}
       Note the final weight matrix $\mathbf{\hat \Omega}^{*(a)}$ is given by $\mathbf{\hat \Omega}^{*(a)}_{i\star }=\frac{ \mathbf{\check\Omega}^{*(a)}_{i\star }}{\left\| \mathbf{\check\Omega}^{*(a)}_{i\star }\right\|_1}$ where $ \mathbf{\check\Omega}^{*(a)}_{ik}=\max\{\mathbf{\tilde \Omega}^{*(a)}_{ik},0\}$. The weight matrix $\mathbf{\Omega}^{*(a)}=\mathbf{\Omega}^{(a)}$ for all $a\in[3]$ has property $\left\|\mathbf{\Omega}^{*(a)}_{i\star }\right\|_1=1=\left\|\mathbf{\hat \Omega}^{*(a)}_{i\star }\right\|_1$. Then, we have
        \begin{equation}\label{eq:hat omega bound}
        \begin{aligned}
            \left\|\mathbf{\hat \Omega}^{*(a)}_{i\star }-\Pi^{\top}\mathbf{\Omega}^{*(a)}_{i\star }\right\|_1&\leqslant\|\mathbf{\hat \Omega}^{*(a)}_{i\star }- \mathbf{\check\Omega}^{*(a)}_{i\star }\|_1+\| \mathbf{\check\Omega}^{*(a)}_{i\star }-\Pi^{\top}\mathbf{\Omega}^{*(a)}_{i\star }\|_1\\
            &= \|\mathbf{\hat \Omega}^{*(a)}_{i\star }\|_1\cdot\left|1-\| \mathbf{\check\Omega}^{*(a)}_{i\star }\|_1\right|+\| \mathbf{\check\Omega}^{*(a)}_{i\star }-\Pi^{\top}\mathbf{\Omega}^{*(a)}_{i\star }\|_1\text{ since } \mathbf{\check\Omega}^{*(a)}_{i\star  }=\mathbf{\hat \Omega}^{*(a)}_{i\star}\| \mathbf{\check\Omega}^{*(a)}_{i\star}\|_1\\
            &=\left|\|\mathbf{\hat \Omega}^{*(a)}_{i\star }\|_1-\|\mathbf{\check \Omega}_{i\star }^{*(a)}\|_1\right|+\| \mathbf{\check\Omega}^{*(a)}_{i\star }-\Pi^{\top}\mathbf{\Omega}^{*(a)}_{i\star }\|_1\\
            &\leqslant\| \Pi^{\top}\mathbf{\Omega}^{*(a)}_{i\star }-\mathbf{\check \Omega}_{i\star }^{*(a)}\|_1+\|\mathbf{\check \Omega}_{i\star }^{*(a)}-\Pi^{\top}\mathbf{\Omega}^{*(a)}_{i\star }\|_1\quad\text{as} \quad \|\mathbf{\hat \Omega}^{*(a)}_{i\star}\|_1=\|\mathbf{\Omega}^{*(a)}_{i\star}\|_1\\
          &  \leqslant 2 \|\mathbf{\check \Omega}_{i\star }^{*(a)}-\Pi^{\top}\mathbf{\Omega}^{*(a)}_{i\star }\|_1
            \\&\leqslant  2 \|\mathbf{\tilde \Omega}^{*(a)}_{i\star }-\Pi^{\top}\mathbf{\Omega}^{*(a)}_{i\star }\|_1
        \end{aligned}
        \end{equation}
        Therefore,
        \begin{align*}   
        \|\mathbf{\hat \Omega}^{*(a)}- \mathbf{\Omega}^{*(a)} \Pi^{(a)}\|_F^2&\leqslant\sum_{i=1}^{N^{(a)}} \|\mathbf{\hat \Omega}^{*(a)}_{i\star }-\Pi^{(a)\top}\mathbf{\Omega}^{*(a)}_{i\star }\|_1^2\\
        &\leqslant 2\sqrt{K^{(a)}}\sum_{i} \|\mathbf{\tilde \Omega}^{*(a)}_{i\star }-\Pi^{(a)\top}\mathbf{\Omega}^{*(a)}_{i\star }\|_2^2\\
        &\lesssim \frac{N^{(a)}\log N_R}{M}\text{ for } a\in\{1,2\}
        \end{align*}
        and
        $$
         \|(\mathbf{\hat \Omega}^{*(3)}- \mathbf{\Omega}^{*(3)} \Pi^{(3)})_{r\star }\|_2\lesssim \left(\frac{\log N_R}{N^{(1)}N^{(2)}M}\right)
        $$
\end{proof}

\section{Matrix A estimation error}
\begin{lemma}[Mode 1 \& 2 matricization]\label{lem: hat A1 2 ERROR} Fix $a\in\{1,2\}$, we obtain $\hat \mA^{(a)}=\mathbf{\hat \Omega}^{*(a)}$. Under the conditions in \ref{lem: hat omega bound}, we have with probability at least $1-o(N_R^{-1})$,
    $$
    \|\hat \mA^{(1)}-\mA^{(1)}\Pi^{(1)}\|_1\lesssim N^{(1)}\sqrt{\frac{\log N_R}{M}},  \quad \|\hat \mA^{(1)}-\mA^{(1)}\Pi^{(1)}\|_F\lesssim \sqrt{\frac{N^{(1)}\log N_R}{M}} , 
    $$
    and 
    $$
    \|\hat \mA^{(2)}-\mA^{(2)}\Pi^{(2)}\|_1\lesssim N^{(2)}\sqrt{\frac{\log N_R}{M}} ,  \quad \|\hat \mA^{(2)}-\mA^{(2)}\Pi^{(2)}\|_F\lesssim \sqrt{\frac{N^{(2)}\log N_R}{M}} , 
    $$
\end{lemma}
\begin{proof}
This can be easily obtained from the results of Lemma \ref{lem: hat omega bound}
$$
  \|\hat \mA^{(1)}-\mA^{(1)}\Pi^{(1)}\|_1\leqslant  C^* \sqrt{N^{(1)}} \|\hat \mA^{(1)}-\mA^{(1)}\Pi^{(1)}\|_F\lesssim N^{(1)}\sqrt{\frac{\log N_R}{M}}
$$
\end{proof}
\begin{lemma}[Mode 3 matricization]\label{lem: hat A3 error}
    Suppose we get $\hat \mA^{(3)}$ from the column-normalized $\text{diag}({\mathbf{\mathbf{\hat \Xi}}^{(3)}}_{\star 1})\mathbf{\hat \Omega}^{*(3)}$. Under conditions in Lemma \ref{lem: hat omega bound} hold, then we have with probability at least $1-o(N_R^{-1})$,
    $$
    \|\hat \mA^{(3)}-\mA^{(3)}\Pi^{(3)}\|_1\lesssim \left(\frac{\log N_R}{N^{(1)}N^{(2)}M}\right)^{1/4}
    $$
\end{lemma}
\begin{proof}
Define \begin{align}\label{def: hat Astar and Astar}
    \hat \mA^*=\text{diag}({{\mathbf{\hat \Xi}}^{(3)}}_{\star 1})\mathbf{\hat \Omega}^{*(3)}\text{ and }\mA^*=\text{diag}({\mathbf{\Xi}^{(3)}_{\star 1} )\mathbf{\Omega}^{*(3)}}.
\end{align} Note $|\mathbf{{\hat \Xi}}^{(3)}_{r1}|\leqslant 2|\mathbf{\Xi}^{(3)}_{r1}|\leqslant C^*f_r$ by \eqref{eq: bound of xi 1 entry}. Then with probability at least $1-o(N_R^{-1})$, for any $r\in[R]$, if $f_r\geqslant \sqrt{\frac{\log(N_R)}{N^{(1,2)}M}}$, we have
\begin{align*}
    \|\hat \mA^*_{r\star}-\Pi^{(3)\top} \mA^*_{r\star}\|_1&\leqslant\|{\mathbf{\hat \Xi}}^{(3)}_{r1}\mathbf{\hat \Omega}^{*(3)}_{r\star}-\Pi^{(3)\top} \mathbf{\Xi}^{(3)}_{r1}\mathbf{\Omega}^{*(3)}_{r\star}\|_1\\
    &\leqslant|\mathbf{\hat \Xi}^{(3)}_{r1}|\|\mathbf{\hat\Omega}^{*(3)}_{r\star}-\Pi^{(3)\top} \mathbf{\Omega}^{*(3)}_{r\star}\|_1+|\mathbf{{\hat \Xi}}^{(3)}_{r1}-\mathbf{\Xi}^{(3)}_{r1}|\|\mathbf{\Omega}^{*(3)}_{r\star}\|_1\\
    &\lesssim f_r\left(\frac{\log N_R}{N^{(1)}N^{(2)}M}\right)^{1/4}\text{ from Lemma \ref{lem: hat omega bound} and } \|\mathbf{\Omega}^{*(3)}_{r\star}\|_1=1
\end{align*}
Note that the final estimator $\hat \mA^{(3)}$ is defined as column-normalized $\hat \mA^*$. 
Also, $\mA^*=\mA^{(3)}\cdot \text{diag}(\tilde{\mV}^{(3)}_{\star 1})$ (by \eqref{eq:definition of Pi} and \eqref{def: hat Astar and Astar}).
We have for each $r\in[R]$ and $k\in[K^{{(3)}}]$,
$$
\hat \mA^{(3)}_{rk}=\frac{\hat \mA^*_{rk}}{\|\hat \mA^*_{\star k}\|_1}\text{ and }  \mA^{(3)}_{rk}=\frac{ \mA^*_{rk}}{\tilde{\mV}^{(3)}_{k1}}
$$
 Given that $\|\mA^{(3)}_{\star k}\|_1=1$, we have $
\|\mA^*_{\star \pi(k)}\|_1=\tilde{\mV}^{(3)}_{k1}$. Consequently, $|\hat \mA^{(3)}_{r\pi(k)}-\mA^{(3)}_{rk}|\leqslant\frac{1}{\|\hat \mA^*_{\star \pi(k)}\|_1}\left|\hat \mA_{r\pi(k)}^*-\mA_{rk}^*\right|+\frac{\left|\|\hat \mA^*_{\star \pi(k)}\|_1-\tilde{\mV}^{(3)}_{k1}\right|}{\|\hat \mA^*_{\star \pi(k)}\|_1}\left|\mA^{(3)}_{rk}\right|$.  Then,
\begin{align}
    \left|\|\hat \mA^*_{\star \pi(k)}\|_1-\tilde{\mV}^{(3)}_{k1}\right|&=\left|\|\hat \mA^*_{\star \pi(k)}\|_1-\|\mA^*_{\star k}\|_1\right|\notag\\&\leqslant\|\hat \mA^*_{\star \pi(k)}-\mA^*_{\star k}\|_1\notag\\
    &\leqslant \sum_{r=1}^R|\hat \mA^*_{r \pi(k)}-\mA^*_{r k}|\leqslant\sum_{r=1}^R\|\hat \mA^*_{r \star}-\Pi^{(3)\top} \mA^*_{r \star}\|_1\label{eq: hat Astar and tilde V error}
\end{align} 
Combining this with $\sum_{r}f_r=K$, $\|\mA^{(3)}_{r\star}\|_1=f_r$ and $\|\hat \mA_{\star k}^*\|_1\geqslant\tilde{\mV}^{(3)}_{k1}/2\geqslant c^*$ by \eqref{eq: bound on first value of tilde V}, we have
$$
\|\hat \mA^{(3)}_{r\star}-\Pi^{(3)\top} \mA^{(3)}_{r\star}\|_1\leqslant C^*\|\hat \mA_{r\star}^*-\Pi^{(3)\top} \mA_{r\star}^*\|_1
\lesssim f_r\left(\frac{\log N_R}{N^{(1)}N^{(2)}M}\right)^{1/4}
$$
Therefore, we have
$$
\|\hat \mA^{(3)}-\mA^{(3)}\Pi^{(3)}\|_1\leqslant \sum_{r}\|\mA^{(3)}_{r\star}\|_1\cdot \max_{r}\left\{\frac{\|\hat \mA^{(3)}_{r\star}-\Pi^{(3)\top} \mA^{(3)}_{r\star}\|_1}{\|\mA^{(3)}_{r\star}\|_1}\right\}\lesssim \left(\frac{\log N_R}{N^{(1)}N^{(2)}M}\right)^{1/4}
$$
\end{proof}
\section{Core Tensor Recovery}
\begin{lemma} \label{lem: G and S}
Let $\D = \mathcal{G} \cdot ( \mA^{(1)}, \mA^{(2)}, \mA^{(3)}) $ be a tensor following the Tensor Topic Model explicited in equation \eqref{eq: TTM}.
   Given its HOSVD decomposition $\mathcal{D}=\mathcal{S}\cdot \left(\mathbf{\Xi}^{(1)},\mathbf{\Xi}^{(2)},\mathbf{\Xi}^{(3)}\right)$ such that $\mathcal{S}\in\mathbb{R}^{K^{(1)}\times K^{(2)}\times K^{(3)}}$, if Assumption \ref{ass: min singular value} is satisfied, then there exist some invertible matrices $\tilde{\mV}^{(1)}$, $\tilde{\mV}^{(2)}$, and $\tilde{\mV}^{(3)}$ such that
    $$
    \mathcal{G}=\mathcal{S}\cdot \left(\tilde{\mV}^{(1)},\tilde{\mV}^{(2)},\tilde{\mV}^{(3)}\right).
    $$

\end{lemma}
\begin{proof}
    Note $\mathbf{\Xi}^{(a)}$ is the matrix of $K^{(a)}$ leading right singular vectors from SVD of mode-a matricization $\mathcal M_a( \D)$ (so $\mathbf{\Xi}^{(a)\top}\mathbf{\Xi}^{(a)} = I_{K^{(a)}}$), and 
    where the $K^{(1)}\times K^{(2)}\times K^{(3)}$ core tensor $\mathcal{S}$ is defined as 
    $$
    \mathcal{S}:=\mathcal{D}\cdot\left(\mathbf{\Xi}^{(1)\top},\mathbf{\Xi}^{(2)\top},\mathbf{\Xi}^{(3)\top}\right)
    $$
    Recall for any $a\in[3]$, by Lemma~\ref{lemma:v_tilde_properties} (bullet point 4), since $\mathbf{\Xi}^{(a)}$ corresponds to the left singular vectors of the matricized data $\mD^{(a)}$, we have $\mathbf{\Xi}^{(a)}=\mA^{(a)}\tilde{\mV}^{(a)}$ where $\tilde{\mV}^{(a)}$ is invertible. Considering in particular the mode-1 matricization of $\D$, we note that:
    \begin{align*}
 \qquad    \mathcal{M}_1(\mathcal{D})&= \mathcal{M}_1(\mathcal{S}\cdot \left(\mathbf{\Xi}^{(1)},\mathbf{\Xi}^{(2)},\mathbf{\Xi}^{(3)}\right)) =\mathbf{\Xi}^{(1)} \mathcal{M}_1(\mathcal{S}\cdot \left(\mathbf{\Xi}^{(2)},\mathbf{\Xi}^{(3)}\right))\\
&=\mathbf{\Xi}^{(1)}\mathcal{M}_1(\mathcal{S})\left(\mathbf{\Xi}^{(2)\top}\otimes \mathbf{\Xi}^{(3)\top}\right) \qquad \text{(Lemma 4 of \cite{zhang2018tensor})}\\
        &=\mA^{(1)}\tilde{\mV}^{(1)}\mathcal{M}_1(\mathcal{S})\left[\left(\tilde{\mV}^{(2)\top}\mA^{(2)\top}\right)\otimes \left(\tilde{\mV}^{(3)\top}\mA^{(3)\top}\right)\right]\\
        &=\mA^{(1)}\tilde{\mV}^{(1)}\mathcal{M}_1(\mathcal{S})\left(\tilde{\mV}^{(2)\top}\otimes \tilde{\mV}^{(3)\top}\right)\left(\mA^{(2)\top}\otimes \mA^{(3)\top}\right)
    \end{align*}
    By \eqref{eq: TTM}, we also have:
    $$\mathcal{M}_1(\mathcal{D})= \mathcal{M}_1(\mathcal{G} \cdot (\mA^{(1)}, \mA^{(2)}, \mA^{(3)}) )=\mA^{(1)} \mathcal{M}_1(\mathcal{G}) ( \mA^{(2)\top} \otimes  \mA^{(3)\top}) \\  $$
Therefore, since $\mA^{(1)\top }\mA^{(1)}$ is invertible (Assumption~\ref{ass: min singular value}), and since the minimum eigenvalue of  
$ ( \mA^{(2)\top} \otimes  \mA^{(3)\top})( \mA^{(2)} \otimes  \mA^{(3)})=  \mA^{(2)\top}\mA^{(2)}  \otimes  \mA^{(3)\top}\mA^{(3)}$ is $ \sigma_{K^{(2)}}(\mA^{(2)\top}\mA^{(2)}) \sigma_{K^{(3)}}(\mA^{(3)\top}\mA^{(3)}) \geq c^*>0$ (Assumption~\ref{ass: min singular value}) and the matrix is thus invertible, we simply take:
$$\mathcal{M}_1(\mathcal{G})=\tilde{\mV}^{(1)}\mathcal{M}_1(\mathcal{S})\left(\tilde{\mV}^{(2)\top}\otimes \tilde{\mV}^{(3)\top}\right)
$$
to get the result.
\end{proof}
\begin{lemma}\label{lemma: S error}
Define $\hat{\mathcal{S}}:=\mathcal{Y}\cdot \left(\mathbf{\mathbf{\hat \Xi}}^{(1)\top},\mathbf{\mathbf{\hat \Xi}}^{(2)\top},\mathbf{\mathbf{\hat \Xi}}^{(3)\top}\right)$, Under the conditions of Lemma \ref{lem: HOSVD error}
then we have with probability at least $1-o(N_R^{-1})$
$$
\|\hat{\mathcal{S}}-\mathcal{S}\left(\mO^{(1)},\mO^{(2)},\mO^{(3)}\right)\|_{F}\lesssim\sqrt{\frac{N^{(1)}N^{(2)}\log N_R}{M}}
$$
\end{lemma}
\begin{proof}
Given the HOSVD form $\mathcal{D}=\mathcal{S}\cdot (\mathbf{\Xi}^{(1)},\mathbf{\Xi}^{(2)},\mathbf{\Xi}^{(3)})$,
\begin{align}
  &\|\hat{\mathcal{S}}-\mathcal{S}\left(\mO^{(1)},\mO^{(2)},\mO^{(3)}\right)\|_{F} =\|\mathcal{M}_1(\hat{\mathcal{S}})-\mO^{(1)}\mathcal{M}_1(\mathcal{S})\left(\mO^{(2)}\otimes \mO^{(3)}\right)^\top\|_{F}\notag\\
  =&\left\|\mathbf{\mathbf{\hat \Xi}}^{(1)\top}\mathcal{M}_1(\mathcal{Y})\left(\mathbf{\mathbf{\hat \Xi}}^{(2)}\otimes \mathbf{\mathbf{\hat \Xi}}^{(3)}\right)-\mO^{(1)}\mathbf{\Xi}^{(1)\top}\mathcal{M}_1(\mathcal{D})\left( \mathbf{\Xi}^{(2)}\otimes \mathbf{\Xi}^{(3)}\right)\left(\mO^{(2)}\otimes \mO^{(3)}\right)^\top\right\|_{F}\notag\\
  \leqslant&\left\|\left(\mathbf{\mathbf{\hat \Xi}}^{(1)}-\mathbf{\Xi}^{(1)} \mO^{(1)\top}\right)^\top\left(\mathcal{M}_1(\mathcal{Y})-\mathcal{M}_1(\mathcal{D})\right)\left(\mathbf{\mathbf{\hat \Xi}}^{(2)}\otimes \mathbf{\mathbf{\hat \Xi}}^{(3)}\right)\right\|_{F}\tag{S.1}\label{eq:S.1}\\
  &+\left\|\left(\mathbf{\mathbf{\hat \Xi}}^{(1)}-\mathbf{\Xi}^{(1)} \mO^{(1)}\right)^\top\mathcal{M}_1(\mathcal{D})\left(\mathbf{\mathbf{\hat \Xi}}^{(2)}\otimes \mathbf{\mathbf{\hat \Xi}}^{(3)}\right)\right\|_{F}\tag{S.2}\label{eq:S.2}\\  &+\left\|\mO^{(1)}\mathbf{\Xi}^{(1)\top}\left(\mathcal{M}_1(\mathcal{Y})-\mathcal{M}_1(\mathcal{D})\right)\left(\mathbf{\mathbf{\hat \Xi}}^{(2)}\otimes \mathbf{\mathbf{\hat \Xi}}^{(3)}\right)\right\|_{F}\tag{S.3}\label{eq:S.3}\\
  &+\left\|\mO^{(1)}\mathbf{\Xi}^{(1)\top}\mathcal{M}_1(\mathcal{D})\left[\left(\mathbf{\mathbf{\hat \Xi}}^{(2)}\otimes \mathbf{\mathbf{\hat \Xi}}^{(3)}\right)-\left( \mathbf{\Xi}^{(2)}\otimes \mathbf{\Xi}^{(3)}\right)\left(\mO^{(2)}\otimes \mO^{(3)}\right)^\top\right]\right\|_{F}\tag{S.4}\label{eq:S.4}.
\end{align}
We use the property of sub-multiplicative matrix norm to bound equations \eqref{eq:S.1} to \eqref{eq:S.4}, respectively. 
\begin{itemize}
    \item Applying Lemmas \ref{lem: HOSVD error}, \ref{lem: Z l2 no} and the fact that for any matrix $\mA\in\mathbb{R}^{n\times m}$ such that $\|\mA\|_F^2=\sum_{j=1}^{n}\|\mA_{j\star}\|^2_2$, we have
    $$\eqref{eq:S.1}\lesssim\sqrt{\frac{\log N_R}{M}}\cdot \sqrt{\frac{N^{(1)}N^{(2)}\log N_R}{M}}$$
    \item Use the fact that $\|\mD^{(1)}\|_F^2\leqslant \|\mD^{(1)}\|_1\leqslant N^{(1)}N^{(2)}$, we apply \eqref{eq: hat xi 1} again to have$$\eqref{eq:S.2}\lesssim\sqrt{\frac{\log N_R}{M}}\cdot \sqrt{N^{(1)}N^{(2)}}$$
    \item It is easily observed that $\eqref{eq:S.3}\leqslant\sqrt{\sum_{i=1}^{N^{(1)}}\|\mZ^{(1)}_{i\star}\|_2^2}\lesssim\sqrt{\frac{N^{(1)}N^{(2)}\log N_R}{M}}$.
    \item Using the fact that $\|\hat \mA\otimes \hat \mB-\mA\otimes \mB\|\leqslant\|(\hat \mA-\mA)\otimes \hat \mB\|+\|\mA\otimes (\hat \mB-\mB)\|$, we then have $$\eqref{eq:S.4}\lesssim\sqrt{N^{(1)}N^{(2)}}\cdot \left(\sqrt{\frac{\log N_R}{M}}+\sqrt{\frac{\log N_R}{N^{(1)}N^{(2)}M}}\right)\lesssim\sqrt{\frac{N^{(1)}N^{(2)}\log N_R}{M}}$$
    
\end{itemize}

\end{proof}
\begin{lemma}\label{lemma: check V error}
   Define $\check \mV^{(3)}:= \text{diag}(\|\hat \mA^*_{\star 1}\|_1,\dots, \|\hat \mA^*_{\star K}\|_1)\hat \mV^{*(3)}$, where $\hat \mA^*=\text{diag}(\mathbf{\mathbf{\hat \Xi}}^{(3)}_{\star 1})\mathbf{\hat \Omega}^{*(3)}$. Under the settings in Lemma \ref{lem: hat omega bound} we have probability at least $1-o(N_R^{-1})$, 
    \begin{align*}
              \|\Pi^{(3)}\check \mV^{(3)}\mO^{(3)}-\tilde{\mV}^{(3)}\|_F\lesssim \left(\frac{\log N_R}{N^{(1)}N^{(2)}M}\right)^{1/4}
    \end{align*}
\end{lemma}
\begin{proof}
    Recall the definition in \eqref{def: Vstar in 3}, $\tilde{\mV}^{(3)}=\text{diag}(\tilde{\mV}^{(3)}_{\star 1})\mV^{*(3)}$. By \eqref{eq: hat Astar and tilde V error} and the fact that $\sum_r f_r=K^{(3)}$, we have for each $k\in[K^{(3)}]$
   $$
   \left|\|\hat \mA^*_{\star \pi(k)}\|_1-\tilde{\mV}^{(3)}_{k1}\right|\lesssim\sum_r f_r \left(\frac{\log N_R}{N^{(1)}N^{(2)}M}\right)^{1/4}\leqslant \left(\frac{\log N_R}{N^{(1)}N^{(2)}M}\right)^{1/4}
$$
Then, let the label permutation $\pi\leftarrow \Pi^{(3)}$,
   \begin{align*}
       \left\|\text{diag}(\|\hat \mA^*_{\star \pi(1)}\|_1,\dots, \|\hat \mA^*_{\star \pi(K)}\|_1)-\text{diag}(\tilde{\mV}^{(3)}_{\star 1})\right\|_F^2&\lesssim \sum_{k\in[K^{(3)}]}\left(\frac{\log N_R}{N^{(1)}N^{(2)}M}\right)^{1/2}\\
       &\lesssim \left(\frac{\log N_R}{N^{(1)}N^{(2)}M}\right)^{1/2}\end{align*}
    By Corollary \ref{coro: V star row-wise bound} and $f_r\geq \left(\frac{\log N_R}{N^{(1)}N^{(2)}M}\right)^{1/2}$, we have
    $$
    \|\Pi^{(3)} \hat \mV^{*(3)}\mO^{(3)}-\mV^{*(3)}\|_F^2\leqslant \sum_{k\in[K^{(3)}]}\|\mO^{(3)\top}\hat \mV^{*(3)}_{\pi(k)\star}-\mV^{*(3)}_{k\star}\|_2^2 \lesssim \left(\frac{\log N_R}{N^{(1)}N^{(2)}M}\right)^{1/2}
    $$
    Note $\|\mV^*\|\leqslant c^*$ by Lemma \ref{lemma: Vstar eigenvalue in C2}, and $\|\text{diag}(\|\hat \mA^*_{\star \pi(1)}\|_1,\dots, \|\hat \mA^*_{\star \pi(K)}\|_1)\|_2\leqslant \max_k\|\hat \mA^*_{\star k}\|_1\leqslant \sum_{r} f_r=K^{(3)}$.
    Combining these results above, we use the fact $\|\hat \mA \hat \mB-\mA \mB\|\leqslant\|(\hat \mA-\mA) \hat \mB\|+\|\mA (\hat \mB-\mB)\|$ again. By definitions, $\check \mV^{(3)}=\text{diag}(\|\hat \mA^*_{\star 1}\|_1,\dots,\|\hat \mA^*_{\star K}\|_1,\dots)\hat \mV^{*(3)}$ and $\tilde{\mV}^{(3)}=\text{diag}(\tilde{\mV}^{(3)}_{\star 1}) \mV^{*(3)}$, we have
    $$
    \|\Pi^{(3)}\check \mV^{(3)}\mO^{(3)}-\tilde{\mV}^{(3)}\|_F\lesssim \left(\frac{\log N_R}{N^{(1)}N^{(2)}M}\right)^{1/4}
    $$
\end{proof}
\begin{lemma}\label{lem: hat G error} Let $\mathcal{\tilde G}=\mathcal{\mathbf{\hat S}}\cdot \left(\check \mV^{(1)},\check \mV^{(2)},\check \mV^{(3)}\right)$ where $\mathcal{\mathbf{\hat S}},\check{\mV}^{(3)}$ are defined  in lemmas \ref{lemma: S error} and \ref{lemma: check V error}, respectively. Let $\check{\mV}^{(1)}=\hat \mV^{*(1)}$ and $\check \mV^{(2)}=\hat \mV^{*(2)}$. Define $\hat \G$ such that $\mathcal{M}_3(\hat \G)$ is the column-normalized of $\mathcal{M}_3(\tilde \G)$.
Under the condition in Lemma \ref{lem: hat omega bound}, with probability at least $1-o(N_R)$, we have

$$
\|\mathcal{G}\cdot \left(\Pi^{(1)\top},\Pi^{(2)\top},\Pi^{(3)\top}\right)-\hat{\mathcal{G}}\|_{1}\leqslant C^* \max\left\{\sqrt{\frac{\log N_R}{M}},\left(\frac{\log N_R}{N^{(1,2)}M}\right)^{1/4}\right\}
$$
    
\end{lemma}
\begin{proof}
By the definitions of $\tilde\G$ and $\G=\mathcal{S}\cdot \left(\tilde{\mV}^{(1)},\tilde{\mV}^{(2)},\tilde{\mV}^{(3)}\right)$ in Lemma \ref{lem: G and S}, we have
    \begin{align}
        &\|\mathcal{G}\cdot \left(\Pi^{(1)\top},\Pi^{(2)\top},\Pi^{(3)\top}\right)-\tilde{\mathcal{G}}\|_{F}\notag\\
        =&\|\Pi^{(1)\top}\tilde{\mV}^{(1)}\mathcal{M}_1(\mathcal{S})\left(\tilde{\mV}^{(2)\top}\Pi^{(2)}\otimes \tilde{\mV}^{(3)\top}\Pi^{(3)}\right)-\check \mV^{(1)}\mathcal{M}_1(\hat{\mathcal{S}})\left(\check \mV^{(2)\top}\otimes \check \mV^{(3)\top}\right)\|_{F}\notag\\
            \leqslant&\left\|\left(\Pi^{(1)\top}\tilde{\mV}^{(1)}-\check \mV^{(1)} \mO^{(1)}\right)\mathcal{M}_1(\mathcal{S})\left[\left(\tilde{\mV}^{(2)\top}\Pi^{(2)}\otimes \tilde{\mV}^{(3)\top}\Pi^{(3)}\right)-\left(\mO^{(2)\top}\check \mV^{(2)\top}\otimes \mO^{(3)\top}\check \mV^{(3)\top}\right)\right]\right\|_{F}\tag{G.1}\label{eq: G.1}\\
            &+\left\|\left(\Pi^{(1)\top}\tilde{\mV}^{(1)}-\check \mV^{(1)} \mO^{(1)}\right)\mathcal{M}_1(\mathcal{S})\left(\mO^{(2)\top}\check \mV^{(2)\top}\otimes \mO^{(3)\top}\check \mV^{(3)\top}\right)\right\|_{F}\tag{G.2}\label{eq: G.2}\\
            &+\left\|\check \mV^{(1)} \mO^{(1)}\mathcal{M}_1(\mathcal{S})\left[\left(\tilde{\mV}^{(2)\top}\Pi^{(2)}\otimes \tilde{\mV}^{(3)\top}\Pi^{(3)}\right)-\left(\mO^{(2)\top}\check \mV^{(2)\top}\otimes \mO^{(3)\top}\check \mV^{(3)\top}\right)\right]\right\|_{F}\tag{G.3}\label{eq: G.3}\\
            &+\left\|\check \mV^{(1)} \left(\mO^{(1)}\mathcal{M}_1(\mathcal{S})\left(\mO^{(2)\top}\otimes \mO^{(3)\top}\right)-\mathcal{M}_1(\hat{\mathcal{S}})\right)\left(\check \mV^{(2)\top}\otimes\check \mV^{(3)\top}\right)\right\|_{F}\tag{G.4}\label{eq: G.4}
    \end{align}
Our plan is to bound the four terms \eqref{eq: G.1} to \eqref{eq: G.4} separately as follows. By the definitions of $\check \mV^{(1)}$, $\check \mV^{(2)}$ and $\mV^{*(a)}=\mV^{(a)}=\tilde{\mV}^{(a)}$ for $a\in\{1,2\}$, we apply Corollary \ref{coro: V star row-wise bound} to have $\left\|\Pi^{(1)\top}\tilde{\mV}^{(1)}-\check \mV^{(1)} \mO^{(1)\top}\right\|_F\leqslant \| \mV^{*(1)}-\Pi^{(1)}\hat \mV^{*(1)} \mO^{(1)}\|_F\lesssim\sqrt{\frac{\log N_R}{N^{(1)}M}}$. Similarly, $\left\|\Pi^{(2)\top}\tilde{\mV}^{(2)}-\check \mV^{(2)} \mO^{(2)\top}\right\|_F\lesssim\sqrt{\frac{\log N_R}{N^{(2)}M}}$. Lemma \ref{lemma: check V error} gives $\left\|\Pi^{(3)\top}\tilde{\mV}^{(3)}-\check \mV^{(3)} \mO^{(3)\top}\right\|_F\lesssim\left(\frac{\log N_R}{N^{(1)}N^{(2)}M}\right)^{1/4}$. By the definition of $\mathcal{S}$, we have
\begin{align*}
\|\mathcal{S}\|_F=\|\mathbf{\Xi}^{(1)\top}\mathcal{M}_1(\mathcal{D})\left(\mathbf{\Xi}^{(2)}\otimes \mathbf{\Xi}^{(3)}\right)\|_{F}\leqslant \|\mD^{(1)}\|_F\leqslant \sqrt{N^{(1,2)}}
\end{align*}
Note $\|\tilde{\mV}^{(1)}\|_F=K^{(1)}\|\tilde{\mV}^{(1)}\|_2\leqslant K^{(1)}\frac{1}{\sigma_{K}(\mA^{(1)})}\leqslant \frac{c^*}{\sqrt{N^{(1)}}}\leqslant c^*$ by assumption \ref{ass: min singular value}. Using Corollary \ref{coro: V star row-wise bound}, we have
\begin{align*}  
\|\check \mV^{(1)}\mO^{(1)\top}\|_F&\leqslant\|\tilde{\mV}^{(1)}\|_F+\| \mV^{*(1)}-\Pi^{(1)}\hat \mV^{*(1)} \mO^{(1)}\|_F\\
&\leqslant C^*\left(\sqrt{\frac{1}{N^{(1)}}}+\sqrt{\frac{\log N_R}{N^{(1)}M}}\right)\leqslant C^*\sqrt{\frac{1}{N^{(1)}}}
\end{align*}
Similarly, $\|\tilde{\mV}^{(2)}\|_F\leqslant c^*$ and $\|\check \mV^{(2)}\mO^{(2)\top}\|_F\lesssim\sqrt{\frac{1}{N^{(2)}}}$. $\|\tilde{\mV}^{(3)}\|_F\leqslant K^{(3)}\|\tilde  V^{(3)}\|_2\leqslant K^{(3)}\frac{1}{\sigma_{K}(\mA^{(3)})}\leqslant C^*$ and then $\|\check \mV^{(3)}\mO^{(3)}\|_F\leqslant C^*$ by lemma \ref{lemma: check V error}. 
\begin{itemize}
    \item \textbf{For \eqref{eq: G.1}:} 
    \begin{align*}
    \eqref{eq: G.1}\leqslant&\left\|\Pi^\top\tilde{\mV}^{(1)}-\check \mV^{(1)} \mO^{(1)\top}\right\|_F\cdot\|\mathcal{S}\|_F\\
    &\cdot \left(\left\|\Pi^\top\tilde{\mV}^{(2)}-\check \mV^{(2)} \mO^{(2)\top}\right\|_F\|\tilde{\mV}^{(3)}\|_F+\left\|\Pi^\top\tilde{\mV}^{(3)}-\check \mV^{(3)} \mO^{(3)\top}\right\|_F\|\check \mV^{(2)} \mO^{(2)\top}\|_F\right)\\
    \lesssim& \sqrt{\frac{\log N_R}{N^{(1)}M}}\cdot \sqrt{N^{(1)}N^{(2)}}\cdot \left(\sqrt{\frac{\log N_R}{N^{(2)}M}}+\left(\frac{\log N_R}{N^{(1)}N^{(2)}M}\right)^{1/4}\sqrt{\frac{1}{N^{(2)}}}\right)\\
    \leqslant &\frac{\log N_R}{M}+\left(\frac{\log N_R}{M}\right)^{3/4} \left(\frac{1}{N^{(1)}N^{(2)}}\right)^{1/4}
\end{align*}
    \item \textbf{For \eqref{eq: G.2}:} $\eqref{eq: G.2}\lesssim\sqrt{\frac{\log N_R}{N^{(1)}M}}\sqrt{N^{(1,2)}}\sqrt{\frac{1}{N^{(2)}}}\leqslant\sqrt{\frac{\log N_R}{M}}$
    \item \textbf{For \eqref{eq: G.3}:}  $\lesssim\sqrt{\frac{1}{N^{(1)}}}\cdot \sqrt{N^{(1,2)}}\left(\sqrt{\frac{\log N_R}{N^{(2)}M}}+\left(\frac{\log N_R}{N^{(1)}N^{(2)}M}\right)^{1/4}\sqrt{\frac{1}{N^{(2)}}}\right)\leqslant\sqrt{\frac{\log N_R}{M}}+\left(\frac{\log N_R}{N^{(1)}N^{(2)}M}\right)^{1/4}$
    \item \textbf{For \eqref{eq: G.4}:} Lemma \ref{lemma: S error} gives
    $$\eqref{eq: G.4}\lesssim \sqrt{\frac{N^{(1,2)}\log N_R}{M}}\sqrt{\frac{1}{N^{(1,2)}}}\leqslant \sqrt{\frac{\log N_R}{M}}$$
\end{itemize}
Combined all of these and $\|\mA\|_1\leqslant \sqrt{K}\|\mA\|_F$, 
we have
$$
\|\mathcal{G}\cdot \left(\Pi^{(1)\top},\Pi^{(2)\top},\Pi^{(3)\top}\right)-{ \mathcal{\tilde G}}\|_{1}\leqslant C^* \max\left\{\sqrt{\frac{\log N_R}{M}},\left(\frac{\log N_R}{N^{(1,2)}M}\right)^{1/4}\right\}
$$
Using the techniques in \eqref{eq:hat omega bound}, we have
$$
\|\mathcal{G}\cdot \left(\Pi^{(1)\top},\Pi^{(2)\top},\Pi^{(3)\top}\right)-\hat{\mathcal{G}}\|_{1}\leqslant C^* \max\left\{\sqrt{\frac{\log N_R}{M}},\left(\frac{\log N_R}{N^{(1,2)}M}\right)^{1/4}\right\}
$$
\end{proof}

\section{Proof of Theorem \ref{theorem: main theorem} and Lemma \ref{lemm: sparse ttm}}
Theorem \ref{theorem: main theorem} is easy to obtain by the results from Lemmas \ref{lem: hat A1 2 ERROR}, \ref{lem: hat A3 error} and \ref{lem: hat G error}. For the sparse tensor topic modeling, the proof is readaptation of the Section E in \cite{tran2023sparse}.
Consider the oracle sets such that
\begin{align*}
    \mathcal{J}_{\pm}:=\left\{r\in[R]: \frac{1}{N^{(1)}N^{(2)}}\sum_{i\in[N^{(1)}],j\in[N^{(2)}]}\mathcal{D}_{ijr}\right\}> \alpha_{\pm}c'\sqrt{\frac{\log N_R}{N^{(1)}N^{(2)}M}}
\end{align*}
where $c'$ is a constant defined as in set $\mathcal{J}$ and $\alpha_{-}>1$ and $\alpha_{+}\in(0,1)$ for some suitably chosen constants.
\begin{proposition}[Lemma E.1 in \cite{tran2023sparse}]\label{prop: sparse} Suppose the matrix $\mA^{(3)}$ follows the column-wise $\ell_q$ sparsity in Assumption \ref{ass: weak lq sparsity}. If $M\gtrsim \log N_R$ and assumption \ref{ass: min singular value} holds, then with probability at least $1-o(N_R^{-1})$,
the event $\mathcal{E}:=\{\mathcal{J}_{-}\subseteq\mathcal{J}\subseteq\mathcal{J}_{+}\}$ occurs, $\min_{r\in[\mathcal{J}]} f_r>\alpha_{+}c'\sqrt{\frac{\log N_R}{N^{(1)}N^{(2)}M}}$, and
$$
\|\mA^{(3)}_{\mathcal{J}^c\star}\|_1\lesssim \frac{1}{1-q}s_0\left(\frac{\log N_R}{N^{(1)}N^{(2)}M}\right)^{\frac{1-q}{2}}
$$
\end{proposition}
We apply this proposition to adapt the proof for sparse tensor topic modeling. Observe that: 
\begin{align*}
    \|\hat \mA^{(3)}-\mA^{(3)}\Pi^{(3)}\|_1=\sum_{r\in\mathcal{J}}\|(\hat \mA^{(3)}-\mA^{(3)}\Pi^{(3)})_{r\star }\|_1+\sum_{r\in\mathcal{J}^c}\|(\hat \mA^{(3)}-\mA^{(3)}\Pi^{(3)})_{r\star }\|_1.
\end{align*}
For the first part, Proposition \ref{prop: sparse} indicates that $\min_{r\in[\mathcal{J}]} f_r\gtrsim\sqrt{\frac{\log N_R}{N^{(1)}N^{(2)}M}}$ which meets the assumption in Lemma \ref{lem: hat A3 error}. Consequently, the first term has the same rate as before, $\left(\frac{\log N_R}{N^{(1)}N^{(2)}M}\right)^{1/4}$. For the second part, where we estimate the rows of $\mA^{(3)}$ corresponding to indices in $\mathcal{J}^c$ as zero. Proposition \ref{prop: sparse} provides the corresponding bound for the second term as well. With this, we complete the proof for the $\ell$-1 error of the estimator $\mA^{(3)}$ in the context of sparse tensor topic modeling.

Remark that the subsequent estimation procedures, including HOSVD and core recovery, rely on the set $\mathcal{J}$ where $\min_{r\in\mathcal{J}}f_r\gtrsim \left(\frac{\log N_R}{N^{(1)}N^{(2)}M}\right)^{1/2}$. Consequently, the error rates for 
 $\mA^{(1)}$   $\mA^{(2)}$ and $\mathcal{G}$
 remain consistent with previous results.

\section{Benchmarks}\label{app:benchmarks}

\subsection{Tensor LDA}

To compare our method with an extension of LDA to the tensor setting, we implement its Bayesian counterpart -- that is, we endow each latent variable with a Dirichlet prior. The full model is detailed below:

\[
\phi_k \sim \text{Dirichlet}(\beta), \quad \forall k \in \{1, \dots, K_3\}
\]
\[
\theta_{n_1} \sim \text{Dirichlet}(\alpha_1), \quad \forall n_1 \in \{1, \dots, N_1\}
\]
\[
\theta_{n_2} \sim \text{Dirichlet}(\alpha_2), \quad \forall n_2 \in \{1, \dots, N_2\}
\]
\[
G_{k_1,k_2, \cdot} \sim \text{Dirichlet}(\alpha_3), \quad \forall k_1 \in \{1, \dots, K_1\}, \, k_2 \in \{1, \dots, K_2\}
\]
\[
\theta_{n_1, n_2, k_3} = \sum_{k1=1}^{K_1} \sum_{k2=1}^{K_2} G_{k_1,k_2,k_3} \cdot \theta_{n_1 k_1} \cdot \theta_{n_2 k_2}, \quad \forall k_3 \in \{1, \dots, K\}
\]
\[
\X_{n_1, n_2, \star} \sim \text{Multinomial}(M, \theta_{n_1, n_2, \star}\phi).
\]

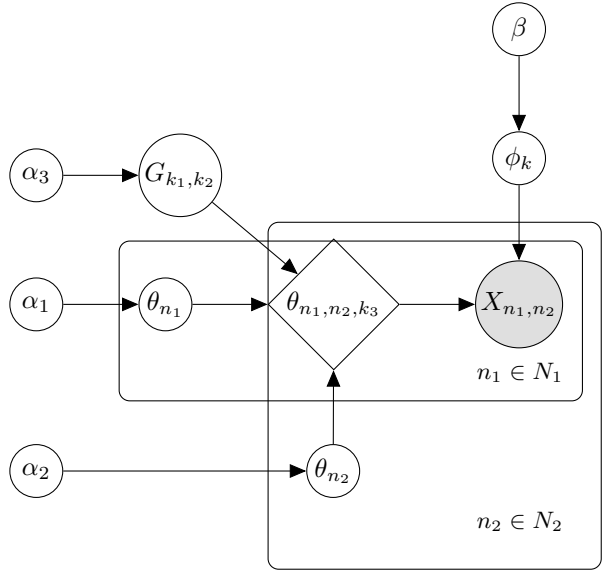
\begin{figure}[H]
\centering
\begin{tikzpicture}

      \node[latent] (alpha3) {\(\alpha_{3}\)};
            \node[latent, below=of alpha3] (alpha1) {\(\alpha_{1}\)};
                  \node[latent, below=of alpha1,yshift=-0.5cm] (alpha2) {\(\alpha_{2}\)};
    \node[latent, right=of alpha3] (G) {\(G_{k_1,k_2}\)};
  \node[latent, right=of alpha1] (theta_n1) {\(\theta_{n_1}\)};
  \node[latent, right=of alpha2, xshift=2.23cm] (theta_n2) {\(\theta_{n_2}\)};
  \node[det, right=of theta_n1] (theta_n1_n2_k3) {\(\theta_{n_1,n_2,k_3}\)};
  \node[obs, right=of theta_n1_n2_k3] (X) {\(X_{n_1,n_2}\)};
    \node[latent, above=of X] (phi) {\(\phi_k\)};
    \node[latent, above=of phi] (beta) {\(\beta\)};
  
  \plate[inner sep=0.25cm] {plate1} {(theta_n1)(X)} {\(n_1 \in N_1\)};
  \plate[inner sep=0.5cm] {plate2} {(theta_n2)(X)} {\(n_2 \in N_2\)};
  
    \edge {alpha3} {G};
        \edge {alpha1} {theta_n1};
                \edge {alpha2} {theta_n2};
                    \edge {beta} {phi};
  \edge {phi} {X};
  \edge {theta_n1} {theta_n1_n2_k3};
  \edge {theta_n2} {theta_n1_n2_k3};
  \edge {G} {theta_n1_n2_k3};
  \edge {theta_n1_n2_k3} {X};
  
\end{tikzpicture}
\caption{Plate Diagram for the tensor LDA model}
\end{figure}

In practice, we set the hyperparameters $\beta, \alpha_1, \alpha_2$ and $\alpha_3$ to 1. To fit this model, we use the variational Bayes provided in RStan\cite{guo2020package}. This provides a fast and scalable solution to the model specified above. Throughout the paper, we report the average of the posterior estimated from the data using this model.

\subsection{Hybrid LDA}

The Hybrid LDA model for tensor topic modeling is a three-step algorithm that relies on the traditional Latent Dirichlet allocation algorithm to estimate the topics. The different latent variables are fitted as follows:
\begin{itemize}
    \item {\it Step 1: Estimation of $\mA^{(3)}$ through LDA}. The first step consists in the estimation of the topic matrix $\mA^{(3)}$ using the matricized matrix $\mD^{(3)} = \mathcal{M}_3(\D) \in \R^{R \times N_1N_2}$. To this end, we use the traditional Latent Dirichlet Allocation model, and recover a posterior distribution for the topic matrix $\mA^{(3)}$ as well as the topic proportions $\mW^{(3)}$. We take $\hat{\mA}^{(3)}$ to be the posterior mean of the LDA topic matrix. We note that this step is performed blindly of any existing correlation between documents corresponding to the same entry in mode 1 or 2 (i.e. for each $i$ and each $j$, the rows of the matrices  $\D_{i\star \star}$ and  $\D_{\star j \star}$  are treated independently). 
    \item {\it Step 2: Estimation of $\mA^{(1)}, \mA^{(2)}$}. Having estimated $\hat{\mA}^{(3)}$, we then estimate $\mA^{(1)}$ and $\mA^{(2)}$ from the estimated mixture matrix $\mW^{(3)} \in \R^{K_3 \times N_1 N_2}$.
    \item {\it Step 3: Estimation of the core $\G$ through constrained regression.} Having estimated all the latent factor matrices, we propose estimating the core through the following estimation procedure:
 \begin{equation}
     \begin{split}
         \hat{\G} \in \text{argmin} \| \mathcal{Y} - \G \cdot (\hat{\mA}_1, \hat{\mA}_2, \hat{\mA}_3)\|^2
     \end{split}
 \end{equation}
    
\end{itemize}

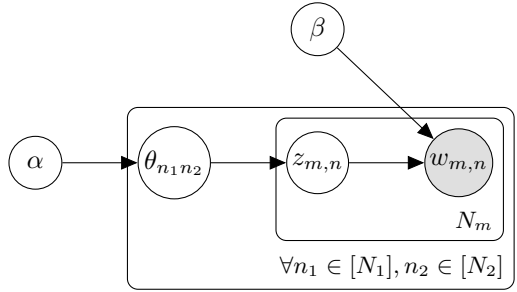
\begin{figure}[H]
    \centering
\begin{tikzpicture}

  \node[latent] (alpha) {\(\alpha\)};
  \node[latent, right=of alpha] (theta) {\(\theta_{n_1n_2}\)};
  \node[latent, right=of theta] (z) {\(z_{m,n}\)};
  \node[obs, right=of z] (w) {\(w_{m,n}\)};
  \node[latent, above=of z] (beta) {\(\beta\)};
  
  \plate {plate1} {(z)(w)} {\(N_m\)};
  \plate {plate2} {(theta)(plate1)} {\(\forall n_1 \in [N_1], n_2\in [N_2]\)};
  
  \edge {alpha} {theta};
  \edge {theta} {z};
  \edge {z} {w};
  \edge {beta} {w};

\end{tikzpicture}
    \caption{Plate Model for the LDA model fitted in Step 1 of the Hybrid LDA approach}
    \label{fig:lda}
\end{figure}

\section{Synthetic experiments: additional results}\label{app:synthetic}
This appendix supplements the synthetic experiments presented in the main text.

\begin{figure}[H]
    \centering
    \includegraphics[width=\textwidth]{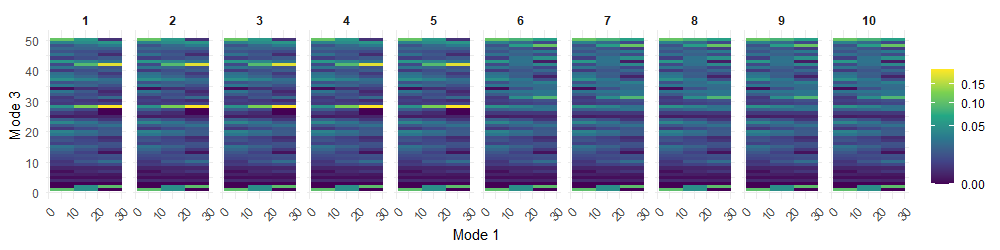}
    \caption{Nonnegative $30\times 10\times 50$ (ground-truth) tensor $\mathcal{D}$ where the slices are taken along the second mode. 
    }
    \label{fig:setting1_D_mixed}
\end{figure}
\noindent

\subsection{Tucker decomposition: Illustration and Comparison}\label{app:syn:tucker}
In this section, we provide a detailed discussion and illustration of why Tucker decomposition is our preferred approach for tensor topic modeling in multivariate data applications, rather than matrix decomposition or the {\it CANDECOMP/PARAFAC} (CP) decomposition. Using the tensor data shown in Figure \ref{fig:setting1_D_mixed}, we compare our Tucker decomposition-based method with Nonnegative CP decomposition implemented in the \texttt{rTensor} package and structural topic modeling (STM) utilizing supervised matrix decomposition from the \texttt{stm} package.
The two key questions we aim to address are:  
(i) Can we identify the presence of two distinct core groups along mode 1 (the document axis)?  
(ii) Can we detect the topic shift within the mixed membership group between time periods 5 and 6 along mode 2 (the time axis)?

\par To address the questions above, we first apply our method to the entire tensor using an input rank of \((2, 2, 3)\), obtaining the Tucker factor matrices \(\mA^{(1)}, \mA^{(2)}, \mA^{(3)}\) and the core tensor \(\mathcal{G}\). The results for this example are presented in Figure \ref{appfig:sim:ours}. 

The factor matrix \(\mA^{(1)}\) effectively identifies two primary groups along the first mode (document axis): group 1 and group 2, corresponding to the first and last ten indices of the first dimension, respectively. Additionally, the middle 10 rows of \(\mA^{(1)}\) represent documents with mixed membership between the two groups. The matrix \(\mA^{(2)}\) captures interactions reflecting topic evolution along the second dimension (e.g., time). 

The core tensor, illustrated in Figure \ref{appfig:sim:ours}, reveals interactions between topics and latent clusters across dimensions. Notably, the topic proportions for document cluster 1, especially those aligned with topic 2, remain stable despite a significant event occurring between time periods 5 and 6. In contrast, documents in cluster 2 exhibit a marked shift in preferences during this period, transitioning from topic 3 to topic 1.

Tucker decomposition offers greater flexibility by allowing different ranks for each mode of the tensor, making it well-suited for datasets with varying levels of complexity across modes such as customer segments, time periods, and product categories. In this synthetic dataset, Tucker decomposition effectively captures these diverse structures. 

In contrast, CP decomposition enforces the same rank across all modes, which can be overly restrictive in scenarios where document clusters (first mode), time (second mode), and product categories (third mode) exhibit distinct underlying patterns. For example, in NNCPD, the matrix \(\mA^{(1)}\) represents clusters associations, \(\mA^{(2)}\) captures latent time clusters, and \(\mA^{(3)}\) reflects word-topic associations. While Figure \ref{app:subplot:NNCPD 2} shows that topic 2 is associated primarily with cluster 1 during time periods 6–10 and topic 1 is linked to cluster 2 in time periods 1–5, it fails to clearly isolate the topic change occurring exclusively within the mixed membership group between time periods 5 and 6. Increasing the rank to 4, as shown in Figure \ref{app:subplot:NNCPD 4}, does not resolve this limitation, underscoring that higher ranks alone cannot fully capture such nuanced changes, and highlighting the need for more sophisticated decompositions that incorporate additional interactions.

Matrix decomposition is inherently limited to two modes (e.g., customers at a given time and products), which fails to preserve the multidimensional relationships in the data. Post-processing is often required to extract interactions between modes, but this approach introduces challenges and may not always yield meaningful insights. Figure \ref{app:subplot:LDA} illustrates results from standard Latent Dirichlet Allocation (LDA) \cite{blei2003latent}, while Figure \ref{app:subplot:STM} presents supervised LDA results from Structural Topic Modeling (STM). The document-topic matrix \(\mW^{(3)}\) in STM demonstrates clearer stripe patterns, reflecting the benefits of incorporating structural information like time stamps and document groups, leading to more accurate reconstructions. However, proper post-processing of \(\mW^{(3)}\) is essential for analyzing interactions between observed changes and customer groups. Figure \ref{app:subplot:LDA slice} shows a failed case where Hybrid-LDA’s core tensor illustrates topic changes over time but fails to capture the transition between time points 5 and 6 in \(\mA^{(2)}\).

By contrast, STM effectively identifies the timing of the event and aligns its core tensor with the expected model behavior, albeit with slightly higher reconstruction error. While STM is robust, as it employs multiple EM algorithm runs to ensure convergence, it requires external information, limiting its applicability in real-world scenarios where such data might be unavailable. In comparison, our method is unsupervised, offering greater flexibility and applicability without relying on additional inputs. For further details on its application to real-world datasets, we refer readers to Section \ref{sec: real data experiments}.
\begin{figure}[H]
    \centering
    \begin{subfigure}[b]{0.10\textwidth}
    \includegraphics[width=\textwidth]{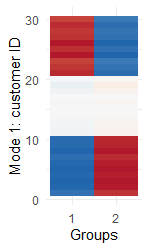}
    \end{subfigure}
    \begin{subfigure}[b]{0.10\textwidth}
    \includegraphics[width=\textwidth]{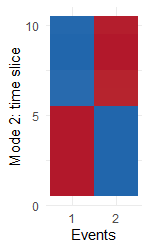}
    \end{subfigure}
    \begin{subfigure}[b]{0.10\textwidth}
    \includegraphics[width=\textwidth]{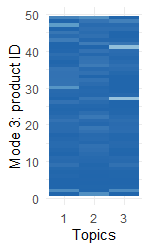}
    \end{subfigure}
    \begin{subfigure}[b]{0.22\textwidth}
    \includegraphics[width=\textwidth]{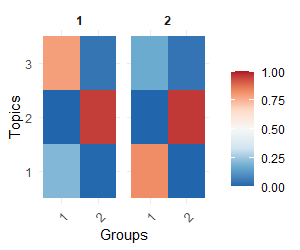}
    \end{subfigure}
    \caption{Tucker components $\mA^{(1)},\mA^{(2)},\mA^{(3)},\mathcal{G}$ (Left to Right respectively) estimated by our method with ranks $(2,2,3)$ for the tensor from Figure \ref{fig:setting1_D_mixed}. Notice that the factor $\mA^{(1)}$ shows the latent groups of documents. The $\mA^{(2)}$ factor shows two events across time slices . The $\mA^{(3)}$ factor shows the product-representation of topics. The core tensor $\mathcal{G}$ showcases the interactions between multiple modes. Each slice in the last subfigure represents the events derived from $\mA^{(2)}$. The reconstruction error is $\|\mathcal{\hat D}-\mathcal{D}\|_1=22.687$ }
    \label{appfig:sim:ours}
\end{figure}


\begin{figure}[H]
    \centering
    \begin{subfigure}[b]{0.4\textwidth}\includegraphics[width=\textwidth]{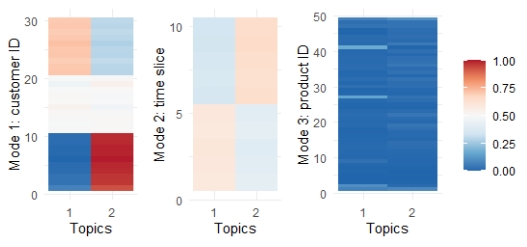}
    \subcaption{NNCPD factors with rank 2}\label{app:subplot:NNCPD 2}\end{subfigure}
    \hspace{1.0cm}
    \begin{subfigure}[b]{0.5\textwidth}\includegraphics[width=\textwidth]{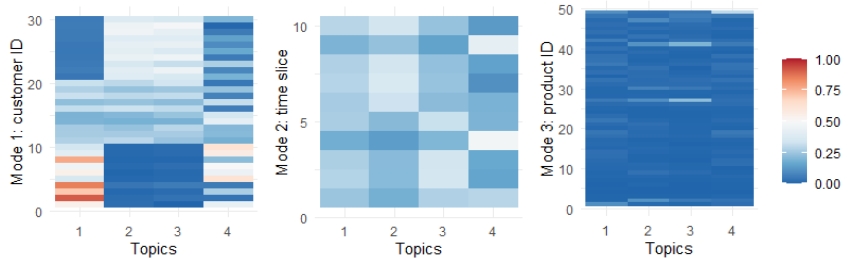}
    \subcaption{NNCPD factors with rank 4}\label{app:subplot:NNCPD 4}\end{subfigure}
    \caption{NNCPD components $\mA^{(1)},\mA^{(2)},\mA^{(3)}$ (Left to right, respectively) for both subplots (a) and (b). Components $\mA^{(1)},\mA^{(2)}$, and $\mA^{(3)}$ representing the factor matrices corresponding to the first (document), second (time), and third (word) modes, respectively. Plot (a) and Plot (b) differ based on the rank used in the decomposition. For rank 2 decomposition, the reconstruction error is $\|\mathcal{\hat D}-\mathcal{D}\|_1=251.540$. For rank 4 decomposition, the reconstruction error slightly increases $\|\mathcal{\hat D}-\mathcal{D}\|_1=253.581$.}
    \label{fig:sim: NNCPD}
\end{figure}

\begin{figure}[H]
    \centering
    \begin{subfigure}[b]{0.25\textwidth}\includegraphics[width=\textwidth]{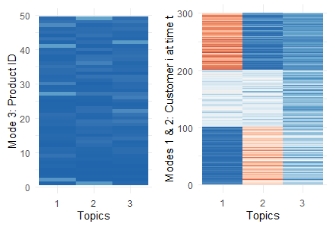}
    \subcaption{LDA factors}\label{app:subplot:LDA}\end{subfigure}
    \hspace{1.7cm}
    \begin{subfigure}[b]{0.3\textwidth}\includegraphics[width=\textwidth]{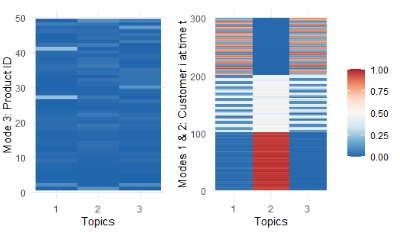}
    \subcaption{STM factors}\label{app:subplot:STM}\end{subfigure}
    \caption{The Matrix LDA components $\mA^{(3)}$ and $\mW^{(3)\top}$ are shown from left to right in both subplots (a) and (b). Here, $\mA^{(3)}$ represents the word-topic matrix, while $\mW^{(3)}$ denotes the topic proportions for each document at different time points. Plot (a) illustrates the results from regular LDA, where the topic distributions are generated without considering external covariates. Plot (b) shows the results from supervised LDA using structural LDA (STM), which incorporates covariate information, such as indicator variables for time stamps and documents. For LDA, the reconstruction error is $\|\mathcal{\hat D}-\mathcal{D}\|_1=82.703$.  After applying STM, the reconstruction error significantly decreases to  $\|\mathcal{\hat D}-\mathcal{D}\|_1=24.504$.}
    \label{fig:sim: matrix lda}
\end{figure}

\begin{figure}[H]
    \centering
    \begin{subfigure}[b]{0.4\textwidth}\includegraphics[width=\textwidth]{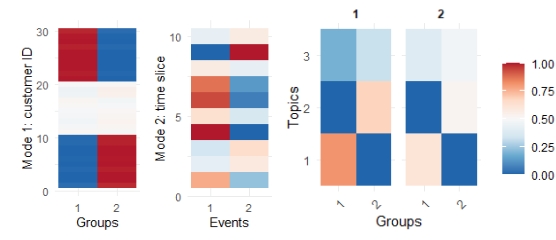}
    \subcaption{LDA post-processing factors}\label{app:subplot:LDA slice}\end{subfigure}
    \hspace{1.2cm}
    \begin{subfigure}[b]{0.4\textwidth}\includegraphics[width=\textwidth]{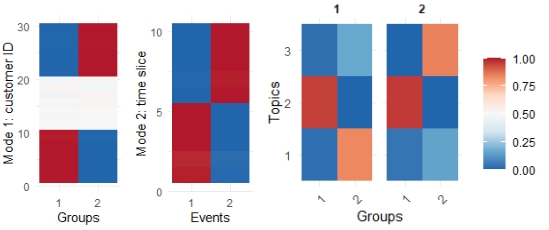}
    \subcaption{STM post-processing factors}\label{app:subplot:STM SLICE}\end{subfigure}
    \caption{The post-processing 
 Matrix LDA components $\mA^{(1)}$ $\mA^{(2)}$ and $\mathcal{G}$ from matrix $\mW^{(3)}$ given from Figure \ref{fig:sim: matrix lda} (shown from left to right in both subplots (a) and (b). Here, $\mW^{(3)}=\mathcal{M}_3(\mathcal{G})\left(\mA^{(1)}\otimes \mA^{(2)}\right)^\top$. The columns in the plots need to be permuted to facilitate a direct comparison. }
    \label{fig:sim: matrix lda post-process}
\end{figure}
\begin{figure}[H]
    \centering
    \includegraphics[width=0.5\linewidth]{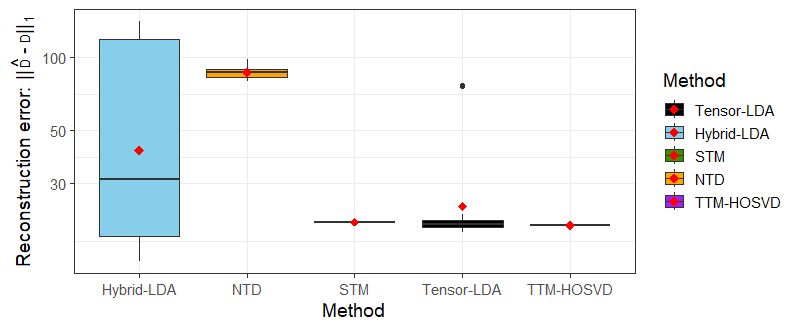}
    \caption{The Box plot of the reconstruction errors over 30 repeated runs on the same dataset from Figure \ref{fig:setting1_D_mixed}. The middle line in the plot represents the median value over 30 runs, while the red point indicates the mean value. STM is a robust method but relies on multiple iterations of the EM (Expectation-Maximization) algorithm—here, we set EM = 20. In this data example, the reconstruction error for STM is slightly higher than ours, and less computationally efficient.}
    \label{fig:sim1_reconstructionerror all}
\end{figure}

\section{Real experiments: additional results}\label{app:real data}
This appendix supplements the synthetic experiments presented in the main text. 
\subsection{Arvix's Paper}\label{appsec:arvix}
\paragraph{Preprocessing:}
The raw repository consists of 1.7 million articles from the vast branches of physics to the many subdisciplines of computer science to everything in between, including math, statistics, electrical engineering, quantitative biology, and economics. The metadata included information such as paper identifiers, titles, abstracts, publication dates, categories, and author information. Each paper in the arXiv repository is assigned one or more categories that follow a hierarchical structure, typically in the format $\text{main}\_\text{category}$.subcategory (e.g., cs.AI for Artificial Intelligence in computer science). For our analysis, we pre-processed the arXiv dataset by randomly selecting up to 500 article abstracts from four distinct categories: mathematics, physics, computer science, and statistics. We then categorized these articles into two main clusters: theoretical and application. The mapping was based on the nature of the research topics.
\begin{itemize}
    \item The theoretical categories included, but were not limited to:
    \begin{itemize}
        \item 
Mathematics: Algebraic Geometry (math.AG), Analysis of PDEs (math.AP), and Combinatorics (math.CO)
\item Statistics: Statistics 
 Theory (stat.TH).
\item Computer Science: Computational Complexity (cs.CC), Formal Languages and Automata Theory (cs.FL), and Logic in Computer Science (cs.LO).
\item Physics: High Energy Physics - Theory (hep-th), General Relativity and Quantum Cosmology (gr-qc), and Quantum Physics (quant-ph).
    \end{itemize}
    \item The application categories encompassed areas such as:
\begin{itemize}
    
\item Computer Science Applications: Artificial Intelligence (cs.AI), Computer Vision and Pattern Recognition (cs.CV), Machine Learning (cs.LG), and Robotics (cs.RO).
\item Applied Physics and Engineering: Applied Physics (physics.app-ph), Medical Physics (physics.med-ph).
\item Statistics: Statistics Applications (stat.AP), Computation (stat.CO), and Machine Learning (stat.ML).
\end{itemize}
\end{itemize}
The selected data has been summarized in the following two tables:
\begin{table}[H]
\centering
\begin{tabular}{|l|c|c|c|c|}
\hline
Category     & CS   & Math & Physics & Stat \\ \hline
Application  & 1505 & 109  & 119     & 150  \\ \hline
Theoretical  & 332  & 2160 & 702       & 27   \\ \hline
\end{tabular}
\caption{Distribution of Application and Theoretical papers by fields.}
\end{table}

\begin{table}[H]
\centering
\begin{tabular}{|l|c|c|c|c|c|}
\hline
Year & CS  & Math & Physics & Stat & Total \\ \hline
2007 & 31  & 162  & 35      & 0    & 228   \\ \hline
2008 & 20  & 111  & 32      & 2    & 165   \\ \hline
2009 & 10  & 45   & 19      & 0    & 74    \\ \hline
2010 & 37  & 143  & 44      & 8    & 232   \\ \hline
2011 & 44  & 138  & 38      & 5    & 225   \\ \hline
2012 & 73  & 176  & 60      & 7    & 316   \\ \hline
2013 & 87  & 161  & 47      & 12   & 307   \\ \hline
2014 & 57  & 160  & 53      & 12   & 282   \\ \hline
2015 & 51  & 91   & 44      & 2    & 188   \\ \hline
2016 & 93  & 121  & 46      & 12   & 272   \\ \hline
2017 & 110 & 132  & 53      & 14   & 309   \\ \hline
2018 & 129 & 139  & 51      & 10   & 329   \\ \hline
2019 & 149 & 126  & 46      & 16   & 337   \\ \hline
2020 & 162 & 111  & 63      & 17   & 353   \\ \hline
2021 & 177 & 105  & 53      & 17   & 352   \\ \hline
2022 & 198 & 115  & 42      & 20   & 375   \\ \hline
2023 & 200 & 114  & 44      & 13   & 371   \\ \hline
2024 & 209 & 119  & 41      & 10   & 379   \\ \hline
\end{tabular}
\caption{Yearly distribution of papers by subjects with totals.  For each year, we randomly selected 500 papers, with replacement. If a year contained fewer than 500 papers, we allowed sampling with replacement to reach the target number of samples.}\label{appsec:arxiv,table by year, sampling}
\end{table}
\begin{figure}[H]
    \centering
    \includegraphics[width=0.45\linewidth]{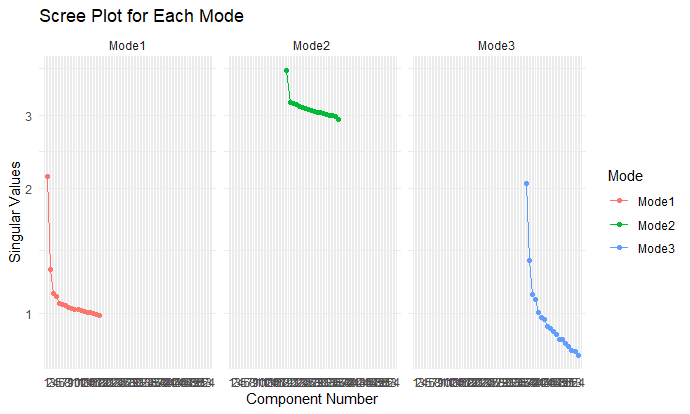}
    \caption{Scree plot of $\mD^{(a)}$ from TTM-HOSVD}
    \label{fig:articles: singular value}
\end{figure}

\begin{figure}[H]
    \centering
    \begin{subfigure}[b]{0.2 \textwidth}
    \includegraphics[width=\textwidth]{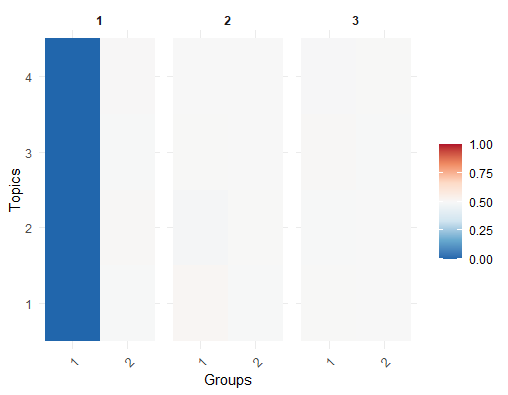}
    \subcaption{Hybrid-LDA}
    \end{subfigure}
    \begin{subfigure}[b]{0.2 \textwidth}
\includegraphics[width=\textwidth]{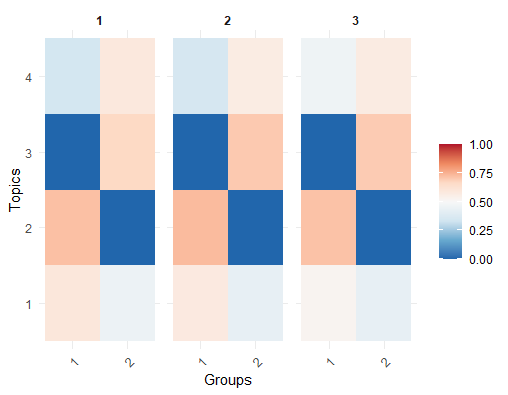} \subcaption{STM}
    \end{subfigure}
        \begin{subfigure}[b]{0.2\textwidth}
\includegraphics[width=\textwidth]{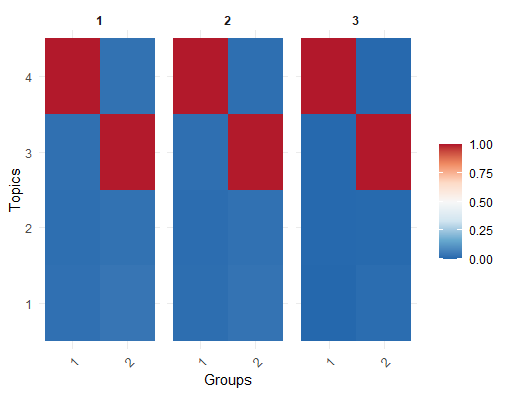} \subcaption{Tensor-LDA}
    \end{subfigure}\\
     \begin{subfigure}[b]{0.2 \textwidth}
\includegraphics[width=\textwidth]{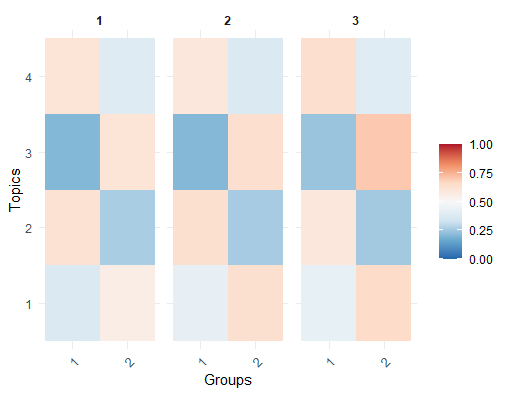} \subcaption{NTD}
    \end{subfigure}
    \begin{subfigure}[b]{0.2\textwidth}
\includegraphics[width=\textwidth]{IMG/real_data/arxiv/ours_core.png} \subcaption{TTM-HOSVD}
\end{subfigure}
    \caption{Core tensor components $\G$ derived from different methods.}
    \label{fig: arxiv topic others core}
\end{figure}

\begin{table}[H]\tiny
\centering
\begin{tabular}{|c|c|c|c|}
\hline
\textbf{Time Slice} & \textbf{[,1]} & \textbf{[,2]} & \textbf{[,3]} \\ \hline
\multicolumn{4}{|c|}{\textbf{Statistics}} \\ \hline
Theoretical Group & 0.11 & 0.08 & 0.16 \\ \hline
Application Group & 0.08 & 0.06 & 0.08 \\ \hline
\textbf{Total} & \textbf{0.19} & \textbf{0.14} & \textbf{0.24} \\ \hline
\multicolumn{4}{|c|}{\textbf{Physics}} \\ \hline
Theoretical Group & 0.14 & 0.16 & 0.15 \\ \hline
Application Group & 0.00 & 0.00 & 0.00 \\ \hline
\textbf{Total} & \textbf{0.14} & \textbf{0.16} & \textbf{0.15} \\ \hline
\multicolumn{4}{|c|}{\textbf{Computer Science}} \\ \hline
Theoretical Group & 0.00 & 0.00 & 0.00 \\ \hline
Application Group & 0.79 & 0.87 & 0.85 \\ \hline
\textbf{Total} & \textbf{0.79} & \textbf{0.87} & \textbf{0.85} \\ \hline
\multicolumn{4}{|c|}{\textbf{Math}} \\ \hline
Theoretical Group & 0.76 & 0.76 & 0.69 \\ \hline
Application Group & 0.13 & 0.07 & 0.07 \\ \hline
\textbf{Total} & \textbf{0.89} & \textbf{0.83} & \textbf{0.76} \\ \hline
\end{tabular}
\caption{Core-tensor Data (Rounded to 2 Decimal Places) derived from TTM-HOSVD}\label{arvix table: core: ours}
\end{table}

\begin{table}[H]\tiny
\centering

\begin{tabular}{|c|c|c|c|c|c|c|}
\hline
\multirow{2}{*}{\textbf{Time Slice}} & \multicolumn{3}{c|}{\textbf{TTM Method}} & \multicolumn{3}{c|}{\textbf{NTD Method}} \\ \cline{2-7}
                                     & \textbf{[,1]} & \textbf{[,2]} & \textbf{[,3]} & \textbf{[,1]} & \textbf{[,2]} & \textbf{[,3]} \\ \hline
\multicolumn{7}{|c|}{\textbf{Slice 1 (Statistics)}} \\ \hline
Theoretical Group                     & 0.11         & 0.08         & 0.16         & 0.14         & 0.18         & 0.19         \\ \hline
Application Group                     & 0.08         & 0.06         & 0.08         & 0.31         & 0.41         & 0.43         \\ \hline
\textbf{Total}                        & \textbf{0.19} & \textbf{0.14} & \textbf{0.24} & \textbf{0.45} & \textbf{0.59} & \textbf{0.62} \\ \hline
\multicolumn{7}{|c|}{\textbf{Slice 2 (Math)}} \\ \hline
Theoretical Group                     & 0.14         & 0.16         & 0.15         & 0.39         & 0.41         & 0.36         \\ \hline
Application Group                     & 0.00         & 0.00         & 0.00         & 0.07         & 0.07         & 0.07         \\ \hline
\textbf{Total}                        & \textbf{0.14} & \textbf{0.16} & \textbf{0.15} & \textbf{0.46} & \textbf{0.48} & \textbf{0.42} \\ \hline
\multicolumn{7}{|c|}{\textbf{Slice 3 (Computer Science)}} \\ \hline
Theoretical Group                     & 0.00         & 0.00         & 0.00         & 0.04         & 0.04         & 0.06         \\ \hline
Application Group                     & 0.79         & 0.87         & 0.85         & 0.37         & 0.41         & 0.50         \\ \hline
\textbf{Total}                        & \textbf{0.79} & \textbf{0.87} & \textbf{0.85} & \textbf{0.41} & \textbf{0.45} & \textbf{0.56} \\ \hline
\multicolumn{7}{|c|}{\textbf{Slice 4 (Physics)}} \\ \hline
Theoretical Group                     & 0.76         & 0.76         & 0.69         & 0.37         & 0.35         & 0.41         \\ \hline
Application Group                     & 0.13         & 0.07         & 0.07         & 0.15         & 0.14         & 0.16         \\ \hline
\textbf{Total}                        & \textbf{0.89} & \textbf{0.83} & \textbf{0.76} & \textbf{0.52} & \textbf{0.50} & \textbf{0.57} \\ \hline
\end{tabular}
\caption{Comparison of Core-Tensor Data (Rounded to 2 Decimal Places) for TTM and NTD Methods}
\end{table}

\begin{table}[H]\tiny
\centering

\begin{tabular}{|c|c|c|c|c|c|c|c|c|c|}
\hline
\multirow{2}{*}{\textbf{Time Slice}} & \multicolumn{3}{c|}{\textbf{STM Method}} & \multicolumn{3}{c|}{\textbf{LDA Method}} & \multicolumn{3}{c|}{\textbf{Tensor LDA Method}} \\ \cline{2-10}
                                     & \textbf{[,1]} & \textbf{[,2]} & \textbf{[,3]} & \textbf{[,1]} & \textbf{[,2]} & \textbf{[,3]} & \textbf{[,1]} & \textbf{[,2]} & \textbf{[,3]} \\ \hline
\multicolumn{10}{|c|}{\textbf{Slice 1 (Statistics)}} \\ \hline
Theoretical Group                     & 0.36         & 0.34         & 0.27         & 0.00         & 0.26         & 0.25         & 0.00         & 0.00         & 0.00         \\ \hline
Application Group                     & 0.21         & 0.18         & 0.19         & 0.25         & 0.25         & 0.25         & 0.00         & 0.00         & 0.00         \\ \hline
\textbf{Total}                        & \textbf{0.57} & \textbf{0.52} & \textbf{0.46} & \textbf{0.25} & \textbf{0.51} & \textbf{0.50} & \textbf{0.00} & \textbf{0.00} & \textbf{0.00} \\ \hline
\multicolumn{10}{|c|}{\textbf{Slice 2 (Math)}} \\ \hline
Theoretical Group                     & 0.52         & 0.54         & 0.52         & 0.00         & 0.23         & 0.25         & 0.00         & 0.00         & 0.00         \\ \hline
Application Group                     & 0.00         & 0.00         & 0.00         & 0.26         & 0.25         & 0.25         & 0.00         & 0.00         & 0.00         \\ \hline
\textbf{Total}                        & \textbf{0.52} & \textbf{0.54} & \textbf{0.52} & \textbf{0.26} & \textbf{0.49} & \textbf{0.50} & \textbf{0.00} & \textbf{0.00} & \textbf{0.00} \\ \hline
\multicolumn{10}{|c|}{\textbf{Slice 3 (Computer Science)}} \\ \hline
Theoretical Group                     & 0.00         & 0.00         & 0.00         & 0.00         & 0.25         & 0.26         & 0.00         & 0.00         & 0.00         \\ \hline
Application Group                     & 0.45         & 0.50         & 0.49         & 0.25         & 0.25         & 0.25         & 1.00         & 1.00         & 1.00         \\ \hline
\textbf{Total}                        & \textbf{0.45} & \textbf{0.50} & \textbf{0.49} & \textbf{0.25} & \textbf{0.50} & \textbf{0.51} & \textbf{1.00} & \textbf{1.00} & \textbf{1.00} \\ \hline
\multicolumn{10}{|c|}{\textbf{Slice 4 (Physics)}} \\ \hline
Theoretical Group                     & 0.12         & 0.13         & 0.21         & 0.00         & 0.25         & 0.24         & 1.00         & 1.00         & 1.00         \\ \hline
Application Group                     & 0.35         & 0.32         & 0.33         & 0.25         & 0.25         & 0.25         & 0.00         & 0.00         & 0.00         \\ \hline
\textbf{Total}                        & \textbf{0.47} & \textbf{0.45} & \textbf{0.54} & \textbf{0.25} & \textbf{0.50} & \textbf{0.50} & \textbf{1.00} & \textbf{1.00} & \textbf{1.00} \\ \hline
\end{tabular}
\caption{Comparison of Core-Tensor Data (Rounded to 2 Decimal Places) for STM, LDA, and Tensor LDA Methods}
\end{table}

\subsection{Sub-communities of the vaginal microbiota in non-pregnant women}\label{sec:app:real_exp:vaginal}
\paragraph{Additional results for (i):}
Here, we include the \texttt{alto} plot to visualize topic refinement and coherence, as well as topic resolution for STM (Structural Topic Model). However, STM performs the worst in both the topic refinement and topic resolution plots, suggesting that it places too much emphasis on enhancing relationships derived from external information. This over-reliance on external covariates makes STM overly sensitive to both the selection of data and changes in the number of topics, ultimately affecting its performance and robustness in recovering $\mA^{(3)}$.
\begin{figure}[H]
    \centering
  \begin{subfigure}[b]{0.45\textwidth}
    \includegraphics[width=\textwidth]{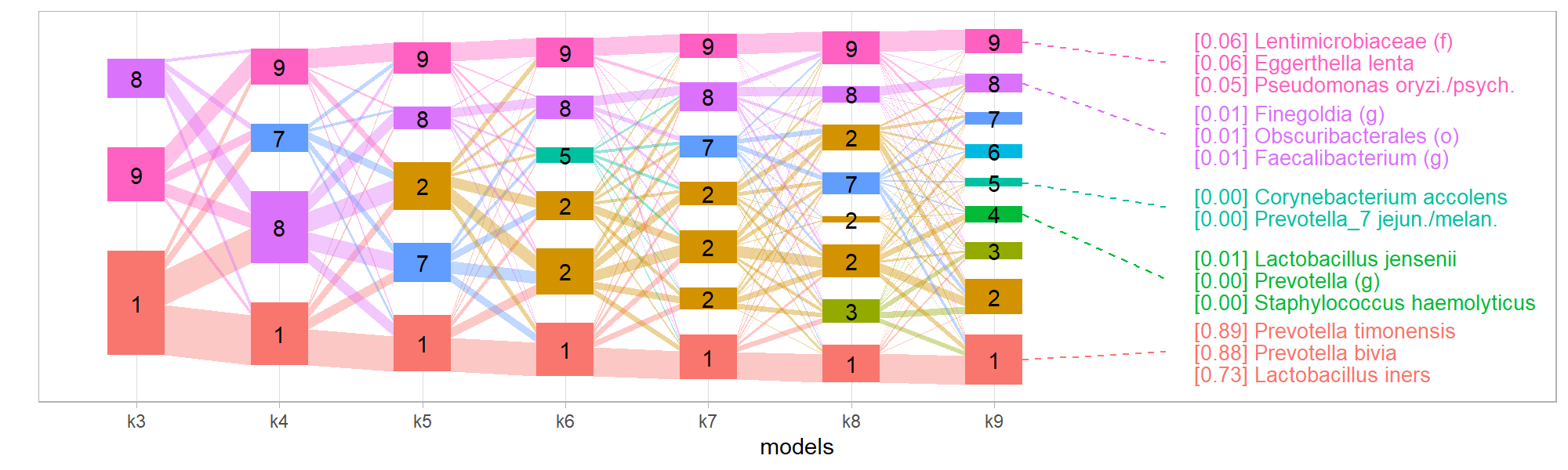}
    \subcaption{Topic refinement}\label{subplot:vaginal:refinement:stm}
    \end{subfigure}
    \hspace{0.3cm}
    \begin{subfigure}[b]{0.4\textwidth}
    \includegraphics[width=\textwidth]{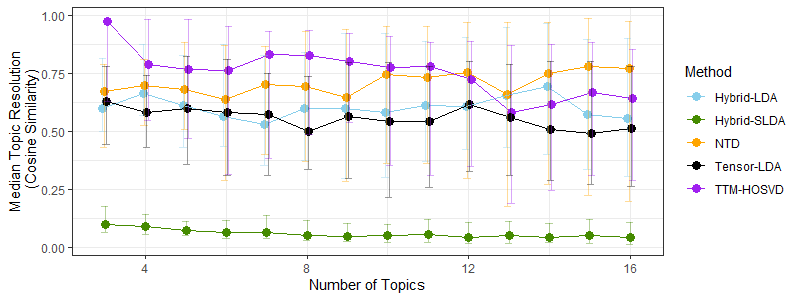}
\subcaption{Topic resolution}\label{subfig: vaginal: resolution:stm}
    \end{subfigure}
    \caption{Topic refinement and resolution as a function of $K^{(3)}$ on the vaginal ecosystem Data \cite{symul2023sub} given from STM. Vertical error bars represent the interquartile range for the average topic resolution scores over 25 trials.}
    \label{fig:vaginal(i)_stm}
\end{figure}

\paragraph{Additional results for (ii):}
Compared to the post-processing matrix methods, such as STM and Hybrid-LDA, the Tucker-based decomposition methods (as shown in Figure \ref{fig: dirichlet others method}), including our method (TTM-HOSVD), Tensor LDA, and NTD, show more significant points in the document-topic matrix $\mW^{(3)}$.  This indicates that the proportion of vaginal microbiota is strongly related to demographic characteristics and menstrual phases, and that these relationships could clearly identified through the menstrual phases and the group of subjects, respectively. 
In this data example, using Tucker decomposition facilitates the detection of this significant relationship, as it could further provides how different phases and subject characteristics influence the vaginal microbiota composition.

\par Upon closer examination of matrix component $\mA^{(2)}$, we observed that STM extracts more significant points than LDA by utilizing external information, as shown in Figures \ref{subplot:dirichlet STM:A2} and \ref{subplot:dirichlet LDA:A2}. Other methods, including ours, also perform well in Dirichlet regression on
$\mA^{(2)}$. However, when we further investigate the heatmap of proportions in $\mA^{(2)}$ in Figure \ref{fig:Heatmap:vaginal A2}, our method outperforms the others in recovering the theoretical menstrual phases. Both Hybrid-LDA and our method show similar patterns, but in Hybrid LDA, day `-6' appears closer to the phases on days `-7' and `-8', which should correspond to ovulation. In contrast, our method captures a mixed pattern on day `-6', reflecting a transition between the ovulation phase (days `-7' and `-8') and the luteal phase (days `-2' to `-5'), as expected. In comparison, STM provides a more mixed distinction between the luteal, menstrual, and follicular phases, making it harder to differentiate between these phases clearly. On the other hand, Tensor LDA tends to assign too much weight to a single phase, causing a disproportionate focus on one menstrual phase, which deviates from the expected patterns. NTD places more weight on the time after menstruation, and associates it more closely with the luteal phase instead of the follicular phase, which is contradiction to the phase progression.
\begin{figure}[H]
    \centering
     \begin{subfigure}[b]{0.285\textwidth}
    \includegraphics[width=\textwidth]{IMG/real_data/microbiota/TEST/OURS_A1_TEST.png}
    \subcaption{$\mA^{(1)}-$TTM-HOSVD} 
    \end{subfigure}
    \begin{subfigure}[b]{0.285\textwidth}
    \includegraphics[width=\textwidth]{IMG/real_data/microbiota/TEST/OURS_A2_TEST.png}
\subcaption{$\mA^{(2)}-$TTM-HOSVD}
    \end{subfigure}
    \begin{subfigure}[b]{0.38\textwidth}
\includegraphics[width=\textwidth]{IMG/real_data/microbiota/TEST/OURS_W_TEST.png}
    \subcaption{$\mW^{(3)}-$TTM-HOSVD}
    \end{subfigure}
    \begin{subfigure}[b]{0.285\textwidth}
    \includegraphics[width=\textwidth]{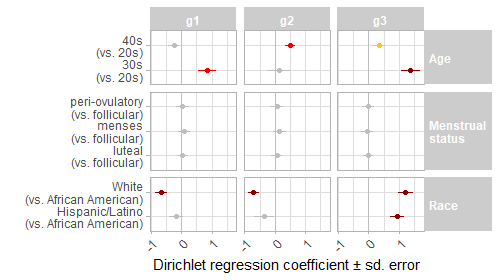}
    \subcaption{$\mA^{(1)}-$Tensor-LDA} \label{subplot:dirichlet TLDA:A1}
    \end{subfigure}
    \begin{subfigure}[b]{0.285\textwidth}
    \includegraphics[width=\textwidth]{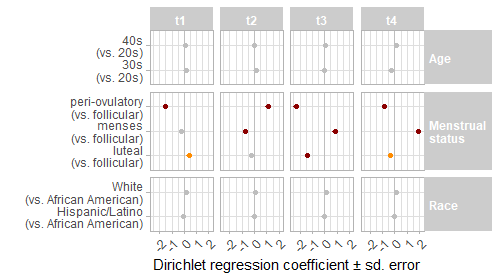}
\subcaption{$\mA^{(2)}-$Tensor-LDA}\label{subplot:dirichlet TLDA:A2}
    \end{subfigure}
    \begin{subfigure}[b]{0.38\textwidth}
\includegraphics[width=\textwidth]{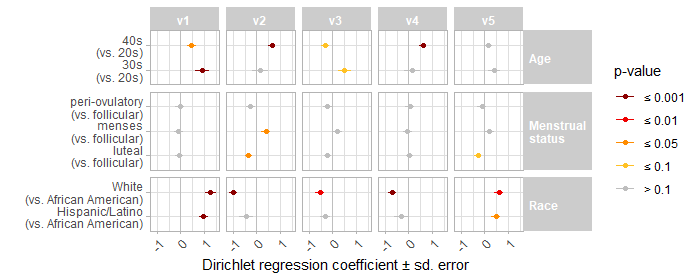}\label{subplot:dirichlet TLDA:W}
    \subcaption{$\mW^{(3)}-$Tensor-LDA}
    \end{subfigure}\\
    \begin{subfigure}[b]{0.285\textwidth}
    \includegraphics[width=\textwidth]{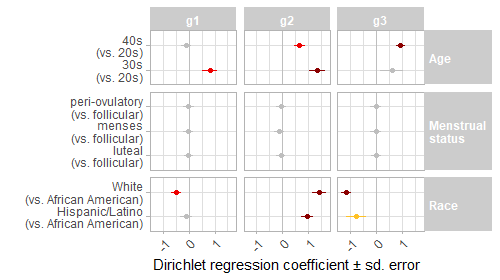}
    \subcaption{$\mA^{(1)}-$Hybrid-LDA} \label{subplot:dirichlet LDA:A1}
    \end{subfigure}
    \begin{subfigure}[b]{0.285\textwidth}
    \includegraphics[width=\textwidth]{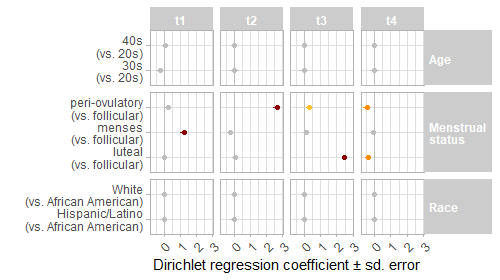}
\subcaption{$\mA^{(2)}-$Hybrid-LDA}\label{subplot:dirichlet LDA:A2}
    \end{subfigure}
    \begin{subfigure}[b]{0.38\textwidth}
\includegraphics[width=\textwidth]{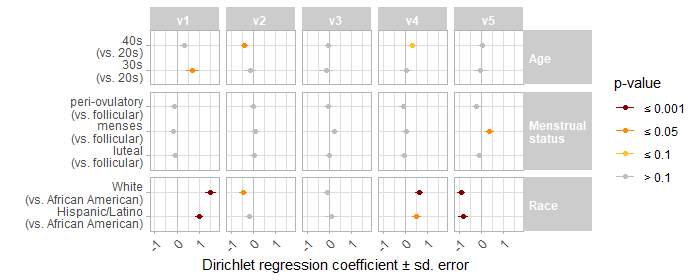}\label{subplot:dirichlet LDA:W}
    \subcaption{$\mW^{(3)}-$Hybrid-LDA}
    \end{subfigure}\\
    \begin{subfigure}[b]{0.285\textwidth}
    \includegraphics[width=\textwidth]{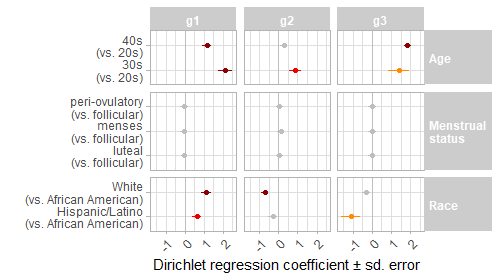}
    \subcaption{$\mA^{(1)}-$STM} \label{subplot:dirichlet STM:A1}
    \end{subfigure}
    \begin{subfigure}[b]{0.285\textwidth}
    \includegraphics[width=\textwidth]{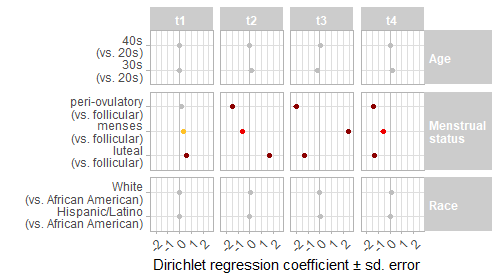}
\subcaption{$\mA^{(2)}-$STM}\label{subplot:dirichlet STM:A2}
    \end{subfigure}
    \begin{subfigure}[b]{0.38\textwidth}
\includegraphics[width=\textwidth]{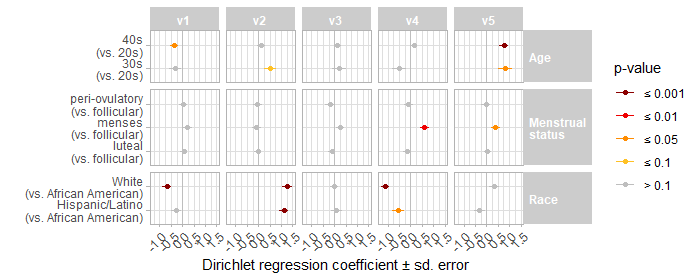}\label{subplot:dirichlet STM:W}
    \subcaption{$\mW^{(3)}-$STM}
    \end{subfigure}\\
    \begin{subfigure}[b]{0.285\textwidth}
    \includegraphics[width=\textwidth]{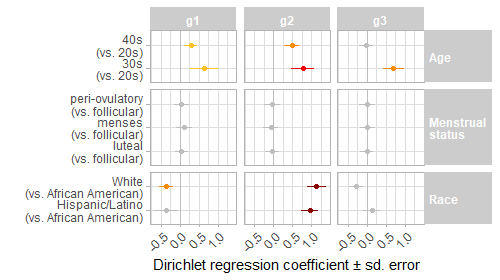}
    \subcaption{$\mA^{(1)}-$NTD} \label{subplot:dirichlet NTD:A1}
    \end{subfigure}
    \begin{subfigure}[b]{0.285\textwidth}
    \includegraphics[width=\textwidth]{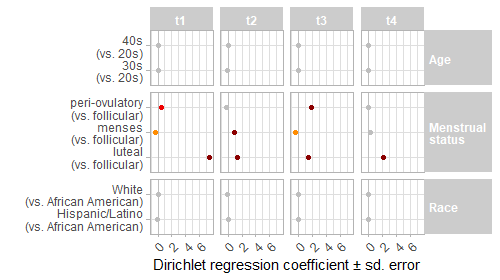}
\subcaption{$\mA^{(2)}-$NTD}\label{subplot:dirichlet NTD:A2}
    \end{subfigure}
    \begin{subfigure}[b]{0.38\textwidth}
\includegraphics[width=\textwidth]{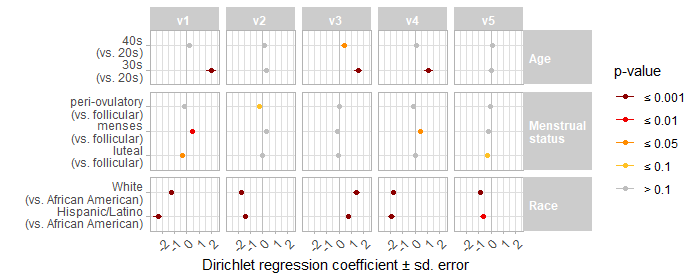}\label{subplot:dirichlet NTD:W}
    \subcaption{$\mW^{(3)}-$NTD}
    \end{subfigure}
    \caption{Dirichlet regression estimated coefficients (x-axis) quantifying the associations between race, menstrual status, age (y-axis) and categorical proportions (horizontal panels). Colours indicate the statistical significance. }
    \label{fig: dirichlet others method}
\end{figure}

\begin{figure}[H]
    \centering
     \begin{subfigure}[b]{0.285\textwidth}
    \includegraphics[width=\textwidth]{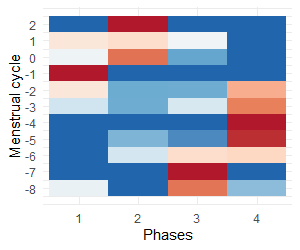}
    \subcaption{$\mA^{(2)}-$TTM-HOSVD} \label{subplot:vaginal: heatmap: ours:A2}
    \end{subfigure}
    \begin{subfigure}[b]{0.285\textwidth}
    \includegraphics[width=\textwidth]{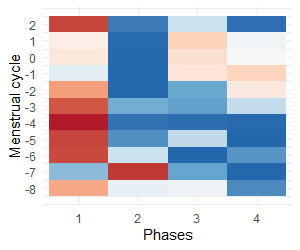}
\subcaption{$\mA^{(2)}-$Tensor-LDA}\label{subplot:vaginal: heatmap: tlda:A2}
    \end{subfigure}
    \begin{subfigure}[b]{0.38\textwidth}
\includegraphics[width=\textwidth]{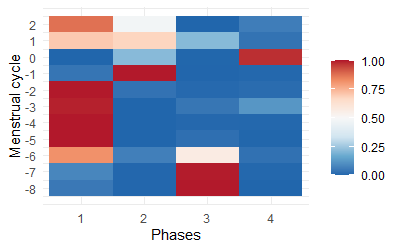}\label{subplot:vaginal: heatmap: NTD:A2}
    \subcaption{$\mA^{(2)}-$NTD}
    \end{subfigure}\\
    \begin{subfigure}[b]{0.285\textwidth}
\includegraphics[width=\textwidth]{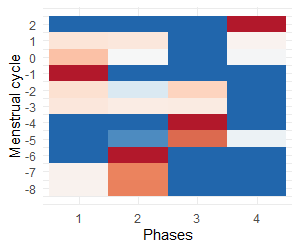}\label{subplot:vaginal: heatmap: lda:A2}
    \subcaption{$\mA^{(2)}-$Hybrid-LDA}
    \end{subfigure}
    \begin{subfigure}[b]{0.285\textwidth}
\includegraphics[width=\textwidth]{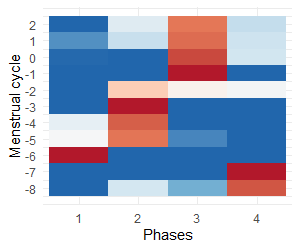}\label{subplot:vaginal: heatmap: stm:A2}
    \subcaption{$\mA^{(2)}-$STM}
    \end{subfigure}
     \begin{subfigure}[b]{0.35\textwidth}
\includegraphics[width=\textwidth]{IMG/real_data/microbiota/output.png}
    \subcaption{menstrual cycle}
    \end{subfigure}
    \caption{Matrix component $\mA^{(2)}$ derived from all methods.  The color scale has been sqrt-transformed. Remind that we select 3 days from the follicular phase, 2 days from the menstrual phase, 4 days from the luteal phase, and 2 days from ovulation. The columns in the plots need to be permuted to facilitate a direct comparison. }
    \label{fig:Heatmap:vaginal A2}
\end{figure}

\paragraph{Additional results for (iii)}From Figure \ref{fig: vaginal plot2}, in group 3 of subjects (which aligns with White) during time phase 1 (the theoretical follicular phase), microbiome topic 1 (dominated by L. gasseri) as characterized by Hybrid-LDA appears more active than microbiome topic 5 (dominated by L. iners). However, this pattern contradicts Figure \ref{fig: vaginal plot1}, where the proportion of topic 5 (v5) in White subjects from day 2 to day 6 is clearly higher than that of topic 1 (v1). It is also inconsistent with other Tucker-decomposition-based methods, which further suggests that the post-processing method can distort the results if the intermediate step of estimating $\mW^{(3)}$ is not sufficiently close to the ground truth. This highlights the importance of Tucker-decomposition-based methods for date with tensor structure. Additionally, Tensor-LDA produces the most sparse core proportion, as the Bayesian framework promotes sparse solutions through priors. In contrast, NTD results in the least sparse core, as KL minimization tends to distribute data more evenly across components, capturing a broader range of interactions. Our HOSVD-based method lies between the two. It captures the change in mixed-topic dynamics, allowing us to model both the prominent topics and the subtle transitions between them.

\begin{figure}[H]
    \centering
    \begin{subfigure}[b]{0.20\textwidth}
    \includegraphics[width=\textwidth]{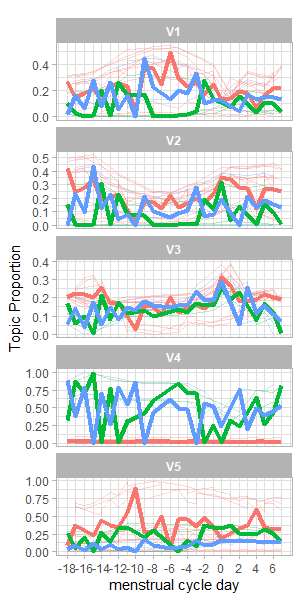}\subcaption{TTM-HOSVD}
    \end{subfigure}
    \begin{subfigure}[b]{0.20\textwidth}
    \includegraphics[width=\textwidth]{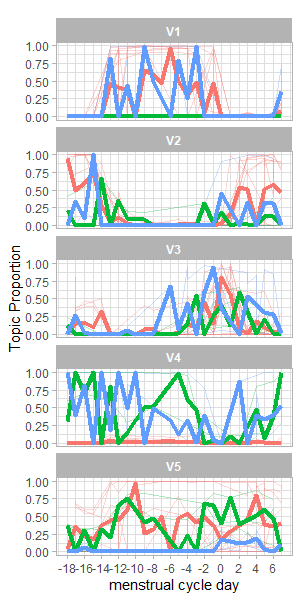}\subcaption{Hybrid-LDA}
    \end{subfigure}
    \begin{subfigure}[b]{0.20\textwidth}
    \includegraphics[width=\textwidth]{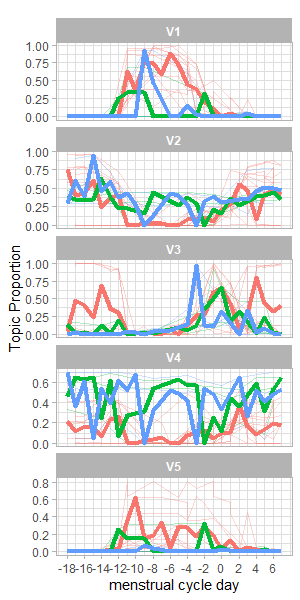}\subcaption{Tensor-LDA}
    \end{subfigure}
    \begin{subfigure}[b]{0.3\textwidth}
    \includegraphics[width=\textwidth]{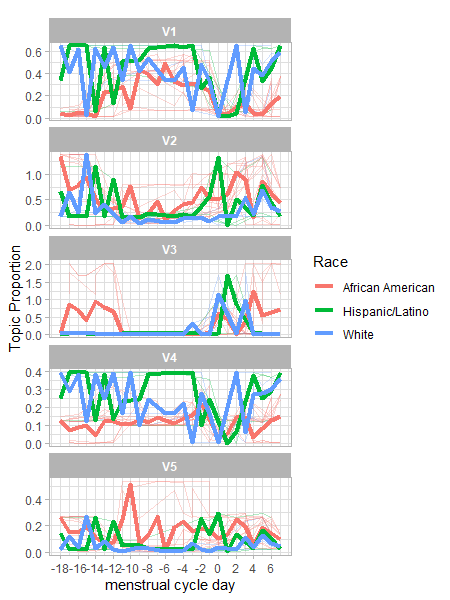}\subcaption{NTD}
    \end{subfigure}
    \caption{The topic proportions $\mW^{(3)}$ by different methods throughout the menstrual cycle}
    \label{fig: vaginal plot1}
\end{figure}

\begin{figure}[H]
    \centering
    \begin{subfigure}[b]{0.27\textwidth}
    \includegraphics[width=\textwidth]{IMG/real_data/microbiota/iii/vaginal_ours_core_algn.png}\subcaption{TTM-HOSVD}
    \end{subfigure}
    \begin{subfigure}[b]{0.195\textwidth}
    \includegraphics[width=\textwidth]{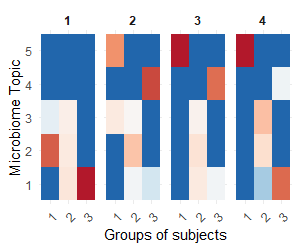}\subcaption{Hybrid-LDA}
    \end{subfigure}
    \begin{subfigure}[b]{0.21\textwidth}
    \includegraphics[width=\textwidth]{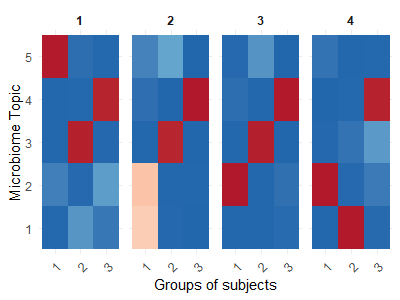}\subcaption{Tensor-LDA}
    \end{subfigure}
    \begin{subfigure}[b]{0.29\textwidth}
    \includegraphics[width=\textwidth]{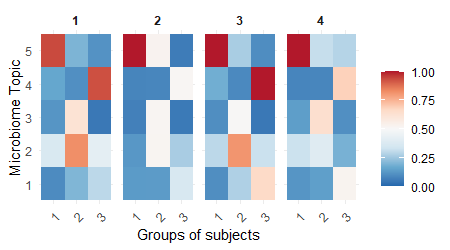}\subcaption{NTD}
    \end{subfigure}
    \caption{Core tensor obtained by different methods. Groups of subject align with the Race (Hispanic/Latino, African American, White). The microbiome topics align with that of our method (TTM-HOSVD).}
    \label{fig: vaginal plot2}
\end{figure}



\subsection{Market-basket analysis}\label{app:market}

\paragraph{Preprocessing} The original Dunnhumby dataset consists of a list of 2 million transactions corresponding to the shopping baskets of 2,500 households at a retailer over a period of 2 years.
We pre-process this dataset using the following steps.
\begin{itemize}
\item We process item names to remove content in parentheses (e.g. "tomatoes (vine)" becomes "tomatoes").
    \item We aggregate items according to their cleaned name and category, thereby eliminating all information pertaining quantities (e.g. ``canned tomatoes 16 oz" and ``canned tomatoes 32 oz" are both equivalent to ``tomatoes'').
    \item We discard items that appear less than 10 times in the dataset, as well as items that are too frequent to be informative (defined here as items occurring in 15\% of all baskets).
    \item We keep baskets that are big enough to determine topics (e.g. bigger than 1 purchase)
    \item We keep households who appear in the dataset at least once every two weeks. We aggregate baskets if they appear several times in a two-week window.
\end{itemize}

\begin{figure}[H]
\centering\includegraphics[width=0.75\linewidth]{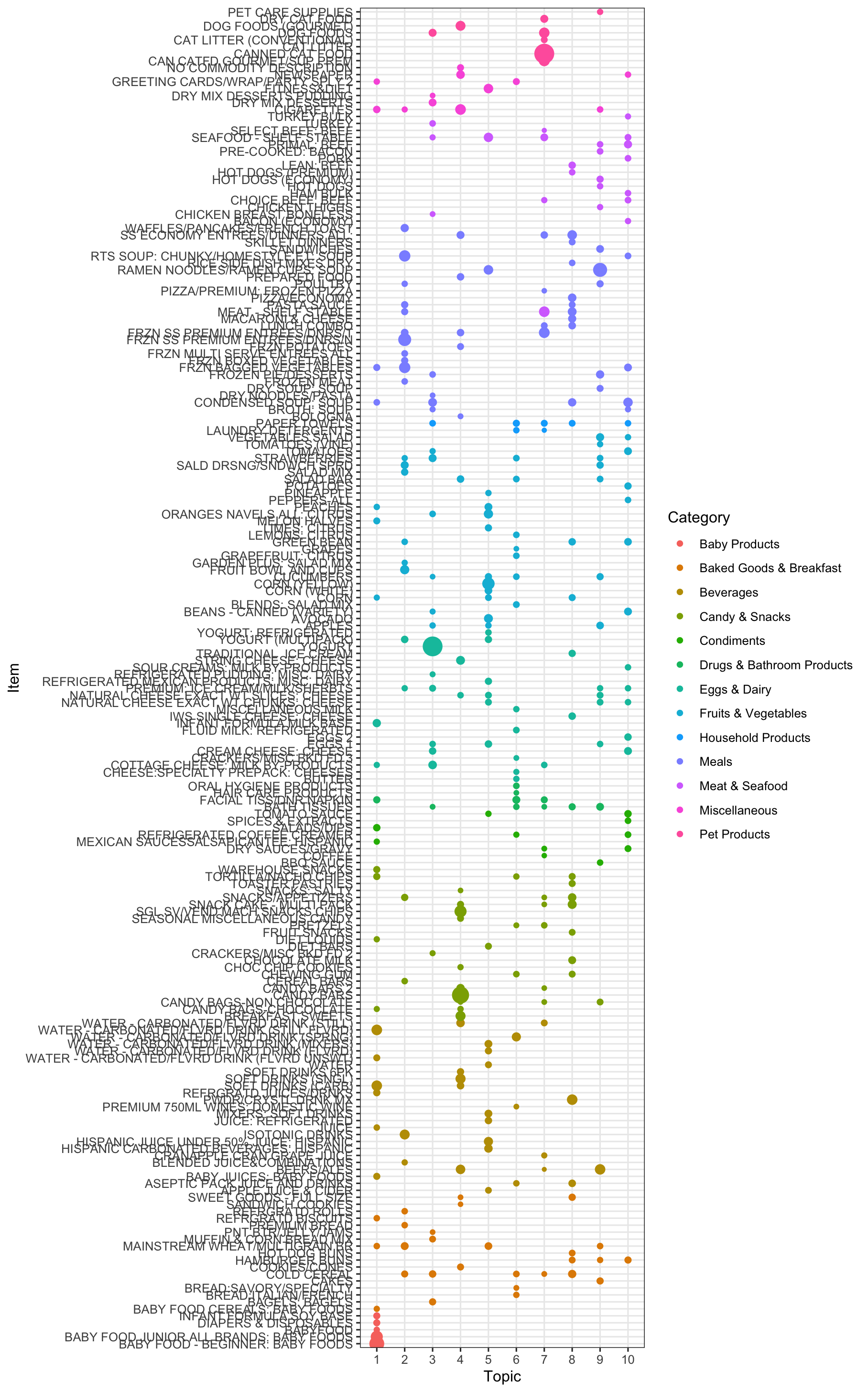}
    \caption{Market-basket topics recovered by STM.}
    \label{fig:market_topic_stm}
\end{figure}

\begin{figure}[H]
    \centering
\includegraphics[width=0.75\linewidth]{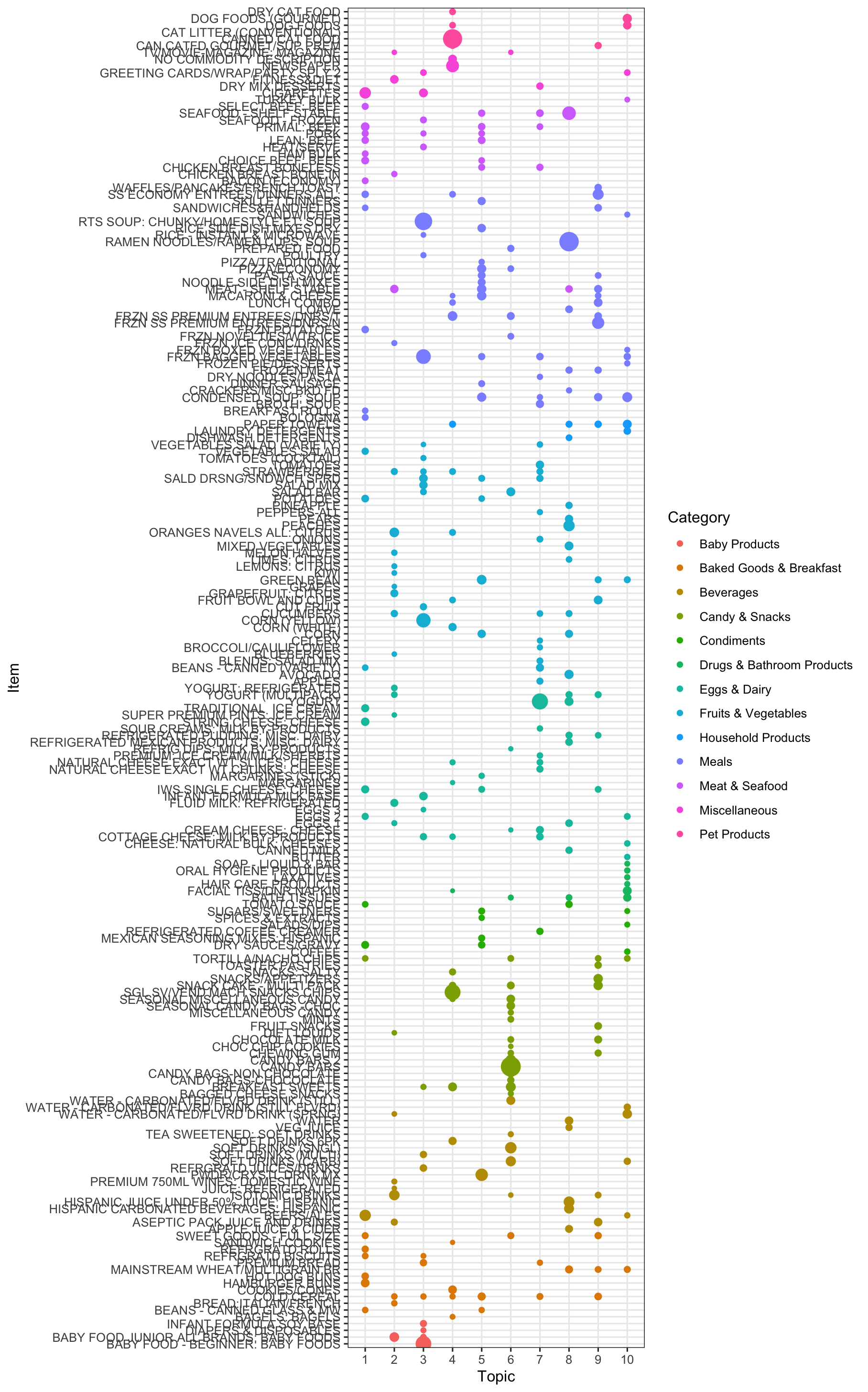}
    \caption{Topics recovered using Hybrid-LDA.}
    \label{fig:market_lda2}
\end{figure}

\begin{figure}[H]
    \centering
\includegraphics[width=0.8\linewidth]{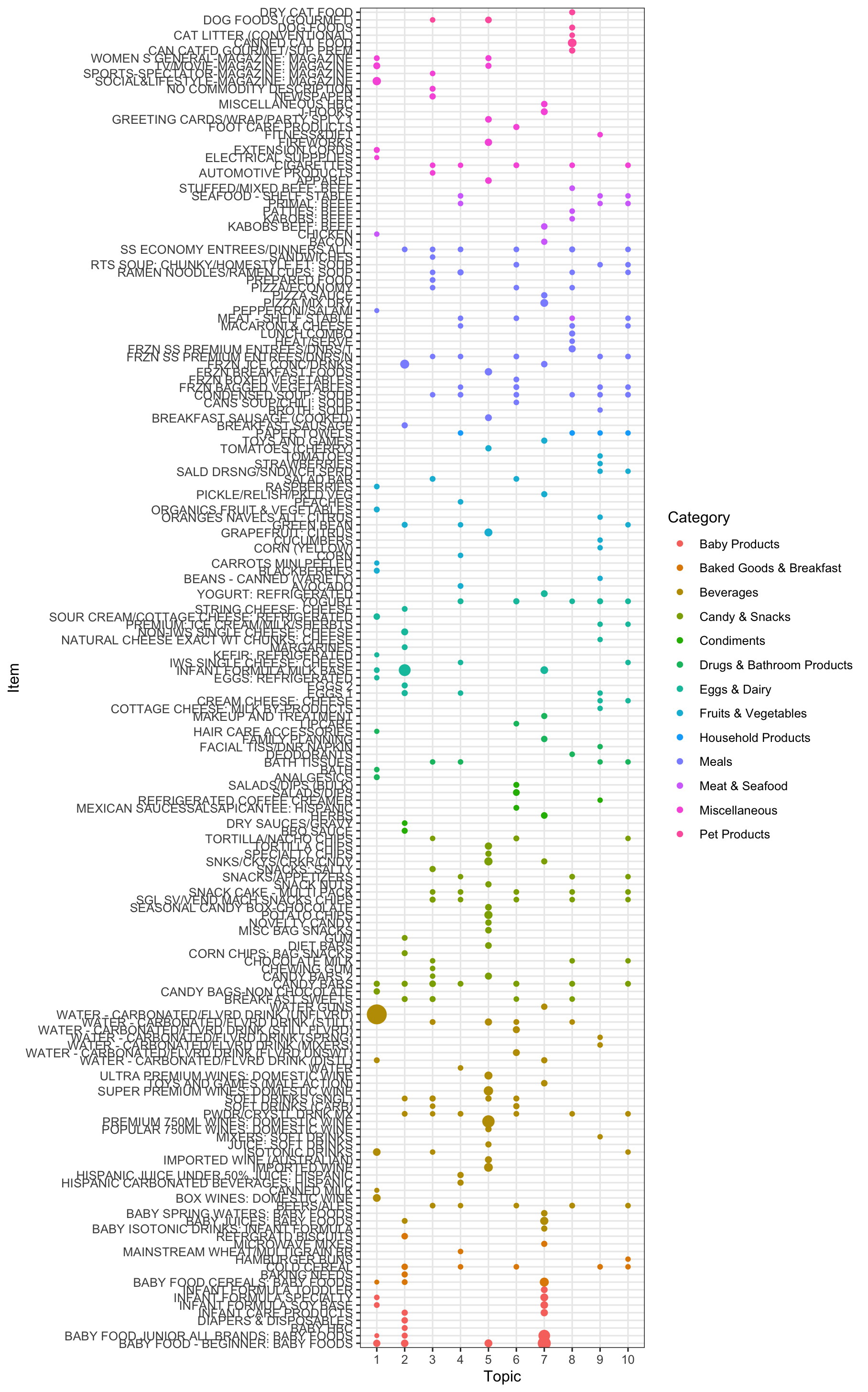}
    \caption{Topics recovered using Topic-Score HOSVD.}
    \label{fig:market-topics-topic-score hosvd}
\end{figure}

\begin{table}[H]
   \resizebox{\textwidth}{!}{ 
\begin{tabular}{|c|l|r|r|r|r|r|r|r|r|r|r|r|r|r|}
\hline
\textbf{Method} & \multicolumn{1}{c|}{\textbf{Topic}} & \multicolumn{1}{c|}{\textbf{\begin{tabular}[c]{@{}c@{}}Baby \\ Products\end{tabular}}} & \multicolumn{1}{c|}{\textbf{\begin{tabular}[c]{@{}c@{}}Baked \\ Goods \\ \& Breakfast\end{tabular}}} & \multicolumn{1}{c|}{\textbf{Beverages}} & \multicolumn{1}{c|}{\textbf{\begin{tabular}[c]{@{}c@{}}Candy \textbackslash\\ \& Snacks\end{tabular}}} & \multicolumn{1}{c|}{\textbf{Condiments}} & \multicolumn{1}{c|}{\textbf{\begin{tabular}[c]{@{}c@{}}Drugs \& \\ Bathroom \\ Products\end{tabular}}} & \multicolumn{1}{c|}{\textbf{Eggs \& Dairy}} & \multicolumn{1}{c|}{\textbf{\begin{tabular}[c]{@{}c@{}}Fruits \\ \& Vegetables\end{tabular}}} & \multicolumn{1}{c|}{\textbf{\begin{tabular}[c]{@{}c@{}}Household \\ Products\end{tabular}}} & \multicolumn{1}{c|}{\textbf{Meals}} & \multicolumn{1}{c|}{\textbf{\begin{tabular}[c]{@{}c@{}}Meat \\ \& Seafood\end{tabular}}} & \multicolumn{1}{c|}{\textbf{\begin{tabular}[c]{@{}c@{}}Miscell-\\ aneous\end{tabular}}} & \multicolumn{1}{c|}{\textbf{\begin{tabular}[c]{@{}c@{}}Pet \\ Products\end{tabular}}} \\ \hline\hline
\multirow{10}{*}{LDA} & 1 &  & 0.06 & 0.04 & 0.01 & 0.02 &  & 0.05 & 0.03 &  & 0.04 & \textbf{0.07} & 0.04 &  \\ \cline{2-15} 
 & 2 & 0.02 & 0.02 & 0.06 & 0.00 &  &  & 0.04 & \textbf{0.08} &  & 0.01 & 0.02 & 0.02 &  \\ \cline{2-15} 
 & 3 & 0.13 & 0.02 & 0.02 & 0.01 &  &  & 0.03 & 0.15 &  & \textbf{0.25} & 0.02 & 0.03 &  \\ \cline{2-15} 
 & 4 &  & 0.03 & 0.01 & \textbf{0.16} &  & 0.00 & 0.02 & 0.04 & 0.01 & 0.04 &  & 0.08 & 0.22 \\ \cline{2-15} 
 & 5 &  & 0.02 & 0.06 &  & 0.03 &  & 0.01 & 0.06 &  & \textbf{0.17} & 0.06 &  &  \\ \cline{2-15} 
 & 6 &  & 0.01 & 0.10 & \textbf{0.38} &  & 0.01 & 0.01 & 0.02 &  & 0.04 &  & 0.00 &  \\ \cline{2-15} 
 & 7 &  & 0.01 &  &  & 0.01 &  & \textbf{0.17} & 0.10 &  & 0.04 & 0.03 & 0.01 &  \\ \cline{2-15} 
 & 8 &  & 0.01 & 0.10 &  & 0.01 & 0.01 & 0.07 & 0.13 & 0.02 & \textbf{0.22} & 0.08 &  &  \\ \cline{2-15} 
 & 9 &  & 0.03 & 0.02 & 0.10 &  &  & 0.02 & 0.03 & 0.01 & \textbf{0.19} &  &  & 0.01 \\ \cline{2-15} 
 & 10 &  & 0.01 & 0.05 & 0.01 & 0.02 & \textbf{0.06} & 0.02 & 0.01 & 0.03 & 0.05 & 0.01 & 0.01 & 0.04 \\ \hline\hline
\multirow{10}{*}{STM} & 1 & \textbf{0.20} & 0.02 & 0.12 & 0.03 & 0.02 & 0.01 & 0.02 & 0.02 &  & 0.02 &  & 0.02 &  \\ \cline{2-15} 
 & 2 &  & 0.05 & 0.04 & 0.02 &  &  & 0.02 & 0.07 &  & \textbf{0.26} &  & 0.01 &  \\ \cline{2-15} 
 & 3 &  & 0.04 &  & 0.01 &  & 0.01 & \textbf{0.31} & 0.05 & 0.01 & 0.04 & 0.02 & 0.02 & 0.01 \\ \cline{2-15} 
 & 4 &  & 0.02 & 0.10 & 0.33 &  &  & 0.03 & 0.01 &  & 0.05 &  & 0.07 & 0.03 \\ \cline{2-15} 
 & 5 &  & 0.01 & 0.11 & 0.01 & 0.01 &  & 0.06 & \textbf{0.18 }&  & 0.03 & 0.03 & 0.03 &  \\ \cline{2-15} 
 & 6 &  & 0.02 & 0.04 & 0.02 & 0.01 & 0.03 & 0.03 & 0.05 & 0.02 &  &  & 0.01 &  \\ \cline{2-15} 
 & 7 &  & 0.01 & 0.02 & 0.03 & 0.01 & 0.01 & 0.01 &  & 0.01 & 0.07 & 0.06 &  & \textbf{0.37} \\ \cline{2-15} 
 & 8 &  & 0.04 & 0.05 & 0.09 &  & 0.01 & 0.02 & 0.02 & 0.01 & \textbf{0.14} & 0.02 &  &  \\ \cline{2-15} 
 & 9 &  & 0.02 & 0.04 & 0.01 & 0.01 & 0.01 & 0.03 & 0.07 &  & \textbf{0.15} & 0.04 & 0.01 & 0.01 \\ \cline{2-15} 
 & 10 &  & 0.01 &  &  & 0.04 &  & 0.05 &\textbf{ 0.06} & 0.01 & \bf 0.06 & \bf 0.06 & 0.01 &  \\ \hline\hline
\multirow{10}{*}{TTM-HOOI} & 1 &  & 0.01 & 0.07 & 0.04 & 0.01 & 0.01 & 0.01 & 0.01 & 0.00 & 0.06 & 0.03 & 0.01 &  \\ \cline{2-15} 
 & 2 &  & 0.01 & 0.04 & 0.07 &  & 0.01 & 0.03 &  & 0.01 & 0.24 & 0.01 &  & 0.01 \\ \cline{2-15} 
 & 3 &  & 0.03 & 0.00 & 0.02 &  & 0.01 & 0.03 & 0.02 & 0.01 & 0.09 & 0.01 & 0.01 &  \\ \cline{2-15} 
 & 4 & 0.40 & 0.01 & 0.06 & 0.04 &  & 0.01 & 0.04 & 0.02 &  & 0.06 &  & 0.01 & 0.02 \\ \cline{2-15} 
 & 5 & 0.01 & 0.02 & 0.02 & 0.04 &  &  & 0.04 & 0.02 &  & 0.13 & 0.01 &  &  \\ \cline{2-15} 
 & 6 &  & 0.01 & 0.13 & 0.08 &  & 0.00 & 0.06 & 0.01 &  & 0.18 & 0.01 & 0.08 & 0.04 \\ \cline{2-15} 
 & 7 &  & 0.02 & 0.01 & 0.05 &  & 0.01 & 0.04 & 0.01 & 0.01 & 0.04 & 0.01 &  &  \\ \cline{2-15} 
 & 8 &  & 0.01 & 0.06 & 0.20 &  & 0.01 & 0.02 & 0.01 & 0.00 & 0.04 & 0.00 & 0.02 & 0.00 \\ \cline{2-15} 
 & 9 &  & 0.00 & 0.02 & 0.04 & 0.00 & 0.01 & 0.01 & 0.00 & 0.01 & 0.10 & 0.01 & 0.00 & 0.23 \\ \cline{2-15} 
 & 10 & 0.01 & 0.02 & 0.01 & 0.02 &  & 0.01 & 0.16 & 0.03 & 0.01 & 0.04 & 0.01 &  &  \\ \hline \hline
\multirow{10}{*}{TTM-HOSVD} & 1 &  & 0.01 & \textbf{0.07} & 0.04 & 0.01 & 0.01 & 0.01 & 0.01 & 0.00 & 0.06 & 0.03 & 0.01 &  \\ \cline{2-15} 
 & 2 &  & 0.01 & 0.04 & 0.07 &  & 0.01 & 0.03 &  & 0.01 & \textbf{0.24} & 0.01 &  & 0.01 \\ \cline{2-15} 
 & 3 &  & 0.03 & 0.00 & 0.02 &  & 0.01 & 0.03 & 0.02 & 0.01 & \textbf{0.09} & 0.01 & 0.01 &  \\ \cline{2-15} 
 & 4 & \textbf{0.40} & 0.01 & 0.06 & 0.04 &  & 0.01 & 0.04 & 0.02 &  & 0.06 &  & 0.01 & 0.02 \\ \cline{2-15} 
 & 5 & 0.01 & 0.02 & 0.02 & 0.04 &  &  & 0.04 & 0.02 &  & \textbf{0.13} & 0.01 &  &  \\ \cline{2-15} 
 & 6 &  & 0.01 & 0.13 & 0.08 &  & 0.00 & 0.06 & 0.01 &  & \textbf{0.18} & 0.01 & 0.08 & 0.04 \\ \cline{2-15} 
 & 7 &  & 0.02 & 0.01 &\textbf{ 0.05} &  & 0.01 & 0.04 & 0.01 & 0.01 & 0.04 & 0.01 &  &  \\ \cline{2-15} 
 & 8 &  & 0.01 & 0.06 & 0\textbf{0.20} &  & 0.01 & 0.02 & 0.01 & 0.00 & 0.04 & 0.00 & 0.02 & 0.00 \\ \cline{2-15} 
 & 9 &  & 0.00 & 0.02 & 0.04 & 0.00 & 0.01 & 0.01 & 0.00 & 0.01 & 0.10 & 0.01 & 0.00 & \textbf{0.23} \\ \cline{2-15} 
 & 10 & 0.01 & 0.02 & 0.01 & 0.02 &  & 0.01 & \textbf{0.16} & 0.03 & 0.01 & 0.04 & 0.01 &  &  \\ \hline \hline
\multirow{10}{*}{TopicScore-HOSVD} & 1 & 0.03 & 0.00 & \textbf{0.83} & 0.02 &  & 0.01 & 0.02 & 0.01 &  & 0.00 & 0.00 & 0.07 &  \\ \cline{2-15} 
 & 2 & 0.05 & 0.03 & 0.01 & 0.02 & 0.01 &  & \textbf{0.20} & 0.00 &  & 0.08 &  &  &  \\ \cline{2-15} 
 & 3 &  &  & 0.02 & 0.04 &  & 0.00 &  & 0.00 &  & 0.02 &  & 0.02 & 0.00 \\ \cline{2-15} 
 & 4 &  & 0.00 & 0.02 & 0.01 &  & 0.00 & 0.01 & 0.01 & 0.00 & 0.02 & 0.01 & 0.00 &  \\ \cline{2-15} 
 & 5 & 0.03 &  & \textbf{0.40} & 0.17 &  &  &  & 0.04 &  & 0.04 &  & 0.06 & 0.01 \\ \cline{2-15} 
 & 6 &  & 0.00 & 0.06 & 0.02 & 0.02 & 0.00 & 0.01 & 0.00 &  & 0.03 &  & 0.01 &  \\ \cline{2-15} 
 & 7 & \textbf{0.47} & 0.07 & 0.08 & 0.01 & 0.01 & 0.02 & 0.04 & 0.01 &  & 0.06 & 0.02 & 0.02 &  \\ \cline{2-15} 
 & 8 &  &  & 0.01 & 0.02 &  & 0.00 & 0.00 &  & 0.00 & 0.04 & 0.01 & 0.00 & 0.07 \\ \cline{2-15} 
 & 9 &  & 0.00 & 0.01 &  & 0.00 & 0.00 & 0.02 & 0.02 & 0.00 & 0.01 & 0.00 & 0.00 &  \\ \cline{2-15} 
 & 10 &  & 0.00 & 0.01 & 0.02 &  & 0.00 & 0.01 & 0.00 & 0.00 & 0.02 & 0.01 & 0.00 &  \\ \hline
\end{tabular}
}
\caption{Results for the market basket analyss}\label{tab:market}
\end{table}



\end{document}